\documentclass[a4paper,10pt]{article} 
\usepackage{amsmath,latexsym,mathrsfs,bbold,
amsbsy,amsfonts,amssymb,amsthm,amscd,amsxtra,amstext,amsopn,dsfont,yfonts}
\usepackage[english]{babel}
\usepackage{bbm,bm,makeidx,psfrag}
\usepackage{indentfirst,fancyhdr,float}
\usepackage{graphicx,epsfig,hyperref}
\usepackage[T1]{fontenc}
\usepackage{calligra}
\DeclareMathAlphabet{\mathcalligra}{T1}{calligra}{m}{n}
\usepackage[T1]{fontenc}
\usepackage{humanist}
\usepackage{rotunda}
\usepackage[utf8]{inputenc}

\usepackage{latexsym}
\usepackage{amsfonts}
\usepackage{amsmath}
\usepackage{amsthm}
\usepackage{amssymb}
\usepackage{mathrsfs}
\usepackage{dsfont}    
\usepackage{bbold}     
\usepackage[english]{babel}
\usepackage{caption}
\usepackage{epsfig}
\usepackage{float}
\usepackage{psfrag}
\usepackage{graphicx}
\usepackage{color}
\usepackage{epsfig}
\usepackage{hyperref}

\newtheorem{hypo}{Hypothesis}[section]

\newtheorem{theorem}{Theorem}
\newtheorem{theo}{Theorem}[section]
\newtheorem{coro}[theorem]{Corollary}
\newtheorem{defi}[theorem]{Definition}
\newtheorem{lemma}[theorem]{Lemma}
\newtheorem{proposition}[theorem]{Proposition}
\newtheorem{rmk}[theorem]{Remark}
\newtheorem{const}[theo]{Constraint}

\newcommand{\zerarcounters}{\setcounter{equation}{0}}

\newcommand{\ZZZ}{\mathds{Z}}
\newcommand{\CCC}{\mathds{C}}
\newcommand{\NNN}{\mathds{N}}

\newcommand{\RRR}{\mathds{R}}
\newcommand{\TTT}{\mathds{T}}
\newcommand{\uno}{\mathds{1}}
\newcommand{\HH}{{\mathcal H}}

\newcommand{\calA}{{\mathcal A}}
\newcommand{\BB}{{\mathcal B}}
\newcommand{\CCCC}{{\mathcal C}}
\newcommand{\DD}{{\mathcal D}}
\newcommand{\calE}{{\mathcal E}}
\newcommand{\calF}{{\mathcal F}}

\newcommand{\calH}{{\mathcal H}}

\newcommand{\calL}{{\mathcal L}}
\newcommand{\MM}{{\mathcal M}}
\newcommand{\NN}{{\mathcal N}}
\newcommand{\calO}{{\mathcal O}}
\newcommand{\calP}{{\mathcal P}}

\newcommand{\RR}{{\mathcal R}}
\newcommand{\SSSS}{{\mathcal S}}
\newcommand{\TT}{{\mathcal T}}
\newcommand{\calU}{{\mathcal U}}
\newcommand{\VV}{{\mathcal V}}
\newcommand{\calW}{{\mathcal W}}
\newcommand{\calX}{{\mathcal X}}
\newcommand{\calY}{{\mathcal Y}}

\newcommand{\gota}{{\mathfrak a}}

\newcommand{\gotc}{{\mathfrak c}}

\newcommand{\gotm}{{\mathfrak m}}

\newcommand{\gotp}{{\mathfrak p}}

\newcommand{\gots}{{\mathfrak s}}

\newcommand{\gotA}{{\mathfrak A}}

\newcommand{\gotM}{{\mathfrak M}}
\newcommand{\gotN}{{\mathfrak N}}

\newcommand{\gotW}{{\mathfrak W}}

\newcommand{\ol}{\overline}

\newcommand{\db}{{\varrho}}
\newcommand{\io}{\infty}
\newcommand{\e}{\varepsilon}
\newcommand{\al}{\alpha}
\newcommand{\de}{\delta}

\newcommand{\uu}{{\bf u}}
\newcommand{\be}{\beta}

\newcommand{\n}{\nu}
\newcommand{\m}{\mu}
\newcommand{\x}{\xi}
\newcommand{\ta}{\mathtt{a}}
\newcommand{\tv}{\mathtt{v}}

\newcommand{\ka}{\kappa}
\newcommand{\g}{\gamma}

\newcommand{\h}{\eta}
\newcommand{\la}{\lambda}
\newcommand{\f}{\varphi}
\newcommand{\s}{\sigma}
\newcommand{\z}{\zeta}

\newcommand{\del}{\partial}

\newcommand{\av}[1]{\langle #1 \rangle}

\newcommand{\oo}{{\omega}}

\newcommand{\ff}{\mathtt{f}}

\newcommand{\cc}{\boldsymbol{c}}

\newcommand{\B}{\boldsymbol{B}}

\newcommand{\ii}{{\rm i}}
\newcommand{\tm}{\mathtt{m}}
\newcommand{\tV}{\mathtt{V}}

\newcommand{\tG}{{\mathtt G}}
\newcommand{\tR}{{\mathtt R}}

\def\leftinv#1{ {#1}^{-1}}

\def\tilde#1{\widetilde{#1}}

\def\ins#1#2#3{\vbox to0pt{\kern-#2 \hbox{\kern#1 #3}\vss}\nointerlineskip}

\makeindex

\begin{document}



\title{\bf KAM for quasi-linear autonomous NLS}
\author{\bf 
R. Feola$^{**}$, M. Procesi$^\dag $
\\
\small
${}^{**}$ SISSA, Trieste, rfeola@sissa.it; 
\\
\small
${}^\dag$ Universit\`a di Roma Tre, procesi@mat.uniroma3.it\footnote{
This research was partially supported by the European Research Council under
FP7 ``Hamiltonian PDEs and small divisor problems: a dynamical systems approach'' grant n. 306414-HamPDEs;
partially supported by PRIN 2012 ``Variational and perturbative aspects of nonlinear differential problems''.
}}


%


\date{}

\maketitle

\begin{abstract}
We consider a class of fully nonlinear Schr\"odinger equations
and we prove the existence and the stability of Cantor families of quasi-periodic, small amplitude solutions.
We deal with reversible \emph{autonomous} nonlinearities and we look for 
 \emph{analytic} solutions.  
Note that this is the first result on analytic quasi-periodic solutions for fully nonlinear PDEs.

\end{abstract}

\tableofcontents

\zerarcounters

\section{Introduction and Main Results}

In this paper we prove the existence of Cantor families of quasi-periodic solutions for the following \emph{fully-nonlinear} Schr\"odinger equation
\begin{equation}\label{6.1}
u_{t}=-i(u_{xx}+\ff(u,u_{x},u_{xx}))\,, \quad  x\in \TTT\,.
\end{equation}
The non linearity $\ff$  is reversible,   gauge preserving and $x$-independent  (see Hypothesis \ref{hyp2aut})
and has a zero of order three in $u=0$, i.e. 
\begin{equation}\label{6.2}
\begin{aligned}
\ff(u,u_{x},u_{xx})&:=\ff^{(3)}(u,u_{x},u_{xx})+g (u,u_{x},u_{xx})
\end{aligned}
\end{equation}
where $\ff^{(3)}$ is homogeneous of degree three while $g$ has a zero of order at least five.

 We will consider 
 two cases:

Case 1. $g$ is  analytic \label{ciao}
as function $\CCC^{3}\to \CCC$ in the ball 
of radius $\mathtt{r}_{0}>0$. Then we fix $\gota>0$ and extend \eqref{6.1} to  $x\in \TTT_{\gota}$. Here $\TTT_{a}$ is the compact subset of the complex torus $\TTT_{\CCC}:=\CCC/2\pi\ZZZ$  
with ${\rm Re}(x)\in\TTT$ and $|{\rm Im}(x)|\leq a$. 

\smallskip

Case 2. $g\in C^q(U_{\mathtt{r}_0},\RRR^2)$, \label{ciao2} 
where $U_{\mathtt{r}_{0}}$ is the ball of radius $\mathtt{r}_{0}$  in 
in the real sense.


%
%
%
%
%
\noindent

\medskip
Since  $\ff$
vanishes of order $3$ at $u=0$, equation \eqref{6.1} can be seen, at least in a neighborhood of the origin, as a perturbation
of the linear Schr\"odinger equation
\begin{equation}\label{6.1111} 
i u_t=u_{xx},
\end{equation}
which is \emph{completely resonant}, i.e. posseses only \emph{periodic} solutions.

Equation \eqref{6.1} can be seen as an infinte dimensional dynamical system with phase space
a scale of complex Hilbert spaces $u\in h^{a,p}$ with
\begin{equation}\label{6.3}
h^{a,p}:=\{u=\{u_{k}\}_{k\in\ZZZ} \; : \|u\|_{a,p}^{2}:=
\sum_{k\in\ZZZ}|u_{k}|^{2}e^{2a|k|}\langle k\rangle^{2p}<\infty\,
\},
\end{equation}
here $p\geq1/2$ while $0< a\leq \gota/2$ in Case 1 while $a=0$ in Case 2. Note that there is an isometric one-to-one correspondance
between a sequence $\{u_{k}\}$ and a function $u=\sum_{k}u_{k}e^{ik\cdot x}$ in $H^{p}(\TTT_{a})$,
i.e. the analytic function on the complex strip $\TTT_{a}$ that are 
$p-$Sobolev on the boundary. We will use the same symbol $u\in h^{a,p}$  to indicate both the sequence and the function.

 A natural question is to know whether equation \eqref{6.1} has periodic, quasi-periodic or
almost periodic solutions \emph{close to} zero, and more precisely solutions bifurcating from a periodic solution of \eqref{6.1111}. We recall that a \emph{quasi}-periodic solution of \eqref{6.1} is an embedding  
\begin{equation}\label{6.6}
\TTT^{d}\ni\f\mapsto v(\f,x) \in h^{a,p}\,,\quad  d\geq1,
\end{equation}
and a frequency vector
$\oo_{\infty}\in \RRR^{d}$
such that $u(t,x)=v(\oo_{\infty}t,x)$ is a solution of the equation of \eqref{6.1} and $v(\f,x)\in H^p(\TTT^{d+1}_a)$.
Note that both the embedding $v$ and the frequency vector $\oo_{\infty}$ are a unknown of the problem. 

Proving existence and stability  for quasi-periodic solutions in infinite dimensional systems is a natural  extension of  KAM theory for lower dimensional tori \cite{Po2}.
The first KAM results for PDEs have been 
obtained by Kuksin \cite{K1} and Wayne \cite{W}. Such results were restricted 
to the case in which the spatial variable ranges in a finite interval with Dirichlet boundary conditions.
In order to 
consider the case of
 periodic boundary conditions, Craig-Wayne used a Lyapunov-Schmidt reduction method in \cite{CW} later 
generalized by Bourgain in \cite{Boj}, \cite{B3}.
Other developments of KAM Theory for PDEs can be found in 
\cite{Po2}, \cite{CY}, \cite{K2}, \cite{KP}. 
For extension of KAM Theory to higher spatial dimension we mention 
the papers by Bourgain in
\cite{B3} for  the nonlinear Schr\"odinger equation on $\TTT^{2}$
with a convolution potential and
 \cite{B5} for an  existence result on $\TTT^{d}$.
 We mention also the remarkable results by Berti and Bolle \cite{BBhe1}, \cite{BB2} which study equations 
 in presence of a more natural multiplicative potential; 
  The latter approach, based on a multi-scale analysis, has been very fruitfully exploited
  in the study of PDEs also on manifolds different 
 from tori. In \cite{BCP} Berti, Corsi and Procesi studied NLW and NLS on  compact
Lie groups and homogeneous manifolds.
There are very few and recent results on reducibility on tori. We mention Geng-You in \cite{GY} for the smoothing NLS, 
Eliasson-Kuksin in \cite{EK} for the non resonant NLS and 
Procesi-Procesi \cite{PP3} for the completely resonant NLS which involves 
deep arguments of normal forms developed in \cite{PP}, \cite{PP1}.
All the aforementioned papers, both using KAM or multi-scale, are on semi linear PDEs with no derivatives in the non linearity. 

More recently KAM theory has been developed also for dispersive semilinear PDEs on the one
 dimensional torus when the nonlinearity contains derivatives of order $\delta \leq n-1$,
here $n$ is
the order of the highest derivative appearing in the linear constant coefficients term. The additional difficulty in this case is that,
due to the presence of derivatives in the nonlinearity, the KAM transformations used to diagonalize the linearized operator might be \emph{unbounded}.
The key idea to overcome such problem has been introduced by Kuksin in \cite{Ku2} 
in order  to deal with non-critical unbounded perturbations, i.e. $\de<n-1$, with the purpose of studying KdV type equations, see also \cite{KaP}.
The key idea is to note that the linear frequencies of KdV have  good separation properties, which allow to control derivatives  in the nonlinearities up to the second order.
This approach, developed for the $KdV$ that has a \emph{strong} dispersion law, has been further exploited by the Chinese school to cover the 
``less'' dispersive case of NLS
in presence of one derivative in the non linearity, i.e. the \emph{critical} case when $\de=n-1$.
In particular  
we mention   Zhang, Gao and Yuan \cite{ZGY}  
 and Liu and  Yuan  \cite{LY} for derivative NLS, 
and  Berti-Biasco-Procesi \cite{BBiP1}-\cite{BBiP2} for the derivative NLW.

Concerning \emph{quasi-linear} or \emph{fully non-linear} PDEs, i.e. $\de=n$, we quote the papers by
  Iooss-Plotnikov-Toland \cite{IPT} and 
  by Baldi \cite{Ba1}, \cite{Ba2} in which is studied the existence of \emph{periodic} solutions.
The first existence results of quasi-periodic solution for quasi-linear PDEs has been  obtained by 
Baldi-Berti-Montalto in  in \cite{BBM}, for the forced case, then in \cite{BBM1} for the autonomous one, see also \cite{Gi}. 
Recently such results has been extended to the forced NLS in \cite{FP} ( reversible setting), \cite{F} (Hamiltonian setting) to the water wave equation in
 \cite{BM1} and to the Kirchoff equation in \cite{Mon}, see also \cite{Mon2}.
 
We remark that all the aforementioned papers on \emph{fully non-linear} PDEs provide existence and stability of quasi-periodic solutions with Sobolev regularity, even when the non-linearity $g$ is an analytic function.  This is due to the strategy proposed in these papers in order to deal with quasi-linear and fully non-linear perturbations. 
Moreover all the results above require a Hamiltonian structure, in the case of autonomous systems.  In \cite{CFP}, we discussed a general strategy in order to deal with both Hamiltonian and reversible equations, in which we treated both analytic and finite regularity cases in a unified way.  We remark that the abovementioned  paper contains only applications to semi-linear PDEs.

The aim of this paper is to  apply the stategy of \cite{CFP} to an autonomous fully nonlinear NLS  and to prove existence of \emph {analytic } solutions (for completeness we also give the result for Sobolev solutions, when the non-linearity has only finite regularity).  In order to avoid the complications coming from double eigenvalues we decided to work in a \emph{reversible} setting.
\\
The first difficulty we have to overcome, before applying any quadratic scheme,  is that  equation is \emph{completely-resonant}, i.e. the solutions
of \eqref{6.1111} are \emph{periodic}, clearly in order to prove the existence of quasi-periodic solutions we need some non-degeneracy hypothesis on $\ff$ (since for instance $\ff=0$ is not acceptable!), more precisely on its leading term $\ff^{(3)}$. Let us first state our reversibility hypotheses explicitly.

\begin{hypo} \label{hyp2aut}
	Assume that $\ff$ is such that
		\begin{itemize}
		
		\item[(i)] $\ff(-\h_{0},{\h}_{1},-\h_{2})=
		-\ff(\h_{0},{\h}_{1},\h_{2})$.
		
		\item[(ii)] $\ff(\h_{0},\h_{1},\h_{2})=
		\ol{{\ff}(\bar{\h}_{0},\bar{\h}_{1},\bar{\h}_{2})}$,
		\item[(iii)]  we require that $\mathtt f $ is gauge preserving, i.e. $e^{-\ii \alpha}\ff(e^{\ii \alpha}\eta_0,e^{\ii \alpha}\eta_1,e^{\ii \alpha}\eta_2  )= \ff(\eta_0,\eta_1,\eta_2  ) $
		\item[(iv)] $
		\del_{\h_{2}}\ff\in\RRR$
	 	where $\del_{\h}= \del_{\rm{Re}(\h)}-i\, \del_{\rm{Im}(\h)}$. 

	\end{itemize}
\end{hypo}

One can check that the reversibility, $x$-independence  and Gauge preserving properties imply that
\begin{equation}\label{6.5}
\begin{aligned}
\ff^{(3)}(u,u_{x},u_{xx})&=\ta_{1}|u|^{2}u+\ta_{2}|u|^{2}u_{xx}+\ta_{3}|u_{x}|^{2}u+
\ta_{4}|u_{x}|^{2}u_{xx}
+\ta_{5}|u_{xx}|^{2}u_{xx}\\
&+\ta_{6} u^2\bar{u}_{xx}+\ta_{7}(u_x)^{2}\bar{u}+\ta_{8}(u_x)^{2}\bar{u}_{xx}
\end{aligned}
\end{equation}
with $\ta_{i}\in \RRR$ for $i=1,2,3,4,5,6,7,8$.
\begin{defi}\label{33}
We   say that $(\ta_1,\ta_{2},\ta_{3},\ta_{4},\ta_{5},\ta_{6},\ta_{7},\ta_{8})\neq 0$ is resonant if either:
\begin{enumerate}
\item $\ta_{5}=\ta_{1}=0$, $\ta_4-\ta_{8}= 0$ and $\ta_{3}-\ta_{2}-\ta_{6}-\ta_{7}=0$.
\item  $\ta_{5}=\ta_{1}=0$, $\ta_4-\ta_{8}= 0$, $\ta_{3}-\ta_{2}-\ta_{6}-\ta_{7}\neq0$  and either:
$\ta_{2}=0$, $\ta_{3}-\ta_{7}=(6d+1)/(2d+1)\ta_{6}$
or $a_{2}\neq0$, $\ta_3-(1+3d)\ta_{2}-\ta_{7}=0$, $\ta_{2}=\ta_{6}/d$.

\item $\ta_{5}=\ta_{1}=0$, $\ta_{3}-\ta_{2}-\ta_{6}-\ta_{7}=0$, $\ta_4-\ta_{8}\neq 0$ and $(2d-1)\ta_{4}=\ta_{8}$.
\end{enumerate} 
\end{defi}


We are now ready to state our main Theorem on the existence of quasi-periodic solutions of $d$ frequencies
which is based on the following ''genericity'' condition.

\begin{defi}[{\bf Genericity}]\label{defgene}
	For any finite $d\in \NNN$, given a non-trivial polynomial $P:\CCC^d\to \CCC$  we say that 
	$x_0\in \CCC^b$ 
	is ``\emph{generic}'' if $P(x_0)\neq 0$.
\end{defi}

\begin{theorem}\label{teoremap}
Consider the equation \eqref{6.1} 
 in case \ref{ciao},  namely  when $\ff$ as in \eqref{6.2} is an analytic function. Assume  the Hypothesis  \ref{hyp2aut} and moreover that  
 $(\ta_1,\ta_{2},\ta_{3},\ta_{4},\ta_{5},\ta_{6},\ta_{7},\ta_{8})$
  is  not resonant.
 There exists 
 a non trivial polynomial 
 such that  
for any $d\in \NNN$ with $d>2$  
and  for any choice of $\tv_1,\ldots,\tv_d\in \NNN$  such that $x_0=(\tv_1,\ldots,\tv_d) $ is \emph{generic} with respect to the polynomial
 the following holds.

There exists  $a=a(d,\ff)$, $\e_{0}=\e_{0}(d,\ff)$    and $\mathtt{c}=\mathtt{c}(d,\ff)\ll 1$
such that   for all $\e\in(0,\e_{0})$,  there exists a Cantor set %
\begin{equation}\label{asyaut}
\CCCC_{\e}\subset\e^2\left[\frac 12, \frac 32\right]^d\,,\quad \e^{-2d}|\CCCC_{\e}|\leq 1-\mathtt{c},
\end{equation}
%
 such that for all $\xi\in\CCCC_{\e}$ the  NLS equation (\ref{6.1})
has a quasi- periodic solution with frequency $\oo^\infty$  given by the embedding  $v(\xi)\in H^1(\TTT_a^{d+1})$:
\begin{equation}\label{esistenzadidio}
v= \sum_{i=1}^d \sqrt{\xi_i} e^{\ii \f_i}\sin(\tv_i x) + o(\sqrt{|\xi|})\,,\quad \omega^{\infty}_i(\xi)= \tv_i^2 + \sum_j \MM_{i}^j \xi_j + o(|\xi|)
\end{equation}
with $\MM$ an invertible matrix. Moreover one has  $v(\f,x)=-{v}(\f,-x)$ and $v(\f,x)=\bar{v}(-\f,x)$,
and the solution  is {\rm linearly stable}.

\end{theorem}

\begin{rmk}
	As far as we know Theorem \ref{teoremap} is the first result of analytic quasi-periodic solutions for a fully non-linear partial differential equation. Some of the key ideas follow closely the ones of \cite{BBM}, \cite{FP}, etc..., however in order to prove the existence of analytic solutions we have to modify that strategy in various non-trivial ways, which we shall illustrate in the next paragraph. While our approach can surely be applied to other equations, such as for instance the KdV equation, with very little modifications,  it does not seem straightforward at all to generalize  to the water wave equation \cite{BM1}.
\end{rmk}

In the case of finite regularity we have a similar result.

\begin{theorem}\label{teoremap2}
Consider the equation \eqref{6.1} 
 in case $2$. 
 For any $d\in \NNN$ with $d>2$  there exists $q=q(d)$ such that for any non linearity $\ff\in C^{q}$ that satisfies 
 Hypothesis  \ref{hyp2aut} and moreover such that  
 $(\ta_1,\ta_{2},\ta_{3},\ta_{4},\ta_{5},\ta_{6},\ta_{7},\ta_{8})$
 is not resonant, 
 there exists 
 a non trivial polynomial 
 such that  
  for any choice of $\tv_1,\ldots,\tv_d\in \NNN$ \emph{generic} with respect to the polynomial
 the following holds.

There exist $p=p(d,\ff)$, $\e_{0}=\e_{0}(d,\ff)$  and $\mathtt{c}=\mathtt{c}(d,\ff)\ll 1$ 
such that   for all $\e\in(0,\e_{0})$,  there exists a Cantor set %
\begin{equation}\label{asyaut1000}
\CCCC_{\e}\subset\e^2\left[\frac 12, \frac 32\right]^d\,,\quad \e^{-2d}|\CCCC_{\e}|\leq 1-\mathtt{c},
\end{equation}
%
 such that for all $\xi\in\CCCC_{\e}$ the  NLS equation (\ref{6.1})
has a quasi- periodic solution with frequency $\oo^\infty$  given by the embedding  $v(\xi)\in H^p(\TTT^{d+1})$:
$$
v= \sum_{i=1}^d \sqrt{\xi_i} e^{\ii \f_i}\sin(\tv_i x) + o(\sqrt{|\xi|})\,,\quad \omega^{\infty}_i(\xi)= \tv_i^2 + \sum_j \MM_{i}^j \xi_j + o(|\xi|)
$$
with $\MM$ an invertible matrix. Moreover one has  
$v(\f,x)=-{v}(\f,-x)$ and $v(\f,x)=\bar{v}(-\f,x)$,
and the solution  is {\rm linearly stable}.
\end{theorem}
The proofs of our two results are very similar, and we shall concentrate on the more difficolt and novel, analytic case.

\begin{rmk}
In stating our Theorems we put some effort in distinguishing the non-resonance conditions on
on $(\ta_1,\ta_{2},\ta_{3},\ta_{4},\ta_{5},\ta_{6},\ta_{7},\ta_{8})$ 
from the genericity conditions on $\tv_1,\ldots,\tv_d$. Informally we are saying that for all non-resonant equations of the form \eqref{6.1} there exist many quasi-periodic solutions, and that for most choices of $d$ sites  $\tv_1,\ldots,\tv_d$  there exist quasi-periodic solutions essentially supported on those sites for all times.
\\
For example given any choice of 
$(\ta_1,\ta_{2},\ta_{3},\ta_{4},\ta_{5},\ta_{6},\ta_{7},\ta_{8})$
such that $\ta_1\neq0$, then the genericity condition can be verified by removing only a co-dimension one algebraic manifold (which may depend on the choice of the parameters $\ta_i$)
in  $\tv_1,\ldots,\tv_d$.

 It may be possible that for some choices of the $\ta_i$ one does not need to impose \emph{any} further genericity condition, however we have not investigated this question. Indeed, even for equations with no derivatives  in the non-linearity such as the quintic NLS it can happen that, for specific choices of the sites $S$, the behavior of the solutions of the non linear equation
differs drastically from the one of the linear equation (see for instance \cite{KT} or \cite{HP}). In order to avoid such phenomena one has
to restrict to ``generic'' choices of $S$, in the sense of Definition \ref{defgene}. 
\end{rmk}
\begin{rmk}\label{nometto}
It is possible that our result can be further refined in order to prove existence of quasi-periodic solutions also for some resonant choices of the $\ta_i$, however some conditions on this parameters are necessary. Indeed it is not possible that all equations of the form \eqref{6.1} have quasi-periodic solutions, as can be seen in the following example.
We start by considering a \emph{linear} Schr\"odinger equation
$$
\ii v_t-v_{xx}=0
$$
and writing $v= u + |u|^2 u $, then we deduce the equations for $u$.
We have
$$
\ii v_t-v_{xx}=(1+2|u|^2)(\ii u_t-u_{xx}) + u^2(\ii \bar u_t -\bar u_{xx})+ 4 |u_x|^2 u + 2 u_x^2 \bar u =0
$$
and after some computations we get
$$
\ii u_t-u_{xx} =  \frac{u^2}{(1+3|u|^4+4 |u|^2)}(4 |u_x|^2 \bar u + 2 \bar u_x^2  u -2 \bar u^2  u_{xx} )- \frac{1+2|u|^2}{1+3|u|^4+4 |u|^2} (4 |u_x|^2 u + 2 u_x^2 \bar u -2 u^2 \bar u_{xx} )
$$
which has the form \eqref{6.2} with 
$$
\ff^{(3)}= - 4 |u_x|^2 u - 2 u_x^2 \bar u +2 u^2 \bar u_{xx},
$$
{and satisfies Hypothesis \ref{hyp2aut}}. Now evidently all the small solutions of this equation are periodic (since the map $u\to v$ is invertible close to zero) and indeed it turns out that such choice of $\ff^{(3)}$ corresponds to resonant $\ta_i$, namely  $\ta_1=\ta_2=\ta_4=\ta_{5}=0$,  $\ta_3=-4$,  $\ta_{6}= 2$, $\ta_{7}=-2$, $\ta_{8}=0$.
\end{rmk}
\paragraph{Description of the paper and strategy of the proof.}
Since the proof  involves many different arguments we explain how the paper is organized.

 \indent{\bf Weak Birkhoff normal form.}
As preliminary step one looks for an approximate solution for the NLS \ref{6.1} which will be the starting point for an iterative algorithm.
Hence in Section \ref{weakuffa} we find a {\em better}  approximate solution.
One first rewrites the NLS as a dynamical system
\begin{equation}\label{papaveri}
\dot{u}=\chi(u)=\sum_{j\in\ZZZ}\chi_{j}(u)\del_{u_j}
\end{equation}
where $\chi$ is vector field defined on the space of sequences $u=\{u_j\}_{j\in \ZZZ}\in h^{a,p}$. In our case $h^{a,p}\leftrightarrow H^{p}(\TTT_{a};\CCC)$
the functions analytic on the toroidal domain $x\in\CCC$ such that ${\rm Re}(x)\in\TTT$ and $|{\rm Im}(x)|\leq a$, for some $a>0$.
For a precise definition see formula \eqref{6.666} in Section \ref{scalelala}.

A  good strategy for finding approximate solutions of a dinamical system is to  perform on the equation a few steps of Birkhoff normal form, here we follow closely the strategy of \cite{BBM1}.

 We select  the sites $S\subset\ZZZ$ and decompose the space of sequences $h^{a,p}$
into two orthogonal subspaces $u=(v,z):=\big(\{u_{k}\}_{k\in S}, \{u_{k}\}_{k\in S^c}\big)$. Then one looks
for a coordinate system such that 
$\{z=0\}$ is an  approximately invariant manifold. More precisely one splits the vector field in \eqref{papaveri} as
$\chi=\Pi_{S}\chi \del_{v}+\Pi_{S^{c}}\chi \del_{z}$, hence,  through a step of Birkhoff normal form,  
one removes all the cubic terms $O(v^3),O(v^2 z)$ from $\Pi_{S}\chi$
and all the term $O(v^3)$ form $\Pi_{S^c}\chi$  that do not commute with the linear part. We then show that
 the dynamics on $z=0$ of the terms $O(v^3)$ is {\em integrable} and {\em anisochronous}.
We perform this step in Proposition \ref{wbnf} of Section \ref{weakuffa} and prove that the corresponding change of variables is close to the identity up to a \emph{finite} dimensional operator.

In principle one could remove also the terms $O(v^2 z)$ from $\Pi_{S^c}\chi$ by performing a stronger normal form, this would give us not only an approximately invariant torus but also information on the linear dynamics in the  normal  directions.  The reason why we do not perform such a step is the same as in \cite{BBM1}, let us briefly illustrate it.

It is well known that in small divisor problems, the main difficulty is in proving estimates on the inverse of the vector field linearized at an approximate torus.  We see in \eqref{skifo}  that the linearized operator of an unbounded vector field as $\chi$ in \eqref{6.666} 
is a non constant coefficients pseudo-differential operator, let us denote it by $\calL$. In the forced case \cite{FP} we analyzed the invertibility of a similar linear operator, by strongly exploiting the pseudo-differential structure.
To prove the invertibility of the linearized operator in the autonomous case we need to use similar arguments. Now applying a map which is  close to the identity up to a \emph{finite} dimensional operator. does {\em not modify} the pseudo-differential structure.
\\
On the other hand, performing a step of Birkhoff normal form as in \cite{PP} which removes the terms $O(v^2 z)$ from $\Pi_{S^c}\chi$, means applying a  map  which is close to the identity up to a small {\em bounded} operator, in general such a map {\em does not} preserve the pseudo-differential structure. 
%

 \smallskip
 
 \indent{\bf Action-angles variables.} We  have underlined that in autonomous cases there are no external parameters to modulate in order to fulfill 
 non-degeneracy conditions. Now thanks to the step of weak Birckhoff norml form we selected an approximatively invariant manifold
 where the dynamics is integrable and non-isocronous.   
 On this manifold we introduce action-angle variables, on the tangential sites (i.e. the sites in $S$), and use  the initial data as parameters
 which will be denoted by $\x$.
 This is done in Section \ref{action} using the change of coordinates in \eqref{aacaut}.
 \\

 \begin{equation}\label{approxuffaB}
\left\{
\begin{aligned}
&\dot\theta=\omega^{(0)}(\x)+G^{(\theta)}(\theta,y,w)\\
&\dot y = G^{(y)}(\theta,y,w)\\
&\dot z_j = \ii \Omega^{(-1)}_jz_j+G^{(z_j)}(\theta,y,z),
\end{aligned}
\right.
\end{equation}
here $\omega^{(0)}(\x)$ is an invertible linear map given by the frequency-amplitude modulation in \eqref{omeghino}, $\Omega^{(-1)}_j= j^2$ and $G(\theta,0,0)$ is appropriately small. We denote the whole vector field by $F$. 
\\
Now, the frequency vector $\oo^{(0)}$ is $O(|\x|)-$close to  integer vector. We require that it is diophantine, by setting
\begin{equation}\label{equazione1000}
|\oo(\x)\cdot l|\geq \frac{\g}{1+|l|^{\tau}}, \quad
\end{equation}
for some $\tau>d$ and $\g\sim  O(|\x|)$.

Now we need to control the linearized operator in the $z$ directions, the leading terms are those coming from the terms $O(v^2z)$ (which we have not removed from $\Pi_{S^c}\chi$) . The resulting matrix, denote it by $\Omega(\theta,\x)$, is of order $O(|\xi|)$ hence in principle {\em not perturbative} w.r.t. $\gamma$. We discuss this in Proposition \ref{birk} of Section \ref{sec7aut} where we study the invertibility of the linearized operator in the normal directions.

A crucial point is the so called ``twist'' condition with respect to the parameters $\x$. 
What we need to check
is that  if one ``moves'' the initial data $\x$ then the frequencies move in a non ``trivial'' way. 
We have said that the map $\x\to \oo(\x)$ is a diffeomorphism, this is the so-called {\em twist} condition.  In order to perform our scheme we also need a {\em twist} condition on  the normal directions, namely on
 the average in $\theta$ of $\Omega(\theta,\x)$. The analysis of this last issue is performed in Lemmata \ref{twist1} and \ref{twist666} in Section \ref{caspita2}.
Note that this is a delicate question, since we are requiring a modulation of infinitely many normal frequencies by only finitely many parameters.
The analysis would be much simpler if one considers a fully nonlinear perturbation, of order at  least four, of the cubic integrable NLS. In such a case, for \emph{any}
choices of the tangential site in $S$, one would obtain that the map $\x\to\oo(\x)$  is a diffeomorphism by exploiting the integrability properties of the system.
Here we need to introduce a notion of ``genericity'' (see Definition \ref{defgene}) which  implies that for ``most'' choices of the cubic terms and ``most''  choices of the tangential sites 
 the frequencies satisfy a ``twist'' condition.
 Interestingly we can produce explicitly  non generic choices of cubic non linearities such that for \emph{any}
 choice of tangential sites the twist condition is false. In particular it turns out that 
 the Jacobian of the map $\x\mapsto \oo(\x)$ has at most rank $2$.
 It would be interesting to investigate whether quasi periodic solutions exist for such ``degenerate'' cases ,see also Remark \ref{nometto}.

%
 
\smallskip
\indent{\bf  KAM scheme.}  Once we are in the setting of \eqref{approxuffaB}, we wish to apply the Abstract KAM theorem of \cite{CFP}. Such theorem gives an explicit (if complicated) set of parameters $\xi$ (denoted by $\calO^{(\infty)}$) for which quasi-periodic solutions exist for \eqref{approxuffaB}, provided that $G$ is {\em tame} and satisfies some smallness conditions at least close to the approximate invariant torus.  In sections \ref{iniz} and \ref{secNMaut} we first introduce the necessary notations and then state the Theorem \ref{thm:kambis}, and verify that all the hypotheses are fulfilled in our setting. We refer to the introduction of \cite{CFP} for a detailed description of the strategy. The theorem is mostly just a quadratically convergent iterative scheme which produces a sequence of changes of variables $\HH_n$ such that $(\HH_n)_{\star} F(\theta,0,0)$ tends to zero (among other properties).  This means that $(\HH_n)^{-1}(\theta,z=0,y=0)$ is an approximately invariant torus, with a better and better approximation, we call this object an approximate solution. 

 The remainder of the paper is devoted to proving that the set $\calO^{(\infty)}$ is non-empty.
Such set is explicitly defined in the Theorem as the intersection of the sets $\calO^{(n)}$ where one has appropriate tame estimates on the inverse of the linearized vector field at the $n$'th approximate solution, see Definition \eqref{pippopuffo3}.
We have to show  that the $\calO^{(n)}$ have positive measure and  that the same holds for their intersection.

Let us denote the   linearized vector field at the $n$'th approximate solution by $\mathfrak L_n$. The strategy in \cite{BBM},\cite{BBM1},\cite{FP}... is to prove the bounds on $\mathfrak L_n$ by constructing a bounded change of variables $Q_n$ 
which approximately diagonalizes it, and then imposing the invertibility of $\mathfrak L_{n}$  by assuming bounds on the eigenvalues and controlling the norm of $Q_n,Q_n^{-1}$.
 In turn the fact that the diagonalizing change of variables exists is ensured by assuming bounds on the differences of the eigenvalues and 
by exploiting the fact that $\mathfrak L_n$ is a pseudo-differential operator. This results in an explicit description of $\calO^{(n)}$ in terms of {\em first and second order Melnikov conditions on the eigenvalues}.

This strategy however has a serious problem in the analytic setting. The change of variables which diagonalizes $\mathfrak L_n$ in the analytic case is NOT bounded from the space in itself but loses some of the analyticity radius.

This can be trivially seen from the following example.  One of the changes of variables used in order to diagonalize is a change of variable $z(x)\to (\TT_\beta z)(x):= z(x+\beta(x))$. 
This change of variables is used in order to conjugate $\mathfrak L_n$ to an operator  whose principal term (the term containing the highest derivatives) is diagonal.
\\
Now it is evident that this operator maps $H^p(\TTT)$ in itself but one cannot expect the same to hold for $H^p(\TTT_a)$, where the radius of analyticity should be  reduced by $\sim |\beta|$. This means that at each step $n$ we lose  some analyticity, of the order of the corresponding $\beta_n$. Now in the strategy of \cite{BBM}, etc. the $\beta_n$ are all   small but more or less all of the same size so that the algorithm would collapse after a finite number of steps.

In order to overcome this difficulty we reason as follows. In \cite{CFP} we have shown that in performing the iterative scheme which produces the changes of variables $\HH_n$ and the approximate solutions we can apply any change of variables which does not ruin our approximation procedure (namely which maps an approximately invariant torus into itself), we call such changes of variables {\em compatible}, see Definition \ref{compa}. With this fact in mind at each step we apply to $(\HH_n)_\star F$ the change of variables $\TT_{\alpha_n}$ which conjugates $\mathfrak L_n$ to an operator  whose principal term  is diagonal. Note that we can apply  the changes of variables  $\TT_{\alpha}$ due to the fact that they preserve the pseudo-differential structure, which we specify in Definition \ref{pseudopseudo}. In this way  our algorithm is closed, moreover not only  $(\HH_{n+1} )_\star F(\theta,0,0)$ becomes smaller at each step but also the principal term of   $\mathfrak L_{n+1}$ becomes closer to being diagonal. This means that $|\alpha_{n+1}|\ll |\alpha_n|$ and our loss of analiticity converges.

In Section \ref{paperina} we first discuss various types of  changes of variables (from which we shall choose the compatible changes of variables explained above). Then we  show how to 
approximately diagonalize  the resulting operator and deduce the estimates on the inverse of a matrix  from Melnikov conditions  on the eigenvalues. 

In section \ref{sbroo} we  use the results of the previous setting in order to define recursively the compatible changes of variables $\calL_n$. Then we show that the sets $\calO_n$ can be expressed in terms of Melnikov conditions  on the eigenvalues. 
Finally in section \ref{caspita2} we compute the measure of the sets $\calO_n$ and of their intersection.

\zerarcounters
\section{Functional setting}

In this Section we introduce the functional spaces on which we work. 
Moreover
we analyze in a specific way the r\^ole of the ``reversibility'' condition and how we use it in Theorems
 \ref{teoremap} and \ref{teoremap2}. 

\subsection{Scales of Sobolev spaces}\label{scalelala}
For the  analytic contest we 
introduce the space of analytic functions that are Sobolev on the boundary 
\begin{equation}\label{space}
H^{p}(\TTT_{a}^{b}; \CCC):=\big\{u=\sum_{l\in\ZZZ^b}u_{l}e^{\ii l\cdot \theta} : 
\|u\|^{2}_{a,p}:=\sum_{l\in\ZZZ^{b}}\langle l\rangle^{2p}|u_{l}|^{2}e^{2a|l|}<\infty\big\}\,.
\end{equation}
for $a>0$ and for some $b\geq1$. Clearly the space $H^{s}(\TTT_{a}^{b})$ is in one-to-one correspondence with the sequences space. We denote the space
of sequences by $h^{a,p}$ (see \eqref{6.3}).
If the parameters $a=0$ then we denote by $H^{p}(\TTT^{b};\CCC)$ the usual Sobolev space of functions defined on $\TTT^{b}$.

In order to prove Theorem \ref{teoremap} and \ref{teoremap2}
it is convenient to study the equation as dynamical system on the phase space $H^{1}(\TTT_{a}; \CCC)$ (or $H^{1}(\TTT; \CCC)$ in the Sobolev case), i.e. look for  
$u(t)\in H^{1}(\TTT_{a}; \CCC)$ quasi-periodic in  $t$. In order to distinguish these two cases, for the autonomous system 
IN the paper we shall  
use  the equivalent notation $h^{a,p}$ to denote the functions in $H^{p}(\TTT_a;\CCC)$.
We shall write $H^{s}(\TTT^{d+1};\CCC)$
to denote functions $v(\f,x)$ defined for $(\f,x)\in \TTT^{d+1}$. 
 
Due to the complex nature of the NLS we need to work on product spaces. 
We will usually denote  
\begin{equation}\label{spaces}
\begin{aligned}
{\bf H}^{s}&:={\bf H}^{s}(\TTT^{d+1};\CCC)=H^{s}(\TTT^{d+1};\CCC)\times H^{s}(\TTT^{d+1};\CCC)\cap \calU,
\end{aligned}
\end{equation}
where
\begin{equation}\label{5}
\calU=\{(h^{+},h^{-})\in H^{s}(\TTT^{d+1};\CCC)\times H^{s}(\TTT^{d+1};\CCC)  \; : \; h^{+}=\ol{h^{-}}\}.
\end{equation}

We also write 
${\bf H}^{s}_{x}$ to denote the phase space of functions
in 
${\bf H}_x^{s}(\TTT ;\CCC)=H^{s}(\TTT^{1};\CCC)\times H^{s}(\TTT^{1};\CCC)\cap \calU,
$
On the product spaces  
${\bf H}^{s}$ we define, with abuse of notation, the norms
\begin{equation}\label{spaces1}
\begin{aligned}
\|z\|_{{\rm H}^{s}}&:=\max\{\|z^{(i)}\|_{s}\}_{i=1,2}, \quad z=(z^{(1)},z^{(2)})\in{\rm H}^{s},\\
\|w\|_{{\bf H}^{s}}&:=\|z\|_{H^{s}(\TTT^{d+1};\CCC)}=\|z\|_{s}, \quad w=(z,\bar{z})\in{\bf H}^{s}, \quad
z=z^{(1)}+iz^{(2)}.
\end{aligned}
\end{equation}
\noindent
For a function $f : \Lambda\to E$ where $\Lambda\subset \RRR^{n}$ and $(E,\|\cdot\|_{E})$ is a Banach space
we define
\begin{eqnarray}\label{lnorm2}
{\it sup \phantom{g}norm:} \; \|f\|_{E}^{sup}&\!\!\!\!\!\!:=\!\!\!\!\!\!&\|f\|^{sup}_{E,\Lambda}:=\sup_{\la\in\Lambda}\|f(\oo)\|_{E},\\
 {\it Lipschitz \phantom{g} semi\!-\!norm:} \;
 \|f\|_{E}^{lip}&\!\!\!\!\!\!:=\!\!\!\!\!\!&\|f\|_{E,\Lambda}^{lip}:=
\sup_{\substack{\oo_{1},\oo_{2}\in\Lambda \\ \oo_{1}\neq\oo_{2}}}
\frac{\|f(\oo_{1})-f(\oo_{2})\|_{E}}{|\la_{1}-\la_{2}|}\nonumber
\end{eqnarray}
%
and for $\g>0$ the weighted  Lipschitz norm
\begin{equation}\label{lnormnorm}
\|f\|_{E,\g}:=\|f\|_{E,\Lambda,\g}:=\|f\|^{sup}_{E}+\g\|f\|_{E}^{lip}.
\end{equation}
In the paper we will work with parameter families of functions in $\HH_s$,
If one deal with parameters family $u=u(\lambda)\in {\rm Lip}(\Lambda,\HH_s)$
where $\calH_{s}={\rm H}^{s}, {\bf H}^{s}$ and  $\Lambda\subset \RRR^{d}$ we simply write $\|f\|_{\HH_{s},\g}:=\|f\|_{s,\g}$,
or $\|u\|_{s,p,\g}$ in the analytic contest. All the discussion above holds for the product space
${\bf h}^{a,p}:=h^{a,p}\times h^{a,p}$.
Along the Thesis we shall write also
\begin{equation*}
a\leq_{s} b \;\;\; \Leftrightarrow \;\;\; a\leq C(s) b \;\;\; {\rm for \; some \; constant}\;\; C(s)>0.
\end{equation*}
Moreover to indicate unbounded or regularizing spatial differential operator we shall write $O(\del_{x}^{p})$ for some 
$p\in \ZZZ$. More precisely we say that an operator $A$ is $O(\del_{x}^{p})$ if
\begin{equation}\label{pseudo}
A : H_{x}^{s}\to H_{x}^{s-p}, \quad \forall s\geq0.
\end{equation}
Clearly if $p<0$ the operator is regularizing.

\noindent
Now we define the subspaces of trigonometric polynomials
\begin{equation}\label{trig}
H_{{n}}=H_{N_n}:=\big\{u\in L^{2}(\TTT^{d+1}) : u(\f,x):=\sum_{|(\ell,j)|\leq N_{n}}u_{j}(\ell)e^{i(\ell\cdot\f+jx)}\big\}
\end{equation}
where $N_{n}:=N_{0}^{(\frac{3}{2})^{n}}$, and the orthogonal projection
$$
\Pi_{n}:=\Pi_{N_{n}}: L^{2}(\TTT^{d+1})\to H_{n}, \quad \Pi^{\perp}_{n}:=\uno-\Pi_{n}.
$$
This definitions can be extended to the product spaces in \eqref{spaces} in the obvious way. We have the following classical result. 
\begin{lemma}
Fo any $s\geq0$ and $\n\geq0$ there exists a constant $C:=C(s,\n)$ such that
\begin{equation}\label{smoothing}
\begin{aligned}
&\|\Pi_{n}u\|_{s+\n,\g}\leq C N_{n}^{\nu}\|u\|_{s,\g}, \;\; \forall u\in H^{s}, \\
& \|\Pi^{\perp}_{n}u\|_{s}\leq C N_{n}^{-\nu}\|u\|_{s+\nu}, \;\; \forall u\in H^{s+\nu}.
\end{aligned}
\end{equation}
\end{lemma}
\noindent
We omit the proof of the Lemma since bounds \eqref{smoothing} are classical estimates for truncated Fourier series
which hold also for the norm in \eqref{lnormnorm} and in the analytic case.

For any
$$
{\bf u}:=(u^{+},u^{-})\in{\bf h}^{a,p}:=h^{a,p}\times h^{a,p}.
$$
 we consider the dynamical system
\begin{equation}\label{6.666}
\dot{\bf u}:=-\ii E\left[{\bf u}_{xx}
+
\left(\begin{matrix} \ff^{+}({\bf u},{\bf u}_{x},{\bf u}_{xx})\\ 
{\ff^{-}({\bf u},{\bf u}_{x},{\bf u}_{xx})}
\end{matrix}
\right)\right]=\chi({\bf u})=\left(
\begin{matrix}
\chi^{+}({\bf u})\\ \chi^{-}({\bf u})
\end{matrix}\right), \qquad  E=\left(\begin{matrix} 1 &0\\ 0 &-1\end{matrix}\right),
\end{equation}
where $\ff^{\pm}$ are defined in such a way that, on the subspace
%
%
%
$\calU:=\{u^{+}=\ol{u^{-}}\}$, the system \eqref{6.666} is equivalent
 to \eqref{6.1}.
 Essentially one look for an extension such that $(\ff^{+},\ff^{-})=(\ff,\bar{\ff})$.
 If $\ff$ is analytic this extension is completely standard, indeed one may Taylor expand $\ff$ as totally convergent series in $u,\bar{u}$ (and their derivatives).
 In the $C^{q}$ case this requires some care, see Section 1 in  \cite{FP} for more details.
 Here the notation of a vector field is the following:
\begin{equation}\label{6.6bis}
\chi({\bf u})=\sum_{\s=\pm}\chi^{\s}(u)\del_{u^{\s}}=\sum_{\s=\pm}\sum_{j\in\ZZZ}\chi^{\s}_{j}\del_{u^{\s}_{j}},
\end{equation}
Note that the map $\calF : {\bf h}^{a,p}\to {\bf h}^{a,p-2}$ defined by
\begin{equation}\label{compoop}
\calF : {\bf u}\mapsto\left(\begin{matrix} \ff^{+}({\bf u},{\bf u}_{x},{\bf u}_{xx})\\ 
{\ff^{-}({\bf u},{\bf u}_{x},{\bf u}_{xx})}
\end{matrix}
\right),
\end{equation}
 is a composition operator. This implies that the linearized operator at some ${\bf u}$ is of the form
 $$
 d_{{\bf u}}\calF({\bf u}) = d_{{\eta}^{\pm}_{2}}\left(\begin{matrix} \ff^{+}({\bf u},{\bf u}_{x},{\bf u}_{xx})\\ 
 {\ff^{-}({\bf u},{\bf u}_{x},{\bf u}_{xx})}
 \end{matrix}
 \right)\del_{xx}+
 d_{{\eta}^{\pm}_{1}}\left(\begin{matrix} \ff^{+}({\bf u},{\bf u}_{x},{\bf u}_{xx})\\ 
  {\ff^{-}({\bf u},{\bf u}_{x},{\bf u}_{xx})}
  \end{matrix}
  \right)\del_{x}+
  d_{{\eta}^{\pm}_{0}}\left(\begin{matrix} \ff^{+}({\bf u},{\bf u}_{x},{\bf u}_{xx})\\ 
   {\ff^{-}({\bf u},{\bf u}_{x},{\bf u}_{xx})}
   \end{matrix}
   \right).
 $$
 Thus $\chi$ linearized at any ${\bf u}$ has a very special multiplicative structure, namely on $\calU$ it acts on functions ${\bf h}(x)=(h^{+}(x),h^-(x))$ as 
\begin{equation}\label{skifo}
\!\!\!\!d_{{\bf u}}\chi({\bf u}) [{\bf h}]= -\ii E\left[\left(\begin{matrix} 1+a_{2}(x) & b_{2}(x) \\ \bar{b}_{2}(x) & 1+\bar{a}_{2}(x)\end{matrix}
 \right)\del_{xx} +\left(\begin{matrix} a_{1}(x) & b_{1}(x) \\ \bar{b}_{1}(x) & \bar{a}_{1}(x)\end{matrix}
  \right)\del_{x} + \left(\begin{matrix} a_{0}(x) & b_{0}(x) \\ \bar{b}_{0}(x) & \bar{a}_{0}(x)\end{matrix}
   \right)\right] {\bf h}(x).
\end{equation}

 \subsection{Reversible structure.}
 Consider the following involution
\begin{equation}\label{b}
S : u(t,x) \to -\bar{u}(t,-x), \quad S^{2}=\uno.
\end{equation}
By Hypothesis \ref{hyp2aut}  it turns out that equation \eqref{6.666}  is reversible with respect the involution \eqref{b}
and hence 
we have
\begin{equation*}
-S \circ \chi({\bf u})=\chi\circ S ({\bf u}).
\end{equation*}

Hence the subspace of ``reversible'' solutions
\begin{equation}\label{reversiblesol}
u(t,x)=-\bar{u}(-t,-x).
\end{equation}
is invariant. Actually we look for \emph{odd reversible}   solutions i.e. $u$ which satisfy (\ref{reversiblesol}) and $u(t,x)=-u(t,-x)$.
  Hence we choose as phase space of \eqref{6.666}
 \begin{equation}\label{phacespace}
 {\bf h}_{\rm odd}^{a,p}:=\left\{(u^+,u^-)\in {\bf h}^{a,p}\, :\; u^\s_{k}=-u^\s_{-k}
 \right\},
 \end{equation}
 essentially the couples of odd functions in $H^{p}(\TTT_{a})$. Then (\ref{reversiblesol}) reads $u(t,x)=\bar u(-t,x)$.

 It shall be convenient to introduce also the following 
  spaces of odd or even functions in $x\in \TTT$. For all $s\geq0$, we set
\begin{equation}\label{SPACES}
\begin{aligned}
X^{s}&:=\left\{u\in H^{s}(\TTT^{d}\times\TTT) : \;\;\; u(\f,-x)=-u(\f,x), \;\; u(-\f,x)=\bar{u}(\f,x)
 \right\},\\ 
Y^{s}&:=\left\{u\in H^{s}(\TTT^{d}\times\TTT) : \;\;\; u(\f,-x)=u(\f,x), \;\; u(-\f,x)=\bar{u}(\f,x)
 \right\},\\ 
Z^{s}&:=\left\{u\in H^{s}(\TTT^{d}\times\TTT) : \;\;\; u(\f,-x)=-u(\f,x), \;\; u(-\f,x)=-\bar{u}(\f,x)
 \right\},
\end{aligned}
\end{equation}
Note that odd reversible solutions means $u\in X^{s}$, moreover an operator reversible w.r.t. the involution $S$ maps $X^s$ to $Z^s$. 
\begin{defi}\label{stizzi}
 We denote with bold symbols the spaces ${\bf G}^{s}:=G^{s}\times G^{s}\cap \mathcal U$ where
$G^{s}$ is $H^s$, $X^{s},Y^{s}$ or $Z^{s}$.

\noindent We denote by  $H_{x}^{s}:=H^{s}(\TTT)$ the Sobolev spaces of functions of $x\in\TTT$ only, same for all the subspaces $G^s_x$ and ${\bf G}^s_x$.
\end{defi}
 
 \begin{rmk}\label{rmkphase}
Given a family of linear operators $A(\f) : H_{x}^{s}\to H_{x}^{s}$ for $\f\in\TTT^{d}$, we can associate it to an operator $A : H^{s}(\TTT^{d+1})\to H^{s}(\TTT^{d+1})$ by considering each matrix element of $A(\f)$ as a multiplication operator. This identifies a subalgebra of linear operators on $H^{s}(\TTT^{d+1})$. An operator 
$A$ in the sub-algebra identifies uniquely its corresponding ``phase space'' operator $A(\f)$. 
With reference to the Fourier basis this sub algebra is called ``T\"opliz-in-time'' matrices (see formul{\ae}  \eqref{eq:2.16aut}, \eqref{pozzo10aut}).

\end{rmk}

Recalling the definitions (\ref{SPACES}), we set,

\begin{defi}\label{reversible}
An operator $R: H^s\to H^s$ is ``{\rm reversible}'' with respect to the reversibility (\ref{b}) if
\begin{equation}\label{rever}
R : X^{s}\to Z^{s}, \quad s\geq0
\end{equation}
We say that $R$ is ``{\rm reversibility-preserving}'' if
\begin{equation}\label{rever2}
R : G^{s}\to G^{s}, \quad {\rm for} \quad G^{s}=X^{s}, Y^{s}, Z^{s}, \quad s\geq0.
\end{equation}
In the same way, we say that $A:{\bf X}^{s}\to {\bf Z}^{s}$, for $s\geq0$ is ``reversible'', while
$A: {\bf G}^{s}\to{\bf G}^{s}$, for ${\bf G}^{s}={\bf X}^{s},{\bf Y}^{s},{\bf Z}^{s}$, $s\geq0$ is ``{reversibility-preserving}''.
\end{defi}

\begin{rmk}
Note that, since  ${\bf X}^{s}= X^s\times X^s \cap\calU$, 
 Definition \ref{reversible} guarantees that
a reversible operator  preserves also the subspace $\calU$, namely $(u,\bar{u})\stackrel{R}{\to}(z,\bar{z})\in{ H}^{s}\times H^s\cap\calU$.
\end{rmk}

\zerarcounters

\section{Weak Birkhoff Normal Form}\label{weakuffa}

In this Section we select a finite dimensional subspace "approximatively"  invariant 
for the system 
\eqref{6.666} from which the solution for the entire system will bifurcate. 
This procedure is necessary since NLS equation is completely resonant near $u=0$.
In other words all the solution of the linear equation are periodic.
Let fix some notation. Given a finite set of distinct numbers $\{j_{1}, \ldots, j_{N}\}=E^+\subset \NNN$ we
define  $E:=\{\pm j_{1}, \ldots,\pm j_{N}\}\subset \ZZZ$. This decomposes naturally $h^{a,p}$ into two orthogonal subspaces $u=(\{u_j\}_{j\in E},\{u_j\}_{j\notin E})$. We write
\begin{equation}\label{6.11aut}
u(x)=\Pi_{E}u+\Pi_{E}^{\perp}u \,,\quad h^{a,p}=\Pi_E h^{a,p}\oplus\Pi_E^\perp h^{a,p}.
\end{equation}
We choose $S^+=\{\tv_1,\ldots,\tv_d\}\subset \NNN$ as above and denote
${ v}=  \Pi_{S}u$  the tangential variables and $z= \Pi_{S}^{\perp}u$ the normal ones.
For a finite dimensional subspace $E:={\rm span}\{e^{ijx} : |j|<C\}, C>0$ we denote $\Pi_{E}$ its $L^{2}$ projector.


As notation we will also indicate with $R(v^{q}z^{r})$ a homogeneous polynomial 
\begin{equation*}
R(v^{q}z^{r}):=M[\overbrace{v^+,\ldots,v^+, v^-,\ldots  v^-}^{q-times},\overbrace{z^+,\ldots,z^+,  z^-,\ldots, z^-}^{r-times}], 
\end{equation*}
with $M$ a $q,r$--multilinear operator in the variables $v^\pm,z^\pm$.
\begin{defi}\label{kresonance}
For any natural $k$ consider  a $2k$-uple  $\vec\jmath=(j_1,\ldots,j_{2k})\in \ZZZ^{2k}$. We say that $\vec\jmath$ is a $k$-{\bf resonance} if 
$$
\sum_{i=1}^{2k}(-1)^{i} j_i=0\,,\quad \sum_{i=1}^{2k}(-1)^{i} j_i^2=0.
$$
We say that a $k$-resonance is {\bf trivial} if $j_i= j_{i+1}$ up to a permutation of the $\{j_{2 l}\}_{l=1}^k$.

\noindent We say that a $2k$-uple is {\bf non-resonant}, $\vec\jmath\in \mathtt N$ if  $$\sum_{i=1}^{2k}(-1)^{i} j_i=0\,,\quad\sum_{i=1}^{2k}(-1)^{i} j_i^2\neq 0$$
\end{defi}
\begin{rmk}
Note that all $2$-resonances are trivial. Indeed if $j_{1}-j_{2}+j_{3}-j=0$ then
$j_{1}^{2}-j_{2}^{2}+j_{3}^{2}-j^{2}=2(j_{1}-j_{2})(j_{2}-j_{3})=0$. \end{rmk}
 \begin{lemma}\label{hyp3autaut} For $S$ generic one has that  that there are no non-trivial $3$-resonances with at least five points in $S$ 
 \end{lemma}
 \begin{proof} We just need to exhibit the polynomial which gives such genericity condition.
 \end{proof}
 
 Given $\vec\jmath= (j_{1},\ldots,j_{n})\in\ZZZ^{n}$ for some $n$ we say that $$\vec\jmath\in \calA_l\,,\quad \text{ if at most }\, l\, \text{ of the} \,j_i\,\text{ are not in }\, S\,. $$
Note that  for each fixed $n$ the set $(j_1,\dots,j_n)\in \calA_1$ is  finite dimensional.
For a finite dimensional subspace $E:={\rm span}\{e^{ijx} : |j|<C\}, C>0$ we denote $\Pi_{E}$ its $L^{2}$ projector.
 
 \begin{proposition}[Weak Birkhoff Normal Form]\label{wbnf}
 There exists an analytic change of variables  
of the form
\begin{equation}\label{weak1} 
 \uu_+=\Phi(\uu)=\uu+\Psi(\uu),\qquad
 \end{equation}
where $\Psi$ is a \emph{finite} rank map.
The map $\Phi(\uu)$ is defined for all \footnote{note that $\epsilon_0$ is fixed in terms of $\mathtt r_0$ and $a,\gotp_1$. }  $\uu\in{\bf h}^{a,p}$  such that $\| \uu\|_{a,\gotp_1} \le \epsilon_0$, and satisfies the bounds:
\begin{equation}\label{tamemod}
\|\Psi(\uu)\|_{a,p} \leq C_p \epsilon_0^2\|\Pi_E\uu\|_{a,p} \,, \quad
\|d_\uu\Psi(\uu)[h]\|_{a,p} \leq C_p \left(\epsilon_0^2 \|\Pi_E h\|_{a,p}+ \epsilon_0\|\Pi_E u\|_{a,p}\|\Pi_E h\|_{a,\gotp_1}\right)
\end{equation} 
for all $  \,\uu \,:\, \| \uu\|_{a,\gotp_1} \le \epsilon_0$.
Actually $\Psi$ is {\em tame modulus} in the sense of \cite{BG}, namely it respects interpolation bounds also for the higher order derivatives.

  Finally
 $\Phi$ preserves ${\bf h}_{\rm odd}^{a,p}\cap \calU$ and  the new vector field  $\Phi_{*}\chi:= \Upsilon $ restricted to $\calU$ is reversible and has the form
 \begin{equation}\label{weak2}
 \begin{aligned}
 & 
 \ol{\Upsilon^{-}}=\Upsilon^{+},\\
 \Pi_{S}\Upsilon^{+}&:=-\ii(v^+_{xx}+A(\uu)+B^+_{1}(\uu)),\\
\Pi_{S}^{\perp}\Upsilon^{+}&:=-\ii(z^+_{xx}+Q(\uu)+B^+_{2}(\uu)),
 \end{aligned}
 \end{equation}
 where 
 \begin{equation}\label{weak4}
 \begin{aligned}
 B^\s_{1}(\uu)&:= 
 \sum_{i,j,k\in S}C_{ikj}u^+_{i} u_i^- u^+_{k} u_k^-u^\s_{j}\del_{u^\s_{j}}+
 \sum_{j\in S}\sum_{\substack{k\in S^{c}, 
}
}
\chi_{kkjj}u^{+}_{k}{u}^{-}_{k}u^{\s}_{j}\del_{u^{\s}_{j}},
+\sum_{q=0}^{3}R(v^{q}z^{5-q})
+\Pi_{S}h^{(>5)}(u)),\\
B_{2}^\s(\uu)&:=\sum_{q=0}^{4}R(v^{q}z^{5-q})
+\Pi_{S}^{\perp}h^{(>5)}(u)),
\end{aligned}
 \end{equation}
 $h^{(>5)}$ collects all terms with degree greater than $5$, and for\footnote{extending polynomials outside $\calU$ is trivial just apply $u\to u^+$ and $\bar u\to u^-$.} $\uu\in \calU$
 \begin{equation}\label{weak3}
 \begin{aligned}
 A(u)&:=\sum_{j\in S}C_{j}^{j}|u_{j}|^{2}u_{j}\del_{u_{j}}+
\sum_{j\in S}(\sum_{\substack{i\in S \\ k\neq j}} C_{j}^{k}|u_{k}|^{2})u_{j}\del_{u_{j}}, \quad
Q(u)=\sum_{j\in S^{c}}\sum_{\substack{j_{1}-j_{2}+j_{3}-j=0, \\
(j_{1},j_{2},j_{3},j)\notin \calA_1}}
\chi_{j_{1}j_{2}j_{3}j}u_{j_{1}}\bar{u}_{j_{2}}u_{j_{3}}\del_{u_{j}},\\
C_{j}^{j}&:=\chi_{jjjj}, \qquad C_{j}^{k}:=\chi_{kkjj}+\chi_{jkkj},\\
\chi_{j_{1}j_{2}j_{3}j}&:=\ta_{1}-\ta_{2}j_{3}^{2}+\ta_{3}j_{1}j_{2}
-\ta_{6}j_{2}^{2}-\ta_{7}j_{1}j_{2}
-\ta_{4}j_{1}j_{2}j_{3}^{2}
+\ta_{8}j_{1}j_{2}^{2}j_{3}
-\ta_{5}j_{1}^{2}j_{2}^{2}j_{3}^{2}
 \end{aligned}
 \end{equation} 
 \end{proposition}
 \begin{proof}
 
 Now consider the equation \eqref{6.1}. As notation for a vector field $F$
 we denote by $F_{j_{1}\ldots j_{2k+1} j}$ the coefficient of the monomial $u^{+}_{j_{1}}u^{-}_{j_{2}}\ldots u^{+}_{j_{2k+1}}\del_{u^{+}_{j}}$.
 We divide $$\chi(\uu)= \NN+ \chi_{3}(\uu)+\chi_ {5}(\uu)+\chi_{>5}(\uu)$$
 a direct computation gives
 $$
 \NN= \ii \sum_{j}  j^2 u_j^+ \partial_{u_j^+} - \ii  \sum_{j}  j^2 u_j^- \partial_{u_j^-}\,
 $$
 while 
 $$
 \chi_{3} (\uu)=  -\ii \sum_{j\in\ZZZ}\sum_{\substack{j_{1}-j_{2}+j_{3}=j \\ 
 j_{i}\in\ZZZ, i=1,2,3}}\chi_{j_{1}j_{2}j_{3}j}u^+_{j_{1}}{u}^-_{j_{2}}u^+_{j_{3}}
 \del_{u^+_{j}}+ \ii \sum_{j\in\ZZZ}\sum_{\substack{j_{1}-j_{2}+j_{3}=j \\ 
  j_{i}\in\ZZZ, i=1,2,3}}\bar \chi_{j_{1}j_{2}j_{3}j}u^-_{j_{1}}{u}^+_{j_{2}}u^-_{j_{3}}
  \del_{u^-_{j}}
 $$ with $\chi_{j_{1}j_{2}j_{3}j}\in \RRR$  defined in \eqref{weak3}, comes from $f$. The other terms collect respectively the part of degree $5$ and $>5$ of $g$.
 
%
%
%

We want to eliminate  from $\chi$ 
 all the monomials $u^+_{j_{1}}{u}^-_{j_{2}}u^+_{j_{3}}
 \del_{u^+_{j}}$ 
or  $u^-_{j_{1}}{u}^+_{j_{2}}u^-_{j_{3}}
   \del_{u^-_{j}}$ such that the list $(j_1,j_2,j_3,j)\in \calA_1\cap \mathtt N$. Note that this is a finite set of monomials. 

We define the transformation $\Phi^{(1)}$ as the time$-1$ flow map generated by the vector field
\begin{equation}\label{seiaut}
\begin{aligned}
 F_{3} (\uu)=  -\sum_{\s=+1,-1} \s\sum_{j\in\ZZZ}\sum_{\substack{j_{1}-j_{2}+j_{3}=j \\ 
 (j_1,j_2,j_3,j)\in \calA_1\cap \mathtt N}}\frac{\chi_{j_{1}j_{2}j_{3}j}}{(j_{1}^{2}-j_{2}^{2}+j_{3}^{2}-j^{2})}u^\s_{j_{1}}{u}^{-\s}_{j_{2}}u^\s_{j_{3}}
 \del_{u^\s_{j}}\\
  -\sum_{\s=+1,-1} \s\sum_{j\in S}\sum_{\substack{j_{1}-j_{2}+j_{3}=j \\ 
 (j_1,j_2,j_3,j)\in (\calA_1)^{c}\cap \mathtt N}}\frac{\chi_{j_{1}j_{2}j_{3}j}}{(j_{1}^{2}-j_{2}^{2}+j_{3}^{2}-j^{2})}u^\s_{j_{1}}{u}^{-\s}_{j_{2}}u^\s_{j_{3}}
 \del_{u^\s_{j}}\\
%
\end{aligned}
\end{equation}
By construction  the transformation $\Phi^{(1)}$ has finite rank.
Moreover one has $$
\chi_{j_{1}j_{2}j_{3}j}\in \RRR, \quad \chi_{(-j_{1})(-j_{2})(-j_{3})(-j)}=\chi_{j_{1}j_{2}j_{3}j},
$$
hence the vector field $F_3$ in \eqref{seiaut} is reversibility preserving.

By construction the push-forward of the vector field $\Phi^{(1)}_{*}\chi:=\calY$ has the form
$\calY=\NN+\calY_{3}+\calY_ {5}+\calY_{>5}$, where $\calY_3$ contains only monomials $u^\s_{j_{1}}{u}^{-\s}_{j_{2}}u^\s_{j_{3}}
 \del_{u^\s_{j}}$  such that $(j_1,j_2,j_3,j)$ is either a trivial resonance or is not in $\calA_1$
or $j\in S$ and at least two among $j_1,j_2,j_3$ are in $S^{c}$ (see the second summand in $B_1^{\s}$ in \eqref{weak4})
 .  The trivial resonances  in $\calA_1$ give  $A(u)$, all the other terms either contribute to $B_1$ or to $Q$.
 More explicitly
  In this way the system 
$\dot{\bf u}=\calY({\bf u})$
possesses an invariant subspace $H_{S}$ and its  dynamics 
is integrable and, as we will see, non-isochronous.

 In order to enter a perturbative regime we need to cancel further term from the vector field. In particular we look for a transformation $\Phi^{(2)}$ such that the field ${\Upsilon}:=\Phi_{*}^{(2)}\calY$
does not contain monomials  $u^\s_{j_{1}}{u}^{-\s}_{j_{2}}u^\s_{j_{3}}u^{-\s}_{j_{4}}u^\s_{j_{5}}
 \del_{u^\s_{j}}$ 
 such that the list $(j_1,j_2,j_3,j_4,j_5,j)\in \calA_1\cap \mathtt N$, as in degree three  this is a finite set of monomials.  $\Phi^{(2)}$ is the time 1 flow of the vector field $F_5$ 
 of the form
 \begin{equation}\label{setteaut}
 F_{5} (\uu)=  - \sum_{j\in\ZZZ}\sum_{\substack{j_{1}-j_{2}+j_{3}-j_4+j_5=j \\ 
 (j_1,j_2,j_3,j_4,j_5,j)\in \calA_1\cap \mathtt N}}\frac{\calY_{j_{1}j_{2}j_{3}j_4j_5j}}{(j_{1}^{2}-j_{2}^{2}+j_{3}^{2}-j_{4}^{2}+j_{5}^{2}-j^{2})}u^+_{j_{1}}{u}^-_{j_{2}}u^+_{j_{3}}{u}^-_{j_{4}}u^+_{j_{5}}
 \del_{u^+_{j}}+ $$
 $$+  \sum_{j\in\ZZZ}\sum_{\substack{j_{1}-j_{2}+j_{3}-j_4+j_5=j  \\ 
 (j_1,j_2,j_3,j_4,j_5,j)\in \calA_1\cap \mathtt N}}\frac{\ol{\calY}_{j_{1}j_{2}j_{3}j_{4} j_5j} }{(j_{1}^{2}-j_{2}^{2}+j_{3}^{2}-j_{4}^{2}+j_{5}^{2}-j^{2})}u^{-}_{j_{1}}{u}^+_{j_{2}}u^-_{j_{3}}{u}^+_{j_{4}}u^-_{j_{5}}
  \del_{u^-_{j}}.
\end{equation}

Again by construction $\Phi^{(2)}$ has finite rank. Moreover
since ${\calY}$ is reversible then $F_{5}$ is reversibility preserving. Finally ${\Upsilon}:=\Phi^{(2)}_{*}\calY=\NN+\calY_3+\Upsilon_{5}+\Upsilon_{>5}$  
contains only monomials such that  $(j_1,j_2,j_3,j_4,j_5,j)$ is either a resonance or is not in $\calA_1$. By 
Lemma \ref{hyp3autaut} all the resonances in $\calA_{1}$ are trivial and hence contribute to the first summand in $B_{1}$. 
Now we perform the last step in order to cancel out from $\calY_3$

For the tame estimates \eqref{tamemod} we refer to \cite{BG}.
\end{proof}
\begin{rmk}\label{quasimolti}
Note that, by construction the change of variables written on functions $H^p(\TTT_a)$ is
$$
{\bf u}(x) \rightsquigarrow {\bf q}(x)= {\bf u}(x)+\sum_{j\in E} \Psi_j(\Pi_E {\bf u}) e^{i j x}\,
$$
hence setting $ {\bf u}=\Phi^{-1}({\bf q})= {\bf q}(x)+\sum_{j\in E} \tilde\Psi_j(\Pi_E {\bf q}) e^{i j x}$ one has 
$$
\Upsilon({\bf q}):= d \Phi ( \Phi^{-1}({\bf q}) )\chi(\Phi^{-1}({\bf q}))=  -\ii E \left[{\bf u}_{xx}
+
\left(\begin{matrix} \ff^{+}({\bf u},{\bf u}_{x},{\bf u}_{xx})\\ 
{\ff^{-}({\bf u},{\bf u}_{x},{\bf u}_{xx})}
\end{matrix}
\right)\right]+ d \Psi ( {\bf u} )\chi({\bf u})\,.
$$
Then ( see Lemma 7.1 of \cite{BBM1} for a detailed proof)
$$
d_{\bf q}\Upsilon({\bf q})= -\ii E \left[ \del_{xx}+ d_{\bf u}\calF ({\bf u}) \right] +\RR({\bf q})
$$
where the first term is described in \eqref{skifo}, while  for some fixed $N$
$$
\RR({\bf q})[h]= \sum_{m=1}^N \left(h(x),a^{(m)}({\bf q};x)\right)_{L^2}b^{(m)}({\bf q};x)\,.
$$
Here $a^{(m)},b^{(m)}$ are functions in $h^{a,p} $ depending on $\bf q$ and such that, for any $m$, one of the $a^{(m)},b^{(m)}$ is equal to $e^{\ii j x}$ for some $j\in E$. In other words $\RR$ is a linear operator sum of two terms, one maps $\Pi_E h^{a,p}$ into $h^{a,p}$ and the other maps $h^{a,p}$ into $\Pi_E h^{a,p}$.
\end{rmk}

\zerarcounters

\section{Action-angles variables}\label{action}
In the previous paragraph we have worked in the Fourier basis and we have shown that the change of variables preserves ${\bf h}^{a,p}_{\rm odd}$. Now we restrict our vector field to   ${\bf h}^{a,p}_{\rm odd}$
defined in  \eqref{phacespace}. In the present setting however 
it will be more convenient to express  ${\bf h}^{a,p}_{\rm odd}$ over $\NNN$ by passing to 
the ``sine''  Fourier basis. 

\noindent
We want to switch the tangential variables to polar coordinates. We set
\begin{equation}\label{aacaut}
\begin{aligned}
2u^\s_{\pm\tv_i}&:=\pm  \s\frac{1}{2\ii}    \sqrt{\x_i+y_{i}}e^{\s \ii \theta_{i}}, &\quad i=1,\ldots,d\\
u^\s_{j}&:=\s\frac{\text{sign}(j)}{2\ii}  z^\s_{|j|}, &\quad j\in S^{c},
\end{aligned}
\end{equation}
this is a well defined , analytic  change of variables for $\xi_i >0$, $|y_i|<\xi_i$.
Our phase space is hence $\TTT^{d}_{s}\times \CCC^{d}\times \ell_{a,p}$. Here
\begin{equation}\label{lap}
\ell_{a,p}\equiv \Pi_S^\perp{\bf h}_{\rm odd}^{a,p}  , \qquad \|w\|^{2}_{a,p}:=\sum_{l\in S^{c}}\langle l\rangle^{2p}|w_{l}|^{2}e^{2a|l|}, \;\; w=(z_{j},\bar{z}_{j})_{j\in S^{c}},
\end{equation}
(see \eqref{phacespace})  is a sequence space indexed in $S^{c}:=\NNN\setminus S^{+}$
with a Hilbert structure w.r.t. the norm in \eqref{lap}.

In the new variables $\calU$ becomes
\begin{equation}\label{est66}
\calU:=\{ (\theta,y,w)\in \CCC^{2d}\times \ell_{0,0}: \quad  \theta\in \TTT^d\,,\quad y\in \RRR^d\,,\quad {z}^-= \bar{z}^+\}.
\end{equation}

For $\e$ small we consider $\x\in \e^{2}\Lambda=\e^{2}[1/2,3/2]^{d}$ and  the domain \begin{equation}\label{domain}
D_{a,p}(s_0,r_0):=\left\{\theta\in\TTT^{d}_{s_{0}},\; |y|\leq r_{0}^{2}, \|w\|_{a_0,\gotp_1}\leq r_{0}
\right\}\subseteq \TTT^{d}_{s_{0}}\times\CCC^{d}\times \ell_{a,p}.
\end{equation}   
One can check that there exist constant $c_{1}$ and $c_{2}$ such that, if
\begin{equation}\label{condizione}
r_{0}<c_{1}\e< \e/2, \qquad \sqrt{2d}c_{2}\ka^{\gotp_1}e^{s+a\ka}\e < \epsilon_{0}, \quad \ka:={\rm max}(|\tv_{i}|), 
\end{equation}
then one has $\Phi^{(\x)} : D_{a,p+\nu}(s_{0},r_{0})\to B_{\epsilon_0}$, where  the vector field $\Upsilon$ is well defined.
We assume that our  parameters $\e,r_{0},s_{0}$
 satisfy \eqref{condizione} so that we can apply $\Phi_{\x}$ to our vector field. 
 In the new variables 
 one has $\uu(x)={\bf v}(\theta,y;x)+w(x)$ with
 \begin{equation}\label{variabili}
 {\bf v}= (v^+, v^-)\,,\quad v^\pm:=\sum_{i=1}^d\sqrt{\x_{i}+y_{i}}e^{\pm \ii \theta_{i}}\sin(\tv_i x), \quad w=(z^+,z^-)\,, z^\s=\sum_{j\in S^{c}}z^\s_{j}\sin(j x),
 \end{equation}
 and 
 \begin{equation}\label{campototale}
 F:=  \Phi^{(\x)}_*\Upsilon = F^{(\theta)}(\theta,y,w)\partial_\theta+  F^{(y)}(\theta,y,w)\partial_y +  F^{(w)}(\theta,y,w)\partial_w,
 \end{equation}
reads
\begin{equation}\label{system}
\begin{aligned}
F^{(\theta_k)}(\theta,y,w)&= \frac{\Upsilon_{{\tv_k}}^+}{2 \ii v_{{\tv_k}}^+}- \frac{\Upsilon_{{\tv_k}}^-}{2 \ii v_{{\tv_k}}^-} =
{\tv_k}^{2}-(\MM(\x+y))_{k}-\frac{e^{-\ii \theta_k}(B_{1}^+)_{{\tv_k}} +e^{\ii \theta_k}(B_{1}^-)_{{\tv_k}}}{\sqrt{\x_{k}+y_{k}}}&\quad k= 1,\dots,d
,\\
F^{(y_k)}(\theta,y,w)&=4 v_{\tv_k}^- \Upsilon_{{\tv_k}}^+ + 4 v_{\tv_k}^+ \Upsilon_{\tv_k}^-= 2\ii  \sqrt{\x_{k}+y_{k}} (e^{-\ii \theta_k}(B_{1}^+)_{{\tv_k}} -e^{\ii \theta_k}(B_{1}^-)_{{\tv_k}}), &\quad k= 1,\dots,d\\
F^{(w)}(\theta,y,w) &)= 2 \Pi_S^\perp \Upsilon 
\end{aligned}
\end{equation}
where $\MM$ is the twist matrix 
\begin{equation}\label{omeghino2}
\begin{aligned}
\MM_{k h}&=\frac14 \big(C_{\tv_{k}}^{\tv_{h}}+C_{\tv_k}^{-\tv_{h}}\big)\quad \text{ for }\quad k,h=1,\ldots,d \,,\quad \tv_k,\tv_h\in S^{+}.
\end{aligned}
\end{equation}
We define ${\oo}^{(0)}\in \RRR^{d}$ the vector of \emph{unperturbed} frequencies as
\begin{equation}\label{omeghino}
{\oo}^{(0)}_j={\oo}^{(0)}_{j}(\x):=\la^{(-1)}_{j}+\la^{(0)}_j(\x), \qquad
 \la^{(-1)}_{j}:=j^{2}, \quad \la^{(0)}_{j}(\x):=-(\MM\x)_{j}, \; j\in S^{+}.
\end{equation}

We set
\begin{equation}\label{ini1}
N_0:={\oo}^{(0)}(\x)\cdot\del_{\theta}+\Omega^{(-1)}w\del_{\theta},
\end{equation}
where $(\Omega^{(-1)})^{\s}_{\s}=i\s {\rm diag}\, j^{2}$, $(\Omega^{(-1)})_{\s}^{-\s}=0$. 
With this notation 
$
F= N_0+G$ has an approximately invariant torus.


\section{Nonlinear functional setting}\label{iniz}

 We set
 $$
V_{a,p} := \CCC^d \times \CCC^d \times \ell_{a, p}\,.
$$

We shall need two parameters, $\gotp_0<\gotp_1$. Precisely $\gotp_0>d/2$ is needed in order to have the
Sobolev embedding and thus the algebra properties, while $\gotp_1$ will be chosen very large and is
needed in order to define the phase space.

\begin{defi}[Phase space]\label{spazio} Given $\gotp_1$ large enough,
	we consider the toroidal domain
	\begin{equation}\label{Dsr}
	\TTT^d_s \times D_{a,p}(r) :=\TTT^d_s\times  B_{r^2} \times \B_{r,a,p,\gotp_1}\,,
	\subset V_{a,p}
	\end{equation}
	where 
	\begin{equation}\nonumber
	\begin{aligned}
	&\TTT^d_s := \big\{ \theta \in \CCC^d \, : \, {\rm Re}(\theta)\in\TTT^{d},\ \max_{h=1, \ldots, d} |{\rm Im} \, \theta_h | < s
	\big\} \,, \\
	&B_{r^2}  := \big\{ y \in \CCC^d \, : \, |y |_1 < r^2 \big\}\,, \qquad
	\B_{r,a,p,\gotp_1}:=\big\{w\in \ell_{a,p}\,:\, \|w\|_{a,\gotp_1}<r\big\}\,,
	\end{aligned}
	\end{equation}
	and we denote by $ \TTT^d := ( \RRR/2\pi \ZZZ)^d $ the $ d $-dimensional torus.
\end{defi}
\noindent
We denote the set of variables $\mathtt{V} := \big\{ \theta_1,\ldots, \theta_d, y_1,\ldots, y_d, w  \big\}$.
Fix some numbers $s_0,a_0\geq 0$ and $r_0>0$. Given $s\leq s_0$, $a,a'\leq a_0$, $r\leq r_0$,
$ p,p'\geq \gotp_0$  consider maps
\begin{equation}\label{faglimale}
\begin{aligned}
f : \TTT^d_s \times D_{a',p'}(r)  & \to V_{a,p} \\
(\theta,y,w) &\to ( f^{(\theta)}(\theta,y,w), f^{(y)}(\theta,y,w),f^{(y)}(\theta,y,w)),
\end{aligned}
\end{equation}
with
\begin{equation*}
f^{(\mathtt v)}(\theta,y,w) = \sum_{l\in \ZZZ^d} f^{(\mathtt v)}_l(y,w)e^{\ii l \cdot \theta}\,,\quad
\mathtt v\in{\mathtt V} \,, 
\end{equation*}
where $f^{(\mathtt v)}(\theta,y,w) \in \CCC$ for $\mathtt v=\theta_i,y_i$ while $f^{(w)}(\theta,y,w) \in  \ell_{a,p}$.
We shall  use also the notation 
$$
f^{(\theta)}(\theta,y,w),f^{(y)}(\theta,y,w)\in\CCC^d.$$
Note that, by \eqref{domain} and \eqref{condizione}, the field $F$ in \eqref{system} has the same form of \eqref{faglimale}.

In order to study the properties of the vector field $F$ we first need to introduce some notations.

\noindent
We define a norm (pointwise  on $y,w$)\ by setting
\begin{equation}\label{totalnorm}
\|f\|_{s,a,p}^2:=\|f^{(\theta)}\|^2_{s,p}+\|f^{(y)}\|^2_{s,p}+\|f^{(w)}\|^2_{s,a,p}
\end{equation}
where
\begin{equation}\label{equation}
\|f^{(\theta)}\|_{s,p}:=\left\{\begin{aligned} &
\frac{1}{s_{0}}\sup_{i=1,\ldots,d}\|f^{(\theta_{i})}(\cdot,y,w)\|_{H^p(\TTT^d_s)}
\quad s\leq s_0\neq 0,
\\
&\sup_{i=1,\ldots,d}\|f^{(\theta_{i})}(\cdot,y,w)\|_{H^{p}(\TTT^{d})}\,,\quad s=s_0=0\end{aligned}\right.
\end{equation}
\begin{equation}\label{equation2}
\|f^{(y)}\|_{s,p}:=\frac{1}{r_0^2}\sum_{i=1}^{d}\|f^{(y_{i})}(\cdot,y,w)\|_{H^{p}(\TTT^{d}_{s})}
\end{equation}
\begin{equation}\label{equation3}
\|f^{(w)}\|_{s,a,p}
:=\frac{1}{r_0}\left[\sum_{ l \in \ZZZ^{d},j\in S^{c}}\langle l,j\rangle^{2\gotp}|(f^{(w)}_{l}(y,w))_{j}|^{2}e^{2s|l|}e^{2a |j|}
\right]^{\frac{1}{2}}
\end{equation}
where $H^{p}(\TTT^{d}_{s})=H^{p}(\TTT^{d}_{s}; \CCC)$ is the standard Sobolev space with norm
\begin{equation}\label{normasob}
\|u(\cdot)\|_{H^{p}(\TTT^{d}_s)}^{2}:=\sum_{l\in\ZZZ^{d}}|u_l|^{2}e^{2s|l|}\langle l \rangle^{2p}, \qquad \av{ l}:=\max\{1,|l|\}.
\end{equation}
Note that trivially $\|\partial_\theta^{p'} u\|_{H^{p}(\TTT^{d}_s)}= \|u\|_{H^{p+p'}(\TTT^{d}_s)}$. 


\begin{rmk}\label{orolo}
Note that, since in this case $\ell_{a,p}=\Pi_S^\perp{\bf h}_{\rm odd}^{a,p}$ 
then fixing $\gotp_0 \geq(d+1)/2$ we have that 
$\|\cdot \|_{s,a,p}$ in 
\eqref{equation3} 
is nothing but the norm of the Sobolev space 
$H^{p}(\TTT_{s}^{d}\times\TTT_{a})$. 
In particular one can check that 
such norm is equivalent to the one introduced in \cite{CFP}.
\end{rmk}

\noindent
It is clear that any $f$ as in \eqref{faglimale} can be identified with ``unbounded'' vector fields
by writing
\begin{equation}\label{vectorfield}
f \leftrightarrow \sum_{\mathtt{v}\in{\mathtt V}}f^{(\tv)}(\theta,y,w)\del_{\mathtt{v}},
\end{equation}
where the symbol $f^{(\tv)}(\theta,y,w)\del_{\tv}$ has the obvious meaning for $\tv=\theta_i,y_i$ while
for $\tv=w$ is defined through its action on differentiable functions $G:\ell_{a,p}\to \CCC$ as
$$
f^{(w)}(\theta,y,w)\del_w G := d_w G[ f^{(w)}(\theta,y,w)].
$$

Similarly, provided that $|f^{(\theta)}(\theta,y,w)|$ is small for all $(\theta,y,w)\in  \TTT^d_s \times D_{a,p}(r)  $
we may lift $f$ to a map \begin{equation}\label{mar}
\Phi:= (\theta +f^{(\theta)},y+f^{(y)}, w+ f^{(w)}) : \, \TTT^d_s \times D_{a',p'} \to
\TTT^d_{s_1} \times \CCC^d\times \ell_{a,p}\,,\quad \text{for some} \; s_1\ge s\,,
\end{equation}
and if we set $\|\theta\|_{s,a,p} := 1$  we can write
$$
\|\Phi^{(\tv)}\|_{s,a,p} = \|\tv\|_{s,a,p}+\|f^{(\tv)}\|_{s,a,p} \,,\; \tv=\theta,y,w\,.
$$
Note that 
$$
\|y\|_{s,a,p}= r_0^{-2}|y|_1\,,\quad  \, \|w\|_{s,a,p}=r_0^{-1}\|w\|_{a,p}.
$$ 

\begin{rmk}\label{azz}
	Note that if there exists ${\mathtt c}={\mathtt c}(d)$ such that if
	$\|f\|_{s,a,\gotp_1}\le {\mathtt c}\rho$ one has
	$$
	\Phi: \TTT^d_s \times D_{a+\db a_0,p}(r ) \to \TTT^d_{s+\rho s_0}\times D_{a,p}(r +\rho r_0).
	$$
\end{rmk}

 We are interested in vector fields defined on a scale of Hilbert spaces; precisely
 we shall fix $\db,\nu,q\ge0$ and consider vector fields
     \begin{equation}\label{vettorazzi}
  F:\TTT^d_s\times D_{a+\db a_0,p+\nu}(r)\times \calO\to V_{a,p}\,,
  \end{equation}
  for some $s<s_0$, $a+\db a_0\le a_0$, $r\le r_0$ and all $p+\nu \le q$. Moreover we require that $\gotp_1$ in Definition \ref{spazio} satisfies $\gotp_1\ge \gotp_0+\nu+1$.
  \begin{rmk}\label{parametronu}
Here $\nu$ represents the loss of regularity of the field $F$.
In the NLS case (for $F$ in \eqref{system}) 
one has $\nu=2$. We shall give same definition for generic $\nu\geq0$
since we shall also need to deal with \emph{bounded} vector field, i.e. $\nu=0$.
\end{rmk}

  \begin{defi}\label{leftinverse}  
Fix $0\le\rho,\db\le1/2$, and consider two  differentiable maps $\Phi= \uno +f$, $\Psi= \uno +g$ as in
\eqref{mar} such that for all
$p\ge \gotp_0$,  $2\rho s_0 \le s \le s_0$, $2\rho r_0 \le r \le r_0$ and $0\leq a\leq a_0(1-2\db )$ one has
\begin{equation}\label{laputtana250}
\Phi,\Psi:\TTT^d_{s-\rho s_0} \times D_{a+\db a_0,p}(r-\rho r_0)\to \TTT^d_{s}\times D_{a,p}(r ).
\end{equation}
If
\begin{equation}\label{sinistra}
\begin{aligned}
\uno=\Psi\circ \Phi: \TTT^d_{s-2\rho s_0}\times D_{a+2\db a_0,p}&(r-2\rho r_0)&\longrightarrow &\
\TTT^d_s\times  D_{a,p}(r)\\
&\ (\theta,y,w)&\longmapsto &\ (\theta,y,w)
\end{aligned}
\end{equation}
we  say that $\Psi$ is a left inverse of $\Phi$ and write $\leftinv{\Phi}:=\Psi$.

Moreover fix  $\nu\geq 0$, $0\le\db'\le1/2$. Then for any 
$F : \TTT^d_{s}\times D_{a+\db' a_0,p+\nu}(r) \to V_{a,p}$, with $0\leq a\leq a_0(1-2\db -\db')$, 
we define the ``pushforward'' of $F$ as
\begin{equation}\label{push}
\Phi_* F:= d\Phi(\leftinv\Phi)[F(\leftinv{\Phi})]: \TTT^d_{s-2\rho s_0} \times D_{a+(2\db+\db') a_0,p+\nu}(r-2\rho r_0)  \to 
 V_{a,p}\,.
\end{equation}
\end{defi}

 We need to introduce  parameters $\x\in\calO_0$ a compact set in $\RRR^d$. Given any compact  $\calO\subseteq\calO_0$ 
 we consider  Lipschitz families of vector fields 
   \begin{equation}\label{vettori}
  F:\TTT^d_s\times D_{a',p'}(r)\times \calO\to V_{a,p}\,,
  \end{equation}
  and say  that they  are \emph{bounded} vector fields when  $p=p'$ and $a=a'$.
  Given a positive number $\g$ we introduce  the weighted Lipschitz norm
 \begin{equation}\label{ancoralip}
 \|F\|_{\vec v,p}=\|F\|_{\g,\calO,s,a,p}:=\sup_{\xi\in \calO}\|F(\x)\|_{s,a,p}
 + \g \sup_{\xi\neq \eta\in \calO}\frac{\|F(\x)-F(\eta)\|_{s,a,p-1}}{|\xi-\eta|}\,.
 \end{equation}
 and we shall drop the labels $\vec v=(\g,\calO,s,a)$ when this does not cause confusion. 
 
  \begin{defi}\label{vecv}
  We shall denote by
 $\VV_{\vec v,p}$  with $\vec v= (\g,\calO,s,a,r)$ the space of  vector fields as in \eqref{vettorazzi}  with $\rho=0$. 
 By slight abuse of notation we denote the norm $\|\cdot\|_{\g,\calO,s,a,p}=\|\cdot\|_{\vec v,p}$ .
 \end{defi}

\begin{rmk}\label{projector}
Note that we have a projector operator defined on 
the whole space
$\CCC^{2d}\times \ell_{a,p}$.
On the space $\ell_{a,p}=\Pi_{S}^{\perp}{\bf h}_{\rm odd}^{a,p}$ one has the projector
 given by 
 $$
 \Pi_{\ell_{K}}w=\{w_{j}\}_{|j|\leq K}.
 $$
 Note that the space $\ell_{a,p}$ we have now defined satisfies Hypothesis $1$ in \cite{CFP}.
\end{rmk}  

\subsection{Polynomial decomposition}\label{poly}

In $\VV_{\vec v,p}$ we identify the closed {\em monomial} subspaces
\begin{equation}\label{sotto}
\begin{aligned}
&\mathcal V^{(\tv,0)} := \{f\in \VV_{\vec v,p}\,:\; f= f^{(\tv,0)}(\theta) \partial_\tv\} \,, \quad \tv\in \mathtt V \\ & 
\mathcal V^{(\tv,\tv')} := \{f\in \VV_{\vec v,p}\,:\; f= f^{(\tv,\tv')}(\theta)[\tv']  \partial_\tv \}\,,\quad
 \tv\in\mathtt V\,,\quad \tv' \in {\mathtt U}:=\{y_1,\ldots,y_d,w\},
\end{aligned}
\end{equation}

As said after \eqref{faglimale}  it will be convenient to use also vector notation so that,
for instance
$$
f^{(y,y)}(\theta)y\cdot\del_y \in\VV^{(y,y)}=\bigoplus_{i\le j=1,\ldots,d}
\VV^{(y_i,y_j)}
$$
with $f^{(y,y)}(\theta)$ a $d\times d$ matrix.

Note that the polynomial vector fields of degree $1$ are
\begin{equation}\label{poly2}
\calP_1:=
\bigoplus_{\tv\in{\mathtt V}}\,
\bigoplus_{\tv_1\in{\mathtt U}\cup\{0\}}\mathcal V^{(\tv,\tv_1)}\,,
\end{equation}
so that, given a polynomial $F\in\calP_1$ we may define its ``projection'' onto a monomial subspace
$\Pi_{\mathcal V^{(\tv,\tv_1)}}$ in the natural way.

Since we are not working on spaces of polynomials, but on vector fields with finite regularity,
we need some more notations.
Given a $C^{2}$ vector field $F\in\VV_{\vec v,p}$, we introduce the notation
\begin{equation}\label{aiuto}
F^{(\tv,0)}(\theta):=F^{(\tv)}(\theta,0,0), \quad F^{(\tv,\tv')}(\theta)[\cdot]:=d_{\tv'}F^{(\tv)}(\theta,0,0)[\cdot], \quad  \tv\in\mathtt V\,,\quad \tv'= y_1,\ldots,y_d,w.
\end{equation}

By Taylor approximation formula any vector field in $\VV_{\vec v,p}$ which is  $C^{2}$ in $y,w$
may be written in a unique way as
sum of its Taylor polynomial in $\calP_1$ plus a $C^{2}$ (in $y,w$)
vector field with a zero of order at least $2$ at $y=0$, $w=0$. 
We think of this as a direct sum of vector spaces and introduce 
the notation
\begin{equation}\label{pro}
\Pi_{\VV^{(\tv,\tv_1)} }F:= F^{(\tv,\tv_1)}(\theta)[\tv_1]\,,
\end{equation}
 we refer to such operators as {\em projections}.
%
%
 
 \begin{defi}
We identify the vector fields in $\VV_{\vec v,p}$ which are $C^{2}$ in $y,w$, with the direct sum
 $$
\calW_{\vec v,p}^{(1)}= \calP_1\oplus \RR_1\,,
 $$
 where $\RR_1$ is the space of $C^{2}$ (in $y,w$) vector fields with a zero of order at least $2$
 at $y=0$, $w=0$. 
On $\calW_{\vec v,p}^{(1)}$ we induce the natural norm for direct sums, namely for 
$$
f=\sum_{\tv\in{\mathtt V}}\,
\left(
f^{(\tv,0)}(\theta)+
\sum_{\tv_1\in{\mathtt U}}
f^{(\tv,\tv_1)}(\theta)[\tv_1]\right)\del_\tv + f_{\RR_1}\,
\qquad f_{\RR_k}\in\RR_{1}\,,
$$
we set
\begin{equation}\label{supernorma}
\|f\|_{\vec v,p}^{(1)}:=\sum_{\tv\in{\mathtt V}}\,
\left(\|f^{(\tv,0)}\|_{\vec{v},p}+
\sum_{\tv_1\in{\mathtt U}}
\|f^{(\tv,\tv_1)}(\cdot)[\tv_1]\|_{\vec v,p}\right) + \|f_{\RR_1}\|_{\vec v,p}\,.
\end{equation}

\end{defi}

  We can and shall introduce in the natural way the polynomial subspaces and the norm \eqref{supernorma}
  also for maps 
  $\Phi= (\theta +f^{(\theta)},y+f^{(y)}, w+ f^{(w)})$ with
  $$
  \Phi: \TTT^d_s\times D_{a',p'}(r)\times \calO \to \TTT^d_{s_1}\times D_{a,p}(r_1)\,,
  $$
  since the Taylor formula holds also for functions of this kind.
  
We also denote
\begin{equation}\label{lemedie}
\begin{aligned}
&\av{\VV{(\tv, \tv_{{1}})}}:=\!\{f\in \VV^{(\tv,\tv_1)}\!:  
f =  \langle f^{(\tv,\tv_1)}\rangle\cdot
 \partial_\tv\},\\
&\VV^{(\tv, \tv_{{1}})}_0:=\{f\in \VV^{(\tv, \tv_1)}\!:  
f =  (f^{(\tv, \tv_{{1}})}-\langle
 f^{(\tv, \tv_{{1}})}\rangle)\cdot \partial_\tv\}\,,
 \end{aligned}
\end{equation}
where $ \av{f}:=\frac{1}{(2\pi)^{d}}\int_{\TTT^d}f(\theta)d\theta$.
  \medskip
  
\noindent  {\bf Tame vector fields}.\,
     We now define vector fields  behaving ``tamely'' 
  when composed with maps $\Phi$. 
  In order to simplify the notation, from now on we set
   \begin{equation}\label{porc}
    \|f\|_{\vec v,p}=\|f\|_{\vec v,p}^{(1)}\,.
    \end{equation}
  
   \begin{defi}\label{tame}
   Fix a large $q\geq \gotp_1$,  and  a set $\calO$.
   Consider a $C^{5}$ vector field
     $$
     F\in\calW_{\vec v,p}^{(1)},\qquad   \vec v= (\g,\calO,s,a,r)\,.
     $$
    We say that $F$ is $C^3$-\emph{ tame} (up to order $q$) 
    if there exists a scale of  constants 
   $C_{\vec{v},p}(F)$,  with $C_{\vec{v},p}(F)\le C_{\vec{v},p_1}(F)$ for $p\le p_1$,
   such that  the following holds.
   
   For all  $\gotp_0\leq p\leq p_1\leq q$ consider any $C^{3}$ map
   $\Phi= (\theta +f^{(\theta)},y+f^{(y)}, w+ f^{(w)})$
   with
   $$
   \Phi:\TTT^d_{s'}\times  D_{a_1,p_1}(r')\times \calO \to \TTT^d_s\times 
   D_{a,p +\nu }(r)\,,\quad \text{ for some } \quad r'\leq r\,, s'\leq s;
   $$   
   and denote
    $ \vec v'= (\g,\calO,s',a,r')$.
   Then

\begin{itemize}
\item[(i)]
   for any $m=0,\ldots,3$ and any $m$ vector fields    
  \begin{equation}\label{miao}
   h_1,\dots,h_m:\; \TTT^d_{s'}\times  D_{a_1,p_1}(r')\times \calO  \to V_{a,p +\nu},
   \end{equation}
   one has  that the symmetric $m$-linear map $d^{m}_{\mathtt U}F(\Phi)$ satisfies (see \eqref{sotto} for the definition of $\mathtt{U}$):
  \begin{equation*}
  \begin{array}{crcl} 
    (\text{T}_m) & 
    \|d^{m}_{\mathtt U}F(\Phi)[h_1,\dots,h_m]\|_{\vec{v}',p} &\leq & 
     \big(C_{\vec{v},p}(F)+C_{\vec{v},\gotp_0}(F)\|\Phi\|_{\vec{v}',{p+\nu}} \big)
   \prod_{j=1}^m\|h_j\|_{\vec{v}',\gotp_0+\nu} \\ [.5em]
   & & &
    +C_{\vec{v},\gotp_0}(F)\sum_{j=1}^m\|{h_j}\|_{\vec{v}',{p+\nu}}\prod_{i\neq j}\|h_i\|_{\vec{v}',{\gotp_0+\nu}}
   \end{array}
   \end{equation*}
  for all $(y,w)\in D_{a_1,p_1}(r')$ and $p\leq q$.
   Here $d_{\mathtt U}F$ is the differential of $F$ w.r.t. the variables ${\mathtt U}:=\{y_1,\ldots,y_d,w\}$
   and the norm is the one defined in \eqref{porc}.

\item[(ii)]  
 For $m=1,2,3$ and given $h_{1},\ldots,h_{m-1}$ as in \eqref{miao},
   consider the linear  maps  $D_{m} : V_{a,p+\nu}\to V_{a,p}$ defined by
   \begin{equation}\label{miao2}
   h\mapsto D_{m}[h]:=d^{m}_{\mathtt{U}}F(\Phi)[h_{1},\ldots, h_{m-1},h],
   \end{equation}
   set moreover 
   $$X^{p}:=H^{p}(\TTT^{d}_{s};\CCC^{d}\times\ell_{a,\gotp_0})\cap 
H^{\gotp_0}(\TTT^{d}_{s};\CCC^{d}\times\ell_{a,p})$$ and 
$$Y^{p}:=
H^{p}(\TTT^{d}_{s};V_{a,\gotp_0})\cap 
H^{\gotp_0}(\TTT^{d}_{s};V_{a,p})$$
for $p\geq\gotp_1$.
We require\footnote{We recall that, given a linear operator $A:X\to Y  $ its adjoint is  $A^*:Y^* \to X^*$. Our condition implies that $(d_{\mathtt U}F(\Phi))^{*}$ is bounded from $Y_1\to X_1$, with $Y_1\subset Y^*$ and $X_1\subset X^*$ 
this is hence a  much stronger condition.}
 \begin{equation}\label{carota}
  \begin{array}{crcl} 
    (\text{T}_m)^* & 
    \|D_{m}^{*}[v]\|_{\g,\calO,X^{p-\nu}} &\leq & 
%
      \big(C_{\vec{v},p}(F)+C_{\vec{v},\gotp_0}(F)\|\Phi\|_{\vec{v}',{p+\nu}} \big)
   \prod_{j=1}^{m-1}\|h_j\|_{\vec{v}',\gotp_0+\nu}
     \|v\|_{\g,\calO,Y^{\gotp_0}}
    \\ 
   & &+ &
    C_{\vec{v},\gotp_0}(F)\sum_{j=1}^{m-2}\|{h_j}\|_{\vec{v}',{p+\nu}}\prod_{i\neq j}
    \|h_i\|_{\vec{v}',{\gotp_0+\nu}}  \|v\|_{\g,\calO,Y^{\gotp_0}}\\
    & & + &  C_{\vec{v},\gotp_0}(F)\prod_{i=1}^{m-1}
    \|h_i\|_{\vec{v}',{\gotp_0+\nu}}  \|v\|_{\g,\calO,Y^{p}}
   \end{array}
   \end{equation}
   \end{itemize}
   
      \noindent
   We  call $C_{\vec{v},p}(F)$ the $p$-\emph{tameness constants of }$F$.
   
    We say that a \emph{bounded} vector field $F$ is adjoint tame  if the conditions  (T$_m$) -$(T_m)^*$
   above hold with $\nu=0$.
   \end{defi}
  \begin{rmk}\label{dioniso}
  Note that in the definition above appear two regularity indices: $3$ being the maximum regularity in $y,w$ and $q$ the one in $\theta$.
  Note that in the $w-$component the norm $\|\cdot\|_{X^{p}}$ is equivalent to the norm $\|\cdot\|_{s,a,p}$. 
  \end{rmk}

\subsection{Normal form decomposition}\label{normi}
In this Section we introduce a suitable decomposition of our vector fields. 

\begin{defi}[{\bf $(\NN,\calX,\RR)-$decomposition}]\label{decomponi}
 \begin{equation}\label{sottoREV}
 \NN:= \VV^{(\theta,0)}\oplus \VV^{(w,w)}\,,
 \quad \calX:=  \VV^{(y,0)}\oplus \VV^{(y,y)} \oplus \VV^{(y,w)}\oplus \VV^{(w,0)}.
 \end{equation}
  We then decompose
  $$
  \calW_{\vec v,p}=C^{5}\cap\calW^{(1)}_{\vec v,p}:=\NN\oplus \calX\oplus\RR
  $$
  where $C^{5}$ is the set of vector fields with $(5)$-regularity in $y,w$,
   $\RR$ contains all of $\RR_{1}$ and all the polynomials generated by monomials not in
  $\NN\oplus \calX$. We shall denote $\Pi_\RR:= \uno-\Pi_\NN-\Pi_\calX$ and
  more generally for $\SSSS= \NN,\calX,\RR$ we shall denote  $ \Pi^\perp_{\SSSS}:= \uno-\Pi_\SSSS$.

\end{defi}

 In order to apply the Abstract KAM Theorem of \cite{CFP} it remains to introduce a suitable subspace of vector field which is called $\calE$, accordingly with the notations
 used in \cite{CFP}.
 
 We first introduce the notion of \emph{Pseudo-differential } vector field.

 \begin{defi}[{\bf Pseudo-differential vector fields}]\label{pseudopseudo}
   We say that a  vector field $F$ is of Schr\"odinger  pseudo-differential type if there exists $\mathtt{N}>0$ such that its differential in a neighborhood of 
   $u_0=(\theta,0,0)$
   has the form, 
 \begin{equation}\label{pippone}
d_w F^{(w)}(u)[\cdot] =  \mathscr{P}\;(u)[\cdot] +{\mathscr{K}}\;(u)[\cdot]
\end{equation}
where
\begin{equation*}
\begin{aligned}
{\mathscr P}(u)&= 
 -\ii E\; \Pi_S^\perp\left[\left(\begin{matrix} 1+a_{2}(u;x) & b_{2}(u;x) \\ \bar{b}_{2}(x) & 1+\bar{a}_{2}(u;x)\end{matrix}
 \right)\del_{xx} +\left(\begin{matrix} a_{1}(u;x) & b_{1}(u;x) \\ \bar{b}_{1}(x) & \bar{a}_{1}(u;x)\end{matrix}
  \right)\del_{x}\right] \Pi_S^\perp\\
  &\qquad -\ii E\;\Pi_S^\perp
   \left[\left(\begin{matrix} a_{0}(u;x) & b_{0}(u;x) \\ \bar{b}_{0}(u;x) & \bar{a}_{0}(u;x)\end{matrix}
   \right)\right] \Pi_S^\perp
 \end{aligned}
 \end{equation*}
 maps  $\ell_{a,p}\to \ell_{a,p}$  with $a_{i}(u),b_{i}(u)\in H^{p}(\TTT^{d}_s\times\TTT_a;\CCC)$ for $i=0,1,2$. Similarly
 \begin{equation}\label{maiale}
\mathscr{K}({u})[h^{(w)}]=\sum_{m=1}^{\mathtt{N}} ( c_{m},  h^{(w)})_{L^{2}} d_m\, 
\end{equation}
is a linear operator  $\ell_{a,p}\to \ell_{a,p}$ of finite rank equal to $N$
with $(\cdot,\cdot)_{L^{2}}$ the usual $L^{2}$ scalar product on $\TTT$ (note that $\ell_{a,p}$ is identified with odd functions in ${\bf h}^{a,p}$)
  where
   and 
  $c_l\in H^{p-2}(\TTT^{d}_s\times\TTT_a;\CCC)$, $d_l\in H^{p}(\TTT^{d}_s\times\TTT_a;\CCC)$ for $l=1,\ldots, \mathtt{N}$.
  Finally we require
  \begin{equation}\label{labellabel}
  \begin{aligned}
  \|a_{i}\|_{\vec{v},p}&\leq C(F)(1+\|u\|_{\vec{v},p+2})\\
  \|d_{u}a_{i}[h]\|_{\vec{v},p}&\leq C(F)(\|h\|_{\vec{v},p+2}+\|u\|_{\vec{v},p+2}\|h\|_{\gotp_0+2}), 
  \end{aligned}
  \end{equation}
  for some constant depending on $F$. The same holds for the other coefficients.
%
%
 \end{defi}
 We remark that the condition that $d_{w}F^{(w)}(u)[\cdot]$ maps $\ell_{a,p+2}$ to $\ell_{a,p}$
 implies some parity conditions in $x\in\TTT$ on the coefficients. 

 \begin{defi}[{\bf The subspace $\calE$}]\label{compatibili}
  Fix  $\vec{v}_0:=(\g_0,\calO_0,s_0,a_0)$ with $\calO_0$
  a compact subset of $\RRR^{d}$.

 We denote by $\calE$ the subset  of $F\in \VV_{\vec{v}_0,p}$ 
 such that

 for $(\theta,y,w)$ restricted to $\calU$ in \eqref{est66},  is
 
 \begin{itemize} 
 \item[] {\sc Tame}: $F$ is tame according to Definition \ref{tame};
 
 \item[] {\sc Pseudo differential}: $F$ is a  pseudo differential vector field according to Definition \ref{pseudopseudo};
 \item[] {\sc Gauge preserving}: the vector field $F$ commutes with $X_{M}:=\sum_{i}\del_{\theta_i}+\ii z\del_{z}-\ii\bar{z}\del_{\bar{z}}$;
 \item[]{\sc Reversible}:  one has
 \begin{equation}\label{est5}
 \begin{pmatrix}
  F^{(\theta)}(-\theta,y,(\bar{z},z))\\F^{(y)}(-\theta,y,(\bar{z},z))\\
  F^{(z^+)}(-\theta,y,(\bar{z},z))\\F^{(z^-)}(-\theta,y,(\bar{z},z))
  \end{pmatrix}
  = - \begin{pmatrix}
  -  F^{(\theta)}(\theta,y,(z,\bar{z}))\\F^{(y)}(\theta,y,(z,\bar{z}))\\
    F^{(z^-)}(\theta,y,(z,\bar{z}))\\F^{(z^+)}(\theta,y,(z,\bar{z}))
    \end{pmatrix},
    \end{equation}
%
%
\item[] {\sc Real-on-Real}: one has,
\begin{equation}\label{est55}
F^{(\theta)}(\theta,y,w)=\ol{F^{(\theta)}(\theta,y,w)}, \;\;
F^{(y)}(\theta,y,w)=\ol{F^{(y)}(\theta,y,w)},\;\;  F^{(z^{+})}(\theta,y,w)=\ol{F^{(z^{-})}(\theta,y,w)}, 
\end{equation}
%
\end{itemize}
  \end{defi}

\begin{rmk}\label{silentnight}
First of all one can note that
the component $F^{(\theta)}$ is even in the variables $\theta$ while $F^{(y)}$ is odd in $\theta$. By condition \eqref{est5}
we see that the field $F$ is \emph{reversible} with respect to the involution
\begin{equation}\label{est4}
S : (\theta_{j},y_{j},z_{j},\bar{z}_{j}) \mapsto (-\theta_{j},y_{j},\bar{z}_{j},z_{j}), \quad j\in \NNN, \quad S^{2}=\uno,
\end{equation}
on the subspace $\calU$.
 \end{rmk}

  \begin{rmk}\label{massamassa}
The condition that $F$ is Gauge preserving is equivalent to requiring that $\Psi_{\tau}^* F=F$ with 
$$
\Psi_\tau (\theta,y,z^+,z^-) = (\theta+ \tau , y,e^{\ii\tau} z^+, e^{-\ii\tau } z^- ). 
$$
Given a map $\Phi$ and a Gauge preserving vector field $F$, then  we recall that $\Phi_* F$ is Gauge preserving if  we have that $$\Phi(\Psi_\tau u)= \Psi_\tau \Phi(u) .$$
We can also check weather
$\Phi$ is the time one flow of a vector field $g$ which is Gauge preserving.
 We write $F$ as in
\eqref{vectorfield}
 ( written using Fourier series in the variables $\theta$)
  \begin{equation}\label{lamassaconserva}
F= \sum_{l,h,\alpha,\beta,\tv}F^{(x)}_{l,h,\al,\beta} e^{\ii \theta\cdot l} y^h z^\al\bar z^\beta \partial_\tv \,,\quad \tv= \theta, y, z^\s_j
\end{equation}
then it 
is Gauge preserving iff
\begin{equation}\label{conservadipomodori}
\sum \ell_i+\sum(\al_j-\be_j)=\left\{\begin{aligned}
&0,\qquad \tv=\theta,y \\
&\s,\qquad \tv=z_j^{\s}
\end{aligned}
\right.
\end{equation}

  \end{rmk}

 We now introduce some  special classes of linear vector fields.
  \begin{defi}[{\bf Finite rank vector fields}]\label{linvec}
%
For $\nu\geq0$
 we say that a vector field $f: \TTT^d_s\times D_{a,p+\nu}(r )\times\calO\to V_{a,p}$ such that 
 \begin{equation}\label{maremma}
  f = \sum_{v\in \mathtt{U}} f^{(v,0)}\del_v + (f^{(y,y)} y +f^{(y,w)}\cdot w )\cdot\del_y \,,\quad 
  f^{(y_i,w)}\in \ell_{a,p-\nu} \,,
 \end{equation}
 (recall \eqref{sotto})
 is \emph{regular} and we set
 \begin{equation}\label{linnorm}
 |f|_{\vec v,p}:=\sum_{u=y,w} \|f^{(u,0)}\|_{\vec v,p} +
  \sup_{i,j=1,\dots,d}\|f^{(y_i,y_j)}\|_{\vec v,p}+ \sup_i\|f^{(y_i,w)}\|_{\vec v,p-\nu}.
 \end{equation}
  We  denote by $\calA=\calA_{\vec v,p}$ with $\vec v= (\g,\calO,s,a,r)$ the set 
  of finite rank vector field $f$ with finite $|\cdot|_{\vec{v},p}$ norm.
   
 We denote by  $\BB$ be the set of bounded vector fields ($\nu=0$) in $\calA_{\vec{v},p}\ni f: \TTT^d_s\times D_{a,p}(r )\to V_{a,p}$.
Finally we denote by $\BB_{\calE}$ the subset of $\BB$ such    
that flow $\Phi_{g}^{t}$ generated by $g\in \BB_{\calE}$ is well defined and $(\Phi_{g}^{t})_{*}$ maps $\calE\to\calE$.

 \end{defi}
 
 \noindent
 We have the following result.
%
%
%
%
%
  \begin{lemma}\label{prestige}
The set $\BB_{\calE}$ defined in \eqref{linvec} coincides with the vector fields $g\in \BB$ such that:
\begin{enumerate}
\item $g$ is Real-on-Real (see \eqref{est55});
\item $g$ satisfies 
 \begin{equation}\label{preservazione}
 g^{(y,0)}(-\theta)=g^{(y,0)}(\theta), \;\; g^{(y,y)}(-\theta)=g^{(y,y)}(\theta),\; 
 g^{(y,z^{+})}(-\theta)={g^{(y,z^{-})}(\theta)},  \; g^{(z^{+},0)}(-\theta)={g^{(z^{-},0)}(\theta)},
 \end{equation}
 when restricted to $\theta\in \TTT^{d}$,
 \item $g$ commutes with $X_{M}:=\sum_{i}\del_{\theta_i}+\ii z\del_{z}-\ii\bar{z}\del_{\bar{z}}$.
\end{enumerate}
%
We say that vector field  $g\in \BB_{\calE}$ is \emph{reversibility-preserving} or 
   equivalently $\calE-$preserving.
 \end{lemma}
 The proof of the Lemma above is postponed to Section \ref{seclin}.

\begin{defi}\label{projectortotale}
 Given $K>0$ and a vector field  $f\in\calA$ 
   we define the projection $\Pi_K f$  as
   \begin{equation}\label{duck33333}
 \begin{aligned}
 &(\Pi_{K}f^{(\tv,0)})(\theta):=\sum_{|\ell|\leq K}f^{(\tv,0)}_{\ell}e^{\ii \ell\cdot \theta}, \qquad \tv=\theta,y\,,\\
 &(\Pi_{K}f^{(w,0)})(\theta):=\sum_{|\ell|\leq K}\Pi_{\ell_{K}}f^{(w,0)}_{\ell}e^{\ii\ell\cdot \theta},
 \quad (\Pi_{K})f^{(y_i,y_j)}(\theta):=\sum_{|\ell|\leq K}f^{(y_i,y_j)}_{\ell}e^{\ii \ell\cdot\theta},i,j=1,\ldots,d\,,\\
 &(\Pi_{K}f^{(y_i,w)})(\theta):=\sum_{|\ell|\leq K}\Pi_{\ell_{K}}f^{(y_i,w)}_\ell e^{\ii \ell\cdot\theta}\,,
 \end{aligned}
 \end{equation}
 %
 %
 %
 (recall \eqref{projector}) and we define $E^{(K)}$ as the subspace of $\calA_{\vec v,p}$ where $\Pi_K$ acts as the identity.
 \end{defi}

We recall also the Definition of \emph{diagonal} vector fields.

\begin{defi}[{\bf Normal form}]\label{norm}
 We say that $N_{0}\in \NN$ is a {\em diagonal} vector field if for all $K>1$
\begin{equation}\label{norm1}
{\rm ad}(N_0)\Pi_{E^{(K)} } \Pi_\calX =\Pi_{E^{(K)} } \Pi_\calX  {\rm ad} (N_0).
\end{equation}
 \end{defi}

%
%
%
%
%
%
%
%
%
%

\subsection{Estimates on F}

In this Section we study the properties of the vector field $F$ introduced in \eqref{system}.

We have the following.

\begin{proposition}[{\bf Properties of $F$}]\label{Properties of $F$}

One has that the vector field  $F$ in \eqref{system} belongs to the subspace $\calE$ in \ref{compatibili}.
\end{proposition}

\begin{proof}
The field $F$ is tame according to Definition \eqref{tame}. This properties follows by 
the fact that the composition operator in \eqref{compoop} satisfies tame estiamtes (see \cite{FP} for more details)
and by the properties of the map $\Psi$ in Proposition \ref{wbnf}.

One has that 
$F$ restricted to $\calU$ is reversible w.r.t. the involution $S$ defined in \eqref{est4} and Real-on-real
(see \eqref{est5}, \eqref{est55}).
This follows by Hypothesis \ref{hyp2aut} (see also \eqref{weak2} in Proposition \ref{wbnf}).
One has that $F$ preserves the Mass again by Hypothesis \ref{hyp2aut}. 
The field $F$ in \eqref{system} is pseudo-differential  according to Definition \ref{pseudopseudo}.
Indeed, recalling the definition of $\chi({\bf u})$ in \eqref{6.666} and its differential in \eqref{skifo}, 
one has that 
the linearized operator in the $w-$direction $d_{w}F(\theta,y,w)$ has the form \eqref{skifo} up to
a finite rank operator of the form \eqref{maiale} for some $N:=\mathtt{N}_0$.
This is actually the form in 
 \eqref{pippone}. Recall also that ${\bf u}={\bf v}+w$ in \eqref{variabili}.
The estimates \eqref{labellabel} follow recalling Lemma $2.19$
in \cite{FP}.
\end{proof}

\noindent
Now we give estimates on the tameness constants of the field $F$.
We set $F_0:=F$ defined on the domain in \eqref{domain}
$D_{a_0,p+\nu}(s_0,r_0)$ where the parameters are given by formula \eqref{condizione}.
We define  $N_0$ as in \eqref{ini1} and notice that it  is \emph{diagonal} according to Definition \ref{norm}.
 Now we define
the vector fields $N_0^{(1)},N_0^{(2)}\in \NN$ as \eqref{6.666}
 \begin{equation}\label{ini3}
\begin{aligned}
N_0^{(2)}&:=\Pi_{S}^{\perp}\left(-i E\left(\begin{matrix} a^{(0)}_{2}(\theta,\x) & b^{(0)}_{2}(\theta,\x) \\ \bar{b}^{(0)}_{2}(\theta,\x) & \bar{a}^{(0)}_{2}(\theta,\x)\end{matrix}
\right)\del_{xx} \right)\Pi_{S}^{\perp}w\cdot\del_{w},\\
N_0^{(1)}&:=
\Pi_{S}^{\perp}\left(-i E\left(\begin{matrix} a^{(0)}_{1}(\theta,\x) & b^{(0)}_{1}(\theta,\x) \\ \bar{b}^{(0)}_{1}(\theta,\x) & \bar{a}^{(0)}_{1}(\theta,\x)\end{matrix}
\right)\del_{x} +
\left(\begin{matrix} a^{(0)}_{0}(\theta,\x) & b^{(0)}_{0}(\theta,\x) \\ \bar{b}^{(0)}_{0}(\theta,\x) & \bar{a}^{(0)}_{0}(\theta,\x)\end{matrix}\right)
\right)\Pi_{S}^{\perp}w\cdot\del_{w}
 \end{aligned}
 \end{equation}
 with coefficients $a_{i}^{(0)},b_{i}^{(0)}$ for $i=0,1,2$ given by  
 \begin{equation}\label{NF2}
\begin{aligned}
a^{(0)}_{2}(\theta,\x)&:=\mathtt{a}_{2}|v|^{2}+\mathtt{a}_{4}|v_{x}|^{2}+2\ta_{5}|v_{xx}|^{2},\quad 
b^{(0)}_{2}(\theta,\x):=    \ta_{6}|v|^{2}+ \ta_{8}(v_{x})^{2}+   \ta_{5}v_{xx}^{2}
\\
a^{(0)}_{1}(\theta,\x)&:=\mathtt{a}_{3}\bar{v}_{x}v+\mathtt{a}_{4}\bar{v}_{x}v_{xx}+2 \ta_{7}\bar{v}v_{x}+2 \ta_{8}v_x\bar{v}_{xx}
,\quad
b^{(0)}_{1}(\theta,\x):=\mathtt{a}_{3}vv_{x}+\mathtt{a}_{4}v_{x}v_{xx}
\\
a^{(0)}_{0}(\theta,\x)&:=2\mathtt{a}_{1}|v|^{2}+\mathtt{a}_{2}\bar{v}v_{xx}+\mathtt{a}_{3}|v_{x}|^{2}
+\ta_{6}\bar{v}\bar{v}_{xx}
,\quad
b^{(0)}_{0}(\theta,x):=\mathtt{a}_{1}v^{2}+\mathtt{a}_{2}vv_{xx}+\ta_{6}{v}\bar{v}_{xx}+\ta_{7}(v_{x})^{2},
\end{aligned}
\end{equation}
where $v$ is the function $v(\theta,y)$ in \eqref{variabili} evaluated at $y=0$ and $E$ defined in \eqref{6.666}.
We set
  \begin{equation}\label{ini2}
\begin{aligned}
 F&=N_0+N_0^{(1)}+N_0^{(2)}+H_{0},
 \end{aligned}
 \end{equation}
where $H_0$ is defined by difference. We can see that the terms $N_0^{(1)},N_0^{(2)}$ come from the cubic term $Q(u)$ defined in \eqref{weak3} while $H_0$ 
comes from the term $B_{2}$.
 
By definition, see \eqref{omeghino} have that ${\oo}^{(0)}$ is $\x-$close  to the integer
vector $\la^{(-1)}$,   hence we fix the size of $\g_0$ by requiring that ${\oo}(\x)$ satisfies 
\begin{equation}\label{ini4}
|{\oo}^{(0)}\cdot l|\geq\frac{\g_0}{\langle l\rangle^{\tau}}, \;\; \forall\, l\in\ZZZ^{d}, \, \tau\geq d+1, \, \g_0 = \mathtt{c} |\x|,
\end{equation}
with the constant $\mathtt{c}\ll 1$. 
%
Fix $\vec{v}_0:=(\g_0,\calO_0,s_0,a_0)$ with 
\begin{equation}\label{sottosottosotto}
\calO_0:=\e^{2}\Lambda=\e^{2}[1/2,3/2]^{d}.
\end{equation}

\noindent
In the following Lemma we analyze the size of the tameness constants of 
the field $F$.

\begin{lemma}\label{INIZIAMO}
Given $\gotp_1$ as in \eqref{Dsr}, fix a constant $\gotp_2>\gotp_1$ and $c_1>0$.
Let $r_0$ in \eqref{domain} such that
\begin{equation}\label{adige4}
|\x|\leq r_0\leq c_1|\x|^{\frac{1}{2}}.
\end{equation}
There exists constant $\gotA_0$, depending 
on the constants $\ta_i$,  for $i=1,\ldots,8$, in \eqref{NF2}, on $\gotp_2$
and on $\mathtt{c}$ in \eqref{ini4}, such that
the vector field $F$ in \eqref{ini2} satisfies 
\begin{equation}\label{ini5}
\begin{aligned}
&\g_{0}^{-1}C_{\vec{v}_{0},\gotp_2}(N_0^{(1)})\leq \gotA_0 , \qquad 
\g_{0}^{-1}C_{\vec{v}_{0},\gotp_2}(N_0^{(2)})\leq \gotA_0, \\
& \g_{0}^{-1}C_{\vec{v}_{0},\gotp_2}(\Pi_{\NN}{H}_0)\leq \gotA_0 |\x|,\quad
 \g_{0}^{-1}C_{\vec{v}_{0},\gotp_2}(\Pi_{\calX}{H}_0)\leq \gotA_0 |\x|^{\frac{1}{2}},\quad 
  \g_{0}^{-1}C_{\vec{v}_{0},\gotp_2}(\Pi_{\RR}{H}_0)\leq \gotA_0 |\x|^{\frac{1}{2}}.
\end{aligned}
\end{equation}

\end{lemma}

\begin{proof}
By definition each term in \eqref{ini2} is tame, now we want to estimate their tameness constant in order  to prove \eqref{ini5}.
Recall that $F_0$ in \eqref{system} is defined in terms of equations \eqref{weak2}, \eqref{weak4} and \eqref{weak3}.
We start to study the terms $N_0^{(1)}$ and $N_0^{(2)}$ that are terms that contribute to $F^{(w)}$ that comes form $Q$ in \eqref{weak3} (such term is linear in $w$ and quadratic in $v$).
For instance we can bound  using the interpolation properties of the norm $\|\cdot\|_{s,a,p}$
\begin{equation*}\label{ini7}
\g_{0}^{-1}\|a_2 |v|^{2}\del_{xx}z\|_{\vec{v}_0,\gotp_2}\leq\frac{1}{\mathtt{c} |\x|}\left(C(\gotp_2)|\x|+C(\gotp_0)|\x|\|z\|_{\vec{v}_0,p+\nu}
\right),
\end{equation*}
where we used that $\|z\|_{a,\gotp_1}\leq r_0$.
Hence  one can check Definition \eqref{tame} with a constant $C_{\vec{v}_0,p}(a_{2}|v|^{2}z_{xx})\leq \gotA_0=\gotA_{0}(\gotp_2,\mathtt{c})$.
All the other terms in \eqref{NF2} can be estimated in the same way. Indeed all those terms are quadratic on $v$ and linear in $z$.
Recall that the norm $\|\cdot\|_{\vec{v}_0,p}$ is a weighted norm and on the $w-$ component the weight is $r_0$ (see \eqref{totalnorm}).
This
determines the constant $\gotA_0$ by setting 
\begin{equation*}\label{ini77}
\g_0^{-1} C_{\vec{v}_0,\gotp_2}(N_0^{(1)}), \g_0^{-1} C_{\vec{v}_0,\gotp_2}(N_0^{(2)})\leq \frac{\gotM_0}{\cc}=\gotA_0(\gotp_2,\mathtt{c}).
\end{equation*}
Now let us estimate the several components of the field ${H}_0$ starting from $\Pi_{\NN}{H}_{0}$.
Consider  $\Pi_{\NN}{H}_0^{(w)}$. All these terms comes form $B_{2}^{+}$ in \eqref{weak2} so that
the linear term in $z$ is at least of degree $4$ in $v$. Hence one has 
\begin{equation*}\label{ini8}
\g_0^{-1}C_{\vec{v}_0,\gotp_2}(\Pi_{\NN}({H}_0^{(w)}))\leq \gotA_{0}|\x|.
\end{equation*}
Note that $\gotA_0$ here could be in principle larger that $\gotA_0$ in the estimates of $N_0^{(1)}$. In this case one choose, with abuse of notation the largest constant for $\gotA_0$.
We now note that $\Pi_{\NN}{H}_0^{(\theta)}:=H_{0}^{(\theta,0)}(\theta)$ is of the form \eqref{system} where $B_1$ is defined in \eqref{weak4}
where the term independent on $z$ of degree minimum has degree $5$ in $v$ hence
\begin{equation*}\label{ini9}
\g_0^{-1}C_{\vec{v}_0,\gotp_2}(\Pi_{\NN}({H}_0^{(\theta)}))\leq \gotA_{0}|\x|.
\end{equation*}
This implies the first bound on $H_0$ in \eqref{ini5}.
Let us study $\Pi_{\calX}{H}_0$. 
By Lemma \ref{miaomiao} we know that each field in $\calE\cap \calX$ is tame with tameness constant controlled by the norm in \eqref{linnorm}.
We have by equation \eqref{weak4} that $\Pi_{\calX}{H}_0^{(w)}=\Pi_{\calX}\Pi_{S}^{\perp}h^{(>5)}(u)$, hence the term
${H}_{0}^{(w,0)}(\theta)$ is of degree at least $6$ in $v$. We have
\begin{equation}\label{ini10}
\g_0^{-1}C_{\vec{v}_0,\gotp_2}(\Pi_{\calX}{H}_{0}^{(w)})\leq \gotA_0 |\x|^{2} r_0^{-1}.
\end{equation}
Now by \eqref{system} and \eqref{weak4} one has
\begin{equation}\label{ini11}
\g_{0}^{-1}|\Pi_{\calX}{H}_0^{(y)}|_{\vec{v}_0,\gotp_2}\leq \gotA_0 |\x|^{2}\sqrt{\x}r_0^{-2}+\gotA_{0}|\x|+\gotA_0|\x|^{2}r_0^{-1}
\end{equation}
To get bound \eqref{ini11} we used 
the fact that the only terms of degree $5$ in $v$ in $B_1$ are integrable,  as one can see in \eqref{weak4}. Substituting in \eqref{system},  such terms of degree $5$ cancel out
in  the $y$ component of $F_0$.
Collecting the bounds \eqref{ini10} and \eqref{ini11} we have, by \eqref{condizione} and \eqref{adige4}
\begin{equation}\label{piccolezzamille}
\g_{0}^{-1}|\Pi_{\calX}{H}_0|_{\vec{v}_0,\gotp_2}\leq \gotA_0 |\x|^{\frac12}.
\end{equation}
Now we study $\Pi_{\RR}{H}_0$. The component $\Pi_{\RR}{H}_0^{(\theta)}$
 is at least  linear in the variables $y,w$ while the terms in $\Pi_{\RR}{H}^{(y)}_0$ comes from the last two summand in \eqref{weak4} (this follows by the definition 
 of $R^{(1)}$ that collects the terms coming form the second summand and the fact that the integrable terms of order $5$ are zero).
  Following the same reasoning of the previous bounds we get
\begin{equation}\label{ini12}
\g_0^{-1}C_{\vec{v}_0,\gotp_2}(\Pi_{\RR}{H}^{(\theta)}_0)\leq \gotA_0
|\x|^{\frac{3}{2}}, \qquad \g_0^{-1}C_{\vec{v}_0,\gotp_2}(\Pi_{\RR}{H}^{(y)}_0)\leq \gotA_0{|\x|}^{\frac{1}{2}},
\qquad \g_0^{-1}C_{\vec{v}_0,\gotp_2}(\Pi_{\RR}{H}^{(w)}_0)\leq \gotA_0|\x|
\end{equation}
without requiring any additional hypotheses on $r_0$. Again in the second bound
we use that the integrable terms in $B_1$ in \eqref{weak4} cancel out. 
Collecting together the bounds in  \eqref{ini12}
we get the \eqref{ini5}.
\end{proof}
%
%
\begin{rmk}
Note that we have separated the
terms $N_0^{(1)},N_0^{(2)}$ that are not ``perturbative'' with respect to the size of the small divisors $\g_0\approx |\x|$. 
\end{rmk}

We have the following Lemma.

\begin{lemma}\label{linearizedatu}
Let $F_0$ be  the field in \eqref{system} and set  $F_0=N_0+G_0$  with $N_0$ in \eqref{ini1}.
Consider 
$d_{w}(F_{(0)})^{(w)}(u)[\cdot]=\mathscr{P}_0+\mathscr{K}_0$ 
the  linearized in a point $u=(\theta,y,w)$ in a neighborhood of $u_0=(\theta,0,0)$.
Then one has that
\begin{equation}\label{datu}
\mathscr{P}_0(u)[\cdot]=(N_0+N_0^{(1)}+N_0^{(2)})^{(w)}[\cdot]+
\Pi_{S}^{\perp}\left(-i E\left[A(\theta,y,w,\x)\right]\right)\Pi_{S}^{\perp},
\end{equation}
where $N_0,N_0^{(1)},N_0^{(2)}$ are defined in \eqref{ini1},\eqref{ini3}, while $A(\theta,y,w,\x)$
has the form \eqref{pippone} with coefficients which satisfy \eqref{labellabel}
with $C(F)\rightsquigarrow C_{\vec{v},\gotp_2}(H_0)\sim \gotA_0|\x|^{\frac{1}{2}} \g_0$ (see Lemma \ref{INIZIAMO}). 
\end{lemma}

\begin{proof}
By definition of $F_0$ in \eqref{ini2} 
and \eqref{ini3}
one has that $F_0(\theta,y,w)=N_0+N_0^{(1)}(\theta,w)+N_0^{(2)}(\theta,w)+H_0(\theta,y,w)$, and hence, recalling that $N_0^{(1)},N_0^{(2)}$ are linear in the $w$ variables, one gets
$(\Pi_{\NN}F_0)^{(w)}=\Omega^{-1}(\x)[\cdot]+
(N_0^{(1)}(\theta,w)+N_0^{(2)})^{(w)}+d_{w}H_0^{(w)}(\theta,y,w)[\cdot]$.
We have that the term $d_{w}H_0^{(w)}(\theta,y,w)[\cdot]$ has the form \eqref{pippone}
thanks Lemma \eqref{Properties of $F$}.
The $\|\cdot\|_{\vec{v},p}$ norm of its
coefficients $a_i,b_i$ with $i=0,1,2$, $c_{i},d_i$ with $i=1,\ldots,N$
can be estimated by follows the same procedure used in Lemma \ref{INIZIAMO}.
One gets the estimates \eqref{labellabel} with $C(F)\rightsquigarrow \gotA_0|\x|^{\frac{1}{2}}\g_0$ which is equivalent to $C_{\vec{v},\gotp_2}(H_0)$ given in Lemma \ref{INIZIAMO}.
Hence Lemma \ref{linearizedatu} follows.
\end{proof}

\zerarcounters

\section{An Abstract KAM Theorem }
\label{secNMaut}
\noindent
In this Section we 
show the the nonlinear functional setting
introduced in Section \ref{iniz}
allows us to prove a KAM result completely analogous to the 
 abstract Theorem 
 proved in \cite{CFP}.
  In order to state the result 
 we need further notations.
Recalling Lemma \eqref{INIZIAMO} we set
\begin{equation}\label{parametri}
\mathtt{G}_0:=\gotA_0, \quad \e_0:=|\x|^{\frac{1}{4}}, \quad
\mathtt{R}_0:=\gotA_0\e_0^{\frac{1}{2}},
\end{equation}
To resume with this notation we have the following bounds on the field $F_0$ in \eqref{ini2}:
\begin{equation}\label{cartabianca}
\!\!\!\g_0^{-1}C_{\vec{v},\gotp_2}(N^{(1)}), \g_0^{-1}C_{\vec{v},\gotp_2}(N^{(2)})\leq \tG_0, \quad
\g_0^{-1}C_{\vec{v},\gotp_2}(H_{0}^{(\theta,0)}), \g_0^{-1}C_{\vec{v},\gotp_2}(\Pi_{\calX}H_0)\leq \e_0, \quad 
\g_0^{-1}C_{\vec{v},\gotp_2}(\Pi_{\NN}^{\perp}H_0)\leq \tR_0.
\end{equation}

We need to 
introduce parameters fulfilling the following constraints.
\begin{const}[{\bf The exponents}]\label{sceltapar}
We fix parameters $\mu_1,\eta,\ka_1,\ka_2,\ka_3,\gotp_1,\gotp_2$ such that the following holds.
\begin{itemize}

\item $\mu_1,\ka_3\ge 0$, $\gotp_2>\gotp_1\geq\gotp_0:=(d+1)/2$, 
\item $\mu:= 5(\mu_1+7)$,
\item Setting $\ka_0:= \mu+6$  and $\Delta\gotp:= \gotp_2-\gotp_1$  one has
\begin{subequations}\label{exp}
\begin{align}
 \ka_1 &> \max(\frac{2}{3}(\ka_0+\ka_3), 2{\ka_0},6\ka_0+\ka_2+1)\,,
 \label{exp1} \\
 \ka_2&> \max\big(4\ka_0, (\ka_0 +2  \max(\ka_1,\ka_3)-\frac{3}{2} \ka_1)\big) \,,
 \label{exp2}\\
 \eta &> \mu+\frac{1}{2} \ka_2+1\,,
 \label{exp3} \\
 \Delta\gotp&>\max\big(\ka_0+\frac{3}{2} \ka_2+\max(\ka_1,\ka_3),(5\ka_0+\frac{3}{2} \ka_2+\ka_1+1)\big)\,,
 \label{exp4}
\end{align} 
\label{eq106}
\end{subequations}
\end{itemize}

\end{const} 
Note the Definition \eqref{sceltapar} is exactly the ``Constraints'' $2.21$ in \cite{CFP} with
$\al=0$, $\nu=2$ and $\chi=3/2$. 
We set 
\begin{equation}\label{lamai}
K_0= \e_0^{-\mathtt{a}} \mathtt{G}_0\,,\qquad \mathtt{a}=\frac{1}{4\ka_0+\ka_2+1},
\end{equation}
with  $\mathtt{G}_0$ defined in \eqref{parametri}.

%
%
%

\begin{lemma}[{\bf Smallness conditions}]\label{smallcondi}
Consider $\e_0,\mathtt{G}_0,\mathtt{R}_0$ in \eqref{parametri} and 
$K_0$ in \eqref{sceltapar}. One has that, 
for $|\x|$ small enough, the following bounds hold: 
\begin{subequations}\label{expexp}
\begin{align}
& 0<\e_0\le\tR_0\le \tG_0, \quad \e_0\tG_0^3,\e_0 \tG_0^2 \tR_0^{-1}<1\label{sss111}\\
&\tG_0^2\tR_0^{-1}\e_0 K_0^{\ka_0}\max (1, 	\tR_0 \tG_0 K_0^{\ka_0+\frac{1}{2}\ka_2})<1\,,
\label{1s1}\\
&\max( K_0^{\ka_1},\e_0 K_0^{\ka_3})  K_0^{\ka_0-\Delta\gotp +\frac{1}{2} \ka_2 }\tG_0\e_0^{-1} 
\le 1\,,
 \label{4s2}\\
& \max( K_0^{\ka_1},\e_0 K_0^{\ka_3})K_0^{\ka_0-\frac{3}{2} \ka_1}\tG_0 \tR_0^{-1}
\le 1\,,
 \label{6s2}
 \end{align}
 \label{small}
 \end{subequations}
\end{lemma}

\begin{proof}
 Let us check the \eqref{sss111}-\eqref{6s2} using 
 \eqref{parametri}. Condition \eqref{sss111} is implied by 
 $$
 \sqrt{\e_0}\gotA_0^{3}\leq1,
 $$
 which holds for $\e_0$ small enough. The \eqref{1s1} reads
 $$
\gotA_0^{3}\e_0K_0^{2\mathtt{b}}{=}\gotA_0^{3+\frac{1}{\mathtt{a}}}K_0^{2\mathtt{b}-\frac{1}{\mathtt{a}}}\leq 1, \quad \mathtt{b}=2\ka_0+\frac{1}{2}\ka_2,
 $$
 which is satisfied thanks to the choice of $\ta$.
The other condition in \eqref{1s1} follows in the same way. Consider condition \eqref{4s2}.
We must have that 
$$
\gotA_0^{1-\frac{1}{\mathtt{a}}}K_0^{\ka_1+\ka_0+\frac{1}{2}\ka_2+\frac{1}{a}-\Delta\gotp}\leq 1,\qquad \gotA_0K_0^{\ka_{3}+\ka_{1}+\ka_{0}+\frac{1}{2}\ka_2-\Delta\gotp}\leq 1,
$$
which is true thanks condition \eqref{exp4} which implies that $\Delta\gotp$ is large enough. The last condition \eqref{6s2} follows in the same way.
\end{proof}

\begin{rmk}
We remark that conditions \eqref{sss111}-\eqref{6s2} are
 the smallness conditions in $(2.44)$ of Constraints $2.21$ in  \cite{CFP}.
\end{rmk}

\noindent
We have the following definition.
\begin{defi}[{\bf Mel'nikov conditions}]\label{pippopuffo3} 
Let $\gamma,\mu_1> 0$, $K\ge K_0$, consider a compact set $\calO \subset \calO_0$ and set
$\vec{v}=(\g,\calO,s,a,r)$ and $\vec{v}^{\text{\tiny 0}}=(\g,\calO_0,s,a,r)$.
Consider a  vector field $F\in \calW_{\vec v^{\text{\tiny 0}},p}$  i.e.
$$
F= N_0+G: \calO_0\times    D_{a,p+\nu}(r)\times\TTT^{d}_{s}\to V_{a,p}\,, 
$$
 which is  $C^{3}$-tame up to order $q=\gotp_2+2$.
   We say $\mathcal O$  satisfies the Mel'nikov conditions  for
   $(F,K,\vec{v}^{\text{\tiny 0}})$ if  
    the following holds.

 \noindent
1.\, For all $\xi\in \mathcal O$   one has  $F(\xi)\in \calE$  and $|\Pi_{\calX}G|_{\vec{v},\gotp_{2}-1}\leq \mathtt{C}\, C_{\vec{v},\gotp_{2}}(\Pi_{\NN}^{\perp}G)$.
 
 \noindent
2.\, Setting $\gotN:= \Pi_K\Pi_\calX {\rm ad }(\Pi_\NN F)$ for all $\xi\in \calO$ there exists a block-diagonal  operator 
$\gotW: E^{(K)}\cap\calX\cap\calE  \to E^{(K)}\cap \BB_\calE $ such that 
 for any vector field $X\in E^{(K)}\cap\calX\cap\calE $

 \begin{itemize}

  \item[(a)] one has that the vector field $g:=\gotW X$ satisfies
 \begin{equation}\label{buoni}
 \begin{aligned}
|\gotW X|_{\vec v,p}&\le
\gamma^{-1}K^{\mu_1}(|X|_{\vec v,p}
+|X|_{\vec v,\gotp_1}
\g^{-1}C_{{\vec v,p}}(G)).
\end{aligned}
 \end{equation}
 
 \item[(b)] setting
 $
u:=(\Pi_{K}{\rm ad}(\Pi_{\NN}F)[\gotW X]-X)
 $ 
one has 
  \phantom{assssasfffff}
  \begin{equation}\label{cribbio4}
  \begin{aligned}
  |u|_{\vec{v},\gotp_1}&\leq \e_0 \g^{-1} K^{-\h+\mu_1} C_{\vec{v},\gotp_1}(G)|X|_{\vec{v},\gotp_1},\\
    |u|_{\vec{v},\gotp_2}&\leq \g^{-1} K^{\mu_1}\left(
    |X|_{\vec{v},\gotp_2}C_{\vec{v},\gotp_1}(G)+|X|_{\vec{v},\gotp_1}C_{\vec{v},\gotp_2}(G)
    \right).
\end{aligned}
  \end{equation}
  \end{itemize}
\end{defi}

\noindent
We set moreover
$r_0>0$ and $a_0, s_0\ge0$, 
 we set for all $\g_0,\mathtt{G}_0,\mathtt{R}_0>0$,
\begin{equation}\label{numeretti}
\begin{aligned}
& K_n= (K_0)^{(3/2)^n},\, 
\g_n= \g_{n-1}(1-\frac{1}{2^{n+2}}), \quad
\mathtt{G}_n=\mathtt{G}_0(1+\sum_{j=1}^{n}2^{-j}),\quad 
\mathtt{R}_n=\mathtt{R}_0(1+\sum_{j=1}^{n}2^{-j})\\
& a_n= a_0(1-\frac{1}{2}\sum_{j=1}^n2^{-j})\,,
 \quad r_n= r_0(1-\frac{1}{2}\sum_{j=1}^n2^{-j}), 
\quad s_n= s_0(1-\frac{1}{2}\sum_{j=1}^n2^{-j})\,, \\
 &\Pi_n := \Pi^{(K_n)}\,,\quad
\Pi_n^{\perp}:=\uno-\Pi_n\,,
\quad  E_n= E^{(K_n)}, \quad \rho_{n}:=\frac{1}{2^{n+8}}, n\geq1, \;\; \rho_0=0, 
\end{aligned}
\end{equation}
Finally, for all $n\ge0$ we denote
$\vec v_n=(\g_{n},\calO_{n},s_{n},a_{n},r_{n})$, $\vec{v}^{\text{\tiny 0}}_n= (\g_{n},\calO_{0},s_{n},a_{n},r_{n})$.

\begin{defi}[{\bf Compatible changes of variables}]\label{compa} Let the parameters in Constraint
\ref{sceltapar} be fixed.
Fix also $\vec v= (\g,\calO,s,a,r)$, $\vec{v}^{\text{\tiny 0}}= (\g,\calO_0,s,a,r)$ with $\calO\subseteq\calO_0$
a compact set,  parameters
$K\ge K_0,\rho<1$. Consider a
 vector field $F= N_0+G\in \calW_{\vec{v}^{\text{\tiny 0}},p}$  which is  $C^{3}$-tame up to
 order $q=\gotp_2+2$ and such that, $$F\in \calE\quad \forall\x\in \calO\,,\qquad |\Pi_{\calX} G|_{\vec{v},\gotp_2-1}\leq \mathtt{C} C_{\vec{v},\gotp_2}(\Pi_{\NN}^\perp G) .$$
  We say that a 
left invertible $\calE$-preserving change of variables 
$$
\calL, \calL^{-1}: \TTT^d_{s}\times D_{a,\gotp_1}(r)\times \calO_0 \to
 \TTT^d_{s+\rho s_0}\times D_{a-\rho a_0,\gotp_1}(r+\rho r_{0})
 $$
 is {\em compatible} with $(F,K,\vec v,\rho)$ if the following holds:
\begin{itemize}
\item[(i)]  $\calL$ is ``close to identity'', i.e.
denoting $\vec{v}^{\text{\tiny 0}}_1:=(\g,\calO_0,s-\rho s_0,a-\rho a_0,r-\rho r_0)$ one has 
\begin{equation}\label{satana}
\begin{aligned}
&\|(\calL-\uno)h\|_{\vec{v}^{\text{\tiny 0}}_1,\gotp_1}\leq \mathtt C  \e_0 K^{-1}
\|h\|_{\vec{v}^{\text{\tiny 0}},\gotp_1}\,.\\
\end{aligned}
\end{equation}

\item[(ii)]  $\calL_*$ conjugates the $C^{3}$-tame vector field $F$ to 
the vector field $\hat{F}:= 
\calL_{*}F=  N_0+ \hat G$
which is $C^{3}$-{tame}; moreover denoting
 $\vec v_2:=(\g,\calO,s-2\rho s_0,a-2\rho a_0,r-2\rho r_0)$
one may choose the tameness constants of $\hat G$ so that
\begin{equation}\label{odio}
\begin{aligned}
&C_{\vec v_2, \gotp_1}(\hat{G})\le C_{\vec v,\gotp_1}(G)(1+\e_0 K^{-1})\,,\\
&C_{\vec v_2,\gotp_{2}}(\hat{G})\le \mathtt C \big(C_{\vec v,\gotp_{2}}(G)+
\e_0 K^{\ka_3}C_{\vec v,\gotp_{1}}(G)\big) \,\\
&
|\Pi_{\calX}\hat{G}|_{\vec{v}_2,\gotp_2-1}
\le \mathtt C \big(C_{\vec v,\gotp_{2}}(\Pi_{\NN}^{\perp}G)+
\e_0 K^{\ka_3}C_{\vec v,\gotp_{1}}(\Pi_{\NN}^{\perp}G)\big) \,.
\end{aligned}
\end{equation}
\item[(iii)] $\calL_*$ ``preserves the $(\NN,\calX,\RR)$-decomposition'', namely one has
\begin{equation}\label{satana2}
\Pi_\NN^\perp (\calL_* \Pi_\NN F) = 0\,, \quad \qquad \Pi_{\calX}(\calL_{*}\Pi_{\calX}^\perp F)=0\,.
\end{equation}
\end{itemize}
\end{defi}

\begin{rmk}
We remark the following facts.  In the choice of parameters in \ref{sceltapar} there is some freedom.
However some parameters are given by the problem we are studying. Indeed the loss of regularity 
$\mu_1$, and the decay parameter $\h$ in Definition \ref{pippopuffo3} are determined in Section \ref{sec7aut}, precisely in Lemma
\ref{computer}. In such Section we will construct a quite explicit set of parameters
which satisfy \eqref{buoni}-\eqref{cribbio4} just for a suitable choice of $\mu_1$ and $\h$.
The same remark holds for the parameter $\ka_3$ in Definition \eqref{compa}. 
Indeed in Section \ref{seclin} we will introduce some changes of coordinates, and in Proposition \eqref{andrea}
in Section \ref{sbroo} we will prove that such transformations are \emph{compatible}, according to
Definition \ref{compa}, provided that $\ka_3$ is chosen in suitable way.
We will also see that, in order to define the transformation of Section \ref{seclin}
we will need to fix a minimum regularity $\gotp_1$.
The other parameters will be chosen according to \eqref{exp1}-\eqref{exp4}.
\end{rmk}

The Abstract result we use in order to prove Theorem \ref{teoremap} is the following

 \begin{theo}[{\bf Abstract KAM}]\label{thm:kambis}
Fix parameters $\e_0,\tR_0,\tG_0,\mu,\eta,\ka_1,\ka_2,\ka_3,\gotp_1,\gotp_2$  satisfying Constraint \ref{sceltapar}. Let $N_0$ be a diagonal vector field (see Definition \ref{norm}) in \eqref{ini1} and consider a  vector field
 \begin{equation}\label{kam1}
 F_0:=N_{0}+G_0 \in \calE\cap\calW_{\vec{v}_{0},p}
 \end{equation} 
 which is  $C^{3}$-tame up to order $q=\gotp_2+2$.

Fix $\g_0>0$ and assume that
 \begin{equation}\label{sizes}
\g_{0}^{-1}C_{\vec{v}_0,\gotp_2}(G_0) \le \tG_0\,,\quad
 \g_0^{-1}C_{\vec{v}_0,\gotp_2}(\Pi_{\NN}^\perp G_0)\le \tR_0\,, \quad
 \g_0^{-1}|\Pi_{\calX}G_0|_{\vec{v}_0,\gotp_1}\le \e_0\,,
 \quad  \g_0^{-1}|\Pi_{\calX}G_0|_{\vec{v}_0,\gotp_2}\le \tR_0\,.
 \end{equation}

For all $n\geq 0$  we define recursively  changes of variables $\calL_n,\Phi_n$  and  compact sets
$ \calO_n$ as follows.

\smallskip

  Set $\HH_{-1}=\HH_0=\Phi_0=\calL_0=\uno$, and for $0\le j \le n-1$ set recursively
  $\HH_j= \Phi_j\circ \calL_j\circ \HH_{j-1}$ and  $F_j:=(\HH_j)_{*}F_0:=N_0+G_{j}$.
  Let $\calL_{n}$ be any change of variables   compatible with $(F_{n-1},K_{n-1}, \vec v_{n-1},\rho_{n-1})$,
    and  $\calO_n$ be any compact set 
   \begin{equation}\label{oscurosignore}
   \calO_{n}\subseteq \calO_{n-1}\,,
   \end{equation}  
  which satisfies the   homological equation for $((\calL_n)_*F_{n-1},K_{n-1},\vec v^{\text{\tiny 0}}_{n-1},\rho_{n-1})$.
   For $n>0$ let $g_n$ be the regular vector field defined in item (2) of Definition \ref{pippopuffo3} and set
    $\Phi_n$ the time-1 flow map generated by $g_n$. 
    
    Then $\Phi_n$ is left invertible and 
     $F_n:= (\Phi_n\circ\calL_n)_*F_{n-1}\in \calW_{\vec v^{\text{\tiny 0}}_n,p}$ is  $C^{3}$-tame up to order $q=\gotp_2+2$.
Moreover the following holds.

      \begin{itemize}
     
  \item[{\bf (i)}]
  Setting $G_n=F_n-N_0$ then
 \begin{equation}\label{lamorte}
 \begin{aligned}
  \Gamma_{n,\gotp_1}&:=\g_{n}^{-1} C_{\vec{v}_n,\gotp_1}(G_n)\leq \tG_n, \quad 
  \Gamma_{n,\gotp_2} :=\g_{n}^{-1}
  C_{\vec{v}_n,\gotp_2}(G_n)\leq  \tG_0 K_n^{\ka_1},\\
\Theta_{n,\gotp_1}&:= \g_{n}^{-1}C_{\vec{v}_n,\gotp_1}(\Pi_{\NN}^\perp G_n)\leq \tR_n,
 \quad  \Theta_{n,\gotp_2}:=
 \g_{n}^{-1}C_{\vec{v}_n,\gotp_2}(\Pi_{\NN}^\perp G_n)\leq  \tR_0  K_{n}^{\ka_1}
 \\
\de_n&:= \g_{n}^{-1} |\Pi_{\calX}G_n|_{\vec{v}_n,\gotp_1}\leq K_0^{\ka_2} \e_0 K_{n}^{-\ka_2},
\quad \g_{n}^{-1} |\Pi_{\calX}G_n|_{\vec{v}_n,\gotp_2}\leq \tR_0 K_n^{\ka_1} 
\\
|g_{n}|_{\vec u_n,\gotp_1}&
     \leq  K_0^{\ka_2} \e_0 \tG_0K_{n-1}^{-\ka_2+\mu+1},
     \quad 
     |g_{n}|_{\vec{u}_{n},\gotp_2}
     \leq 
     \tR_0 \tG_0^{-1} K_{n-1}^{-3+\frac{3}{2} \ka_1} 
 \end{aligned}
 \end{equation}
where 
$\vec u_{n}=(\g_{n},\calO_{n},s_{n}+12\rho_n s_0 ,a_{n}+12\rho_{n} a_0,r_{n}+12\rho_n r_0)$.
 \item[{\bf (ii)}]
 The sequence $\calH_n$ converges
     for all $\x\in\calO_0$ 
  to some change of variables
  \begin{equation}\label{dominio}
  {\calH}_\io={\calH}_\io(\x):  D_{{a_{0}},p}({s_{0}}/{2},{r_{0}}/{2})\longrightarrow D_{\frac{a_{0}}{2},p}({s_{0}},{r_{0}}).
  \end{equation}
  
 \item[{\bf (iii)}]
 Defining $F_{\infty}:=(\calH_\io)_{*}F_0$ one has 
 \begin{equation}\label{fine}
 \Pi_\calX F_{\infty}=0
 \quad \forall \xi \in \calO_\io:=\bigcap_{n\geq0}\calO_n
 \end{equation}
 and 
 $$
 \g_0^{-1}C_{\vec{v}_\io,\gotp_1}(\Pi_{\NN}F_{\infty}-N_0)\le 2\tG_0, \quad
  \g_0^{-1}C_{\vec{v}_\io,\gotp_1}(\Pi_{\RR}F_{\infty})\le 2\tR_0
 $$
 with $\vec{v}_\io:=(\g_0/2,\calO_\io,s_{0}/2,a_{0}/2)$. 
 
 \end{itemize}
\end{theo}

\begin{rmk}
We remark that hypothesis \eqref{sizes} is clearly satisfied 
by the  field $F$ in \eqref{system} thanks to
 Lemma \ref{INIZIAMO} and  the choices in \eqref{parametri}.
\end{rmk}
 
This is exactly the result stated in Theorem $2.25$ in \cite{CFP}. 
In particular we are in the setting of Example $3$ in Section $4$ of \cite{CFP}.
Such rather abstract result is based on Polynomial and normal form decomposition as
detailed in Sections \ref{poly} and  \ref{normi}, and definitions of \emph{tame} and \emph{regular}
vector fields (see Def. $2.13$, $2.18$ of \cite{CFP}). 
The main point is that regular vector fields must be tame and their tameness constants 
have to satisfy a series of properties.
Actually in \cite{CFP} is not fixed a particular choice of regular fields but are only listed 
the required properties. The proof of Theorem \ref{thm:kambis} consists in 
showing that we are in the setting of \cite{CFP}
section $4$ example $3$.
We shall underline the minor differences.

The $(\NN,\calX,\RR)-$decomposition  is the same, hence we have the following.

\begin{lemma}\label{decompototale}
The $(\NN,\calX,\RR)-$decomposition in Definition \ref{decomponi}
satisfies all the properties of Definition $2.17$ in \cite{CFP}.
 Note moreover then $(\NN, \calX,\RR)-$decomposition  is
 \emph{triangular} according to Definition $3.1$ in \cite{CFP}.
\end{lemma}

\begin{proof}
This is proved in Section $4.3$ of \cite{CFP}.
\end{proof}

\noindent
The main difference with respect to \cite{CFP} is in the Definition \ref{tame}
since in the present setting we require  not only the properties $(T_m)$  (which were proposed in \cite{CFP}) but also properties $(T_{m})^{*}$ on the adjoint of the linearized vector field.
The reason is that if one uses Definition $2.13$ of \cite{CFP},
it is not true that finite rank vector fields defined in \eqref{linvec} 
satisfy all the properties of Def. $2.18$ in \cite{CFP}, i.e. are regular.
Actually we have the following.
\begin{lemma}\label{miaomiao}
The finite rank vector fields of Definition \ref{linvec} satisfy all the conditions
of Definition $2.18$ in \cite{CFP}, with respect to the tameness constants of Definition \ref{tame}.
Moreover $|\cdot|_{\gotp_1}$ is a {\em sharp} tameness constant, namely there exists a $\mathtt{c}, \mathtt{C}$  depending  on $\gotp_0,\gotp_1,\gotp_2,d$ such that for any $f$
\begin{equation}\label{tameconst333}
C_{\vec{v},p}(f)=\mathtt{C}|f|_{\vec{v},p},
\end{equation}
is a tameness constant and for any tameness constant $ C_{\vec{v},\gotp_1}(f)$
one has 
 \begin{equation}\label{tameconst3}
 |f|_{\vec{v},p} \le \mathtt{c} \, C_{\vec{v},p}(f), \quad\quad \gotp_1\leq p\leq \gotp_2.
\end{equation}
Finally if $f$ is of finite rank, then it is pseudo-differential according to Definition \ref{pseudopseudo}.
\end{lemma}
\begin{proof}
This Lemma is the analogous of Lemmata $4.8$ and $4.13$ of \cite{CFP}. Note that the present definition of finite rank vector fields (see \ref{linvec})
is different from the one in \cite{CFP} (the norm \eqref{linnorm} here is stronger).
However also the the Definition \ref{tame} is stronger, hence we get \eqref{tameconst3}.
Let us give a sketch 
of the proof.  The only non trivial component is  $f^{(y_{i},w}\cdot w$. 
The adjoint of the differential of this term is  the map $\lambda\to f^{(y_i,w)} \lambda$ with $\lambda\in H^p(\TTT^d_s)$. 
Then the result follows by  \eqref{carota} and the definition of $|\cdot|_{\vec{v},p}$.

In Section $4.3$ in \cite{CFP}, in order to get the analogous results,
see Lemma $4.13$ of \cite{CFP},
we had to introduce the notion of ``adjoint-tame'' vector fields, which is 
better formalized in Definition \ref{tame}. The other properties stated in Definition $2.18$ in \cite{CFP}  follow exactly as in Lemma $4.15$.
The fact that $f$ is pseudo-differential follows trivially from the fact that $d_{w}f^{(w)}(u)=0$.
\end{proof}

\begin{lemma}\label{cellulare}
If $F\in \calX$ is $C^{1}$ tame (according to Definition \ref{tame}), then
$F$ is finite rank according to Definition \ref{linvec}.
\end{lemma}

\begin{proof}
By the definition of $\calX$ in \eqref{sottoREV}
$F$ has, at least formally, the correct structure as \eqref{maremma}.
The only thing we have to prove is that $F^{(y_{i},w)}$ belongs to $\ell_{a,p-2}$. 
By definition $F^{(y_{i},w)}[\cdot]=d_{w}F^{(y_{i})}(\theta,0,0)[\cdot]$. By the property
The adjoint of the differential is then the map $\lambda\to F^{(y_i,w)} \lambda$ with $\lambda\in H^p(\TTT^d_s)$. 
Then the result follows by $(T_{1})^{*}$ in \eqref{carota}.
\end{proof}

Note that the strong property above is not required in \cite{CFP} and will simplify many of our proofs.

We remark that property \eqref{tameconst3} is one of the key points together with
the properties of tame vector fields detailed in the following Lemma.

\begin{lemma}\label{temavec}
Tame vector fields of Definition \ref{tame} satisfy all the properties detailed 
In Appendix $B$ of \cite{CFP}. 
In particular: 
\begin{itemize}
\item[(i)] {\bf Commutator.}
Consider any two $C^3$-\emph{tame} vector fields $F,G\in \mathcal W_{\vec v,p}$, then 
the commutator $[G,F]$ is a  $C^{2}$-tame vector field up to order $ q-1$ with scale of
constants 
 \begin{equation}\label{commu2}
C_{\vec{v},p}([G,F])\leq 
\mathtt C (C_{\vec{v},p+\nu_G+1}(F)C_{\vec{v},\gotp_{0}+\nu_F+1}(G)+C_{\vec{v},\gotp_{0}+\nu_G+1}(F)C_{\vec{v},p+\nu_F+1}(G)),
 \end{equation}
 where $\nu_F$, $\nu_G$ are the loss of regularity of $F$, $G$ respectively.

\item[(ii)] {\bf Composition.}
 Given a tame vector field $f\in \VV_{\vec{v},p}$ with scale of constants $C_{p}(f)$
 of the form \eqref{vectorfield} and given 
 a map $\Phi(\theta)=\theta+\be(\theta):\TTT^d_{s'}\to\TTT^d_s$ 
 with $\|\be\|_{s,\gotp_0}\ll 1$
 then 
 the composition $f\circ \Phi$ is a tame vector field with constant 
 \begin{equation}\label{pillole}
 C_{p}(f\circ \Phi)\leq C_{p}(f)+C_{\gotp_0}(f)\|\be\|_{s,p+\nu+3}.
 \end{equation}
 Moreover if $f$ is a finite rank vector field, i.e. it satisfies \eqref{maremma}, then
  \begin{equation}\label{pillole2}
  |f\circ \Phi|_{\vec{v}_1,p}\leq |f|_{\vec{v},p}+|f|_{\vec{v},\gotp_0}\|\be\|_{s,p+\nu+3}.
  \end{equation}
  where $\vec{v}_1=(\la,\calO,s',a)$.

\item[(iii)] {\bf Conjugation.}
Consider  a tame left invertible map $\Phi=\uno+f$  with tame inverse $\Phi^{-1}= \uno+g $ 
as in Definition \ref{leftinverse} such that \eqref{sinistra} holds.
Assume that the fields $f, g$ 
belong to $\calP_1$ (see  \eqref{poly2}) and 
are such that $C_{\vec{v},p}(f)=C_{\vec{v},p}(g)\leq \rho$ for $\rho>0$ small. 
Consider a constant $\gotp_0\geq0$ and any  tame vector field $F : D_{a,p+\nu}(s,r) \to V_{a,p}$, then the pushforward
\begin{equation}\label{push2}
G:=\Phi_* F: D_{a,p+\nu}(s- 2\rho s_0,r-2\rho r_0)  \to  V_{a- 2\delta a_0,p}
\end{equation}
is tame with scale of constants
 \begin{equation}\label{dafare}
 C_{\vec{v}_{1},p}(G)\leq (1+C_{\vec{v}_{2},\gotp_{0}+\nu+1}(f))C_{\vec{v},p}(F)+C_{\vec{v},\gotp_{0}}(F)(1+C_{\vec{v}_{2},\gotp_{0}+\nu+1}(f))C_{\vec{v}_{2},p+\nu+1}(f),
 \end{equation}
 with $\vec{v}:=(\la,\calO,s,a)$, $\vec{v}_{1}:=(\la,\calO,s-2\rho s_0,a-2\de a_0)$ and 
 $\vec{v}_{2}:=(\la,\calO,s-\rho s_0,a-\de a_0)$.

\end{itemize}

\end{lemma}
\begin{proof}
	We start by proving that $d_U F[G]$ is  $C^2$-tame. The fact that $(T_m)$ holds  for $m=0,1,2$ is proved in \cite{CFP}, Lemma B.1.  Let us now prove that $(T_m)^*$ holds, see \eqref{carota}, for $m=1,2$.
Essentially this amouts to using the chain rule and then the definition of tameness. Indeed  
\begin{align*}
 & d^2_{\mathtt U} (d_{\mathtt U} F [G])(\Phi)[h_1,h_2]=  \\
& d^3_{\mathtt U}F(\Phi)[G(\Phi),h_1,h_2] + d^2_{\mathtt U}F(\Phi)[d_{\mathtt U}G(\Phi)[h_1],h_2]+ d^2_{\mathtt U}F(\Phi)[d_{\mathtt U}G(\Phi)[h_2],h_1] + d_{\mathtt U}F(\Phi)[d^2_{\mathtt U}G(\Phi)[h_1,h_2]]\,;
\end{align*} so that
\begin{equation}\label{albialbi}
D_2[\cdot]:= d^3_{\mathtt U}F(\Phi)[G(\Phi),h_1,\cdot] + d^2_{\mathtt U}F(\Phi)[d_{\mathtt U}G(\Phi)[h_1],\cdot]+ d^2_{\mathtt U}F(\Phi)[d_{\mathtt U}G(\Phi)[\cdot],h_1] + d_{\mathtt U}F(\Phi)[d^2_{\mathtt U}G(\Phi)[h_1,\cdot]].
\end{equation}
Property $(T_2)^{*}$ follows by reasoning as follows. In the first summand in \eqref{albialbi} one uses properties $(T_3)^{*}$ on the field $F$ and $(T_0)$ on $G$.
In the last summand in \eqref{albialbi}  we denote by $A[\cdot]= d_{\mathtt U}F(\Phi)[\cdot]$ and $B[\cdot]=  d^2_{\mathtt U}G(\Phi)[h_1,\cdot]$ then
$(AB)^*= B^* A^*$ and hence $(T_2)^{*}$ follows by $(T_1)^*$ on $F$ and $(T_2 )^*$ on $G$. The second and third summands use the same ideas. This concludes the proof of item {\it (i)}. Regarding items {\it (ii)-(iii)}, the tameness properties $(T_m)$ follow by Lemmata $A.5$ and $B.3$ respectively. Regarding the properties $(T_m)^*$ we proceed as in \eqref{albialbi}. Finally one can prove Remark $B.5$ and Lemma $B.6$ by using  items {\it (i)-(iii)} exactly as in \cite{CFP}.

\end{proof}
The last property we check is the compatibility, according to Definition $2.19$ in \cite{CFP}, 
of the space $\calE$ with 
respect to the $(\NN, \calX,\RR)-$decomposition.
This amounts to show the following.

 \begin{lemma}\label{ciprovo}
The following properties hold:
 \begin{itemize}

 \item[(i)]
  any $F\in \calE\cap \calX$ is a finite rank vector field according with Definition \ref{linvec};
  
  \item[(ii)]
  for any $F\in \calE\cap\calP_{1}$ one has
$\Pi_{\calU}F\in \calE$ for $\calU=\NN,\calX,E^{(K)}$; for $\calU=\NN$.

 \item[(iii)]  
one has
  \begin{equation}\label{proloco}
\forall g\in \BB_\calE,F\in \calE:\quad  \Pi_{\calX}[g,F]\in \calE\,,\quad \forall g,h\in \BB_\calE:\quad  \Pi_\calX[g,h]\in \BB_\calE\,.
\end{equation}

\end{itemize}
\end{lemma} 

\begin{proof}
Item $(i)$
follows by Lemma \ref{cellulare}. Let us check item $(ii)$. It is trivial for $\calU=\calX,E^{(K)}$ and
for $\calU=\NN$ follows by definition \ref{pseudopseudo}. To prove item $(iii)$
we reason as follows. By Lemma \ref{cellulare}
the only think one need to get the result is condition \eqref{preservazione} for $[g,F]$ if $ g\in \BB_\calE,F\in \calE$.
It follows by the parity conditions \eqref{est5} and \eqref{preservazione}.
%
%



\end{proof}

For the proof of Theorem \ref{thm:kambis} we refer the reader to the paper \cite{CFP}. 
Indeed it is
 Theorem  $4.16$ in Section $4.3$ of \cite{CFP}
 and we have proven in the previous Lemmata that all the properties used in that paper 
 hold in this slightly different setting. We omit the proof since it would be a word by word repetition.
 

\zerarcounters

\section{Pseudo-differential Vector Fields}\label{paperina}
In this Section we study how a vector field in $\calE$ defined in \ref{compatibili}
changes under special changes of variables.
First we prove some lemmata of conjugation  of non linear vector fields. 
 Then, in Section
\ref{sec7aut}, we analyze properties of the linearized operator of compatible vector field of Definition \ref{compatibili}. In particular we study how to invert it.
Roughly speaking we study some particular changes of variables for which the subspace $\calE$ of vector fields introduce in Definition \ref{compatibili}
is stable.


\paragraph{The decay norm.}
In order to give quatitative estimates we introduce an appropriate norm
on linear operators on $\ell_{a,p}$.
We recall that we have identified the sequence space $\ell_{a,p}$
with the function space $\Pi_{S}^{\perp}{\bf h}^{a,p}_{{\rm odd}}\subset {\bf h}^{a,p}$, by 
 writing this space in the Fourier sine basis.

A linear operator on $\ell_{a,p}$ is represented by a matrix $A:=(A_{j}^{k})_{j,k\in S^{c}}$.
It is very natural to extend the operator $A$ to the whole space ${\bf h}^{a,p}$
with the Fourier 
exponential basis
by setting 
\begin{equation}\label{tilde}
\tilde{A}_{h}^{r}:=A_{|h|}^{|r|}, |h|,|r|\in S^{c}, \quad \tilde{A}_{h}^{r}=0, \;\; {\rm otherwise}.
\end{equation}
Note that such extension preserves  $\Pi_{S}^{\perp}{\bf h}^{a,p}_{{\rm odd}}$.
Moreover it is compatible with the composition of operators.
We define a ``decay'' norm on the operator $A$ induced by 
the usual decay norm  on linear operators on ${\bf h}^{a,p}$
by setting 
$|A|_{s,a,p}^{{\rm dec}}=|\tilde{A}|^{{\rm dec}}_{s,a,p}$.
The nice fact is that we may deduce all the properties on  compositions, inteprolations,
by the corresponding ones on the standard decay norms. See for instance
\cite{BBM}, \cite{BCP}.
More precisely we have the following definition.

\begin{defi}\label{decaynorm}
({\bf (s,a,p)-decay norm}). Given an infinite dimensional matrix 
$A:=(A_{i}^{i'})_{i,i'\in\{\pm1\}\times S^{c}\times\ZZZ^{d}}$, 
we define the norm of off-diagonal decay
\begin{equation}\label{decayaut}
\begin{aligned}
(|A|^{{\rm dec}}_{s,a,p})^{2}:=\sup_{\s,\s' =\pm1}(|A_{\s}^{\s'}|^{{\rm dec}}_{s,a,p})^{2}&:=
\sup_{\s,\s'=\pm1}\sum_{(h,\ell)\in\ZZZ\times\ZZZ^{d}}\langle h,\ell\rangle^{2p}e^{2|h|a}e^{2|\ell|s}
\sup_{k- k'=(h,\ell)}|\tilde{A}_{\s,k}^{\s',k'}|^{2},
\end{aligned}
\end{equation}
where the extension $\tilde{A}$ is defined in \eqref{tilde}.
If one has that $A:=A(\x) $ for $\x\in\calO\subset\RRR^{d}$, we define for $\vec{v}:=(\g,\calO,s,a)$
\begin{equation}\label{2.1aut}
\begin{aligned}
&|A|^{{\rm dec}}_{\vec{v},p}:=\sup_{\x\in\calO}|A|^{{\rm dec}}_{s,a,p}+\g\sup_{\x_{1}\neq\x_{2}}\frac{|A(\x_{1})-A(\x_{2})|^{{\rm dec}}_{s,a,p}}{|\x_{1}-\x_{2}|}
\end{aligned}
\end{equation}
\end{defi}

The decay norm we have introduced in (\ref{decayaut}) is suitable for the problem we are studying.
 Note that
\begin{equation*}
\;\;\;
\forall\; p\leq p' \;\; \Rightarrow \;\; |A_{\s}^{\s'}|^{{\rm dec}}_{s,a,p}\leq|A_{\s}^{\s'}|^{{\rm dec}}_{s,a,p'}.
\end{equation*}
The same for the other parameters.
Moreover norm (\ref{decayaut}) gives information on the polynomial and exponential off-diagonal decay of the matrices,
indeed setting $k-k'=(h,\ell)$ 
\begin{equation}\label{decay2aut}
\begin{aligned}
&|A_{\s,k}^{\s,k'}|\leq \frac{|A_{\s}^{\s'}|^{{\rm dec}}_{s,a,p}}{\langle k-k' \rangle^{p }e^{2|\ell|s+2|h|a}}, \quad\forall\; k,k'\in\ZZZ_+\times\ZZZ^{d},\; 
\end{aligned}
\end{equation}
In the following we consider the decay norm $|\cdot|_{s,a,p}$ in
we have introduced in (\ref{decayaut})
in order to deal with linear operators on $\Pi_S^\perp{\bf h}_{\rm odd}^{a_0,p+\nu}$.
 We have the following important property.
\begin{lemma}[{\bf Decay along lines}]\label{linesdecay}
Let $M=(M_{i}^{i'})_{i,i'\in S^{c}\times\ZZZ^{d}}$ be a linear operator on $\ell_{a,p}=\Pi_S^\perp{\bf h}_{\rm odd}^{a,p}$. Then one has
\begin{equation}\label{dec1000}
|M|^{{\rm dec}}_{s,a,p} \leq C \max_{i\in S^{c}\times\ZZZ^{d} }\|M_{\{i\}}\|_{s,a,p+d+1}.
\end{equation}
\end{lemma}
For the proof see \cite{BCP}.

The decay norm satisfies the following  classical interpolation estimates.

\begin{lemma}\label{ bubboleaut}{\bf Interpolation.} For all $p\geq \gotp_{0}>(d+1)/2$ 
there are $C(p)\geq C(\gotp_{0})\geq1$ 
such that 
if $A=A(\x)$ and $B=B(\x)$ depend on the parameters $\x\in \calO$
in a Lipschitz way, then for $v=(\g,\calO,s,a)$
\begin{subequations}
\begin{align}
|AB|^{{\rm dec}}_{\vec{v},p}&\leq C(s)|A|^{{\rm dec}}_{\vec{v},\gotp_0}|B|^{{\rm dec}}_{\vec{v},p}
+C(s_{0})|A|^{{\rm dec}}_{\vec{v},p}|B|^{{\rm dec}}_{\vec{v},\gotp_0},\label{eq:2.11a}\\
\|Ah\|_{\vec{v},p}&\leq C(p)(|A|^{{\rm dec}}_{\vec{v},\gotp_0}\| h\|_{\vec{v},p}
+|A|^{{\rm dec}}_{\vec{v},p}\| h\|_{\vec{v},\gotp_0}),\label{eq:2.13b}\end{align}
\end{subequations}
\end{lemma}
\noindent
%
%
%
%

\paragraph{T\"opliz-in-time matrices}

We study now a special class of operators , the so-called 
{\em T\"opliz in time} matrices, i.e.
\begin{equation}\label{eq:2.16aut}
A_{i}^{i'}=A_{(\s,j,l)}^{(\s',j',l')}:=A_{\s,j}^{\s'j'}(l-l'), 
\quad {\rm for} \quad i,i'\in \{\pm1\}\times\ZZZ_+\times\ZZZ^{d},
\end{equation}
To simplify the notation in this case, we shall write
$A_{i}^{i'}=A_{k}^{k'}(\ell)$,
 $i=(\s,j,l)\in \{\pm1\}\times\ZZZ_+\times \ZZZ^{d}$, 
$i'=(\s',j',l')\in \{\pm1\}\times \ZZZ_+\times\ZZZ^{d}$,
and $l-l'=\ell$.

They are relevant because
one can identify the matrix $A$ with a one-parameter family of operators, acting on the space
$\ell_{a,p}\times\ell_{a,p}$,
which depend on the time, namely
\begin{equation}\label{pozzo10aut}
A(\theta):=(A_{\s,j}^{\s',j'}(\theta))_{\substack{\s,\s'\in {\pm1} \\ j,j'\in\ZZZ_+}}, \quad 
A_{\s,j}^{\s',j'}(\theta):=\sum_{\ell\in\ZZZ^{d}}A_{\s,j}^{\s',j'}(\ell)e^{i\ell\cdot\theta}.
\end{equation}
Moreover to obtain the stability result on the solutions we will strongly use this property.

\begin{lemma}\label{1.4aut} If $A$ is a T\"opliz in time matrix as in (\ref{pozzo10aut}), and 
$\gotp_{0}:=(d+2)/2$, then one has
\begin{equation}\label{aaaaa}
\begin{aligned}
|A(\theta)|_{a,p}:=
\left(\sup_{\s,\s'}\sum_{h\in\ZZZ_{+}}
\langle h\rangle^{2p}e^{2|h|a}\sup_{k-k'=h}|\tilde{A}_{\s,k}^{\s',k'}(\theta)|^{2}
\right)^{\frac{1}{2}}
\leq C(\gotp_{0})|A|^{{\rm dec}}_{s,a,p+\gotp_{0}}, \quad  \quad \forall \; \theta\in\TTT^{d}_{s}.
\end{aligned}
\end{equation}

\end{lemma}

%

\begin{rmk}\label{opmolt}
The class of linear operators with \emph{decay} is strictly stronger than just being bounded. In particular contains an important class operator, the so called multiplication
operators. See Remark 4.34 in Section 4 of \cite{FP}.
More precisely if $M$ is the multiplication operator by a function 
$m(\theta,x)\in H^{\gotp_0}(\TTT^{d};{\bf h}_{{\rm odd}}^{a,p})\cap H^{p}(\TTT^{d};{\bf h}_{{\rm odd}}^{a,\gotp_0})$ then one has
\begin{equation}\label{opmolobrutto}
|\Pi_{S}^{\perp} M \Pi_{S}^{\perp}|^{{\rm dec}}_{s,a,p}\leq \|m\|_{s,a,p}.
\end{equation}

 \end{rmk}

\begin{rmk}\label{posacenere}
$F: \calO\times D_{a,p+\nu}(s+\rho s_0,r+\rho r_0)\to V_{a,p}$ a vector field of pseudo-differential type
in ${\calE}$ of Def. \ref{compatibili}, 
and the associated linear operator $\mathscr{P}$ defined in \eqref{pippone}.
Set $\vec{v}=(\g,\calO,s,a)$
and assume that for some $\gotp_1\geq\gotp_0$ 
$$
\|a_{i}\|_{\vec{v},\gotp_1},\|b_{i}\|_{\vec{v},\gotp_1}\leq 1,\quad i=0,1,2,
$$
then
one has that 
for any 
$w\in H^{\gotp_0}(\TTT_{s}^{d};\ell_{a,p+2})\cap H^{p+2}(\TTT_{s}^{d};\ell_{a,\gotp_0})$,
\begin{equation}\label{posacenere2}
\|\mathscr{P}w\|_{\vec{v},p}\leq C(p)\left( \|w\|_{\vec{v},p+2}+\big[\sum_{i=0}^{2}
 \|a_{i}\|_{\vec{v},p}+ \|b_{i}\|_{\vec{v},p}
\big]\|w\|_{\vec{v},\gotp_1}
\right), \quad C(p)\geq1.
\end{equation}
Similar estimates hold for the term $\mathscr{R}$ in \eqref{pippone}.
In other words
the operator $\mathscr{P}$ is tame, with tameness constants given by the $\|\cdot\|_{\vec{v},p}$ norms
of its coefficients. 
\end{rmk}


\subsection{Compatible changes of variables}\label{seclin}
In this Section we study the properties of field in $\calE$ of Def. \ref{compatibili}.
In particular we study the action on $\calE$ of the following class of linear changes of variables:

\begin{itemize}
	\item{\emph{Finite rank changes of variables.}}
	the transformation $\Phi:=\uno+f$ with $f\in\BB_{\calE}$,
	\vspace{.3em}
	\item{\emph{Diffeomorphism of the torus 1.}} the transformation $\Phi:= \Pi_S^\perp \TT_\al \Pi_S^\perp$ where
	$\TT_{\al}u:=u(\theta,x+\al(\theta,x))$,
	\vspace{.3em}
	\item{\emph{Diffeomorphism of the torus 2.}} the transformation $\TT_{\be }: \theta\to \theta+\be(\theta)$
	\item{\emph{Multiplication operator.}} the transformation $\Phi:=\Pi_S^\perp M(\theta,x)  \Pi_S^\perp$ where $M(\theta,x)$ is a multiplication operator
	on the space $ H^p(\TTT^{d+1}_a)$.
\end{itemize}

\begin{proof}[{\bf Proof of Lemma \ref{prestige}}]
By the properties of Finite rank vector fields, see Lemma \ref{miaomiao}, we have that
$g\in \BB_\calE$ generates a flow $\Phi^{t}_{g}=\uno+f_{t}$ where $f_{t}=f\in \BB$ satisfies \eqref{preservazione}.  It is well known that if  $g,F$ are Gauge preserving, i.e. commutes with $X_M$,
then so does  $(\Phi_{*}F)$.

\noindent
Now we have that $(\Phi_{*}F)(u)=F\circ{\Phi^{-1}}+d_uf(\Phi^{-1})[F\circ{\Phi^{-1}}]$. 
Since the transformation $\Phi$ is tame 
by Lemma \ref{temavec} one has that the push-forward is tame with tameness constants given by \eqref{dafare}.
The fact that \eqref{est5} holds for $(\Phi_{*}F)(u)$ follows by the parity conditions
\eqref{est5} on $F$ and \eqref{preservazione}.  It remains to check that the push-forward of $F$
is pseudo-differential according to Definition \ref{pseudopseudo}. We evaluate
the linearized at some point in a neighborhood of $(\theta,0,0)$ of  the $w$ component; we have
\begin{equation}\label{posacenere20}
d_{w}[(\Phi)_{*}F]^{(w)}(u)[h^{(w)}]=
(d_{w}F^{(w)})(\Phi^{-1}(u))[h^{(w)}]+\sum_{i=1}^{d}d_{w}F^{(\theta_i)}(\Phi^{-1}(u))[h^{(w)}]\del_{\theta_i}f^{(w,0)}.
\end{equation}
We claim that 
\begin{equation}\label{label}
d_{w}(\Phi_{*}F)^{(w)}(u)[h]=\mathscr{P}_++\mathscr{K}_+
\end{equation}
has the form \eqref{pippone} with coefficients $a_{i}^{+},b_{i}^{+},c_{i}^{+},d_{i}^{+}$.
 It is clear that term $\mathscr{P}_{+}$ comes from the first summand in \eqref{posacenere20}
 and 
has the same form of the term $\mathscr{P}$ in $d_{w}F(u)$ but with coefficients
 ${a}^{+}_{i}(x)=a_{i}(\Phi^{-1}(u);x)$
(same for $b_i$). 
Both the $a_{i}^{+},b_{i}^{+}$
belong to $H^{\gotp_0}(\TTT_{s}^{d};{\bf h}^{a,p})\cap H^{p}(\TTT_{s}^{d};{\bf h}^{a,\gotp_0})$
and satisfy the require bounds \eqref{labellabel} by the chain rule.

The second summand in \eqref{posacenere20} is  clearly a finite rank operator
and contributes to the coefficients
$c_{i}^{+},d_{i}^{+}$ in $\mathscr{K}^{+}$ in \eqref{maiale}.  Let us show that they belong to
$H^{\gotp_0}(\TTT_{s}^{d};{\bf h}^{a,p-2})\cap H^{p-2}(\TTT_{s}^{d};{\bf h}^{a,\gotp_0})$.
It follows from the tameness of $F$ by applying $(T_{1})^{*}$ in order to bound say $c_{i}^{+}$,
and $(T_{2})^{*}$ for its derivatives. The bounds on $d_{i}^{+}=\del_{\theta_i}f^{(w,0)}$ follows 
$f\in \BB$.

\end{proof}

\noindent
We have the following result.

\begin{lemma}[{\bf Finite rank changes of variables}]\label{torus}
Fix $0<\rho\ll 1$ and $\gotp_1\geq (d+1)/2$.
Consider $F: \calO\times D_{a,p+2}(s,r)\to V_{a,p}$ a vector field 
in ${\calE}$ of Def. \ref{compatibili}, 
then for all $\Phi=\uno+f$  as in Definition \ref{leftinverse} and with $f\in\BB_{\calE}$ (see
Lemma \ref{prestige})
the following holds.
 The operator 
 $d_{w}(\Phi_{*}F)^{(w)}(u)[h]$ defined in \eqref{label}
 satisfies the following bounds.
 Set $\vec{v}=(\g,\calO,s,a, r), \vec{v}_{1}=(\g,\calO,s-\rho s_0,a,r-\rho r_0)$, then we have 
 \begin{equation}\label{torus2}
 \begin{aligned}
  \|{a}^{+}_i-a_{i}\|_{\vec{v}_1,p}&\leq 
 C(p)( |f|_{\vec{v},p}\|a_{i}\|_{\vec{v},\gotp_1}+|f|_{\vec{v},\gotp_1}\|a_{i}\|_{\vec{v},p}),
    \\
 \|{a}^{+}_i\|_{\vec{v},p}
&\leq (1+C(p)|f|_{\vec{v},\gotp_1})\|a_{i}\|_{\vec{v},p}+C(p)|f|_{\vec{v},p}\|a_{i}\|_{\vec{v},\gotp_1},\\
\|d_{u}a_{i}^{+}[h]\|_{\vec{v}_1,p}&\leq 
(1+C(p)|f|_{\vec{v},\gotp_1})\|d_{u}a_{i}[h]\|_{\vec{v},p}+
C(p)|f|_{\vec{v},p}
\|d_{u}a_{i}[h]\|_{\vec{v},\gotp_1},
\end{aligned}
\end{equation}
for some $C(p)>0$ depending only on $p,d,\gotp_1,\gotp_2$.
The same bounds hold for all the other coefficients in \ref{pseudopseudo}.
\end{lemma}
\begin{proof}
By definition  ${a}^{+}_{i}(x)=a_{i}(\Phi^{-1}(u);x)$ then 
$$
\|a_{i}(\Phi^{-1}(u);x)-a_{i}(u)\|_{\vec{v}_1,p}\leq \|d_{u}a(u^{*})[g(u)]\|_{\vec{v}_1,p}.
$$
Recall that $\Phi^{-1}(u)=u+g(u)$.
The  bound follows by estimates \eqref{labellabel} on $a$ and the estimates on $g$ (see \ref{leftinverse}).
The second bound in \eqref{torus2} trivially follows. for the third we only need to use the chain rule.
\end{proof}
\noindent
\noindent
Now we study diffeomorphisms of the $x$ variable and of the 
$\theta$ variables. In the Sobolev case this amounts to studying 
transformations of the real torus $\TTT^{d+1}$.  See  Appendix  in \cite{BBM}.
For completeness we give the statement of an equivalent technical result
 for the analytic case. For the proof we refer the reader to the Appendix in \cite{CFP}.

\begin{lemma}[{\bf Diffeomorphism}]\label{lem.diffeo}
Given $b>0$ and $\gotp_0>(b+1)/2$ then
there exists $c>0$ small depending only on $\gotp_0$ such that for any 
$p\geq \gotp_0,\z>0$ and any
 $\be\in H^{p+1}(\TTT^{b}_\z)$ such that
\begin{equation}\label{diffeo2}
\|\be\|_{\z,\gotp_{0}+1}\leq 
 c \de , 
\quad 0<\de<\frac{\z}{2}, 
\end{equation}
the following holds.
Let us consider $\Phi:\TTT^b_\z\to\TTT^b_{2\z}$ of the form
\begin{equation}\label{diffeo}
x\mapsto x+\be(x)=\Phi(x),
\end{equation}
then

\noindent
(i) There exists $\Psi:\TTT^b_{\z-\de}\to\TTT^b_{\z}$ of the form $\Psi(y)=y+\tilde{\be}(y)$ with
$\tilde{\be}\in W^{p,\io}(\TTT^b_{\z-\de})$ 
satisfying
\begin{equation}\label{diffeo3}
\|\tilde{\be}\|_{\z-\de,\gotp_{0}+1}\leq\frac{\de}{2}, \quad 
\|\tilde{\be}\|_{\z-\de,p}\leq2\|\be\|_{\z,p}\,,
\end{equation}
such that for all $x\in\TTT^b_{\z-2\de}$ one has $\Psi\circ \Phi(x)=x$.

\noindent
(ii) For all $u \in H^{\z,p}(\TTT^b_\z)$, the composition $(u\circ \Phi)(x)=u(x+\be(x))$ satisfies
\begin{equation}\label{diffeo4}
\begin{aligned}
\|u\circ \Phi\|_{\z-\de,p}&\leq (1+C(p)\|\be\|_{\z,\gotp_0+1})\|u\|_{\z,p}
+C(p)\|\be\|_{\z,p+1}\|u\|_{\z,\gotp_0}),\\
\|u\circ\Phi -u\|_{\z-\de,p}&\leq C(p) \frac{1}{\de} (\|\be\|_{\z, p+1}\|u\|_{\z,\gotp_0}+
\|\be\|_{\z, \gotp_0+1}\|u\|_{\z,p})
\end{aligned}
\end{equation}
%
All the estimates above hold also for $\z=0$ except the second in \eqref{diffeo4}.

 \end{lemma}

\begin{lemma}\label{massauffa}
Consider a function $\al\in H^{s}(\TTT^{d}_s\times \TTT_a) $ such that
\begin{equation}\label{massapilota}
\sum_i\partial_{\theta_i}\al=0\;\Leftrightarrow \;\al(\theta,x)=\sum_{\ell\in \ZZZ^{d},j\in\ZZZ}\al_{j}(\ell)e^{\ii \ell\cdot\theta}e^{\ii jx}, \quad \al_{j}(\ell)\neq0,\;\; 
{\rm only \; if}\;\;\sum_{i=1}^{d}\ell_{i}=0;
\end{equation}
 define the map $\TT_{\al}$ on $H^{s}(\TTT^{d}_s\times \TTT_a)$
as
\begin{equation}\label{massapilota2}
\TT_{\al}u:=
u(\theta,x+\al(\theta,x)),\;\; \forall \; u\in H^{s}(\TTT^{d}_s\times \TTT_a).
\end{equation}
and consider $\Phi_\al: (\theta,y,w)\mapsto (\theta,y,w_1):=(\theta,y,\Pi_{S}^{\perp}\TT_{\al}w)$.
Given any Gauge preserving vector field $F$ we have that $\Phi_{\al}^*F$ is Gauge preserving.
\end{lemma}

\begin{proof}
	By Remark \ref{massamassa}, we just check that $\Phi_\al(\Psi_\tau u)= \Psi_\tau \Phi_\al(u)$  where
	$$
	\Psi_\tau (\theta,y,z^+,z^-) = (\theta+ \tau \vec{1} , y,e^{\ii\tau} z^+, e^{-\ii\tau } z^- ). 
	$$
By \eqref{massapilota} we have $\al(\theta+\tau\vec{1},x)= \al(\theta,x)$, hence we have 
	$$
	 \Phi_\al(\theta+\tau \vec{1},y,e^{\ii\tau} z^+, e^{-\ii\tau } z^-) = (\theta+\tau \vec{1},y,\Pi_S^\perp\TT_{\al}e^{\ii\tau} z^+, \Pi_S^\perp\TT_{\al} e^{-\ii\tau } z^-)= (\theta+\tau \vec{1},y,e^{\ii\tau} \Pi_S^\perp\TT_{\al} z^+, e^{-\ii\tau } \Pi_S^\perp\TT_{\al} z^-).
	 $$

\end{proof}
\noindent
We have the following.

\begin{lemma}[{\bf Diffemorphism of the torus 1}]\label{diffeodiffeo}
Fix $\rho>0$.
Consider $F: \calO\times D_{a,p+2}(s+\rho s_0,r+\rho r_0)\to V_{a,p}$ 
a vector field 
in ${\calE}$ (see \ref{compatibili}), 
and a function 
$\al \in 
H^{p}(\TTT^{d}_{s}\times \TTT_a; \CCC)$, 
such that $\al=\al(u)$ with $u$ in a neighborhood of $(\theta,0,0)$,
which restricted to the real subspace is even in $\theta$, odd in $x$ and satisfies \eqref{massapilota}.
Assume that
\begin{equation}\label{piccolino}
\|\al\|_{\vec{v},\gotp_1}\leq c \rho,\quad \gotp_1=\gotp_0+\gotm_1,\quad \gotm_1=d+2\gotp_0+10,
\end{equation}
for some $c>0$ small and 
 \begin{equation}\label{labellabel200}
  \begin{aligned}
  \|\al\|_{\vec{v},p}&\leq \|\al\|_{\vec{v},p}(1+\|u\|_{\vec{v},p+2})\\
  \|d_{u}\al[h]\|_{\vec{v},p}&\leq \|\al\|_{\vec{v},p}
  (\|h\|_{\vec{v},p+2}+\|u\|_{\vec{v},p+2}\|h\|_{\gotp_0+2}).
  \end{aligned}
  \end{equation}
Then setting $\TT_{\al}$ as in \eqref{massapilota2}
and defining the map 
\begin{equation}\label{hobbit}
\Phi : \theta_{+}=\theta,\;\; y_{+}=y,\;\; w_{+}= \Pi_S^\perp \TT_\al w=\Pi_S^\perp \TT_\al \Pi_{S}^{\perp}w,
\end{equation}
one has that for all $\rho$ small enough $\Phi_{*}F: D_{a-2\rho a_0,p+2}(s+\rho s_0,r+\rho r_0)\to V_{a-2\rho a_0,p}$ is 
in $\calE$. 
 Moreover  $d_w(\Phi_{*}F)^{(w)}(u)[h]=\mathscr{P}_{+}+\mathscr{K}_{+}$  has the form \eqref{pippone}. 
 The operator  ${\mathscr{P}}_{+}$ has coefficients ${a}^{+}_{i}, {b}^{+}_{i}$ 
given by
\begin{equation}\label{diff2}
\begin{aligned}
&1+{a}^{+}_2=\TT_{\al}^{-1}[(1+a_{2})(1+\al_{x})^{2}], \qquad {b}^{+}_2=\TT_{\al}^{-1}[(b_{2})(1+\al_{x})^{2}],\\
&{a}^{+}_1=\TT_{\al}^{-1}[(1+a_{2})\al_{xx}]-i\TT_{\al}^{-1}[\oo\cdot\del_{\theta}\al]+\TT_{\al}^{-1}[a_{1}(1+\al_x)],\qquad {b}^{+}_{1}=
\TT_{\al}^{-1}[b_{1}(1+\al_x)],\\
&{a}^{+}_{0}=\TT_{\al}^{-1}[a_{0}], \qquad {b}^{+}_0=\TT_{\al}^{-1}[b_{0}],
\end{aligned}
\end{equation}
where $\TT_{\al}^{-1}u=u(\theta,y+\tilde{\al}(\theta,y)$ is the inverse of the diffeomorphism $x\to x+\al(\theta,x)$.
In particular, setting $\vec{v}:=(\g,\calO,s+\rho s_0,a)$,  $\vec{v}_1:=(\g,\calO,s+\rho s_0,a-\rho a_0)$ and  $\vec{v}_2:=(\g,\calO,s+\rho s_0,a-2\rho a_0)$, the following holds.
For $i=0,1,2$,
 the coefficients $a_{i}^{+},b_{i}^{+}$, of $\mathscr{P}_{+}$ satisfy
 ${a}^{+}_i={a}^{+}_{i}(\theta,y,w_{+};s)$, with $w_{+}=\TT_{\al}w$ and $s=x+\al(\theta,x)$ .
One has that 
\begin{equation}\label{diff3}
\begin{aligned}
 \|{a}^{+}_2-a_{2}\|_{\vec{v}_2,p}&\leq \rho^{-1}C(p)(
\|\al\|_{\vec{v},p+2} +\|\al\|_{\vec{v}_1,p}\|a_2\|_{\vec{v},\gotp_1}+\|\al\|_{\vec{v}_1,p}\|a_2\|_{\vec{v},\gotp_1})\\
  \|{a}^{+}_1-a_{1}\|_{\vec{v}_2,p}&\leq \rho^{-1}C(p)(
\|\al\|_{\vec{v},p+2} +\|\al\|_{\vec{v}_1,p}(\|a_2\|_{\vec{v},\gotp_1}+\|a_1\|_{\vec{v},\gotp_1})
+\|\al\|_{\vec{v}_1,\gotp_1}(\|a_2\|_{\vec{v},p}+\|a_1\|_{\vec{v},p}))\\
\|b_{1}^{+}-b_{1}\|_{\vec{v}_{2},p}&\leq \rho^{-1}C(p) ( \|\al\|_{\vec{v},p+1}\|b_1\|_{\vec{v},\gotp_0}
+\|\al\|_{\vec{v},\gotp_1}\|b_{1}\|_{\vec{v},p}), 
\end{aligned}
\end{equation}
for 
$C(p)>0$ depending only on $p,d,\gotp_1,\gotp_2,\rho$. The coefficients ${a}^{+}_0,{b}^{+}_0$
satisfy the same estimates as $b_1$.
The operator
 $\mathscr{K}_+$ has rank $\mathtt{N}+4|S|$ and the coefficients ${c}^{+}_i,{d}^{+}_i$ are such that
\begin{equation}\label{uffadecay}
\begin{aligned}
\|c_{i}^{+}-c_{i}\|_{\vec{v}_{2},p}\leq \rho^{-1}{C(p)} ( \|\al\|_{\vec{v},p+1}\de_{\gotp_0}
+\|\al\|_{\vec{v},\gotp_1}\de_{p}), \quad i=1,\ldots,N+4|S|,
\end{aligned}
\end{equation}
where $\de_{p}:=\max\{\|a_{i}\|_{\vec{v},p},\|b_{i}\|_{\vec{v},p},\|c_{i}\|_{\vec{v},p},\|d_{i}\|_{\vec{v},p}\}$,
and where we define $c_{i}=0$ for $i=\mathtt{N}+1,\ldots,\mathtt{N}+5|S|$.
Same for $d_{i}$. The coefficients $a^{+}_{i},b^{+}_{i}$ satisfies bounds like \eqref{labellabel}
with $C(F)\rightsquigarrow C(F)(1+\|\al\|_{\vec{v},p})$. 
\end{lemma}
\begin{proof}
The vector field $\Phi_{*}F$ is clearly tame, indeed in the new variables the system reads
\begin{equation}\label{diff4}
\left\{\begin{aligned}
&\dot{\theta}_{+}=F^{(\theta)}(\theta_{+},y_+,(\Pi_{S}^{\perp}\TT_{\al})^{-1}w_+)\\
&\dot{y}_{+}=F^{(y)}(\theta_{+},y_+,(\Pi_{S}^{\perp}\TT_{\al})^{-1}w_+)\\
&\dot{w}_{+}=[F^{(\theta)}(\theta_{+},y_+,(\Pi_{S}^{\perp}\TT_{\al})^{-1}w_+)\cdot\del_{\theta}\al]\del_{x_+}w_{+}+\Pi_{S}^{\perp}\TT_{\al}F^{(w)}(\theta_{+},y_+,(\Pi_{S}^{\perp}\TT_{\al})^{-1}w_+).
\end{aligned}\right.
\end{equation}

We need to study the linearized operator $d_{w}(\Phi_{*}F)^{(w)}(u)[\cdot]$ at $u$ in a neighborhood of 
$(\theta,0,0)$.
First we note the following.
By \eqref{diffeo4} one has that the  map $\Pi_{S}^{\perp}\TT_{\al}$ is near to the identity and can be written as
\[
\Pi_{S}^{\perp}\TT_{\al}\Pi_{S}^{\perp}=\Pi_{S}^{\perp}(\uno+\calL_{\al})\Pi_{S}^{\perp}.
\]
Moreover it is invertible and one has that
\begin{equation}\label{inversaT}
(\Pi_{S}^{\perp}\TT_{\al}\Pi_{S}^{\perp})^{-1}=\Pi_{S}^{\perp}(\uno+\calL_{\al}^{(-1)})\Pi_{S}^{\perp}+Q_{S},
\end{equation}
where 
\[
\TT_{\al}^{-1}\circ\TT_{\al}:=(\uno+\calL_{\al}^{(-1)})\circ(\uno+\calL_{\al})=\uno,
\]
and where $Q_{S}$ is a linear operator of the form \eqref{maiale} with $\mathtt{N}=|S|$ (here $|S|$ is the cardinality of the set $S$) with coefficients $\tilde{c}_{i},\tilde{d}_{i}$, for $i=1,\ldots, |S|$ which satisfy
bounds like
\[
\|\tilde{c}_{i}\|_{\vec{v}_2,p}\leq C(p)\rho^{-1}\|\al\|_{\vec{v},p+1}.
\]
We remark also that we can write
\begin{equation}\label{inversaT2}
\begin{aligned}
&(\Pi_{S}^{\perp}\TT_{\al}\Pi_{S}^{\perp})^{-1}=
(\uno+\calL_{\al}^{(-1)})\Pi_{S}^{\perp}-\Pi_{S}(\uno+\calL_{\al}^{(-1)})\Pi_{S}^{\perp}+Q_{S},\\
&\Pi_{S}^{\perp}\TT_{\al}\Pi_{S}^{\perp}=
\Pi_{S}^{\perp}(\uno+\calL_{\al}-\calL_{\al}\Pi_{S}).
\end{aligned}
\end{equation}
Let us start by the term $\Pi_{S}^{\perp}\TT_{\al}d_{w}F^{(w)}(u)(\Pi_{S}^{\perp}\TT_{\al})^{-1}$,
which comes from the linearization of the second term in the equation for $w_{+}$ in \eqref{diff4}.
We know that $d_{w}F^{(w)}$ has the form \eqref{pippone},
hence we can write
\begin{equation}\label{pippip}
\begin{aligned}
d_{w}F^{(w)}&=\Pi_{S}^{\bot}L\Pi_{S}^{\bot}+\mathscr{K}, \\
L&= -\ii E\;\left[\left(\begin{matrix} 1+a_{2}(u;x) & b_{2}(u;x) \\ \bar{b}_{2}(x) & 1+\bar{a}_{2}(u;x)\end{matrix}
 \right)\del_{xx} +\left(\begin{matrix} a_{1}(u;x) & b_{1}(u;x) \\ \bar{b}_{1}(x) & \bar{a}_{1}(u;x)\end{matrix}
  \right)\del_{x}\right]\\
  &\qquad -\ii E\;
   \left[\left(\begin{matrix} a_{0}(u;x) & b_{0}(u;x) \\ \bar{b}_{0}(u;x) & \bar{a}_{0}(u;x)\end{matrix}
   \right)\right] .
\end{aligned}
\end{equation}
By equations \eqref{inversaT2} and \eqref{pippip} we have
\begin{equation}\label{coniugazioneprecisa}
\begin{aligned}
\Pi_{S}^{\perp}\TT_{\al}(\Pi_{S}^{\bot}L\Pi_{S}^{\bot}+\mathscr{K})(\Pi_{S}^{\perp}\TT_{\al})^{-1}&=
\Pi_{S}^{\perp}\TT_{\al}L\TT_{\al}^{-1}\Pi_{S}^{\perp}+\Pi_{S}^{\perp}\TT_{\al}\mathscr{K}(\Pi_{S}^{\perp}\TT_{\al})^{-1}+\\
&+\Pi_{S}^{\perp}\TT_{\al}\Pi_{S}^{\bot}L\Pi_{S}^{\bot}Q_{S}+
\Pi_{S}^{\perp}\TT_{\al}(-L\Pi_{S}-\Pi_{S}L\Pi_{S}^{\perp})(\Pi_{S}^{\perp}\TT_{\al})^{-1}+\\
&+
\Pi_{S}^{\perp}\calL_{\al}\Pi_{S}(\Pi_{S}^{\bot}L\Pi_{S}^{\bot})
(\Pi_{S}^{\perp}\TT_{\al})^{-1}+
\Pi_{S}^{\perp}\TT_{\al}L(-\Pi_{S}\calL_{\al}^{(-1)})\Pi_{S}^{\perp}.
\end{aligned}
\end{equation}

The first term in the r. h. s. of \eqref{coniugazioneprecisa} can be studied explicitly.
The equation \eqref{diff2} follows by a direct computation (see Section $3$ of \cite{FP} for further details). All the other terms in \eqref{coniugazioneprecisa} have the form \eqref{maiale}
with $N\rightsquigarrow N+4|S|$. 
Indeed, for $h\in \ell_{a,p}$ we have 
\[
\Pi_{S}^{\perp}\TT_{\al}L\Pi_{S}(\Pi_{S}^{\perp}\TT_{\al})h=
\Pi_{S}^{\perp}\TT_{\al}L\sum_{j\in S}(\TT_{\al}h,e^{\ii jx})_{L^{2}}e^{\ii jx}=
\sum_{j\in S}(h,(\TT_{\al})^{*}e^{\ii jx})\Pi_{S}^{\perp}\TT_{\al}L(e^{\ii jx}),
\]
where $(\TT_{\al})^{*}$
is the adjoint, with respect to the $L^{2}$ scalar product, of $\TT_{\al}$. One sets 
$c_{j}(x)=(\TT_{\al})^{*}e^{\ii jx}$ and $d_{j}(x)=\Pi_{S}^{\perp}\TT_{\al}L(e^{\ii jx})$
to get the form \eqref{maiale}. The other terms can be studied in the same way.
Of course the rank of the new operator $\mathscr{K}^{+}$ is no more $\mathtt{N}$, but it is increased 
proportionally with the cardinality of $S$.

Hence the new term $\mathscr{K}^{+}$ has coefficients
$c^{+}_{i},d_{i}^{+}$ for $i=1,\ldots,\mathtt{N}+5|S|$.
Bounds \eqref{diff3} and \eqref{uffadecay}
follows
by applying 
Lemma \ref{lem.diffeo}. 
%
%
%
%
%
The field $\Phi_{*}F$ is in the subspace $\calE$ since
the function $\al$ is even in 
$\theta$, odd in $x$ and by Lemma \eqref{massauffa}.

\end{proof}

\begin{rmk}\label{normadecayfiniterank}
By estimates \eqref{uffadecay}, and using Lemma \ref{linesdecay}, one gets
\begin{equation}\label{normadecayfiniterank2}
|\mathscr{K}^{+}-\mathscr{K}|^{{\rm dec}}_{\vec{v}_{2},p}\leq 
(\mathtt{N}+5|S|)
\rho^{-1}{C(p)} ( \|\al\|_{\vec{v},p+1}\de_{\gotp_0}
+\|\al\|_{\vec{v},\gotp_1}\de_{p}).
\end{equation}

\end{rmk}

\begin{lemma}[{\bf Diffeomorphism of the torus 2}]\label{diffeodiffeo2} Consider $F$ as in Lemma \ref{diffeodiffeo}
and the transformation $\TT_{\be }: \theta\to \theta+\be(\theta)$ with 
$\|\be\|_{\vec{v},\gotp_1}\leq \de,$ for some $\de>0$ small, which restricted to the real subspace is even in $\theta$ and satisfy \eqref{massapilota},
and where $\gotp_1=\gotp_0+\gotm_1$ defined in \eqref{piccolino}.
Then setting $\TT_{\be}u(\theta+\be(\theta),x):=u(\theta_{+},x)$ and defining the map 
\begin{equation}\label{mordor}
\Phi : \theta_{+}=\theta+\be(\theta),\;\; y_{+}=y,\;\; w_{+}=w,
\end{equation}
one has that for some $\rho$ small $\Phi_{*}F: D_{a,p+2}(s-2\rho s_0,r+\rho r_0)\to V_{a,p}$ is tame 
in $\calE$ (see \ref{compatibili})
for $\de\ll \rho$.
 Moreover  $d(\Phi_{*}F)(u)[h]:=\mathscr{P}_++\mathscr{K}_+$  has the form \eqref{pippone} where $\mathscr{P}$ has coefficients ${a}^{+}_{i}, {b}^{+}_{i}$ such that
 \begin{equation}\label{torus222}
 \|{a}^{+}_{i}-a_{i}\|_{\vec{v}_2,p}\leq  C(p,\rho)\|a_{i}\|_{\vec{v},p}\|\be\|_{\vec{v},\gotp_0+1}+\|\be\|_{\vec{v},p+1}\|a_i\|_{\vec{v},\gotp_1},
\end{equation}
where $\vec{v}:=(\g,\calO,s+\rho s_0,a)$,  $\vec{v}_1:=(\g,\calO,s,a)$ and  $\vec{v}_2:=(\g,\calO,s-2\rho s_0,a)$. Same for ${b}^{+}_i$.
Moreover  $\mathscr{K}_+$ with rank $N$  and coefficients whoch satisfy the bound \eqref{uffadecay} with $\al\rightsquigarrow \be$.
Finally if the function $\be$ satisfies bounds \eqref{labellabel} with $C(F)\rightsquigarrow\|\be\|_{\vec{v},p}$, then the coefficients $a^{+}_{i},b^{+}_{i}$ satisfies bounds like \eqref{labellabel}
with $C(F)\rightsquigarrow C(F)(1+\|\be\|_{\vec{v},p})$. 
\end{lemma}
\begin{proof}
One can reason as in Lemma \ref{diffeodiffeo} and use Lemma \ref{lem.diffeo}.
\end{proof}

\begin{lemma}[{\bf Multiplication operator}]\label{multiopop}
Consider $F$ as in Lemma \ref{diffeodiffeo} and  the map $\Phi:=\Pi_S^\perp M(\theta,x)  \Pi_S^\perp$
where $M(\theta,x)=\uno+A(\theta,x) : H^{p}(\TTT_{s}^{d}\times\TTT_{a})\times H^{p}(\TTT_{s}^{d}\times\TTT_{a})
\to H^{p}(\TTT_{s}^{d}\times\TTT_{a})\times H^{p}(\TTT_{s}^{d}\times\TTT_{a})$
with $\|A\|_{\vec{v},\gotp_{1}}\leq \de$ with $\de>0$ small
and $\gotp_1=\gotp_0+\gotm_1$ as in \eqref{piccolino}.
Assume also that $A(\theta,x)$ is even both in $\theta\in\TTT^{d}$ and $x\in \TTT$ when restricted to the real subspace and condition \eqref{massapilota} holds. Finally assume that $A(\theta,x)$ satisfies bounds like
\eqref{labellabel} with $C(F)\rightsquigarrow \|A\|_{\vec{v},p}$. Then 
one has that $\Phi_{*}F$ is tame in $\calE$ with tameness constant given by \eqref{dafare}
where $C_{\vec{v},p}(G)\rightsquigarrow \|A\|_{\vec{v},p}$. Moreover, writing
$$
d_{w}F(u)=\Pi^{\perp}_{S}\left[(-iE+L_2)\del_{xx}+L_{1}\del_x+L_{0}\right]\Pi_{S}^{\perp}+\mathscr{K}(u)
$$
where
\begin{equation}\label{lupo}
L_{i}(\theta,y,w;x)=-iE\left(\begin{matrix}a_{i}(u;x) & b_{i}(u;x) \\ \bar{b}_{i}(u;x) & \bar{a}_{i}(u;x)\end{matrix}\right), \;\;i=0,1,2,
\end{equation}
the linearized operator    $d(\Phi_{*}F)(u)[h]:=\mathscr{P}_++\mathscr{K}_+$  has the form \eqref{pippone} where 
\begin{equation}\label{eq:3.12}
\begin{aligned}
&\mathscr{P}_+=\Pi_{S}^{\perp}\left(M^{-1}(-iE+L_{2}) M \del_{xx}+\left[2 M^{-1}(-iE+L_{2})\del_{x}M+ M^{-1}L_{1} M\right]\del_{x} \right.\\
&\left.+\left[ M^{-1}(F^{(\theta)}\cdot\del_{\theta}M)+ M^{-1}(-iE+L_{2})\del_{xx} M+ M^{-1}L_{1}\del_{x}M+M^{-1}L_{0}M\right] \right)\Pi_{S}^{\perp}
\end{aligned}
\end{equation}
and $\mathscr{K}_+$ has the form \eqref{maiale} with $\mathtt{N}\rightsquigarrow \mathtt{N}+5|S|$ and with coefficients  $c_{i}^{+},d_{i}^{+}$, $i=1,\ldots,\mathtt{N}+5|S|$.
The coefficients $a^{+}_{i},b_{i}^{+}$ of $\mathscr{P}_+$ and $c_{i}^{+},d_{i}^{+}$ of $\mathscr{K}_+$
satisfy
\begin{equation}\label{uffadecay2}
\|{a}^{+}_{i}-a_{i}\|_{\vec{v}_2,p}\leq \|a_{i}\|_{\vec{v},p}\|A\|_{\vec{v},\gotp_0}+
\|A\|_{\vec{v},p}\|a_i\|_{\vec{v},\gotp_1}, 
\end{equation}
same for ${b}^{+}_i,{c}^{+}_i,{d}^{+}_i$. The bounds \eqref{labellabel} hold with 
$C(F)\rightsquigarrow C(F)(1+\|A\|_{\vec{v},p})$.
\end{lemma}
\begin{proof}
First of all we note that $\Phi$ is a tame map. It is sufficient to apply the definition and use the fact that $\|\cdot\|_{\vec{v},p}$ is equivalent 
to $\|\cdot\|_{H^{p}(\TTT_{b}^{d}\times \TTT_{a})}$ on the functions $u(\theta,x)$. Hence also the push-forward is tame. Equation
\eqref{eq:3.12} follows again by an explicit calculation, and the tame bounds on the coefficients follow by the tameness of the transformation and of the coefficients $L_{i}$.
The bounds \eqref{uffadecay2} follows by interpolation properties of the decay norm.
One cam study of the finite rank term $\mathscr{K}^{+}$ 
following the reasoning used in Lemma \ref{diffeodiffeo}.
\end{proof}

\begin{rmk}
The new finite rank term $\mathscr{K}^{+}$
satisfies the bound \eqref{normadecayfiniterank2} with $\al\rightsquigarrow A$.

\end{rmk}

\begin{rmk}[{\bf Loss of regularity}]\label{loss}
Note that transformations  in Lemma \ref{torus} do not loose analyticity. On the contrary the diffeormorphisms of the torus in Lemmata \ref{diffeodiffeo} and \ref{diffeodiffeo2}
loose regularity in the analytic case. 
\end{rmk}

Lemmata \ref{torus}, \ref{diffeodiffeo}, \ref{diffeodiffeo2} and \ref{multiopop} guarantees that the structure of pseudo-differential operator
of the linearized in a neighbourhood of $u=(\theta,0,0)$
persists under the change of variables we described above.
We also have decomposed the linearized
operator in homogeneous decreasing symbols of order two, one and zero. By Lemmata above we note that also such decomposition is stable, and we are able to control 
the tameness constant of each symbol. This not \emph{a priori} obvious.

\subsection{Reduction to diagonal plus bounded operator}\label{settesette}
The following Lemma is the key result of this Section. We give the result on a special class of vector field.
Consider   vector field $F : \calO\times D_{a,p+\nu}(s,r)\to V_{a,p}$  
with $F=N_0+G$
and with $\nu=2$.  For $\g_0/2\leq \g\leq 2\g_0$ (see \eqref{ini4}),
 set $\vec{v}=(\g,\calO,s,a)$ and
  assume that $F$ is in $\calE$ (see \ref{compatibili}) and has the form
\begin{equation}\label{regu22}
F:=(1+h)\Big(\hat{N}_0+N^{(1)}+N^{(2)}+H
\Big)
\end{equation}
%
%
%
for some $h: \calO\times \TTT^{d}_{s}\to \CCC$ with
\begin{equation}\label{barad9}
\|h\|_{\vec{v},p}<+ \infty, \quad p\leq \gotp_2
\end{equation}
and
where $\hat{N}_0=\oo\cdot\del_{\theta}+\tilde{\Omega}^{-1}w\cdot\del_{w}$ with $\tilde{\Omega}^{-1}=m \Omega^{(-1)}$, $\oo\in \RRR^{d}$ is diophantine and
\begin{equation}\label{barad}
|\oo-{\oo}^{(0)}|_{\g}\leq  o(|\x|), \quad |m-1|_{\g}\leq O(|\x|),
\end{equation}
\begin{equation}\label{barad20}
{N}^{(2)}:=\Big(\Pi_{S}^{\perp}\left(-\ii E\left(\begin{matrix} a_{2}(\theta,x) &b_{2}(\theta,x)\\ \bar{b}_{2}(\theta,x)& a_{2}(\theta,x) \end{matrix}\right)
\right)\del_{xx} \Pi_{S}^{\perp}\Big)w\cdot\del_{w}
\end{equation}
\begin{equation}\label{angband2}
N^{(1)}:=  \Big(-\ii E\Pi_{S}^{\perp}\left[\left(\begin{matrix} a_{1}(\theta,x) & b_{1}(\theta,x) \\ \bar{b}_{1}(\theta,x) & \bar{a}_{1}(\theta,x)\end{matrix}
  \right)\del_{x} + \left(\begin{matrix} a_{0}(\theta,x) & b_{0}(\theta,x) \\ \bar{b}_{0}(\theta,x) & \bar{a}_{0}(\theta,x)\end{matrix}
   \right) \right]\Pi_S^\perp+\mathscr{K}\Big)w\cdot\del_{w},
\end{equation}
where $\mathscr{K}$ of the form \eqref{maiale} with coefficients $c_{j},d_{j}$, $j=1,\ldots,N$
 and ${\oo}^{(0)}$ defined in 
\eqref{ini4}.
Moreover we { assume that}
\begin{equation}\label{barad21}
d_{w}H^{(w)}(0)[\cdot]=0.
\end{equation}
and 
\begin{equation}\label{barad22}
\begin{aligned}
C_{\vec{v},p}(H)\leq C_{\vec{v},p}(\Pi_{\NN^{\perp}}G), \quad C_{\vec{v},p}(N^{(2)})\leq C_{\vec{v},p}(G).
\end{aligned}
\end{equation}

Note that by the reversibility condition one has that on $\calU$ the function $h(\theta)$ is real and \emph{even} in $\theta$.
Write
\begin{equation}\label{corvo}
 \g^{-1}\|H^{(\theta,0)}\|_{\vec{v},p}:=\de^{(1)}_{p}, \qquad
{\rm max}\{\g^{-1}\|a_{2}\|_{\vec{v},p},\g^{-1}\|b_{2}\|_{\vec{v},p}\}:=\de^{(2)}_{p}, 
\end{equation}
We recall that by Remark \ref{posacenere} we have that $\de^{(2)}_{p}$ controls the tameness constant
$C_{\vec{v},p}(N^{(2)})$.
In the same way the norm of the functions $a_{i},b_{i},c_{i},d_{i}$ controls the constant $C_{\vec{v},p}(N^{(1)})$.
In the following we will use the following notation:
we set
\begin{equation}\label{indiani}
\begin{aligned}
C_{\vec{v},p}(N^{(2)})&:=\max\{\|a_{2}\|_{\vec{v},p},\|b_{2}\|_{\vec{v},p}\},\\
C_{\vec{v},p}(N^{(1)})&:=\max_{\substack{i=0,1\\ j=1,\ldots,N}}
\{\|a_{i}\|_{\vec{v},p},\|b_{i}\|_{\vec{v},p},\|c_{j}\|_{\vec{v},p},\|d_{j}\|_{\vec{v},p}\}
\end{aligned}
\end{equation}
In Remark \ref{orolo} we already fixed the  parameter $\gotp_{0}>(d+1)/2$ 
so that the norm $\|\cdot\|_{s,a,p}$ for $p\geq \gotp_0$ defines a Banach algebra on the space ${H^{p}(\TTT_{s}^{d}\times \TTT_{a})}$.

Some comments are in order. 
Note that the term $N^{(2)}$ is  pseudo-differential operator of order $2$, i.e. maps $\ell_{a,p+2}$ to $\ell_{a,p}$
for any $p\in \RRR$. The term $N^{(1)}$ is an operator which maps $\ell_{a,p+1}$ to $\ell_{a,p}$.
In the following we shall use always this notation.
The vector field $F$ in \eqref{regu22} we study has a particular structure.
In Section \ref{sbroo} we will show the following fact: we construct a sequence of maps $\calL_{n}$  
which are compatible according to Definition \ref{compa}. Then we will apply  
Theorem \ref{thm:kambis} to $F_0$ defined in \eqref{system}, 
then we will show by induction that 
the  sequence of fields $F_{n}$ for $n\geq1$ given by Theorem \ref{thm:kambis}
have actually the structure in \eqref{regu22}.
In the following Lemma
we show how to construct a map $\calL_{+}$ which reduces the size of the coefficients $a_{2},b_{2}$
in \eqref{barad20} for vector fields of the form \eqref{regu22}.

\begin{lemma}[{\bf Regularization}]\label{regularization}
Fix $K>0$ large, $K_{+}=K^{\frac{3}{2}}$ and $\rho_{+}>0$ small
and consider $F$ as in \eqref{regu22}.
Fix 
\begin{equation}\label{posacenere10}
\h_{1}:=2\gotp_0+2\tau+4,
\end{equation}
and $\vec{v}=(\g,\calO,s,a)$,  $\vec{v}_1:=(\g,\calO,s-\rho_+ s_0,a-\rho_+ a_0)$ and  $\vec{v}_2:=(\g,\calO,s-2\rho_+ s_0,a-2\rho_+ a_0)$.
Let $\gotp_1=\gotp_0+\gotm_{1}$ defined in \eqref{piccolino}.
Then, if 
\begin{equation}\label{piccolopassero}
\rho_+^{-1}K^{\h_{1}}\max\{\de^{(1)}_{\gotp_1},\de^{(2)}_{\gotp_1}\}\leq \epsilon
\quad \|h\|_{\vec{v},\gotp_1}\leq \g_0 \tG_0,   
\end{equation}
where $\g_0$,  $\tG_0$ as in \eqref{ini4},\eqref{parametri}, 
$\gotp_1,\gotp_2$ in \ref{sceltapar}
and $\epsilon=\epsilon(d,\gotp_0)$  is small enough 
then there exists a tame map 
\begin{equation}\label{mappamappa}
\calL_+= : \calO_0\times D_{a+\rho_+ a_0,p+2}(s-\rho_+ s_0,r-\rho_+ r_0)\to D_{a,p+2}(s,r),
\end{equation}
with 
\begin{equation}\label{trans2}
\begin{aligned}
\|(\calL_{+}-\uno)u\|_{\vec{v}_{1},p}&\leq C\rho_{+}^{-1}K^{\h_1}\max\{\de_{p}^{(1)},\de_{p}^{(2)}\}
\|u\|_{\vec{v},\gotp_0}
+
C\rho_{+}^{-1}K^{\h_1}\max\{\de_{\gotp_1}^{(1)},\de_{\gotp_1}^{(2)}\}\|u\|_{\vec{v},p}
, \;\; p\leq \gotp_2.
\end{aligned}
\end{equation}
Moreover for any $\x\in \calO_0$ such that
\begin{equation}\label{lamerd}
|\oo\cdot l|\geq \frac{\g}{\langle l\rangle^{\tau}}, \quad |l|\leq K_{+}
\end{equation}
\begin{equation}\label{mappaccia}
\hat{F}:=(\calL_+)_{*}F : D_{a-2\rho_+ a_0,p+2}(s-2\rho_+ s_0,r-2\rho_+ r_0)\to V_{a-2\rho a_0,p}
\end{equation}
is in $\calE$ (see Def. \ref{compatibili}) and has the form $N_{0}+\hat{G}$ and 
\begin{equation}\label{regu5}
\begin{aligned}
\hat{F}:=(1+h_+)\Big(\hat{N}^{+}_0+\hat{N}^{(1)}+\hat{N}^{(2)}+\hat{H}
\Big)
\end{aligned}
\end{equation}
with 
\begin{equation}\label{argento}
\|{h}_+\|_{\vec{v}_{2},\gotp_1}\leq \|h\|_{\vec{v},\gotp_1}(1+K_{+}^{(\gotp_0+2\tau+2)}\de_{\gotp_1}^{(2)}), \qquad \|{h}_+\|_{\vec{v}_{2},\gotp_2}\leq
\|h\|_{\vec{v},\gotp_1}(1+K_{+}^{(\gotp_0+2\tau+2)}\de_{\gotp_1}^{(2)})+K_{+}^{\gotp_0+2\tau+2}\de_{\gotp_2}^{(2)},
\end{equation}
and $\hat{N}^{+}_0=\oo_+\cdot\del_{\theta}+\tilde{\Omega}_{+}^{-1}w\cdot\del_{w}$ with $\tilde{\Omega}^{-1}=m_{+} \Omega^{(-1)}$, and
\begin{equation}\label{barad1000}
|\oo_{+}-{\oo}|_{\g}\leq  \de_{\gotp_1}^{(1)}, \quad |m_+-m|_{\g}\leq \de_{\gotp_1}^{(2)}.
\end{equation}
One has that $N_{+}^{(2)}$ as the form \eqref{barad20} with coefficients $a_{2}^{+},b_{2}^{+}$, $N^{(1)}_+$ has the form \eqref{angband2} with coefficients
$a_{i}^{+},b_{i}^{+}$ for $i=0,1$, $\mathscr{K}_+$ of the form \eqref{maiale} with rank $N+\mathtt{C}|S|$,
for some absolute constant $\mathtt{C}$,  and coefficients $c^{+}_i,d^{+}_i$ for $ i=1,\ldots, N+\mathtt{C} |S|$. In particular, recalling the notation introduced in \eqref{indiani}, one has
%
 \begin{equation}\label{regu6}
 \begin{aligned}
  C_{\vec{v}_{2},\gotp_{1}}(\hat{N}^{(1)})&\leq  (1+K_{+}^{\h_1}\de_{\gotp_1}^{(2)})C_{\vec{v},\gotp_1}(N^{(1)})
\\
C_{\vec{v}_{2},\gotp_{2}}(\hat{N}^{(1)})&\leq 
(1+K_{+}^{\h_1}\de_{\gotp_1}^{(2)})(C_{\vec{v},\gotp_2}(N^{(1)})+
K_{+}^{2\gotp_0+2\tau+2}(\de_{\gotp_2}^{(1)}+\de_{\gotp_2}^{(2)}) C_{\vec{v},\gotp_{1}}(G)),
 \end{aligned}
 \end{equation}
 \begin{equation}\label{regu66}
 \begin{aligned}
 \g_0^{-1}C_{\vec{v}_{2},\gotp_1}(\hat{N}^{(2)})&\leq (1+K_{+}^{\h_1}\de_{\gotp_1}^{(2)})\Big[K_{+}^{-(\gotp_2-\gotp_1-2\gotp_0-2\tau-2)}\de^{(2)}_{\gotp_2}+K_{+}^{2\gotp_0+2\tau+2}(\de_{\gotp_1}^{(2)})^{2}
\Big],\\
C_{\vec{v}_{2},\gotp_2}(\hat{N}^{(2)})&\leq C_{\vec{v},\gotp_2}(N^{(2)})+K_{+}^{\h_1}C_{\vec{v},\gotp_1}(N^{(2)}).
 \end{aligned}
\end{equation}
Moreover the following estimates hold: set $\hat{G}=\hat{F}-N_0$, then
 \begin{equation}\label{passero}
 \begin{aligned}
 C_{\vec{v}_{2},\gotp_{1}}(\hat{G})&\leq  (1+K_{+}^{\h_1}\de_{\gotp_1}^{(2)})C_{\vec{v},\gotp_1}(G)
\\
C_{\vec{v}_{2},\gotp_{2}}(\hat{G})&\leq 
(1+K_{+}^{\h_1}\de_{\gotp_1}^{(2)})(C_{\vec{v},\gotp_2}(G)+
K_{+}^{\h_1}(\de_{\gotp_2}^{(1)}+\de_{\gotp_2}^{(2)}) C_{\vec{v},\gotp_{1}}(G)),
\end{aligned}
\end{equation}
 \begin{equation}\label{passero2}
\g_0^{-1}\|\hat{H}^{(\theta,0)}\|_{\vec{v}_2,\gotp_1}\leq   (1+K_{1}^{\h_1}\de_{\gotp_1}^{(2)})\Big[
K_{1}^{-(\gotp_2-\gotp_1)}\de_{\gotp_1}^{(1)}+K_{1}^{\h_1}\de_{\gotp_1}^{(2)}\de_{\gotp_1}^{(1)}
\Big],
\end{equation}
Finally one has
\begin{equation}\label{merlo1}
\g_{0}^{-1}\|a_{2}-a_{2}^{+}\|_{\vec{v}_{2},\gotp_1}, \g_{0}^{-1}\|b_{2}-b_{2}^{+}\|_{\vec{v}_{2},\gotp_1}\leq 
K_{+}^{\h_1}\de_{\gotp_1}^{(2)}\Big[K_{+}^{-(\gotp_2-\gotp_1-2\gotp_0-2\tau-2)}\de^{(2)}_{\gotp_2}+K_{+}^{2\gotp_0+2\tau+2}(\de_{\gotp_1}^{(2)})^{2}
\Big],
\end{equation}
\begin{equation}\label{merlo}
\begin{aligned}
\|a_{i}-a_{i}^{+}\|_{\vec{v}_{2},\gotp_1},\|b_{i}-b_{i}^{+}\|_{\vec{v}_{2},\gotp_1}&\leq  K_{+}^{\h_1}\de_{\gotp_1}^{(2)}
C_{\vec{v},\gotp_1}(N^{(1)}),
\qquad i=0,1,\\
\|c_{i}-c_{i}^{+}\|_{\vec{v}_{2},\gotp_1}, \|d_{i}-d_{i}^{+}\|_{\vec{v}_{2},\gotp_1}&\leq  K_{+}^{\h_1}\de_{\gotp_1}^{(2)}C_{\vec{v},\gotp_1}(N^{(1)}),
\qquad i=1,\ldots,N+\mathtt{C}|S|
\end{aligned}
\end{equation}
where $c_{i}=d_i=0$ for $i=N+1,\ldots, N+\mathtt{C}|S|$. 
\end{lemma}

\begin{proof}[{\bf Proof of Lemma \eqref{regularization}}]
The proof follows the scheme used in Section $3$ of \cite{FP}. 

\noindent
{\bf Step 1.} The first step is to diagonalize the second order coefficient by eliminating the terms $b_{2}$ through a multiplication operator.
We use a transformation of the form
\begin{equation}\label{minas3}
\Phi_{1} : \theta\to\theta,\;\; y\to y,\;\; w\to \Pi_{S}^{\perp}M_{A}\Pi_{S}^{\perp}w.
\end{equation}
The eigenvalues of
$$
\left(\begin{matrix} m+a_{2}(\theta,x) & b_{2}(\theta,x) \\ \bar{b}_{2}(\theta,x) & m+\bar{a}_{2}(\theta,x)\end{matrix}
 \right)
$$
are $\la_{1,2}=\sqrt{(m+a_{2})^{2}-|b_2|^{2}}$. Hence we set
$\tilde{a}_{2}(\f,x)=\la_{1}-m$. We have that $\tilde{a}_{2}\in\RRR$ because $a_{2}\in\RRR$ and $a_i,b_i$ are small.
The diagonalizing matrix is
\begin{equation}\label{fiorellino}
\frac{1}{2m}\left(\begin{matrix}2m+a_{2}+\tilde{a}_{2}& -b_{2} \\
-\bar{b}_{2}& 2m+a_{2}+\tilde{a}_{2}\end{matrix}
\right)\,:=\uno+A.
\end{equation}
We define $M_{A}^{-1}:=\uno+\Pi_{K_+}A$, hence
\begin{equation}\label{finiamola}
\|\Pi_{K_+}A\|_{\vec{v},p}\leq C_{\vec{v},p}(\Pi_{K_+}N^{(2)}).
\end{equation}
The bound on the inverse $M_A$ follows simply by the fact that
\begin{equation}\label{finiamola2}
{\rm det}(\uno+A):=\frac{(|b_{2}|^{2}-(2c+a_2+\tilde{a}_{2})^{2})}{4m^{2}}.
\end{equation}
The bounds on the truncated matrix is the same. One can also prove that
\begin{equation*}
\|(\uno+\Pi_{K_+}A)^{-1}-(\uno+A)^{-1}\|_{\vec{v},p}\leq \|\Pi_{K_+}^{\perp}A\|_{\vec{v},p}+\|A\|_{\vec{v},\gotp_1}\|A\|_{\vec{v},p}.
\end{equation*}
Since $F$ is in $\calE$ then the map $\Phi_{1}$ defined in \eqref{minas3} satisfies all the hypothesis of Lemma
\ref{multiopop}. 
Such result guarantees that 
the field $F^{(1)}=(\Phi_{1})_{*}F=N_0+G^{(1)}=(1+h)(\tilde{N}_0+N_{1}^{(1)}+N_{1}^{(2)}+H^{(1)})$ (see notations in \eqref{regu22}) is in $\calE$ and one has
\begin{equation}\label{multiopop10}
C_{\vec{v},p}(G^{(1)})\leq C_{\vec{v},p}(G)+K_{+}^{\nu+1}\|A\|_{\vec{v},p}C_{\vec{v},\gotp_1}(G).
\end{equation}
In particular the term $(H^{(1)})^{(\theta,0)}$ satisfies a bounds like \eqref{multiopop10} with $G\rightsquigarrow (H)^{(\theta,0)}$.
Moreover using equation \eqref{eq:3.12} of Lemma \ref{multiopop} with $M\rightsquigarrow M_{A}=\uno+\Pi_{K_+}A$ 
we obtain that
\begin{equation}\label{minas}
N_{1}^{(2)}=-i E\Pi_{S}^{\perp}\left[\left(\begin{matrix} \tilde{a}_2(\theta,x) & 0 \\ 0 &\tilde{a}_2(\theta,x)\end{matrix}
 \right)\del_{xx}\right]\Pi_{S}^{\perp}
 -i E\Pi_{S}^{\perp}\left[\left(\begin{matrix} {a}^{(1)}_{2}(\theta,x) & {b}^{(1)}_{2}(\theta,x) \\ \bar{b}^{(1)}_{2}(\theta,x) &{a}^{(1)}_{2}(\theta,x)\end{matrix}
 \right)\del_{xx}\right]\Pi_{S}^{\perp}
%
\end{equation}
and bounds \eqref{uffadecay2} imply on the field $F^{(1)}$ bounds like \eqref{regu6}-\eqref{merlo}. More explicitly
one has
\begin{equation*}
\begin{aligned}
\|a_{2}^{(1)}\|_{\vec{v},p},\|b_{2}^{(1)}\|_{\vec{v},p}\leq C_{\vec{v},p}(\Pi_{K_{+}}^{\perp}(N^{(1)}))+C_{\vec{v},p}(N^{(2)})C_{\vec{v},\gotp_1}(N^{(2)}).
\end{aligned}
\end{equation*}
Moreover the field  $N_{1}^{(1)}$
has the form 
 \eqref{angband2} with coefficients $a_{i}^{(1)},b_{i}^{(1)}$ 
  (given by Lemma \ref{multiopop}) and $\mathscr{K}^{(1)}$ of the form \eqref{maiale}.
Finally by \eqref{finiamola} one has
\begin{equation}\label{finiamola4}
\begin{aligned}
C_{\vec{v},p}(N^{(1)}_{1})&\leq C_{\vec{v},p}(N^{(1)})+C_{\vec{v},p+2}(\Pi_{K_{+}}N^{(2)})C_{\vec{v},\gotp_1}(N^{(1)})+
C_{\vec{v},p+2}(\Pi_{K_+}N^{(2)})\\
&+C_{\vec{v},\gotp_1+2}(\Pi_{K_+}N^{(2)})C_{\vec{v},p}(H)+C_{\vec{v},p+2}(\Pi_{K_+}N^{(2)})C_{\vec{v},\gotp_1}(H),
\end{aligned}
\end{equation}
and the same bounds holds on the single coefficients $a_{i}^{(1)},b_{i}^{(1)}$ for $i=0,1$.
%
Hence
\begin{equation}\label{knock}
\begin{aligned}
C_{\vec{v},\gotp_1}(N^{(1)}_{1})&\leq (1+K_{+}^{2}\de_{\gotp_1}^{(2)})C_{\vec{v},\gotp_1}(N^{(1)})+K_{+}^{3}C_{\vec{v},\gotp_1}(H)\de_{\gotp_1},\\
C_{\vec{v},\gotp_2}(N^{(1)}_{1})&\leq C_{\vec{v},\gotp_2}(N^{(1)})+C_{\vec{v},\gotp_1}(N^{(1)})K_{+}^{3}\de_{\gotp_1}^{2}+K_{+}^2(C_{\vec{v},\gotp_2}(H)\de_{\gotp_1}^{(2)}+
C_{\vec{v},\gotp_1}(H)\de_{\gotp_2}^{(2)})
\end{aligned}
\end{equation}
One the coefficients $c_i^{(1)},d_{i}^{(1)}$ of $\mathscr{K}^{(1)}$ the bound \eqref{uffadecay2} holds.
Lemma \ref{multiopop} implies also bounds \eqref{merlo} on the coefficients of the field $F^{(1)}$.
In particular $b_{2}^{(1)}$ satisfies \eqref{merlo1}.

\noindent
{\bf Step 2 - Change of the space variable} Now we want to eliminate the dependence on $x$ of the coefficients $\tilde{a}_{2}$ of the field $F^{(1)}$. We use a change of variables $\Phi_{2}$ as in 
\eqref{hobbit} of Lemma
\ref{diffeodiffeo}. We are looking for $\al(\theta,x)$ such that the coefficient of the second order differential operator does not depend on
$x$, namely by equation \eqref{diff2}
\begin{equation}\label{25}
\TT_{\al}^{-1}[(1+\tilde{a}_{2})(1+\al_{x})^{2}]=1+a_{2}^{(2)}(\theta),
\end{equation}
for some function $a_{2}^{(2)}(\theta)$. Since $\TT_{\al}$ operates only on the space variables, the $(\ref{25})$
is equivalent to 
\begin{equation}\label{26}
(1+\tilde{a}_{2}(\theta,x))(1+\al_{x}(\theta,x))^{2}=1+m_{2}(\theta).
\end{equation}
Hence 
we have to set
\begin{equation}\label{27}
\al_{x}(\theta,x)=\rho_{0}, \qquad \rho_{0}(\theta,x):=(1+m_{2})^{\frac{1}{2}}(\theta)(1+\tilde{a}_{2}(\theta,x))^{-\frac{1}{2}}-1,
\end{equation}
that has solution $\al$ periodic in $x$ if and only if $\int_{\TTT}\rho_{0} dx=0$. This condition
implies
\begin{equation}\label{28}
m_{2}(\theta)=\left(\frac{1}{2\pi}\int_{\TTT}(1+\tilde{a}_{2}(\theta,x))^{-\frac{1}{2}}\right)^{-2}-1,
\end{equation}
which satisfies the bound
\begin{equation}\label{stimaM2}
\|m_{2}(\theta)\|_{\vec{v},p}\leq C(s)\g_0 \de_{p}^{(2)}.
\end{equation}

Then we have the ``approximate'' solution (with zero average) of $(\ref{27})$
\begin{equation}\label{29bis}
\al(\theta,x):=(\del_{x}^{-1}\Pi_{K_+}\rho_{0})(\theta,x),
\end{equation}
where $\del_{x}^{-1}$ is defined by linearity as
$
\del_{x}^{-1}e^{ikx}:=\frac{e^{ikx}}{ik}, \quad \forall \; k\in\ZZZ\backslash\{0\}, \quad \del_{x}^{-1}=0.
$
In other word $\del_{x}^{-1}h$ is the primitive of $h$ with zero average in $x$.
The function $\al$ (that is a trigonometric polynomial) satisfies
\begin{equation}\label{finiamola5}
\|\al\|_{\vec{v},\gotp_1+\gotp_0}\leq K_+^{\gotp_0}\de_{\gotp_1}^{(2)}, \quad \|\al\|_{\vec{v},\gotp_2+\gotp_0}\leq K_+^{\gotp_0}\de_{\gotp_2}^{(2)}
\end{equation}
 For more details on the estimates on $\al$ we refer to \cite{FP}.

With this definition of the function $\al$ and by Lemma \ref{lem.diffeo} one has that
$\TT_{\al} : \TTT_{a-(\rho/4) a_0}\to \TTT_{a}$ if, by \eqref{finiamola5},
  $\de_{\gotp_1}^{(2)}$ is small enough. 
  Moreover we know, since $F^{(1)}$ is in $\calE$ that $\tilde{a}_{2}$ is even in $\theta$  and in $x$ and satisfies \eqref{massapilota}, hence the function $\al$ satisfies the hypotheses of Lemma \ref{diffeodiffeo}.
Setting 
\begin{equation}\label{barad4}
F^{(2)}:=(\Phi_{2})_{*}F^{(1)}=N_0+G^{(2)}=(1+h)(\hat{N}_0+N_{2}^{(1)}+N_{2}^{(2)}+H^{(2)}),
\end{equation}
again one has that $F^{(2)}$ is in $\calE$ and 
\begin{equation}\label{multiopop11}
C_{\vec{v},p}(G^{(2)})\leq (1+\|\al\|_{\vec{v},\gotp_1+\nu+1})\Big(C_{\vec{v},p}(G^{(1)})+\|\al\|_{\vec{v},p+\nu+1}C_{\vec{v},\gotp_1}(G^{(1)})\Big),
\end{equation}
where we used $\|\al\|_{\vec{v},\gotp_1+\nu+1}\leq K^{\nu+1}\de_{\gotp_1}^{(2)}\leq 1$.
Moreover we have obtained
  \begin{equation}\label{minas2}
N_{2}^{(2)}=-i E\Pi_{S}^{\perp}\left[\left(\begin{matrix} m_{2}(\theta) & 0 \\ 0 & m_{2}(\theta)\end{matrix}
 \right)\del_{xx}\right]\Pi_{S}^{\perp}
  -i E\Pi_{S}^{\perp}\left[\left(\begin{matrix} {a}^{(2)}_{2}(\theta,x) & {b}^{(2)}_{2}(\theta,x) \\ \bar{b}^{(2)}_{2}(\theta,x) &{a}^{(2)}_{2}(\theta,x)\end{matrix}
 \right)\del_{xx}\right]\Pi_{S}^{\perp}
\end{equation}
where
\begin{equation}\label{finiamola33}
\begin{aligned}
\|a_{2}^{(2)}\|_{\vec{v},p},\|b_{2}^{(2)}\|_{\vec{v},p}\leq C_{\vec{v},p}(\Pi_{K_{+}}^{\perp}(N^{(2)}))+C_{\vec{v},p}(N^{(2)})C_{\vec{v},\gotp_1}(N^{(2)}).
\end{aligned}
\end{equation}
The field $N_{2}^{(1)}$ in \eqref{barad4}
has the form \eqref{angband2} with coefficients $a^{(2)}_{i}, b^{(2)}_{i}$ for $i=0,1$. The bound \eqref{diff3} hold  with $a_{1},b_{1}\rightsquigarrow a^{(1)}_{i}, b^{(1)}_{i}$ where $\al$ is the one given in  \eqref{29bis}.
More explicitly one has
\begin{equation}\label{door}
\begin{aligned}
C_{\vec{v},\gotp_1}(N_{2}^{(1)})&\leq (1+K_{+}^{\gotp_0+2}\de_{\gotp_1}^{(2)})C_{\vec{v},\gotp_1}(N_1^{(2)})+K_{+}^{\gotp_0+2}\de_{\gotp_1}^{(1)}C_{\vec{v},\gotp_1}(N_1^{(1)}),\\
C_{\vec{v},\gotp_1}(N_{2}^{(1)})&\leq C_{\vec{v},\gotp_2}(N_{1}^{(1)})+K_{+}^{\gotp_0+2}\de_{\gotp_2}C_{\vec{v},\gotp_1}(N_{1}^{(1)})+
K_{+}^{\gotp_0+2}(C_{\vec{v},\gotp_2}(N_{1}^{(2)})\de_{\gotp_1}^{(2)}+C_{\vec{v},\gotp_1}(N_{1}^{(2)})\de_{\gotp_2}^{(2)})
\end{aligned}
\end{equation}
Moreover  the coefficients $c_i^{(2)},d_{i}^{(2)}$ of $\mathscr{K}^{(2)}$ satisfy the bound \eqref{uffadecay}.
%
  For more details on the estimates on $\al$ we refer to \cite{FP}.
  Lemma \ref{diffeodiffeo} implies also bounds \eqref{merlo} on the coefficients of the field $F^{(2)}$.
  In particular coefficients $a_{2}^{(2)},b_{2}^{(2)}$
 satisfy bounds \eqref{merlo1}.
 
 \noindent
 First we remark that, by Lemmata \ref{diffeodiffeo} and \ref{multiopop}, we have that the rank of the finite rank term is increased, and $\mathscr{K}^{2}$ has rank $N+\mathtt{C}|S|$ for some absolute constant $\mathtt{C}>0$.
Secondly we note that in this two steps the function $h(\theta)$ did not change. 

\noindent
{\bf Step 3 - Time reparameterization.} In order to eliminate the dependence on $\theta$ of  $a_{2}^{(2)}$ we use a special
 diffeomorphism of the torus
\begin{equation}\label{32}
\TT_{\be} :  \theta \to \theta_{+}=\theta+\oo\be(\theta), \quad \theta\in\TTT_{s}^{d}, \quad \be(\theta)\in\RRR,
\end{equation}
where $\al$ is a small real valued function, $2\pi-$periodic in all its arguments. 
We consider a transformation $\Phi_{3}$
of the form \eqref{mordor} and we set 
\begin{equation}\label{barad5}
F^{(3)}:=(\Phi_{3})_{*}F^{(2)}=N_0+G^{(3)}
\end{equation}
and  $N^{(3)}:=\Pi_{\NN}(\Phi_{3})_{*}F^{(2)}$, $X^{(3)}:=\Pi_{\calX}(\Phi_{3})_{*}F^{(2)}$ and $R^{(3)}:=\Pi_{3}(\Phi_{3})_{*}F^{(2)}$. 
We also have that
\begin{equation}\label{barad6}
\Pi_{\NN}(\Phi_{3})_{*}F^{(2)}:=(\Phi_{3})_{*}\Pi_{\NN}F^{(2)}, \quad \Pi_{\calX}(\Phi_{3})_{*}F^{(2)}:=(\Phi_{3})_{*}\Pi_{\calX}F^{(2)}, 
\quad \Pi_{3}(\Phi_{3})_{*}F^{(2)}:=(\Phi_{3})_{*}\Pi_{\RR}F^{(2)},
\end{equation}
Let us study in detail the from of $N^{(3)}$. First of all we have
\begin{subequations}\label{torrenera}
\begin{align}
((\Phi_{3})_{*}\Pi_{\NN}F^{(2)})^{(\theta)}(\theta_{+})&=\TT_{\be}^{-1}\Big[(1+\oo\cdot\del_{\theta}\be)(1+h)[\oo+(H^{(2)})^{(\theta,0)}]\Big](\theta_{+}),\label{thor}\\
((\Phi_{3})_{*}\Pi_{\NN}F^{(2)})^{(w)}(\theta_{+},w_{+})&=\TT_{\be}^{-1}\Big[(1+h)\Big(\hat{N}_0^{(w)}+N_{2}^{(1)}+N_{2}^{(2)}+ d_{w}H^{(2)} w\Big)\label{martello}
\Big]=\\
&-i E\Pi_{S}^{\perp}\left(\begin{matrix} \TT_{\be}^{-1}(1+h)(m+m_2) & 0 \\ 0 & \TT_{\be}^{-1}(1+h)(m+m_{2})\end{matrix}
 \right)\del_{xx}\TT_{\be}\Pi_{S}^{\perp} \nonumber\\
&-i E\Pi_{S}^{\perp}\TT_{\be}^{-1}\left[\left(\begin{matrix} a^{(2)}_{1} & b^{(2)}_{1} \\ \bar{b}^{(2)}_{1} & \bar{a}^{(2)}_{1}\end{matrix}
  \right)\del_{x} + \left(\begin{matrix} a^{(2)}_{0} & b^{(2)}_{0} \\ \bar{b}^{(2)}_{0} & \bar{a}^{(2)}_{0}\end{matrix}
   \right)\right] \TT_{\be}\Pi_S^\perp+\TT_{\be}^{-1}\mathscr{K}^{(2)}\TT_\be\nonumber
\end{align}
\end{subequations}
%
%
%
%
 
 Our aim is to find $\be$ in such a way the coefficients of the second order is 
proportional with respect to the coefficients of $\oo$. This is equivalent to require that
\begin{equation}\label{helm}
(1+h(\theta))(m+m_{2}(\theta))=m_{+}(1+\oo\cdot\del_{\theta}\be(\theta))(1+h(\theta)),
\end{equation} 
for some $m_{+}:=m+\gotc$. By setting 
\begin{equation}\label{37}
\gotc=\frac{1}{(2\pi)^{d}}\int_{\TTT^{d}}(m_{2}(\theta))d\theta,
\end{equation}
which satisfies also
\begin{equation}\label{stimaC}
|\gotc|\leq C(s)\g_0 \de_{\gotp_1}^{(2)}.
\end{equation}
we can find the (unique) approximate solution of $(\ref{helm})$ with zero average
\begin{equation}\label{38}
\be(\theta):=\frac{1}{m+\gotc }(\oo\cdot\del_{\theta})^{-1}(m+\Pi_{K}(m_{2})-m-\gotc)(\theta),
\end{equation}
where $(\oo\cdot\del_{\theta})^{-1}$ is defined by linearity
$
(\oo\cdot\del_{\f})^{-1}e^{i\ell\cdot\f}:=\frac{e^{i\ell\cdot\f}}{i\oo\cdot\ell}, \; \ell\neq0, \quad
(\oo\cdot\del_{\f})^{-1}1=0.
$
Note that $\be$ is trigonometric polynomial supported on $|\ell|\leq K$.
As one can check (see also \cite{FP} for more details) the function $\be$ in \eqref{38} satisfy the bound 
\begin{equation}\label{barad3}
\|\be\|_{\vec{v},p+\gotp_0}\leq \g_{0}^{-1}\|\Pi_{K}m_{2}\|_{\vec{v},p+\gotp_{0}+2\tau+2},
\end{equation}
with $\vec{v}=(\g,\calO,s-(\rho/4),a-(\rho/4) a_0)$ and $\g_0$ is the diophantine constant of $\oo$. 
Hence we have
\begin{equation}\label{barad303}
\|\be\|_{\vec{v},\gotp_1+\gotp_0}\leq K_{+}^{\gotp_0+2\tau+2}\de_{\gotp_1}^{(2)}, \qquad \|\be\|_{\vec{v},\gotp_2+\gotp_0}\leq K_{+}^{\gotp_0+2\tau+2}\de_{\gotp_2}^{(2)}
\end{equation}

This implies that
the  transformation  $\TT_{3}$ maps $\TTT^{d}_{s-(\rho/4) s_{0}}\to\TTT^{d}_{s}$ if $\de_{\gotp_1}^{(2)}$ is sufficiently smaller than $\rho$ (see condition \eqref{piccolopassero}).
We set
\begin{equation}\label{barad7}
\tilde{h}_{+}:=\TT_{\be}^{-1}(\oo\cdot\del_{\theta}\be+h+h \oo\cdot\del_{\theta}\be),
\end{equation}
so that one has
 for $\vec{v}=(\g,\calO,s-\rho s_0,a-\rho a_{0})$ 
\begin{equation}\label{ora}
\begin{aligned}
\|\tilde{h}_{+}- h\|_{\vec{v},p}&\leq \|\oo\cdot\del_{\theta}\be\|_{\vec{v},p}+\|\oo\cdot\del_{\theta}\be\|_{\vec{v},2\gotp_0}\|h\|_{\vec{v},p}\leq \g_0^{-1}K_{+}^{\gotp_0+2\tau+2}(\|a_{2}^{(2)}\|_{\vec{v},p}+
\|a_{2}\|_{\vec{v},\gotp_0}\|h\|_{\vec{v},p}),\\
\end{aligned}
\end{equation}
which  implies the bound \eqref{argento}. 
The function $\be$ in \eqref{38} satisfies the hypotheses of Lemma \ref{diffeodiffeo2}. By applying Lemma \ref{diffeodiffeo2} to $F^{(2)}$ we get
 \begin{equation}\label{barad8}
 F^{(3)}=N_0+G^{(3)}=(1+\tilde{h}_{+})(\tilde{N}_{0}+N_{3}^{(1)}+N_{3}^{(2)}+H^{(3)}),
 \end{equation} 
 with (recalling \eqref{ini1})
 \begin{equation}\label{rock}
 \tilde{N}_0:=\oo \cdot\del_{\theta}+m_{+}\Omega^{(-1)}w\cdot\del_{w}, \qquad  m_+=m+\gotc,
 \end{equation}
  \begin{equation}\label{minas200}
N_{3}^{(2)}=
  -i E\Pi_{S}^{\perp}\left[\left(\begin{matrix} {a}^{(3)}_{2}(\theta,x) & {b}^{(3)}_{2}(\theta,x) \\ \bar{b}^{(3)}_{2}(\theta,x) &{a}^{(3)}_{2}(\theta,x)\end{matrix}
 \right)\del_{xx}\right]\Pi_{S}^{\perp}
\end{equation}
where
\begin{equation}\label{finiamola333}
\begin{aligned}
C_{\vec{v},p}(N^{(2)}_3)&\leq C_{\vec{v},p}(\Pi_{K_{+}}^{\perp}(N_2^{(2)}))+
\|\be\|_{\vec{v},p+\gotp_0}C_{\vec{v},\gotp_1}(\Pi_{K_{+}}^{\perp}(N_2^{(2)}))\\
&+\max\{\|a_2^{(2)}\|_{\vec{v},p},\|b_2^{(2)}\|_{\vec{v},p}\}+\|\be\|_{\vec{v},p+\gotp_0}\max\{\|a_2^{(2)}\|_{\vec{v},\gotp_1},\|b_2^{(2)}\|_{\vec{v},\gotp_1}\}\\
\end{aligned}
\end{equation}
and the coefficients $a_2^{(3)},b_{2}^{(3)}$ satisfies the same bound (see Remark \ref{posacenere}). 
The field $N_{3}^{(1)}:=(\Phi_3)_*(\hat{N}_{0}+N_2^{(1)}+N_{2}^{(2)})-(\tilde{N}_0+N_{3}^{(2)})$
collects all the terms of order at most $O(\del_{x})$ plus a finite rank term  $\mathscr{K}^{(3)}$ and has the form
\eqref{angband2} with coefficients
%
%
$a_{i}^{(3)},b_{i}^{(3)}$ for $i=0,1$ 
and $\mathscr{K}^{(3)}$ as in \eqref{maiale}.
Lemma \ref{diffeodiffeo2} implies that
\begin{equation}\label{finiamola334}
\|a_i^{(3)}\|_{\vec{v},p},\|b_i^{(3)}\|_{\vec{v},p}\leq \|a_i^{(2)}\|_{\vec{v},p}+\|\be\|_{\vec{v},p+\gotp_0}\|a_{i}^{(2)}\|_{\vec{v},\gotp_1}.
\end{equation}
The coefficients of $\mathscr{K}^{(3)}$ satisfy similar bounds as \eqref{finiamola334} (see \eqref{uffadecay}).

%
More explicitly one has
\begin{equation}\label{door2}
\begin{aligned}
C_{\vec{v},\gotp_1}(N_{3}^{(2)})&\leq (1+K_{+}^{\gotp_0+2\tau+2}\de_{\gotp_1}^{(2)})\Big[
K_{+}^{-(\gotp_2-\gotp_1)}C_{\vec{v},\gotp_2}(N_2^{(2)})+\max\{\|a_2^{(2)}\|_{\vec{v},\gotp_1},\|b_2^{(2)}\|_{\vec{v},\gotp_1}\}
\Big],\\
C_{\vec{v},\gotp_2}(N_{3}^{(2)})&\leq
C_{\vec{v},\gotp_2}(N_2^{(2)})+K_{+}^{-(\gotp_2-\gotp_1)+\gotp_0+2\tau+2}\de_{\gotp_2}^{(2)}C_{\vec{v},\gotp_2}(N_2^{(2)})\\
&+
K_{+}^{\gotp_0+2\tau+2}\de_{\gotp_{2}}^{(2)}\max\{\|a_2^{(2)}\|_{\vec{v},\gotp_1},\|b_2^{(2)}\|_{\vec{v},\gotp_1}\},\\
C_{\vec{v},\gotp_1}(N_{3}^{(1)})&\leq(1+K_{+}^{\gotp_0+2\tau+2}\de_{\gotp_1}^{(2)})C_{\vec{v},\gotp_1}(N_2^{(1)}),\\
C_{\vec{v},\gotp_2}(N_{3}^{(1)})&\leq C_{\vec{v},\gotp_2}(N_2^{(1)})+K_{+}^{\gotp_0+2\tau+2}\de_{\gotp_2}^{(2)}C_{\vec{v},\gotp_1}(N_2^{(1)}).
\end{aligned}
\end{equation}
By equation \eqref{37} we also obtain the bound $\sup_{\calO}|m_{+}-m|\leq\de_{\gotp_1}^{(2)}$ on the constant $m_{+}$ hence \eqref{barad1000} is satisfied.
Using the estimates given by Lemmata \ref{diffeodiffeo}, \ref{diffeodiffeo2}, \ref{multiopop} and collecting the estimates in \eqref{finiamola}, \eqref{finiamola5}
and \eqref{barad303}, we get
on the field
$G^{(3)}$ in \eqref{barad8}, the estimates
\begin{equation}\label{monaco311}
\begin{aligned}
C_{\vec{v},\gotp_1}({G}^{(3)})&\leq (1+K_+^{(\gotp_0+2\tau+2)}\de_{\gotp_1}^{(2)})C_{\vec{v},\gotp_1}(G), \\
C_{\vec{v},\gotp_2}({G}^{(3)})&\leq (1+K_+^{(\gotp_0+2\tau+2)}\de_{\gotp_1}^{(2)})C_{\vec{v},\gotp_2}({G})+
(1+K_+^{(\gotp_0+2\tau+2)}\de_{\gotp_1}^{(2)})C_{\vec{v},\gotp_1}({G})K_{+}^{\gotp_0+2\tau+2}\de_{\gotp_2}^{(2)}.
\end{aligned}
\end{equation}
Finally we have
\begin{equation}\label{door3}
\begin{aligned}
\|(H^{(3)})^{(\theta,0)}\|_{\vec{v},\gotp_1}&\leq \g_0(1+K_+^{(\gotp_0+2\tau+2)}\de_{\gotp_1}^{(2)})\de_{\gotp_1}^{(1)},\\
\|(H^{(3)})^{(\theta,0)}\|_{\vec{v},\gotp_2}&\leq\|H^{(\theta,0)}\|_{\vec{v},\gotp_2}+\g_0K_{+}^{\gotp_0+2\tau+2}\de_{\gotp_2}^{(2)}\de_{\gotp_1}^{(1)}.
\end{aligned}
\end{equation}
\noindent
{\bf Step 4 - Torus diffeomorphism.}
The aim of the final step is reduce ``quadratically'' the size of the term $(H^{(3)})^{(\theta,0)}$.
We define the trasformation 
\begin{equation}\label{helm2}
\Phi_{4} : (\theta,y,w) \to (\theta+g(\theta),y,w),
\end{equation}
and we call $\TT_{g}$ the transformation $\TT_{g}u=u(\theta+g(\theta),x)$.
We set $\hat{F}=(\Phi_{4})_{*}F^{(3)}=N_0+\hat{G}$ and we study its projection on the subspace $\NN$.
By a direct calculation one can note that
$\Pi_{\NN}(\Phi_{4})_{*}F^{(3)}=(\Phi_{4})_{*}\Pi_{\NN}F^{(3)}$, $\Pi_{\calX}(\Phi_{4})_{*}F^{(3)}=(\Phi_{4})_{*}\Pi_{\calX}F^{(3)}$ and $\Pi_{\RR}(\Phi_{4})_{*}F^{(3)}=(\Phi_{4})_{*}\Pi_{\RR}F^{(3)}$.
For convenience we write
\begin{equation}\label{monaco200}
\begin{aligned}
(\Phi_{4})_{*}\Pi_{\NN}F^{(3)}&:=(\Phi_{4})_{*}\Big(
(1+\tilde{h}_{+}) (\tilde{N}_0+(H^{(3)})^{(\theta,0)})
\Big)\\
&+(\Phi_{4})_{*}\Big((1+\tilde{h}_+)(N_{3}^{(1)}+N^{(2)}_{3}
\Big)=\\
&=(1+\TT_{g}^{-1}\tilde{h}_+)\left[(\Phi_{4})_{*}(\tilde{N}_0+(H^{(3)})^{(\theta,0)})
+ 
(\Phi_{4})_{*}(N_{3}^{(1)}+N^{(2)}_{3})
\right].
\end{aligned}
\end{equation}
We set 
\begin{equation}\label{barad1001}
h_{+}:=\TT_{g}^{-1}\tilde{h}_+.
\end{equation} 
By Lemma \ref{diffeodiffeo2}, definition \eqref{barad7} and estimate \eqref{ora} we have that \eqref{argento}
holds.
Moreover one can write
$$
(\Phi_{4})_{*}N^{(3)}=N^{(3)}+[g,N^{(3)}]+\int_{0}^{1}\int_{0}^{t}(\Phi^{(4)})_{*}^{s}[g,[g,N^{(3)}]]d s d t.
$$
where $N^{(3)}:=(\tilde{N}_{0}+(H^{(3)})^{(\theta,0)})$.
We look for $g(\theta)$ such that the $\theta$ component of $N^{(3)}+[g,N^{(3)}]$ has the form in \eqref{regu5} with the size of $H_{+}^{(\theta,0)}$ ``much smaller'' 
of the size of $H^{(\theta,0)}$. More explicitly one has that
\begin{equation}\label{merda}
(N^{(3)}+[g,N^{(3)}])^{(\theta)}=(\oo_{+}+\tilde{H}^{(\theta,0)}(\theta)),
\end{equation}
where , denoting by $\langle\cdot\rangle$ the average in the variable $\theta$,
\begin{equation}\label{merda2}
\begin{aligned}
\oo_{+}&:=\oo+\langle (H^{(3)})^{(\theta,0)}\rangle,\\
\tilde{H}^{(\theta,0)}&:= (H^{(3)})^{(\theta,0)}(\theta)-\langle (H^{(3)})^{(\theta,0)}\rangle+\oo\cdot\del_{\theta}g(\theta)
+ (H^{(3)})^{(\theta,0)}\del_{\theta}g(\theta)-\del_{\theta} (H^{(3)})^{(\theta,0)}g(\theta)
\end{aligned}
\end{equation}
and we look for $g(\theta)$ such that
\begin{equation}\label{merda3}
\Pi_{K}\left[ (H^{(3)})^{(\theta,0)}(\theta)-\langle (H^{(3)})^{(\theta,0)}\rangle\right]+\oo\cdot\del_{\theta}g(\theta)=0.
\end{equation}
Equation \eqref{merda3} is satisfied by choosing 
\begin{equation}\label{merda5}
g(\theta):=(\oo\cdot\del_{\theta}^{-1})^{-1} \Pi_{K}\left[(H^{(3)})^{(\theta,0)}(\theta)-\langle (H^{(3)})^{(\theta,0)}\rangle\right].
\end{equation}
One has the following estimates hold 
\begin{equation}\label{merda6}
\begin{aligned}
\|g\|_{\vec{v},\gotp_1}&
{\leq}
K_{1}^{2\tau+2}\g^{-1}C_{\vec{v},\gotp_1}((H^{(3)})^{(\theta,0)})\leq K_{1}^{2\tau+2}(1+K_{+}^{\gotp_0+2\tau+2}\de^{(2)}_{\gotp_{1}})\de_{\gotp_1}^{(1)}
,\\
\|g\|_{\vec{v},\gotp_2}&\leq K_{+}^{2\tau+1}(\de_{\gotp_2}^{(1)}+K_{+}^{\gotp_0+2\tau+2}\de_{\gotp_2}^{(2)}\de_{\gotp_1}^{(1)}).
\end{aligned}
\end{equation}
Moreover by the first of \eqref{merda2} we have that \eqref{barad1000} holds.
 Now for $\de_{\gotp_1}^{(1)},\de_{\gotp_1}^{(2)}$ small enough (see condition \eqref{piccolopassero}) one can use Lemma \ref{lem.diffeo} to conclude that
\begin{equation}\label{merda8}
\TT_{g}: \TTT^{d}_{s-(\rho_+/4)s_0}\to \TTT^{d}_{s-(\rho_+/2)s_0}.
\end{equation}
In particular $g$ in \eqref{merda5} satisfies hypotheses of Lemma \ref{diffeodiffeo2}, which applied to 
$F^{(3)}$ implies that

$$
\hat{F}=(\Phi_{4})_{*}F^{(3)} : D_{a-\rho_+ a_0,p+\nu}(r-\rho_+ r_0,r-\rho_+ r_0)\to V_{a-\rho_+ a_0,p}
$$
is in the subspace $\calE$.
We set $\calL_+:=\Phi_{4}\circ\Phi_{3}\circ\Phi_{2}\circ\Phi_{1}$. By the discussion above we have that
$$
\calL_{+} : D_{a-\rho_+ a_{0},p+2}(s-\rho_+ s_{0},r_0-\rho_+ r_{0}) \to D_{a,p+2}(s,r),
$$
the estimate \eqref{trans2} follows by collecting the bounds \eqref{finiamola}, \eqref{finiamola5}, \eqref{barad303},
\eqref{merda6} and Lemma \ref{lem.diffeo}.
%
%
\noindent
We check all the bounds on the new vector field $(\calL_{+})_{*}F$.
We have
\begin{equation}\label{ordine100}
\begin{aligned}
\hat{F}&:=N_{0}+((\Phi_{4})_{*}\tilde{N}_0-\tilde{N}_0)+(\Phi_{4})_{*}(\tilde{h}_+\tilde{N}_0)+(\Phi^{(4)})_{*}\Big(
(1+\tilde{h}_+)(N_{3}^{(1)}+N_{3}^{(2)}+(H^{(3)})^{(\theta,0)})
\Big)\\
&+(\Phi_{4})_{*}((1+\tilde{h}_+)\Pi_{\NN^{\perp}}H^{(3)}),
\end{aligned}
\end{equation}
with $\tilde{N}_{0}$ defined in \eqref{rock}.
Fix now $\vec{v}=(\g,\calO,s-\rho_+ s_{0},a-\rho_+ a_{0})$.
From this first splitting we have that $\hat{F}=\tilde{N}_0+\hat{G}$ and on $\hat{G}$ the following estimates hold:
\begin{equation}\label{tartaruga100}
\begin{aligned}
C_{\vec{v},\gotp_{1}}(\hat{G})&\leq  (1+K_{+}^{\h_1}\de_{\gotp_1}^{(2)})C_{\vec{v},\gotp_1}(N^{(1)}+N^{(2)}+H)
\\
C_{\vec{v},\gotp_{2}}(\hat{G})&\leq (1+K_{+}^{\h_1}\de_{\gotp_1}^{(2)})(C_{\vec{v},\gotp_2}(N^{(1)}+N^{(2)}+H)+K_{+}^{2\gotp_0+2\tau+2}(\de_{\gotp_2}^{(1)}+\de_{\gotp_2}^{(2)})C_{\vec{v},\gotp_{1}}(G)),
\end{aligned}
\end{equation}
for $\h_1$ as in \eqref{posacenere10}
which implies \eqref{passero}.
In order to prove \eqref{tartaruga100} one 
reasons as follows.
%
 One has that
\begin{equation*}\label{tartaruga2}
\begin{aligned}
((\Phi^{(4)})_{*}\tilde{N}_0-\tilde{N}_0)&=\int_{0}^{1} (\Phi_{4})_{*}^{s}[g,\tilde{N}_0]=\int_{0}^{1} (\Phi_{4})_{*}^{s}\Pi_{K_+}[g,\tilde{N}_0]=\\
&=\int_{0}^{1} (\Phi_{4})_{*}^{s}\Pi_{K_{+}}[g,\tilde{N}_0 +(H^{(3)})^{(\theta,0)}]-
\int_{0}^{1}(\Phi_{4})_{*}^{s}\Pi_{K_+}[g, (H^{(3)})^{(\theta,0)}]\\
&=\int_{0}^{1} (\Phi_{4})_{*}^{s}r+\int_{0}^{1}(\Phi_{4})_{*}^{s}\Pi_{K_+}[g,(H^{(3)})^{(\theta,0)}],
\end{aligned}
\end{equation*}
where by \eqref{merda3} 
\begin{equation}\label{tartaruga33}
r=\Big( \Pi_{K_+}^{\perp}(H^{(3)})^{(\theta,0)}
+ (H^{(3)})^{(\theta,0)}\del_{\theta}g(\theta)-\del_{\theta} (H^{(3)})^{(\theta,0)}g(\theta)
\Big)\cdot\del_{\theta}
\end{equation}
Use \eqref{merda6} to estimate $g$, \eqref{door3} to estimate $(H^{(3)})^{(\theta,0)}$, hence
using  \eqref{monaco311} and the smallness of $\de$ 
one gets the estimates \eqref{tartaruga100}
with $\h_1$ in \eqref{posacenere10}.
 Trivially one has also
\begin{equation}\label{tartaruga44}
\begin{aligned}
C_{\vec{v},\gotp_{1}}(\Pi_{\NN^{\perp}}\hat{G})&\leq  (1+K_{+}^{\h_1}\de_{\gotp_1}^{(2)})C_{\vec{v},\gotp_1}(\Pi_{\NN^{\perp}}G)
\\
C_{\vec{v},\gotp_{2}}(\Pi_{\NN^{\perp}}\hat{G})&\leq (1+K_{+}^{\h_1}\de_{\gotp_1})(C_{\vec{v},\gotp_2}(\Pi_{\NN^{\perp}}G)+K_+^{\gotp_0+2\tau+2}\tG_0C_{\vec{v},\gotp_{2}}(\Pi_{\NN^{\perp}}G )),
\end{aligned}
\end{equation}

Now we want to rewrite the field $\hat{F}$ in a form more similar to \eqref{barad8} by using \eqref{monaco200}.
We have
\begin{equation}\label{tartaruga66}
\begin{aligned}
\hat{F}&:=(1+\Phi_{4}^{-1}\tilde{h}_+)\left[\hat{N}_{0}^{+}+N_{4}^{(1)}+N_{4}^{(2)}+H^{(4)}
\right],\\
&\hat{N}_0^{+}:=({\oo}+\langle (H^{(3)})^{(\theta,0)}\rangle)\cdot\del_{\theta}+m_+{\Omega}^{-1}w\cdot\del_{w},\\
&N_{4}^{(1)}:=(\Phi_{4})_{*}N_{3}^{(1)}, \quad N_{4}^{(2)}:=(\Phi_{4})_{*}N_{3}^{(2)},\\
&H^{(4)}:=(\Phi_{4})_{*}H^{(3)}+(\Phi_{4})_{*}\tilde{N}_0-\hat{N}_0^{+}.
\end{aligned}
\end{equation}
%
This is another way to write \eqref{ordine100}. But now we give a precise estimates on the low norm of the component $\theta$ of the field $H^{(4)}$ on $\NN$.
First of all we have
$$
H^{(4)}:=\tilde{N}_0-\tilde{N}_0^{+}+[g,\tilde{N}_0]+\int_{0}^{1}\int_{0}^{t}(\Phi_{4})_{*}^{(s)}[g,[g,\tilde{N}_0]]+
H^{(3)}+\int_{0}^{1}(\Phi_{4})_{*}^{(s)}[g,H^{(3)}].
$$
Now if we look at the the component $(H^{(4)})^{(\theta,0)}$, by using equation \eqref{merda3}, we obtain
\begin{equation}\label{tartaruga77}
\begin{aligned}
(H^{(4)})^{(\theta,0)}&:=
\Pi_{K_+}^{\perp} (H^{(3)})^{(\theta,0)}(\theta)
+\int_{0}^{1}\int_{0}^{t}(\Phi_{4})_{*}^{(s)}[g,[g,\tilde{N}_0]]+\int_{0}^{1}(\Phi_{4})_{*}^{(s)}[g,H^{(3)}]\\
&=\Pi_{K_+}^{\perp} (H^{(3)})^{(\theta,0)}(\theta)+\int_{0}^{1}(\Phi_{4})_{*}^{(s)}[g,H^{(3)}]
+\int_{0}^{1}\int_{0}^{t}(\Phi_{4})_{*}^{(s)}[g,-\Pi_{K_+}(H^{(3)})^{(\theta,0)}],
\end{aligned}
\end{equation}
and hence, using \eqref{piccolopassero}, we get
\begin{equation}\label{monaco400}
\begin{aligned}
\|(H^{(4)})^{(\theta,0)}\|_{\vec{v},\gotp_{1}}&\leq 
K_{+}^{-(\gotp_2-\gotp_1)}C_{\vec{v},\gotp_2}((H^{(3)})^{(\theta,0)})+K_+^{\nu+2}\|g\|_{\vec{v},\gotp_1}\|(H^{(3)})^{(\theta,0)}\|_{\vec{v},\gotp_1}\\
&\stackrel{(\ref{door3}),(\ref{merda6})}{\leq}
K_{1}^{-(\gotp_2-\gotp_1)}(1+K_{1}^{\gotp_0+2\tau+1}\de_{\gotp_1}^{(2)})(\de_{\gotp_2}^{(1)}
+
\de_{\gotp_1}^{(1)}K_{+}^{\gotp_0+2\tau+2}\de_{\gotp_1}^{(2)})\\
&+(1+K_{+}^{\gotp_0+2\tau+1}\de_{\gotp_1}^{(2)})K_{+}^{\gotp_0+2\tau+2}\de_{\gotp_1}^{(2)}\de_{\gotp_1}^{(1)},
\end{aligned}
\end{equation}
which implies \eqref{passero2}.
Now by \eqref{tartaruga66} we have, by using \eqref{door2} and \eqref{finiamola33}, that
\begin{equation}\label{passero30}
\begin{aligned}
C_{\vec{v},\gotp_1}(N_{4}^{(2)})&\leq (1+K_{+}^{2\gotp_0+2}\de_{\gotp_1}^{(2)})^{3}\Big[K_{+}^{-(\gotp_2-\gotp_1-2\gotp_0-2\tau-2)}\de^{(2)}_{\gotp_2}+K_{+}^{2\gotp_0+2\tau+2}(\de_{\gotp_1}^{(2)})^{2}
\Big],\\
C_{\vec{v},\gotp_2}(N_{4}^{(2)})&\leq C_{\vec{v},\gotp_2}(N^{(2)})+K_{+}^{2\gotp_0+2\tau+4}C_{\vec{v},\gotp_1}(N^{(1)}).
\end{aligned}
\end{equation}
which implies \eqref{regu66}.
In the same way , using \eqref{door2}, \eqref{door} and \eqref{finiamola4} and Lemma \ref{diffeodiffeo2} 
the \eqref{regu6} follows.
Bounds in \eqref{merlo} follow by the discussion above recalling the results of Lemmata \ref{diffeodiffeo}, \ref{diffeodiffeo2} and \ref{multiopop}.
We conclude by setting
$$
\hat{N}^{(1)}=N_{4}^{(1)}, \quad \hat{N}^{(2)}=N_{4}^{(2)} \quad \hat{H}=H_{4},
$$
defined in \eqref{tartaruga66}.

%

\end{proof}

\begin{coro}[{\bf Compatible changes of variable}]\label{senonseicompa}
Under the hypotheses of Lemma \ref{regularization}, assume also that
\begin{equation}\label{piccolopassero1000}
\rho_+^{-1}K^{\h_{1}}\max\{\de^{(1)}_{\gotp_1},\de^{(2)}_{\gotp_1}\}\leq 
\e_0 K^{-1},
\quad \|h\|_{\vec{v},\gotp_1}\leq 2\g_0 \tG_0,   
\quad \max\{\de_{\gotp_2}^{(1)},\de_{\gotp_2}^{(2)}\}\leq \e_0 K^{\ka_1},\quad
\|h\|_{\vec{v},\gotp_2}\leq \g_0 \mathtt{G}_0 K^{\ka_1},
\end{equation}
where $\g_0$,$\e_0$,  $\tG_0$ as in \eqref{ini4},\eqref{parametri}, $\gotp_1=\gotp_0+\gotm_1$
as in \eqref{piccolino} and 
$\ka_1,\gotp_2$ in \ref{sceltapar}.
We have that 
the map $\calL_{+}$ in \eqref{mappamappa} is \emph{compatible}, according to Definition \ref{compa},
provided that
\begin{equation}\label{numeretti1000}
\ka_{3}:=\ka_1+\h_1,
\end{equation}
with $\h_1$ in \eqref{posacenere10}.
\end{coro}

\begin{proof}
The \eqref{satana2} is trivial. Bound \eqref{satana} follows by
\eqref{trans2} and \eqref{piccolopassero1000}.
Bounds in 
\eqref{odio} follow by \eqref{piccolopassero1000} and \eqref{passero}.
By the discussion in Lemma \ref{regularization} we also have that the subspace $\calE$ is preserved.
\end{proof}

\subsection{Inversion of the linearized operator in the normal directions }\label{sec7aut}

We consider a vector  field $F=N_0+G$ of the form 
\eqref{regu22} with all the properties in equations \eqref{barad9}-\eqref{corvo} and which satifies the hypotheses of Lemma \ref{regularization}.
We  set
\begin{equation}\label{corvo2}
\de_{p}:=\g^{-1}C_{\vec{v},p}(\Pi_{\calX}H).
\end{equation}
Now we apply Lemma \ref{regularization} to the field $F$ and we obtain the field $\hat{F}=(\calL_{+})_*F=(1+h_{+})(\hat{N}_0^{+}+\hat{N}^{(1)}+\hat{N}^{(2)}+\hat{H})$ 
 in \eqref{regu5}.
We want to describe a set 
$\calO'$ which satisfies the Mel'nikov conditions in \ref{pippopuffo3} for $(\hat F,K,\vec v_2^{(o)})$.
The parameters $K,\g_0$
are the same of Lemma \ref{regularization} while $\vec{v}_2^{(0)}$ is $\vec v_2$ of Lemma \ref{regularization} with $\cal O$ replaced by $\calO_0$. 
One can note that the conditions \eqref{buoni} and \eqref{cribbio4} on the operator $\gotW$ are equivalent to find an ``approximate'' solution
$g\in {\BB}_{\calE}$
of the equation
\begin{equation}\label{corvo3}
\Pi_{K_+}\Pi_{\calX}[g,N]=\Pi_{K_+}X, \qquad X \in\calX\cap\calE, 
\end{equation}
and where 
\begin{equation}\label{corvo4}
N=(1+h_+)\Big(\hat{N}_0^{+}+\hat{N}^{(1)}+\hat{N}^{(2)}+\Pi_{\NN}\hat{H}\Big),
\end{equation}
with $\hat{N}_0^{+}$ defined in \eqref{tartaruga66} and $h_{+}$ in \eqref{argento}.
Indeed by an explicit calculation equation \ref{corvo3} defining $\oo(\theta):=\hat{F}^{(\theta,0)}$ and $\Omega(\theta):=\hat{F}^{(w,w,)}[\cdot]$, becomes
\begin{equation}\label{corvo5}
\begin{aligned}
\Pi_{K_+}\Big(\oo(\theta)\cdot \del_{\theta}g^{(y)}+g^{(y,w)}\cdot \Omega(\theta)w\Big)\cdot\del_y+\Pi_{K_+}\Big(\oo(\theta)\cdot\del_{\theta}g^{(w,0)}-\Omega(\theta)g^{(w,0)}
\Big)\cdot\del_{w}=\Pi_{K_+}X.
\end{aligned}
\end{equation}
We recall that $\oo(\theta):=(1+h_+)(\oo_{+}+\hat{H}^{(\theta,0)})$ and 
$\Omega(\theta):=(1+h_{+})(m_{+}\Omega^{(-1)} + \hat{N}^{(1)}+\hat{N}^{(2)})$. 
By construction one of the effect of the map
$\calL_+$ is that the sizes of the terms $\hat{H}^{(\theta,0)}$ and $\hat{N}^{(2)}$  are ``much'' smaller with respect to the size of 
${H}^{(\theta,0)})$ and ${N}^{(2)}$ (see equations \eqref{passero2} and \eqref{regu66}). 

In the course of our algorithm we shall only need to find an approximate solution of \eqref{corvo5}, up to an error which is of the order of  $\hat{H}^{(\theta,0)}$ and $\hat{N}^{(2)}$.

Hence we ingnore those terms in \eqref{corvo5}. We are reduced  to solving
\begin{equation}\label{corvo6}
\begin{aligned}
&\oo_+\cdot\del_{\theta}g^{(y,0)}(\theta)=\Pi_{K_{+}}X^{(y,0)}\frac{1}{1+\Pi_{K_+}h_+},\\
&\oo_+\cdot\del_{\theta}g^{(y,y)}(\theta) y=\Pi_{K_{+}}X^{(y,y)}\frac{1}{1+\Pi_{K_+}h_+} y,\\
&\oo_+\cdot\del_{\theta}g^{(w,0)}(\theta)-\tilde{\Omega}(\theta)g^{(w,0)}=\Pi_{K_{+}}X^{(w,0)}\frac{1}{1+\Pi_{K_+}h_+},\\
&\oo_+\cdot\del_{\theta}g^{(y,w)}(\theta)\cdot w+g^{(y,w)}(\theta)\cdot \tilde{\Omega}(\theta)w=\Pi_{K_{+}}X^{(y,w)}w\frac{1}{1+\Pi_{K_+}h_+},
\end{aligned}
\end{equation}
with
\begin{equation}\label{corvo9}
\tilde{\Omega}(\theta):=(1+h_{+})(m_{+}\Omega^{(-1)}+\hat{N}^{(1)}). 
\end{equation} 
To solve the first two equations it is enough to ask that 
\begin{equation}\label{corvo77}
|\oo_{+}\cdot k|\geq\frac{\g}{\langle k\rangle^{\tau}}, \quad k\in\ZZZ^{d}, \;\; |k|\leq K_+.
\end{equation}
In other word thanks to the bound \eqref{corvo77} we are able to estimate the operator $W_0:=(\oo_+\cdot\del_{\theta})^{-1}$.

%
\noindent
The third equation has the form 
\begin{equation}\label{corvo8}
L u=\Big(\oo_+\cdot\del_{\theta}-\tilde{\Omega}(\theta)\Big)g=f, 
\end{equation}
for $g=g(\theta,x),f=f(\theta,x)$ maps
$$
g,f : \TTT_{s}^{d}\to \ell_{a,p},
$$
Hence we need to invert the an operator of the form  $L$ in \eqref{corvo8}.
\\
In order to solve the fourth equation in \eqref{corvo6}
we need to invert the operator $\Big(\oo_+\cdot\del_{\theta}+\tilde{\Omega}^{T}(\theta)\Big)$
. Here $\tilde{\Omega}^{T}(\theta)$ is traspose w.r.t. te bilinear form $a\cdot b=\sum_j a_j b_j$.
%
%
We remark that we need to invert  the latter operator on the space $\ell_{a,p}$ and not only on its dual
in order to get bounds like \eqref{buoni}.

We briefly explain our strategy.
In this Section, more precisely in Lemma \ref{kamreduc}, we will show that, under some
non degeneracy conditions on the eigenvalues of $\tilde{\Omega}(\theta)$ (see \eqref{corno7}),
one can construct a family of invertible operators $Q=Q(\theta)$ on $\ell_{a,p}$ such that
\begin{equation}\label{approxdiag} 
Q^{-1}LQ=Q^{-1}(\oo_{+}\cdot\del_{\theta}-\tilde\Omega(\theta))Q=\oo_{+}\cdot\del_{\theta}-(D+M),
\end{equation}
where $D$ is diagonal an $M$ a ``small'' remainder. This is actually what we get in \eqref{corno8} in Lemma \ref{kamreduc}.  Equation \eqref{approxdiag} implies that 
\begin{equation}\label{approxdiag2}
\begin{aligned}
\tilde{\Omega}(\theta)&:=-Q[\oo_{+}\cdot\del_{\theta}(Q^{-1})]+Q(D+M)Q^{-1},\\
{\tilde{\Omega}}^{T}(\theta)&:=-[\oo_{+}\cdot\del_{\theta}(Q^{-1})]^{T}Q^{T}+(Q^{-1})^{T}(D+M^{T})Q^{T}.
\end{aligned}
\end{equation}
Then we have
\begin{equation}\label{approxdiag3}
\begin{aligned}
Q^{T}(\oo_{+}\cdot\del_{\theta}+\tilde{\Omega}^{T}(\theta))(Q^{T})^{-1}&=
\oo_{+}\cdot\del_{\theta}+Q^{T}[\oo_{+}\cdot\del_{\theta}(Q^{T})^{-1}]+Q^{T}\tilde{\Omega}^{T}(\theta)(Q^{T})^{-1}\\
&\stackrel{(\ref{approxdiag2})}{=}\oo_{+}\cdot\del_{\theta}+Q^{T}[\oo_{+}\cdot\del_{\theta}(Q^{T})^{-1}]
-Q^{T}[\oo_{+}\del_{\theta}Q^{-1}]^{T}+Q^{T}(Q^{-1})^{T}(D+M^{T})
\end{aligned}
\end{equation}
Now one can show that $(Q^{T})^{-1}=(Q^{-1})^{T}$. Indeed by definition 
one has
\[
x\cdot y=(Q Q^{-1}x)\cdot y=(Q^{-1}x)\cdot (Q^{T}y)=x\cdot \big[(Q^{-1})^{T}Q^{T}y\big]
\]
for $x,y\in\ell_{a,p}$. This means that equation \eqref{approxdiag3} reads
\begin{equation}\label{approxdiag4}
Q^{T}(\oo_{+}\cdot\del_{\theta}+\tilde{\Omega}^{T}(\theta))(Q^{T})^{-1}=
\oo_{+}\cdot\del_{\theta}+D+M^{T}.
\end{equation}
In other words, if the matrix $Q$, diagonalizes approximatively
the operator $\oo_{+}\cdot\del_{\theta}-\tilde{\Omega}(\theta)$, then its adjoint $Q^{T}$
diagonalizes approximatively the operator $\oo_{+}\cdot\del_{\theta}+\tilde{\Omega}^{T}(\theta)$.
This procedure  clearly leaves the spectrum invariant.
As eplained above, we need to 
check that the matrix $Q^{T}$ acts on  $\ell_{a,p}$ (and not only on its dual space).
Actually this property is guaranteed by the fact that $Q$ is the identity plus a matrix with
finite decay norm
$|\cdot|^{\rm dec}_{s,a,p}$ (see Definition  \ref{decaynorm}).
Hence of course the adjoint matrix $Q^{T}$ satisfies the same bounds of $Q$.
The discussion above implies that if one can solve the third equation 
in \eqref{corvo6} in such a way bounds like \eqref{buoni} and \eqref{cribbio4} hold,
then one can do the same for the fourth equation
and give similar bounds. This is why all the rest of the Section 
is devoted to the study of equation \eqref{corvo8}.

\begin{defi}\label{spaces200}
Consider the spaces $X,Y,Z$ in \eqref{SPACES}. We define
\begin{equation}\label{fiore}
\begin{aligned}
{\bf G}:=\Big\{ {\bf u}:=(u,\bar{u})\; : \; u\in G \;\;{\rm s.t.} \;\;  u=\sum_{ l \in\ZZZ^{d},j\in S^{c}}u_{j}(l) e^{\ii l\cdot\theta} e^{\ii j\cdot x}
\Big\},
\end{aligned}
\end{equation}
for $G=X,Y,Z$ of $G=H^{s}(\TTT^{d}_{s}\times \TTT: \CCC)$ endowed with the norm $\|\cdot\|_{s,a,p}$ defined in 
\end{defi}
\noindent
We study is the invertibility of the operator $L$.
\begin{equation}\label{corvo10}
\begin{aligned}
L:=\Pi_{S}^{\perp}\oo_{+}\cdot\del_{\theta}+\Pi_{S}^{\perp}\left(-\ii E\left(\begin{matrix}m_{+} & 0\\ 0 & m_{+} \end{matrix}\right)\del_{xx}
-\ii E\left( \begin{matrix} a_{1} & b_{1}\\ -\bar{b}_1 & \bar{a}_{1} \end{matrix}\right)\del_{x}
-\ii E \left(\begin{matrix} a_{0} & b_{0}\\ -\bar{b}_0 & \bar{a}_{0} \end{matrix}
\right)
\right)\Pi_{S}^{\perp}+\hat{\mathscr{K}},
\end{aligned}
\end{equation}
where we rename $\hat{a}_{i},\hat{b}_{i}\rightsquigarrow{a}_{i},{b}_{i}$ the coefficients of $\hat{N}^{(1)}$.
In particular we have that $L$ in \eqref{corvo10} is the linearized operator of a field $F$ belonging to the subspace $\calE$
of compatible vector field in \eqref{compatibili}.
This means that $L$ is \emph{tame}, \emph{gauge preserving}, \emph{pseudo-differential}, \emph{reversible}
and \emph{real-on-real}, i.e. it belongs itself to $\calE$.

The inversion of $L$ stands on two fundamental 
results. The first is the following:

\begin{proposition}\label{birk}
Fix $\e_0=|\x|^{\frac{1}{4}}$ with $\x\in\e^{2}\Lambda$ (see \eqref{domain} and \eqref{parametri}), recall 
the definition of the parameters $\mathtt{R}_{0},\mathtt{G}_{0}$ in \eqref{parametri}
and that $\g_0:=\mathtt{c} |\x|$ (see \eqref{ini4}) and that $\gotp_1=\gotp_0+\gotm_1$ in \eqref{piccolino}.
Consider 
 $L$ defined for $\x\in\calO_0$ in $\calE$ of the form 
\begin{equation}\label{birk3}
\begin{aligned}
L &=\Pi_{S}^{\perp}\oo_+\cdot\del_{\theta}\uno+\Pi_{S}^{\perp}
\left(-\ii E\left(\begin{matrix}m_{+} & 0\\ 0 & m_{+} \end{matrix}\right)\del_{xx}
\right)\Pi_{S}^{\perp}+\hat{\mathscr{K}}\\
&+\Pi_{S}^{\perp}\left(
-\ii E\left( \begin{matrix} a_1 & b_{1} \\ \bar{b}_1 & \bar{a}_{1} \end{matrix}\right)\del_{x}
-\ii E \left(
 \begin{matrix} a_{0} & b_{0} \\ \bar{b}_0 & \bar{a}_{0} \end{matrix}
\right)
\right)
\Pi_{S}^{\perp},
\end{aligned}
\end{equation}
with $|m_{+}-1|_{\g}\leq C|\x|$, and, for $\x\in \calO\subseteq\calO_0$
\begin{equation}\label{birk11}
a_i=a^{(0)}_{i}+{a}'_{i}, \quad b_i:=b^{(0)}_{i}+{b}'_{i}, \;\; i=0,1,
\end{equation}
where
$\mathscr{\hat{K}}$ is of the form \eqref{maiale} with coefficients $c_i,d_i$,
 the coefficients 
$a^{(0)}_{i},b^{(0)}_{i}$ for $i=0,1$ are given by formul\ae  \eqref{NF2} while 
\begin{equation}\label{birk4}
\begin{aligned}
&\|a'_{i}\|_{\vec{v},\gotp_2},\|b'_{i}\|_{\vec{v},\gotp_2},
\|c_{i}\|_{\vec{v},\gotp_2},\|f_{i}\|_{\vec{v},\gotp_2}\leq C|\x| \mathtt{R}_0,\\
&\|a'_{i}\|_{\vec{v},\gotp_1},\|b'_{i}\|_{\vec{v},\gotp_1}\leq C|\x| \e_0, \quad i=0,1,\\
&\|c_{i}\|_{\vec{v},\gotp_1},\|f_{i}\|_{\vec{v},\gotp_1}\leq C|\x| \e_0, \quad i=1,\ldots,N,
\end{aligned}
\end{equation}
for some constant $C=C(\gotp_2,d)>0$. 
Then there exists a tame  map
\begin{equation}\label{birk5}
\SSSS =\uno+\Psi : {\bf H}\to {\bf H}
\end{equation}
with
\begin{equation}\label{birk6}
C_{\vec{v},\gotp_1}(\Psi)\leq C|\x|,
\end{equation}
such that the conjugated $L^{+}:=\SSSS^{-1} L \SSSS$ is in $\calE$ and has the form
\begin{equation}\label{birk7}
\begin{aligned}
L^{+} &=\Pi_{S}^{\perp}\oo_+\cdot\del_{\theta}+\Pi_{S}^{\perp}
\left(-\ii E\left(\begin{matrix}m_{+} & 0\\ 0 & m_{+} \end{matrix}\right)\del_{xx}-\ii E\left(\begin{matrix}{\rm diag}_{j\in\ZZZ_+} r^{(0)}_{j} & 0\\ 0 & {\rm diag }_{j\in\ZZZ_+} r^{(0)}_{j} \end{matrix}\right)
\right)\Pi_{S}^{\perp}+\hat{\mathscr{K}}^{+}\\
&+\Pi_{S}^{\perp}\left(
-\ii E\left( \begin{matrix} 0 & b^{+}_{1} \\ \bar{b}^{+}_1 & 0 \end{matrix}\right)\del_{x}
-\ii E \left(
 \begin{matrix} a^{+}_{0} & b^{+}_{0} \\ \bar{b}^{+}_0 & \bar{a}^{+}_{0} \end{matrix}
\right)
\right)
\Pi_{S}^{\perp},
\end{aligned}
\end{equation}
where  $r_j^{(0)}\in \RRR$
is such that $|r_{j}^{(0)}|_{\g}\leq C|\x|$, $\hat{\mathscr{K}}^{+}$ of the form \eqref{maiale} with coeffcients  $c^+_i,d^+_i$ for $i=1,\ldots,N+\mathtt{C}_{1}|S|$ ( where $|S|$ is the cardinality of the set $|S|$).
Moreover
\begin{equation}\label{birk8}
\|a^+_0-a_0\|_{\vec{v},\gotp_1}\leq C|\x|^{2}\e_0, \quad \|a^+_0-a_0\|_{\vec{v},\gotp_2}\leq C|\x|^{2}\mathtt{R}_0
\end{equation}
same for $b_{j}^{+}$, $j=0,1$ and 
\begin{equation}\label{birk4444}
\|c_{i}^{+}-c_{i}\|_{\vec{v},\gotp_1}\leq C|\x|^{2}\e_0, \quad \|c_{i}^{+}-c_{i}\|_{\vec{v},\gotp_2}\leq|\x|^{2}\mathtt{R}_0,
\end{equation}
same for 
$d^+_i$.

\end{proposition}

\begin{proof}
We divide the proof into two steps.

\noindent 
{\bf Step 1 - Descent Method.} In this step we want to eliminate the term  $a_1:=a^{(0)}_{1}+{a}'_{1}$ in the operator of order $O(\del_x)$. We follows the 
strategy used in Step $4$ of Section $3$ in \cite{FP}. We introduce a change of coordinates of the form
\begin{equation}\label{birk10}
\SSSS_1:=\uno+\Psi_1:=\uno+\Pi_{S}^{\perp}\left(\begin{matrix}s(\theta,x) & 0\\ 0 & \bar{s}(\theta,x) \end{matrix}\right) \Pi_{S}^{\perp}
\end{equation}
for a function $s $ small enough in such a way $\SSSS_1$ is invertible.
By a direct calculation we have that the coefficients 
\begin{equation}\label{birk12}
\begin{aligned}
a_{1}^{(1)}&:=2c_+\frac{s_{x}}{1+s}+a_{1},\qquad
a^{(1)}_{0}:=\frac{-i(\oo_+\cdot\del_{\f}s)+c_+ s_{xx}}{1+s}+a_{0},\\
b_{1}^{(1)}&:=b_{1}\frac{1+\bar{s}}{1+s}, \qquad
b^{(1)}_{0}:=b_{0}\frac{1+\bar{s}}{1+s}.
\end{aligned}
\end{equation}
We look for $s$ such that $a_1^{(1)}\equiv0$. Recall that by the reversibility one has on $\calU$ that $a_{1}$ has zero average in $x$. Hence,
by setting $1+s=\exp{(q(\theta,x))}$, one has that $a_{1}^{(1)}=0$ becomes
\begin{equation}\label{birk13}
\begin{aligned}
{\rm Re}(q_{x})&=-\frac{1}{2c_+}{\rm Re}(a_{1}), \quad \;\;
{\rm Im}(q_{x})=-\frac{1}{2c_+ }{\rm Im}(a_{1}),
\end{aligned}
\end{equation}
that have unique (with zero average in $x$) solution
\begin{equation}\label{birk16}
\begin{aligned}
{\rm Re}(q)&=-\frac{1}{2c_+}\del_{x}^{-1}{\rm Re}(a_{1}), \quad \;\;
{\rm Im}(q)=-\frac{1}{2c_+}\del_{x}^{-1}{\rm Im}(a_{1}).
\end{aligned}
\end{equation}
One has that the solution $q$ is satisfies the estimates
\begin{equation}\label{birk14}
\begin{aligned}
\|q\|_{\vec{v},p}&\leq C \|a_1\|_{\vec{v},p}, \\
\|q\|_{\vec{v},\gotp_1}&\leq C|\x|,
\end{aligned}
\end{equation}
where we used the estimate $|m_{+}-1|_\g\leq C |\x|$. Clearly the function $s$ satisfies the same estimates in \eqref{birk14}.
Hence we have obtained 
\begin{equation}\label{birk77}
\begin{aligned}
L_1 &:=\SSSS_1^{-1} L\SSSS=\oo_+\cdot\del_{\theta}+\Pi_{S}^{\perp}
\left(-\ii E\left(\begin{matrix}m_{+} & 0\\ 0 & m_{+} \end{matrix}\right)\del_{xx}
\right)\Pi_{S}^{\perp}+\hat{\mathscr{K}}^{(1)}\\
&+\Pi_{S}^{\perp}\left(
-\ii E\left( \begin{matrix} 0 & b^{(1)}_{1} \\ \bar{b}^{(1)}_1 & 0 \end{matrix}\right)\del_{x}+
-\ii E \left(
 \begin{matrix} a^{(1)}_{0} & b^{(1)}_{0} \\ \bar{b}^{(1)}_0 & \bar{a}^{(1)}_{0} \end{matrix}
\right)
\right)
\Pi_{S}^{\perp}.
\end{aligned}
\end{equation}
Now since the transformation $\SSSS_1=\uno+O(|\x|)$ trivially (see Lemma \eqref{multiopop}) one has again  that 
\begin{equation*}
a^{(1)}_0=a^{(0)}_{0}+{a}''_{0}, \quad b^{(1)}_i:=b^{(0)}_{i}+{b}''_{i}, \;\; i=0,1,
\end{equation*}
 the coefficients $a^{(1)}_0,b^{(1)}_j$ with $j=0,1$, and  the coefficients $c^{(1)}_i,d^{(1)}_i$ for $i=1,\ldots,N+\mathtt{C}_1|S|$ of $\hat{\mathscr{K}}^{(1)}$ satisfy
 the bounds \eqref{birk444}. 
 The study of the term $\hat{\mathtt{K}}^{(1)}$ follows by following the same reasoning used in Lemma \ref{diffeodiffeo}.
 Moreover by equation \eqref{birk16} one has that $q$ is even in $x$ and hence the transformation $\SSSS_1$ does not 
 change the parity of the coefficients. Moreover satisfies condition \eqref{massapilota} which implies that $\SSSS_{1}$
 satisfies the hypotheses of Lemma \ref{multiopop}.

\noindent
{\bf Step 2 - Linear Birkhoff Normal Form.}
In this step we look for a reversibility preserving map
\begin{equation}\label{corvo20}
\tilde{\SSSS}_2:=\uno+\Psi_{1}:=
\uno+\left(\begin{matrix} (\Psi_1)_{1}^{1}& (\Psi_1)_{1}^{-1} \\ (\Psi_1)_{-1}^{1} & (\Psi_1)_{-1}^{-1}
\end{matrix}
\right) : \Pi_{S}^{\perp}{\bf h}^{a,p}_{\rm odd}\to \Pi_{S}^{\perp}{\bf h}^{a,p}_{\rm odd},
\end{equation}
which eliminates the coefficients $a_0^{(0)},b_i^{(0)}$ for $i=0,1$.  First we write
\begin{equation}\label{corvo21}
L_1:=\oo_+\cdot\del_{\theta}+\Pi_{S}^{\perp}\left(-\ii E \left(\begin{matrix}m_+ & 0\\0& m_+ \end{matrix}
\right)\del_{xx}+B
\right)\Pi_{S}^{\perp}+R,
\end{equation}
where
\begin{equation}\label{corvo22}
\begin{aligned}
B&:=\Pi_{S}^{\perp}\left(-\ii E \left(\begin{matrix}0 & b_1^{(0)} \\ \bar{b}_1^{(0)} & 0 \end{matrix}
\right)\del_{x}
-\ii E \left(\begin{matrix}a_{0}^{(0)} & b_0^{(0)} \\ \bar{b}_0^{(0)} & \bar{a}_0^{(0)} \end{matrix}
\right)
\right)\Pi_{S}^{\perp}, \\
R&:=
 \Pi_{S}^{\perp}\left(-\ii E \left(\begin{matrix}0 & b''_1 \\ \bar{b}''_1 & 0 \end{matrix}
\right)\del_{x}
-\ii E \left(\begin{matrix}a''_{0} & b''_0 \\ \bar{b}''_0 & \bar{a}''_0 \end{matrix}
\right)
\right)\Pi_{S}^{\perp}+\hat{\mathscr{K}}^{(1)}
\end{aligned}
\end{equation}
We have that
\begin{equation}\label{corvo23}
\begin{aligned}
L_{1}\tilde\SSSS_2&-\tilde\SSSS_2\Pi_{S}^{\perp}(\oo_{+}\cdot\del_{\theta}-\ii E m_{+}\del_{xx})\Pi_{S}^{\perp}=\\
&=\Pi_{S}^{\perp}\left[\oo_{+}\cdot\del_{\theta}\uno+[-\ii Em_{+}\del_{xx}, \Psi_1]+B
\right]\Pi_{S}^{\perp}+\tilde{R}
\end{aligned}
\end{equation}
where 
\begin{equation}\label{corvo2444}
\tilde{R}:=B\Psi_1+R(\uno+\Psi).
\end{equation}
We look for  $\Psi_{1}$
such that
\begin{equation}\label{corvo24}
\oo_{+}\cdot\del_{\theta}\uno+[-\ii Em_{+}\del_{xx}, \Psi_1]+B=0.
\end{equation}
In Fourier space, using the exponential basis both in time and space, equation \eqref{corvo24} reads
\begin{equation}\label{corvo25}
i\oo_{+}\cdot l -i\s c_{+}(j^{2}-(\s \s')k^{2})(\Psi_{1})_{\s,j}^{\s',k}(l)=-B_{\s,j}^{\s',k}(l), \quad l\in\ZZZ^{d},\;\; j,k\in S^{c}, \;\;\s,\s'=\pm 1.
\end{equation}
Now by \eqref{corvo22} we have that the opeator $B$ depends only on the terms defined in \eqref{NF2}.
Moreover by \eqref{aacaut} and \eqref{variabili} we have that the function $v(\theta,x)$ has the form
\begin{equation}\label{corvo26}
\begin{aligned}
v(\theta,x)&:=\sum_{i=1}^d\sqrt{\x_{i}+y_{i}}e^{\ii \ell(\mathtt{v}_i)\cdot\theta}\sin(\tv_i x), \\
&\ell : S_{+} \in \ZZZ^{d}, \quad \ell(\mathtt{v}_i)=e_i,
\end{aligned}
\end{equation}
where $e_{i}=(0\ldots,1,\ldots,0)$ is the $i-th$ vector of the canonical basis of $\RRR^{d}$. 
Definitions in \eqref{aacaut} and \eqref{variabili} are given in the sine basis in space, since we deal with odd function of $x$. On the other hand in this case 
it is more convenient to use the exponential basis also in $x$.  
It is sufficient to change the definition in \eqref{aacaut} by recalling that for $u$ odd in space then the Fourier coefficients (in space) has the property $u_j=-u_{-j}$.
Hence ${\bf v}=(v^{+},v^{-})$ and $w=(z^{+},z^{-})$ are 
\begin{equation}
\begin{aligned}
v^{\s}&=\sum_{ \mathtt{v}_{i}\in S}{\rm sign}(\mathtt{v_{i}})e^{\ii \s\mathtt{v}_{i}\cdot x } \sqrt{\x_{|\mathtt{v}_{i}|}+y_{|\mathtt{v}_i|} }e^{\ii \s \ell(v_{i})\cdot \theta}, \x_{\mathtt{v}_{i}}=\x_i, y_{\mathtt{v}_i}=y_i,
\ell(-|\mathtt{v}_i|)=-\ell(|\mathtt{v}_i|),\\
z^{\s}&=\sum_{j\in S^{c}} u^{\s}_{j}e^{\ii\s j\cdot x}, \;\; u^{\s}_{-j}=-u^{\s}_{j}
\end{aligned}
\end{equation}
that are equivalent to definitions in \eqref{aacaut}. With this formalism we have that
\begin{equation}\label{corvo30}
\begin{aligned}
B_{\s,j}^{\s,k}(l)&:=-\ii \s(2a_1 -j_{1}^{2}a_{2}-j_{1}j_{2}a_{3}-a_{5}(j_1^{2}j_{2}^{2}+j_{2}^{2}))\sqrt{\x_{j_1}}\sqrt{\x_{j_2}}, \\
&\quad {\rm for} \;\;  j_{1},\,j_{2},\,j-k\in S, \;\;\; j_{1}-j_{2}+k-j=0, \;\; l=\ell(j_{1})-\ell(j_{2}), 
\end{aligned}
\end{equation}
and $B_{\s,j}^{\s,k}(l)=0$ otherwise, and
\begin{equation}\label{corvo31}
\begin{aligned}
B_{\s,j}^{-\s,k}(l)&:=-\ii \s\left(a_1-j_1^{2}a_{2}-(a_{3}j_1-a_{4}j_1^{2}j_2)(-\s)k
\right) 
, \\
&\quad {\rm for} \;\;  j_{1},\,j_{2},\,j+k\in S, \;\;\; j_{1}-j_{2}+k-j=0, \;\; l=\ell(j_{1})-\ell(j_{2}), 
\end{aligned}
\end{equation}
and $B_{\s,j}^{-\s,k}(l)=0$ otherwise. We define the solution of equation \eqref{corvo24} as
\begin{equation}\label{corvo32}
\begin{aligned}
(\Psi_{1})_{\s,j}^{\s',k}(l):=\left\{\begin{aligned}&
\frac{B_{\s,j}^{\s',k}(l)}{i(\oo_{+}\cdot l-\s m_{+}(j^{2}-(\s \s')k^{2}) )},\;{\rm if} \;\; \oo^{(-1)}\cdot l+\s j^{2}-\s' k^{2}\neq0,\\
&0 \;\; {\rm otherwise}
\end{aligned}\right.
\end{aligned}
\end{equation}
Actually the operator $\Psi_{1}$ in \eqref{corvo32} is the solution of 
\begin{equation}\label{corvo33}
\oo_{+}\cdot\del_{\theta}\uno+[-\ii Em_{+}\del_{xx}, \Psi_1]+B=[B], \qquad [B]:=\left\{\begin{aligned}& B_{\s,j}^{\s,j}(0), l=0, j=k, \s,=\s' \\
&0\;\; {\rm otherwise},
\end{aligned}\right.
\end{equation}
indeed for such values of $l,j,k$ the denominators in \eqref{corvo32} are zero.
Now $\Psi_1$ is well posed and solves \eqref{corvo33}.
Indeed by \eqref{corvo30} and \eqref{corvo31} we have that $B_{\s,j}^{\s',k}(\ell)=0$ unless $|l|\leq 2$, hence since $l=\ell(j_{1})-\ell(j_{2})$ one has
$$
|\oo^{(-1)}\cdot l+k^{2}-j^{2}|=|j_1^2-j_2^{2}+k^2-j^{2}|\geq1, \; {\rm for} \;\; j\neq \pm k , j_1\neq\pm j_2,
$$
where one uses Lemma \ref{hyp3autaut}. Moreover for $\s=-\s'$ one has 
$$
|\oo^{(-1)}\cdot l+k^{2}+j^{2}|\geq1, \; \; {\rm for \;\; any} \;\; j,k, j_1,j_2.
$$
By using that $\oo_{+}=\oo^{(-1)}+O(\x)$ and that $m_+=1+O(|\x|)$  we have that the denominators in \eqref{corvo32} satisfy
\begin{equation}\label{den}
|\oo_{+}\cdot l-\s c_{+}(j^{2}-(\s \s')k^{2})|\geq|m_{+}|\oo^{(-1)}\cdot l+\s(j^{2}-k^{2})|-|l||\oo_+-m_{+}\oo^{(-1)} |  |\geq 
\frac{1}{2},
\end{equation}
for $|\x|$ small enough.
\begin{lemma}\label{denn}
Set $d_{\s,j}^{\s',k}(l)=(\oo_{+}\cdot l-\s m_{+}(j^{2}-(\s \s')k^{2}) )$. One has that
\begin{equation}\label{den3}
|d_{\s,j}^{\s',k}(l)|
\geq \left\{\begin{aligned} 
& C(j^{2}+k^{2}), \quad \s= -\s' ,\\
& C (|j|+|k|), \quad {\rm if}\;\;\; \s=\s'\;  j\neq k,\\
& C, \quad {\rm if} \;\;\;\s=\s'\;  j=\pm k , \; l\neq0
\end{aligned}\right.
%
\end{equation}
\end{lemma}

\begin{proof}
If one assume $j^{2}+k^{2}>\tilde{C}>0$ then, since $|\oo_{+}|\leq |\oo^{(-1)}|+1$, one has
\begin{equation}
|d_{\s,j}^{-\s,k}(l)|\geq \frac{1}{4}|j^{2}+k^{2}|-2|\oo_+|\geq \frac{1}{8}(j^{2}+k^{2}).
\end{equation}
If $j^{2}+k^{2}\leq\tilde{C}$ but $j-k\in S$ then one can use equation \eqref{den} to obtain the result. Finally if $j-k\in S^{c}$
then one has $B_{\s,j}^{-\s,k}(l)=0$. The second bound is obtained following the same reasoning and using the fact that
$|j^{2}-k^{2}|=|(j-k)(j+k)|\geq|j|+|k|$. The last bound is trivial.
\end{proof}

The following properties is a consequence of Lemma \ref{denn}.

\begin{lemma}\label{den2}
Let us define the operator $A:=\Psi_1-\{\Psi_1\}$ where
$\{\Psi_1\}_{\s,j}^{\s',k}(l)=\Psi_{\s,j}^{\s,j}(l)$ for $\s=\s'$, $j=k$ and $l\neq0$.
Then one has that $|A\del_x|^{\rm dec}_{\vec{v},p}+|\del _xA|^{\rm dec}_{\vec{v},p}\leq C(p)|\x|$ where $|\cdot|^{\rm dec}_{\vec{v},p}$ is defined in \eqref{2.1aut}
with $j\in \ZZZ$ instead of $\ZZZ_+$.
\end{lemma}

\begin{proof} 
One has that
\begin{equation}
|\Psi_1\del_x |_{s,a,p}=\sup_{\s,\s'=\pm}\sum_{h\in\ZZZ, l\in\ZZZ^{d}}\langle j,l\rangle^{2p}e^{2|l|s}e^{2|j|a}\sup_{j-k=k}|(\Psi_1)_{\s,j}^{\s',k}(l) k|^{2}\leq C(p)|\x|
\end{equation}
since $(\Psi_1)_{\s,j}^{\s',k}(l)=0$ outside the set $|l|\leq 2$ and $|j-k|\leq C_{S}$
and the decay norm of $B$ is controlled by the norm of its coefficients $a_0^{(0)},b_1^{(0)},b_0^{(0)}$. 
In particular note that we used Lemma \ref{denn} in the following way. For instance we have the estimate
\begin{equation} 
|(\Psi_1)_{\s,j}^{-\s,k}(l) k|\leq C\frac{1}{(j^{2}+k^{2})}|B_{\s,j}^{-\s,k}(l)k|
\end{equation}
and one uses the gain of two derivatives of the denominator to control the two derivatives in the numerator. Hence one control the coefficients using $\|b_1^{(0)}\|_{\vec{v},p}+\|b_0^{(0)}\|_{\vec{v},p}$
The bounds second term and the Lipschitz estimate follows in the same way.
\end{proof}

 By Lemma \ref{den2} follows that for $|\x|$ small the map $\tilde{\SSSS}_{2}$ is invertible. Moreover we have the following Lemma
 \begin{lemma}\label{reversibilit}
Consider a linear  operator $A=(A)_{\s}^{\s'}$ for $\s,\s'=\pm1$ on the spaces ${\bf G}:=G\times G$ where $G=X,Y,Z$ the spaces defined in 
\eqref{reversible}. One has that $A$ is reversibility preserving if and only if for any $\s,\s'=\pm 1$, $l\in \ZZZ^{d}$, $j,k\in\ZZZ$
\begin{equation}\label{cornacchia}
A_{\s,j}^{\s',k}(l)=\ol{A_{\s,j}^{\s',k}(l)}, \quad A_{\s,-j}^{\s',-k}(l)=A_{\s,j}^{\s',k}(l), \quad \ol{A_{\s,j}^{\s',k}(-l)}=A_{-\s,j}^{-\s',k}(l).
\end{equation}
An operator $B$ is reversible, i.e. maps ${\bf X}\to {\bf Z}$ if and only if
\begin{equation}\label{cornacchia2}
B_{\s,j}^{\s',k}(l)=-\ol{B_{\s,j}^{\s',k}(l)}, \quad B_{\s,-j}^{\s',-k}(l)=B_{\s,j}^{\s',k}(l), \quad \ol{B_{\s,j}^{\s',k}(-l)}=B_{-\s,j}^{-\s',k}(l).
\end{equation}

\end{lemma}

The proof of Lemma \ref{reversibilit} is similar to the proof of Lemma $(4.36)$ in \cite{FP}. Clearly in this case the differences stands in the fact that
we developed in Fourier coefficients using the exponential basis in $x$. 
By this Lemma and an explicit computation, we have that $\Psi_{1}$ is reversibility preserving since $B$ is reversible.
Now we can define the map
\begin{equation}\label{mappa}
\SSSS_{2}:=\exp(\Psi_1),
\end{equation}
the time$-1$ flow of the field $\Psi_1$. Clearly again $\SSSS_1-\uno=O(|\x|)$. Hence 
by equation
\eqref{corvo33} we obtain 
$$
L^{+}:=\SSSS_2^{-1}L_1\SSSS_2 =\Pi_{S}^{\perp}(\ii E\Omega_{+}^{-1}(\x)+R_{+})\Pi_{S}^{\perp}
$$
with
\begin{equation}\label{mappa2}
\Omega_{+}^{-1}:={\rm diag}_{j\in \ZZZ_{+}}\left(\left(\begin{matrix}
m_{+}j^{2} & 0\\ 0 & m_+ j^{2}\end{matrix}\right)
+\left(\begin{matrix}B_{1,j}^{1,j}(0) & B_{1,j}^{1,-j}(0)  \\
B_{1,-j}^{1,j}(0) & B_{1,-j}^{1,-j}(0)\end{matrix}\right)
\right)
\end{equation}
and $R_{+}$ as the form \eqref{birk7}. 
Note that we have defined $\Omega_{+}^{-1}$ as infinite dimensional matrix with index $\ell\in\ZZZ^{d}$ and $j\in\ZZZ_{+}$.
It is an operator one the space of sequences $\{z_{j}\}_{j\in\ZZZ}$. But by our condition of reversibility
we work on sequences such that $z_{j}=-z_{-j}$. Hence we can rewrite the matrix $\Omega^{(-1)}_{+}$ as an operator 
acting on the space of ``odd'' sequences 
as a diagonal matrix
\begin{equation}\label{strega10}
\Omega_{+}^{-1}:={\rm diag}_{j\in \ZZZ_{+}}\left(c_{+}j^{2}+r_{j}^{(0)}
\right), \qquad r_{j}^{(0)}:= B_{1,j}^{1,j}(0)-B_{1,j}^{1,-j}(0),
\end{equation}
and $r_{j}^{(0)}$ is real by the reversibility of the field $B$.
Hence, setting $\SSSS=\SSSS_2\circ\SSSS_{1}$, the Lemma is proved.
\end{proof}

\begin{rmk}\label{laghetto2}
The terms $r_0^{j}$  are of order $O(|\x|)$. In particular they are
the integrable terms that cannot be cancelled through a Birkhoff
transformation. Moreover such terms are the corrections of order $O(|\x|)$ to $j^{2}$ that we have considered in \eqref{integra2}
of Section \ref{caspita2}.
\end{rmk}

The following Lemma is the last important abstract result we will use in order to invert the operator of the type $L$ in \eqref{corvo10}.

\begin{lemma}\label{kamreduc}
Fix $\e_0=|\x|^{\frac{1}{4}}$ with $\x\in\e^{2}\Lambda:=\calO_0$ (see \eqref{domain} and \eqref{parametri}) and recall that 
$\g_0:=\mathtt{c} |\x|$ (see \eqref{ini4}) and $\gotm_1$ in \eqref{piccolino}.
Consider a reversible, tame linear operator 
 $L$ defined for $\x\in \calO_0$ of the form 
\begin{equation}\label{corno1}
L=
\oo_{+}\cdot\del_{\theta}+ \DD+\RR : \ell_{a,p+2}\to \ell_{a,p}
\end{equation}
where
\begin{equation}\label{corno2}
\begin{aligned}
\DD&:=-\ii E{\rm diag}_{j\in \ZZZ_+\cap S^{c}}\Big( c_{+}j^{2}+r_{j}^{(0)}\Big),\\
\RR&:=E_{1}D+E_{0}=E_1(L)D+E_0(L)
\end{aligned}
\end{equation}
with $D:={\rm diag}_{j\in\ZZZ_+\cap S^{c}}\{ \ii j\}$, and where, if we write $k=(\s,j,p)\in\{\pm1\}\times\NNN\times \ZZZ^{d}$,with 
$ q=0,1$,
\begin{equation}\label{corno3}
\begin{aligned}
E_{q}&=\left((E_{q})_{k}^{k'}\right)_{k,k'\in\{\pm1\}\times S^{c} \times \ZZZ^{d}}=
\left((E_{q})_{\s,j}^{\s',j'}(l)\right)_{k,k'\in\{\pm1\}\times S^{c}\times \ZZZ^{d}},\\
&(E_{1})_{\s,j}^{\s,j'}(l)\equiv0, \quad \forall\; j,j'\in\ZZZ_+\cap S^{c}, \;\; l\in\ZZZ^{d}.
\end{aligned}
\end{equation}
%
Assume that $|m_{+}-1|_{\g}\leq C|\x|$, and $|r_{+}^{j}|_{\g}\leq C|\x|$. Fix parameters
\begin{equation}\label{corno4}
\ka_4=7\tau+3, \quad \ka_{5}=7\tau+5,\quad \gotm_2=\gotm_1+\ka_{5},\quad
\end{equation}
$\gotp_1,\gotp_2$ as in \ref{sceltapar}
and take an arbitrary $N>0$ large.  Assume that
\begin{equation}\label{corvo7}
\begin{aligned}
|E_1|^{{\rm dec}}_{\vec{v},\gotp_1}+|E_0|^{{\rm dec}}_{\vec{v},\gotp_1}&\leq C|\x| \e_0, \\ 
|E_1|^{{\rm dec}}_{\vec{v},\gotp_2}+|E_0|^{{\rm dec}}_{\vec{v},\gotp_2}&\leq C\e_0 \tG_0, 
\end{aligned}
\end{equation}
with $\tG_0$ in \eqref{parametri}.
There exists a constant $C_0=C_0(\gotp_2,d)>0$ such that, if
\begin{equation}\label{corno5}
K_0^{C_0}\g^{-1}C|\x|\e_0 \leq \epsilon
\end{equation}
and $\epsilon=\epsilon(d,\gotp_2)$ is small enough then the following holds.
There exists a sequence of purely imaginary numbers
\begin{equation*}
\mu_{\s,j}^{N}(\x):=-\ii \s (c_{+}j^{2}+r_{j}^{N}), \quad \s=\pm1, \; j\in \ZZZ_{+}\cap S^{c},
\end{equation*}
with
$$
|r_{j}^{N}|_{\g}\leq C |\x|, \quad |r_{j}^{N}-r_{j}^{(0)}|_{\g}\leq C|\x|\e_0
$$
defined on $\calO_0$ and 
 such that for any $\x\in \Lambda^{2\g}_{N}$, defined as
 \begin{equation}\label{corno7}
 \Lambda_{N}^{2\g}:=\left\{
\begin{array}{ll}
\x\in \calO: &\;  |\oo_+\cdot l \!+\!\mu^{N}_{\s,j}(\x)-\!\mu^{N}_{\s',j'}(\x)| \geq \frac{2\g|\s j^{2}-\s' j'^{2}|}{\langle l \rangle^{\tau}}
  \\  
& \forall l \in\ZZZ^{d},\; |l|\leq N \;\; \forall (\s,j),(\s',j')\in\{\pm1\}\times(\ZZZ_+ \cap S^{c})
\end{array}
\right\},
 \end{equation}
 there exists a bounded, reversibility preserving, linear operator $\Phi_{N}=\Phi_{N}(\x)$ 
 depending on $\theta\in \TTT^{d}_{s}$ and acting on $\ell_{a,p}$
 such that
 \begin{equation}\label{corno8}
 L_{N}:=\Phi_{N}^{-1}\circ L\circ \Phi_{N}:=\oo_{+}\cdot\del_{\theta}+\DD_N+\RR_N,
 \end{equation}
 where
 \begin{equation}\label{corno9}
 \begin{aligned}
 \DD_{N}&:={\rm diag}_{\s\in \{\pm1\}, j\in\ZZZ_+}(\mu_{\s,j}^{N}),\\
 \RR_{N}:&=E_{1}^{N}D+E_{0}^N,
 \end{aligned}
 \end{equation}
 
 \begin{equation}\label{corno10}
 \begin{aligned}
 |E_1^{N}|^{{\rm dec}}_{\vec{v}_1,p}+|E_0^{N}|^{{\rm dec}}_{\vec{v}_1,p}&\leq \Big(|E_1|_{\vec{v},p+\ka_5}+|E_0|_{\vec{v},p+\ka_5}
 \Big)
 N^{-\ka_4}, \quad \vec{v}_1:=(\g, \Lambda_{N}^{2\g},s,a),\\
  |E_1^{N}|^{{\rm dec}}_{\vec{v}_1,p+\ka_5}+|E_0^{N}|^{{\rm dec}}_{\vec{v}_1,p+\ka_5}&
   \Big(|E_1|^{{\rm dec}}_{\vec{v},p+\ka_5}+|E_0|^{{\rm dec}}_{\vec{v},p+\ka_5}
 \Big)N,\quad \gotp_0\leq p \leq \gotp_2-\ka_5.
  \end{aligned}
 \end{equation}
 Moreover one has that
 \begin{equation}\label{corno11}
 |\Phi_{N}^{\pm1}-\uno|^{{\rm dec}}_{\vec{v}_1,p}\leq \g^{-1}\Big(|E_1|^{{\rm dec}}_{\vec{v},p}
 +|E_0|^{{\rm dec}}_{\vec{v},p}\Big).
 \end{equation}

\end{lemma}

Before giving the proof of the Lemma we make some remarks. This Lemma essentially can be applied to 
operators $L^{+}$ of the form \eqref{birk7}. Indeed our strategy is to use Proposition \ref{birk} as a preliminary step
before using 
a \emph{KAM}-like scheme in order to diagonalized the linear operator $L$. Lemma \ref{kamreduc} provides an approximate diagonalization, but anyway
the order of the approximation $N$ is arbitrary large. The conditions on the parameters in \eqref{corno7} are the Second order Melnikov conditions. 
Clearly such conditions depends on $N$ (see formula \eqref{corno7}).
 In particular to obtain a partial diagonalization one can  ask for the conditions \eqref{corno7} only for $|l|\leq N$. On the contrary 
 in order to completely diagonalize one has to ask the the lower bounds in \eqref{corno7} for any $l\in\ZZZ^{d}$. Our choice is less restrictive but it is sufficient 
  we are just looking for an
 approximate inverse of $L$. 
 Lemma \eqref{kamreduc} is the equivalent result of Theorem $4.27$ in Section $4$ of \cite{FP}. The proof of the Lemma above is based on the following  Iterative Lemma.
Take $L$ as in \eqref{corno1} and define
\begin{equation}\label{piccolo11}
\e_{p}^{0}:=|E_{1}|^{{\rm dec}}_{\vec{v},p}+|E_{0}|^{{\rm dec}}_{\vec{v},p}, \quad {\rm for} \quad p\geq0.
\end{equation}
%

\begin{lemma}[{\bf KAM iteration}]\label{teo:KAMKAM}
There exist constant $C_{0}>0$, $K_{0}\in\NNN$ large, such that
if
\begin{equation}\label{eq:4.155}
K_{0}^{C_{0}}\g^{-1}\e_{\gotp_{0}+\gotm_2}^{0}\leq1,
\end{equation}
with $\gotm_2$ in \eqref{corno4}
then, for any $\nu\geq0$, one has:

\noindent
$({\bf S1.})_{\nu}$ Set
$\Lambda^{\g}_{0}:=\calO_{0}$ and for $\nu\geq1$ 
\begin{equation}\label{eq:419bis10}
\begin{aligned}
&\Lambda_{\nu}^{\g}:=
\left\{\begin{array}{ll}
\x\in\Lambda_{\nu-1}^{\g} : 
&|\oo\cdot\ell\!+\! \mu_{\s,j}^{\nu-1}(\x)\!-\! \mu_{\s',j'}^{\nu-1}(\x)|\geq
\frac{\g|\s j^{2}-\s'j'^{2}|}{\langle\ell\rangle^{\tau}},\\
&\forall |\ell|\leq K_{\nu-1}, \! (\s,j),(\s',j')\in\{\pm1\}\times\NNN \end{array}
\right\},
\end{aligned}
\end{equation}
For any $\x\in\Lambda_{\nu}^{\g}=\Lambda_{\nu}^{\g}(L)$, there exists an invertible map $\Phi_{\n-1}$ of the form $\Phi_{-1}=\uno$ and for $\nu\geq1$,  $\Phi_{\nu-1}:=\uno+\Psi_{\nu-1}: {\bf H}^{s}\to{\bf H}^{s}$, with the following properties.

 The maps $\Phi_{\nu-1}$, $\Phi_{\nu-1}^{-1}$
are reversibility-preserving according to Definition \ref{reversible},
moreover $\Psi_{\nu-1}$ is T\"oplitz in time, $\Psi_{\nu-1}:=\Psi_{\nu-1}(\f)$ (see (\ref{eq:2.16aut})) and satisfies the bounds:
\begin{equation}\label{eq:4.2210}
\begin{aligned}
|\Psi_{\nu-1}|^{{\rm dec}}_{\vec{v}_\nu,p}\leq 
\e_{p+\ka_5}^{0} K_{\nu-1}^{2\tau+1}K_{\nu-2}^{-\al},\quad \vec{v}_\nu:=(\g,\Lambda^{\g}_{\nu},s,a).
\end{aligned}
\end{equation}
Setting, for $\nu\geq 1$, $L_{\nu}:=\Phi_{\nu-1}^{-1}L_{\nu-1}\Phi_{\nu-1}$, we have:
\begin{equation}\label{eq:4.1610}
\begin{aligned}
L_{\nu}&=\oo\cdot\del_{\f} \uno +\DD_{\nu}+\RR_{\nu}, \qquad
\DD_{\nu}={\rm diag}_{(\s,j)\in\{\pm1\}\times\NNN}\{\mu_{\s,j}^{\nu}\},
\\
\quad
\mu_{\s,j}^{\nu}&=\mu_{\s,j}^{0}(\x)+r_{\s,j}^{\nu}(\x), \quad 
\mu_{\s,j}^{0}(0)=-\s \ii(m_+ j^{2}+r_0^{j})
\end{aligned}
\end{equation}
and
\begin{equation}\label{eq:4.1810}
\RR_{\nu}=E_{1}^{\nu}(\x)D+E_{0}^{\nu}(\x),
\end{equation}
where $\RR_\nu$ is reversible and the matrices
$E_{q}^{\nu}$ satisfy (\ref{corno3}) for $q=1,2$.
For $\nu\geq0$ one has $r_{\s,j}^{\nu}\in i \RRR$,  $r_{\s,j}^{\nu}=-r_{-\s,j}^{\nu}$ and the following bound holds:
\begin{equation}\label{eq:4.2010}
|r_{\s,j}^{\nu}|_{\g}:=|r_{\s,j}^{\nu}|_{\Lambda_{\nu}^{\g},\g}\leq|\x|\de C.
\end{equation}
Finally, if we define
\begin{equation}\label{piccolo210}
\e_{p}^{\nu}:=|E_{1}^{\nu}|^{{\rm dec}}_{\vec{v}_\nu,p}+|E_{0}^{\nu}|^{{\rm dec}}_{\vec{v}_\nu,p},
\quad \forall p\geq0,
\end{equation}
one has $\forall\;p\in[\gots_{0},\gotp_2-\ka_5]$ ($\ka_5$ is defined in \eqref{corno4}) and $\nu\geq0$
\begin{equation}\label{eq:4.2110}
\begin{aligned}
\e_{p}^{\nu}&\leq \e_{p+\ka_5}^{0}K_{\nu-1}^{-\ka_4},\\
\e_{p+\ka_5}^{\nu}&\leq \e_{p+\ka_5}^{0}K_{\nu-1}.
\end{aligned}
\end{equation}

\noindent
$({\bf S2})_{\nu}$ For all $j\in\NNN$ there exists  Lipschitz extensions
$\tilde{\mu}_{h}^{\nu}(\cdot) : \calO_0 \to \RRR$ of $\mu_{h}^{\nu}(\cdot):\Lambda_{\nu}^{\g}\to \RRR$,
such that for $\nu\geq1$,
\begin{equation}\label{eq:4.2310}
|\tilde{\mu}_{\s,j}^{\nu}-\tilde{\mu}_{\s,j}^{\nu-1}|_{\g}\leq
\e_{\gots_{0}}^{\nu-1}, \qquad \forall\; k\in\{\pm1\}\times\NNN.
\end{equation}

\noindent
$({\bf S3})_{\nu}$ Let $L_1$ and $L_2$ as in \eqref{corno1}, defined on $\calO_0$ such that \eqref{corno5} and \eqref{eq:4.155}
 hold.
Then for $\nu\geq0$, for any $\x\in\Lambda_{\nu}^{\g_{1}}(L_1)\cap\Lambda_{\nu}^{\g_{2}}(L_2)$,
with $\g_{1},\g_{2}\in[\g/2,2\g]$, one has
\begin{subequations}\label{eq:4.2410}
\begin{align}
&
|E_{0}^{\nu}(L_{1})\!-\!E_{0}^{\nu}(L_{2})|^{{\rm dec}}_{\vec{v},\gotp_{0}}\leq
 N_{\nu-1}^{-\ka_5} |E_{0}(L_{1})\!-\!E_{0}(L_{2})|^{{\rm dec}}_{\vec{v},\gotp_{0}},
\label{eq:4.24a10}\\
&
\! \max_{i=1,0}|E_{i}^{\nu}(L_{1})\!-\!E_{i}^{\nu}(L_{2})|^{{\rm dec}}_{\vec{v},\gotp_{0}+\ka_5}\!\leq\!
 N_{\nu-1}\!\max_{i=1,0}|E_{i}(L_{1})\!-\!E_{i}(L_{2})|^{{\rm dec}}_{\vec{v},\gotp_{0}+\ka_5}
\end{align}
\end{subequations}
with $\vec{v}:=(\g,\Lambda_{\nu}^{\g_{1}}(L_1)\cap\Lambda_{\nu}^{\g_{2}}(L_2),s,a)$,
and moreover, for $\nu\geq1$, for any $p\in[\gotp_{0},\gotp_{0}+\ka_5]$, for any $(\s,j)\in\{\pm1\}\times\NNN$
and for any $\x\in\Lambda_{\nu}^{\g_{1}}(L_1)\cap\Lambda_{\nu}^{\g_{2}}(L_2)$,
\begin{subequations}\label{eq:4.24bis10}
\begin{align}
&\!\!\!\!|(r_{\s,j}^{\nu}(L_{2})-r_{\s,j}^{\nu}(L_{1}))\!-\!(r_{\s,j}^{\nu-1}(L_{2})\!-\!r_{s,j}^{\nu-1}(L_{1}))|
\!\leq\!\max_{i=1,0}|E_{i}^{\nu-1}(L_{1})\!-\!E_{i}^{\nu-1}(L_{2})|^{{\rm dec}}_{s,a,\gotp_{0}},\\
&|(r_{\s,j}^{\nu}(L_{2})-r_{\s,j}^{\nu}(L_{1}))|\leq  C \max_{i=1,0} |E_{i}(L_{1})\!-\!E_{i}(L_{2})|^{{\rm dec}}_{\vec{v},\gotp_{0}}.
\label{aaaa}
\end{align}
\end{subequations}

\noindent
$({\bf S4})_{\nu}$ Let $L_{1},L_{2}$ be as in $({\bf S3})_{\nu}$ and $0<\rho<\g/2$.
For any $\nu\geq0$ one has
\begin{equation}\label{eq:4.2510}
 CK_{\nu-1}^{\tau}\max_{i=1,0} |E_{i}(L_{1})\!-\!E_{i}(L_{2})|^{{\rm dec}}_{\vec{v},\gotp_{0}}
 \leq \rho
\quad \Rightarrow \quad 
\Lambda_{\nu}^{\g}(L_{1})\subset \Lambda_{\nu}^{\g-\rho}(L_{2}),
\end{equation}

\end{lemma}

\begin{proof}[Proof of Lemmata \ref{kamreduc} and \ref{teo:KAMKAM}]
The proof is the same that the one of Theorem $4.27$ 
in \cite{FP}
 which is based on the iterative scheme 
in Lemma $4.38$ 
proved in \cite{FP}. 
Here Lemma \ref{kamreduc} follows from Lemma \ref{teo:KAMKAM}. The proof of Lemma \ref{teo:KAMKAM} is similar to the one of Lemma 
 $4.38$ in \cite{FP}.
Indeed by hypothesis 
the operator $L$ in \eqref{corno1} has the same class of operators defined in Definition 
$4.37$ 
of \cite{FP} 
and moreover
smallness condition in \eqref{corvo7} is the equivalent of the smallness condition of $\g^{-1}\e$ in 
Theorem $4.27$ 
in \cite{FP}.
One difference is that here the frequency $\oo_+$ depends on parameters $\x\in \RRR^{d}$, while in 
\cite{FP}
there is only one-dimensional parameters $\la\in \RRR$ that modulate $\oo$. Anyway 
there are no differences in the proof since Kirszbraun's Theorem on Lipschitz extension of functions 
holds in $\RRR^{d}$ (see Lemma A.2 in \cite{Po2}).
The proofs of items $(S3)_\nu,(S4)_{\nu}$ of Lemma \ref{teo:KAMKAM} are the same of those
of items $(S3)_\nu,(S4)_{\nu}$ of Lemma 
$4.37$ in \cite{FP}. 
The difference is in the fact that in 
\cite{FP} 
one considers the same 
linear operator $L$ that is the linearized of the same non linear operator on two different points $u_1$ and $u_2$. Moreover the difference
of $L(u_1)$ and $L(u_2)$ is given in terms of  the difference of $u_1$ and $u_2$. In other words the operators are close to each other
if the points $u_i$ are close. Here one gives the estimates on the differences of $r_{\s,j}^{\nu}(L_1)$ and
$r_{\s,j}^{\nu}(L_2)$  directly in terms of the differences of the two operators $L_1$ and $L_2$.

Another difference is that in Section 4 in \cite{FP}
one gets a complete diagonalization. This is obtained by applying
infinitely many changes of coordinates that approximatively diagonalize.
In
this case, to prove formula \eqref{corno8} it is enough to consider $\Phi_{N}$ the composition of a \emph{finite} number of
changes of coordinates. That is why the set of parameters in \eqref{corno7} is defined for $|l|\leq N$.
The last difference is that here the sites $j\in S^{c}$ instead of $\ZZZ_+$. 
\end{proof}

\begin{rmk}{\bf Approximate eigenvalues}\label{approxeig}
In Theorem 
$4.27$ in \cite{FP}, 
given an operator $L$ one construct the eigenvalues $\m_{\s,j}^{\infty}$
as limit of  some ``approximate'' eigenvalues $\mu_{\s,j}^{\nu}$, for $\nu\geq1$. Here 
we have that we stop the sequence of $\mu_{\s,j}^{\nu}$ after the number of steps one need in order to get 
the approximation of order $N$ in \eqref{corno9},\eqref{corno10} and the one defines $\mu_{\s,j}^{N}$ as the last term of such sequence.
Moreover in 
\cite{FP} 
the operator $L$
 is the linearize operator of a field $F_0$ in some point $u$. Theorem 
 $4.27$ 
 provides also Lipschitz dependence of
 the approximate eigenvalues $\m_{\s,j}^{\nu}(u)$ with respect the point $u$.  
 Here the situation is different. 
 As we will see the operator $L$ comes form the linearization in zero of some vector field $F_{1}$. Hence while in 
\cite{FP}
 one has to control the difference between the eigenvalues of $L(u_1)$ and $L(u_2)$, i.e
the linearized operator of $F_0$ in two different functions $u_1,u_2$, here we need to control
the differences between  the eigenvalues of the linearized operators of two different fields $F_1, F_2$. 
If one knows that $L_1$ is ``close'' to $L_2$ (clearly one has to explain the meaning of ``close'')
then the bounds on the eigenvalues follows trivially.
\end{rmk}

\begin{coro}\label{invertiamoo}
For  ${\bf g}\in{\bf Z}$,
 consider the equation
\begin{equation}\label{eq:4.4.7}
L_{N}{\bf u}={\bf g},
\end{equation}
with $L_{N}$ in \eqref{corno8}.
Let us define 
\begin{equation}\label{primedimerda}
P_{N}^{2\g}:=
\left\{\begin{array}{ll}
\x\in\calO : &|\oo_{+}\cdot l+\mu^{N}_{\s,j}(\x)|\geq
\frac{2\g j^{2}}{\langle l \rangle^{\tau}}, \\
& \forall \in\ZZZ^{d}, \;\; |l|\leq N\;\;\forall\; (\s,j)\in\{\pm1\}\times\ZZZ_{+}\cap S^{c}
\end{array}
\right\}.
\end{equation}
If $\x\in \Lambda^{2\g}_{N}\cap P^{2\g}_{N}$  (defined respectively in \eqref{corvo7} and \eqref{primedimerda}),  then there exists 
${\bf h}=(h,\bar{h})\in{\bf X}$
such that
\begin{equation}\label{corno12}
\begin{aligned}
\|{\bf h}\|_{\vec{v},p}&\leq \g^{-1}\|{\bf g}\|_{\vec{v},p+2\tau+1},\\
\|L_N {\bf h}-{\bf g}\|_{\vec{v},\gotp_0}&\leq C\|{\bf g}\|_{\vec{v},\gotp_1}\de N^{-\ka_4},\qquad  v:=(\g,\Lambda^{2\g}_{N}\cap P^{2\g}_{N},a,s).\\
\|L_N {\bf h}-{\bf g}\|_{\vec{v},p}&\leq 
\g^{-1}
   \Big(|E_1|^{{\rm dec}}_{\vec{v},p+\ka_5}+|E_0|^{{\rm dec}}_{\vec{v},p+\ka_5}
 \Big)N^{-\ka_4}\|{\bf g}\|_{\vec{v},\gotp_1}+\g^{-1}C|\x| \de N^{-\ka_4}\|{\bf g}\|_{\vec{v},p+2\tau+1},
\end{aligned}
\end{equation}


\end{coro}

\begin{proof}
First of all we can define
\begin{equation}\label{corno13}
{\bf h}:=(\oo_{+}\cdot\del_{\theta}+\DD_{N})^{-1}{\bf g},
\end{equation}
since $\DD_N$ is diagonal and hence it is trivial to define its inverse. Let us check  estimate \eqref{corno12}.
Following the same strategy of Lemma $5.44$ in \cite{FP} one get the bound
\begin{equation}\label{corno14}
\|{\bf h}\|_{\vec{v},p}\stackrel{(\ref{primedimerda})}{\leq} \g^{-1}\|{\bf g}\|_{\vec{v},p+2\tau+1}
\end{equation}
and that ${\bf h}\in {\bf X}$. Eq. \eqref{corno13} implies $L_N {\bf h}-{\bf g}=\RR_{N}{\bf h}$ and moreover one has
\begin{equation}\label{corno15}
\begin{aligned}
 \|\RR_{N}{\bf h}\|_{\vec{v},p}\stackrel{(\ref{eq:2.13b})}{\leq}
 ( |E_1^{N}|^{{\rm dec}}_{\vec{v},p}+|E_0^{N}|^{{\rm dec}}_{\vec{v},p})\|{\bf h}\|_{\vec{v},\gotp_0}
 +
  (|E_1^{N}|^{{\rm dec}}_{\vec{v},\gotp_0}+|E_0^{N}|^{{\rm dec}}_{\vec{v},\gotp_0})\|{\bf h}\|_{\vec{v},p}.
\end{aligned}
\end{equation}
By using \eqref{corno14}, \eqref{corno10} and \eqref{corvo7}  we have that \eqref{eq:2.13b} implies \eqref{corno12}.
\end{proof}

We collect the results of Section \ref{paperina} in the following Lemma.

\begin{lemma}\label{computer}
Consider the operator $L$ in \eqref{birk3} and assume 
bounds \eqref{birk4} with $\gotp_1=\gotp_0+\gotm_2$
with $\gotm_{2}$ defined in \eqref{corno4}. Fix any $N\geq1$ and
for $\x\in \Lambda_{N}^{2\g}\cap P_{N}^{2\g}$ (see \eqref{corno7}, \eqref{primedimerda})
consider the maps $\SSSS$, $\Phi_{N}$ defined in \eqref{birk5} and \eqref{corno8} respectively
and set $\gotM:=\SSSS\circ\Phi_{N}$. Then, the map $\gotM$ is reversibility preserving
according to Def. \eqref{reversible}, and
 for any $f\in {\bf X}$ one has that
\begin{equation}\label{computer2}
\begin{aligned}
\|\gotM f\|_{\vec{v},p}\leq \|f\|_{\vec{v},p}+\g^{-1}\e^{0}_{p}\|f\|_{\vec{v},p},\quad \gotp_1\leq p\leq \gotp_2,
\end{aligned}
\end{equation}
with $\e^{0}_{p}$ defined in \eqref{piccolo11}. Moreover, setting for any $g\in {\bf X}$
\begin{equation}\label{computer3}
{\bf h}=\gotM(\oo_{+}\del_{\theta}+\DD_{N})^{-1}\gotM^{-1} {\bf g},
\end{equation}
one has that ${\bf h}$ and $\big[{L} {\bf h}-{\bf g}\big]$ satisfy bounds like \eqref{corno12}.
\end{lemma}

\begin{proof}
The result follows by collecting the results of Proposition \ref{birk}, Lemma \ref{kamreduc}
and Corollary \ref{invertiamoo}.
More precisely, we apply Proposition \ref{birk} to the operator $L$ in \eqref{birk3}.
Consider the operator $L^{+}$ in \eqref{birk7} and set
\begin{equation}\label{accendino}
E_{1}=E_{1}(L^{+}):=-\ii E\Pi_{S}^{\perp}\left( \begin{matrix} 0 & b^{+}_{1} \\ \bar{b}^{+}_1 & 0 \end{matrix}\right)
\Pi_{S}^{\perp}, \quad 
E_{0}=E_{0}(L^{+}):=
-\ii E \Pi_{S}^{\perp}\left(
 \begin{matrix} a^{+}_{0} & b^{+}_{0} \\ \bar{b}^{+}_0 & \bar{a}^{+}_{0} \end{matrix}
\right)\Pi_{S}^{\perp}+\hat{\mathscr{K}}^{+}.
\end{equation}
By Remarks \ref{opmolt}, \ref{posacenere}, Lemma \ref{linesdecay} and by bounds \eqref{birk8}, \eqref{birk4444} 
one has that hypothesis \eqref{corvo7} hold. 
We set moreover
$$
\DD=\Pi_{S}^{\perp}
\left(-\ii E\left(\begin{matrix}m_{+} & 0\\ 0 & m_{+} \end{matrix}\right)\del_{xx}-\ii E\left(\begin{matrix}{\rm diag}_{j\in\ZZZ_+} r^{(0)}_{j} & 0\\ 0 & {\rm diag }_{j\in\ZZZ_+} r^{(0)}_{j} \end{matrix}\right)
\right)\Pi_{S}^{\perp}
$$
hence \eqref{corno2} holds, then we can 
 we apply Lemma \ref{kamreduc}
to $L^{+}$. By Corollary \ref{invertiamoo} we get the thesis.

\end{proof}

\zerarcounters
\section{The sets of ``good'' parameters}\label{sbroo}

In this Section we conclude the proof of Theorem \ref{teoremap}. 
In Sections 
\ref{weakuffa} and \ref{action} essentially we rewrite the \eqref{6.1} as a infinite dynamical system given by the vector field in \eqref{system}.
In this way we are allowed to apply Theorem 
\eqref{thm:kambis} to the vector field $F_0$ defined in \eqref{system}. 
The analysis performed in Section \ref{iniz} guarantees that one can satisfies hypotheses \eqref{sizes} 
of the Abstract theorem.
In order to apply Theorem \eqref{thm:kambis} one need to identify the sequences of maps $\calL_{n}$
with properties \eqref{odio},\eqref{satana} and \eqref{satana2}
%
%
and give a more explicit formulation of the sets of ``good'' parameters 
defined in \eqref{oscurosignore} in order to estimate the measure of such sets.

On the field $F_0$ we cannot apply directly Lemma \eqref{regularization} just because $N^{(2)}$ is not ``small enough'' and we are not able to prove 
that $\calL$ is close to the identity. We overcome this problem using an algebraic arguments.
We strictly follows the strategy of Lemma \eqref{regularization}, we will underline the fundamental differences. 
Roughly speaking  the aim of the following Lemma is to conjugate $F_0$ to a vector field
for which the term $N^{(2)}$ has constant coefficients of order $O(|\x|)$ plus  terms of order at least 
$O(|\x|^{\frac{3}{2}})$.
Clearly a procedure like this cannot be iterated infinitely many times. We just perform  it one time
in order to fullfil the smallness hypothesis of Lemma \ref{regularization}. In such Lemma one reduces the size of the term in $N^{(2)}$ ``quadratically''.
We have the following result.

\begin{lemma}[{\bf Preliminary step}]\label{step0}
Consider the field $F_0$ defined in \eqref{ini2}. 
Consider $K_0$ as in \ref{sceltapar}, $\e_0$ in \eqref{parametri},  $\rho_1$ of definition \eqref{numeretti}.
and set $\vec{v}=(\g,\calO,s,a)$,  $\vec{v}_1:=(\g,\calO,s-\rho_1 s_0,a-\rho_1 a_0)$ and  $\vec{v}_2:=(\g,\calO,s-2\rho_1 s_0,a-2\rho_1 a_0)$.
Fix $K_{1}=K_0^{\frac{3}{2}}$ and $\h_1$ as in \eqref{posacenere10}.
Assume
\begin{equation}\label{mamma}
\rho_1^{-1}\e_0 K_0^{C}\leq \epsilon
\end{equation}
for some $C$ depending only on $d,\tau$. Then, if $\epsilon$ is small enough,  there exists a tame map 
\begin{equation}\label{transtrans}
\calL_1=\uno+f : \calO_0\times D_{a_0+\rho_1 a_0,p+2}(s_0-\rho_1 s_0,r-\rho_1 r_0)\to D_{a_0,p+2}(s_0,r_0),
\end{equation}
with 
\begin{equation}\label{trans2000}
|f|_{\vec{v}_1,p}\leq CK^{\h_1}\max\{\|a_{2}\|_{\vec{v},p},\|b_{2}\|_{\vec{v},p},\|H^{(\theta,0)}\|_{\vec{v},p}, \;\}, \;\; p\leq \gotp_2
\end{equation}
that satisfies \eqref{odio}, \eqref{satana} and \eqref{satana2}
with $\ka_{3}$ as in \eqref{numeretti1000}.
We set
$$
\hat{F}_0:=(\calL_{1})_{*}F_0=N_0+\hat{G}:  D_{a_0-2\rho_1 a_0,p+2}(s_0-2\rho_1 s_0,r-2\rho_1 r_0)\to V_{a_0-2\rho_1 a_0,p},
$$
and moreover $\hat{F}$ is in $\calE$ (see \ref{compatibili}) and has the form
\begin{equation}\label{regu55}
\begin{aligned}
\hat{F}:=(1+h_1)\Big(\hat{N}^{1}_0+\hat{N}_1^{(1)}+\hat{N}_1^{(2)}+\hat{H}_1
\Big).
\end{aligned}
\end{equation}
On the set of $\x\in\calO_0$ such that $|{\oo}^{(0)}\cdot l|\geq \g/\langle l\rangle^{\tau}$ for $|l|\leq K_1$ (see \eqref{ini1}) one has the following.
The function $h_1$ satisfies bounds 
\begin{equation}\label{argento1000}
\|h_{1}\|_{\vec{v}_2,p}\leq K_{1}^{\h_1}\|a_2\|_{\vec{v},p}.
\end{equation}
One has that
$\hat{N}^{1}_0=\oo_1\cdot\del_{\theta}+\tilde{\Omega}_{1}^{-1}w\cdot\del_{w}$ with $\tilde{\Omega}_1^{-1}=m_{0} \Omega^{(-1)}$, and
\begin{equation}\label{barad100}
|\oo_{1}-{\oo}^{(0)}|_{\g}\leq  \tG_0|\x|^{2}, \quad |m_0-1|_{\g}\leq \tG_0|\x|, \quad
m_0:=\frac{1}{(2\pi)^{d+1}}\int_{\TTT^{d+1}}a_2 d xd \theta
\end{equation}
with $a_{2}\equiv a^{(0)}_{2}$ in \eqref{NF2}
and that $\hat{N}_{1}^{(2)}$ as the form \eqref{barad20} with coefficients $a_{2}^{(1)},b_{2}^{(1)}$, $\hat{N}^{(1)}_1$ has the form \eqref{angband2} with coefficients
$a_{i}^{(1)},b_{i}^{(1)}$ for $i=0,1$, $\hat{\mathscr{K}}_1$ of the form \eqref{maiale} with $\mathtt{N}_{1}=\mathtt{N}_0+\mathtt{C}|S|$ for $\mathtt{C}$ and absolute constant, and  where $\mathtt{N}_0$ is the rank of $\mathscr{K}_0$
(see Prop. \ref{Properties of $F$}).
On the field $N_1^{(2)}$ the bounds \eqref{regu6} hold with $K_{+}\rightsquigarrow K_1$, 
$\de_{\gotp_2}^{(2)}=C_{\vec{v},\gotp_2}(N^{(2)})=\g_0\tG_0$ and $\de_{\gotp_2}^{(1)}=\g_0^{-1}\|H_0^{(\theta,0)}\|_{\vec{v},\gotp_2}$ given in \eqref{cartabianca}.
With the same convention also the bounds \eqref{regu66}, \eqref{passero}, \eqref{passero2} and \eqref{merlo} hold.

\end{lemma}

\begin{proof}
 Recall that, using the notations of Lemma \eqref{regularization}, for the field $F_0$ the function $h\equiv0$
and the constant $m\equiv1$.
Note that the definition of $\de_{p}^{(1)}$ is the same of Lemma \ref{regularization}. It controls the norm of the $H_0^{(\theta,0)}$
divided by the size of the small divisor $\g_0$.
The definition of $\de_{p}^{(2)}$ has changed. Indeed in this case we have set $\de_{p}^{(2)}\approx C_{\vec{v},p}(N_0^{(1)})$
without divide by $\g_0$. This is due to the fact that $\g_0^{-1}C_{\vec{v},\gotp_1}(N_0^{(2)})=\tG_0$ that is not small.
In Lemma \ref{regularization} we used the smallness of $\de_{\gotp_1}^{(2)}=\g_0^{-1}C_{\vec{v},\gotp_1}(N_0^{(2)})$ in order to prove that $\calL_+$ is close to the identity.
In this case to get the result we need to use different arguments.
However we follows the same strategy used in Lemma \ref{regularization} and we perform the same four steps of that Lemma.
Concerning step 1 and step 2
we apply
the same transformations defined exactly in the same way. In this case there is no small divisors in the equations that define 
transformation $\Phi_1$ and $\Phi_2$. Hence the same estimates of Lemma \ref{regularization} hold with the 
definition in \eqref{cartabianca}. 
The step 3 has to be analyzed more carefully.  Indeed if one looks at the equation \eqref{38} for $\be$ one sees that
one has to control the inverse of the operator $\tilde{\oo}\cdot\del_{\theta}$ ($\tilde{\oo}$ in \eqref{ini1}).  By using the diophantine condition \eqref{ini4}, one get 
that
$$
\|\be\|_{\vec{v},\gotp_1+\gotp_0}\leq K_1^{\gotp_0+2\tau+1}\g_0^{-1}\|m_{2}\|_{\vec{v},\gotp_1}
\leq K_1^{\gotp_0+2\tau+1}\g_0^{-1}C_{\vec{v},\gotp_1}(N_0^{(1)})\leq K_+^{\gotp_0+2\tau+1}\tG_0,
$$
that is not small. We need to estimates $\be$ in a different way. We first analyze the form of the coefficient $a_2^{(2)}$ .
By equation \eqref{28} we have
\begin{equation}\label{288}
m_{2}(\theta)=\left(\frac{1}{2\pi}\int_{\TTT}(1+\tilde{a}_{2}(\theta,x))^{-\frac{1}{2}}\right)^{-2}-1,
\end{equation}
where $\tilde{a}_2=\sqrt{(1+a_2)^{2}-|b_2|^{2}}-1$  , with $a_2$ and $b_2$ the coefficients of $N_0^{(1)}$.
We have that we can write
\begin{equation}\label{rondine}
\tilde{a}_{2}=a_{2}(\theta,x)+b(\theta,x), \qquad b(\theta,x):= \sqrt{(1+a_2)^{2}-|b_2|^{2}}-(1+a_{2})
\end{equation}
and we note also that
\begin{equation}\label{rondine2}
\|b\|_{\vec{v},\gotp_1}\leq (C_{\vec{v},\gotp_1}(N_0^{(1)}))^{2}\leq C |\x|^{2},
\end{equation}
for some constant $C$. Moreover by an explicit computation we can write
\begin{equation}\label{monaco10}
\begin{aligned}
m_{2}&:=\frac{1}{2\pi}\int_{\TTT}\tilde{a}_{2}d x+\frac{d}{2\pi(1-\int_{\TTT}\tilde{a}_2d x+\hat{c}  )},\\
&d:=2\pi \hat{c}-\left(\int_{\TTT}\tilde{a}_2\right)^{2}+\left(\int_{\TTT}\tilde{a}_2\right)\hat{c},\\
&\hat{c}:=\frac{1}{\pi}\int_{\TTT}c+\frac{1}{4\pi^{2}}\left(\int_{\TTT}\tilde{a}_2\right)^{2}-\frac{1}{2\pi}\left(\int_{\TTT}\tilde{a}_{2}\right)\left(\int_{\TTT}c\right)+\frac{1}{4\pi}\left(\int_{\TTT}c\right)^{2},\\
&c:=\frac{1}{2}\left((1+\tilde{a}_2)^{-\frac{1}{2}}-1+\frac{1}{2}\tilde{a}_2
\right).
\end{aligned}
\end{equation}
Clearly one has
\begin{equation}\label{monaco11}
\begin{aligned}
\|m_{2}-\frac{1}{2\pi}\int_{\TTT}\tilde{a}_{2}\|_{\vec{v}_0,p}&\leq C(p)\|a_{2}\|_{\vec{v}_{0},\gotp_1}\|a_{2}\|_{\vec{v}_{0},p}, \\
\|m_{2}-\frac{1}{2\pi}\int_{\TTT}\tilde{a}_{2}\|_{\vec{v}_0,\gotp_1}&\leq C|\x|^{2}.
\end{aligned}
\end{equation}
Clearly using \eqref{rondine}  one has
\begin{equation}\label{rondine3}
m_{2}=\frac{1}{2\pi}\int_{\TTT}a_{2}d x+ s(\theta,x), \qquad \|s\|_{\vec{v},\gotp_1}\stackrel{(\ref{monaco11}),(\ref{rondine2})}{\leq} C|\x|^{2}.
\end{equation}
Roughly speaking this implies that in low norm $\gotp_1$ one has $m_{2}\approx a_{2}+O(|\x|^{2})$.
Now equation \eqref{28} becomes
\begin{equation}\label{38bis}
\be(\theta):=\frac{1}{1+\gotc }({\oo}^{(0)}\cdot\del_{\theta})^{-1}(1+\Pi_{K_{1}}\left(\frac{1}{2\pi}\int_{\TTT}a_{2}d x+s \right)-1-\gotc)(\theta), \quad \gotc=\frac{1}{(2\pi)^{d+1}}\int_{\TTT^{d+1}}m_{2}(\theta,x)d\theta d x,
\end{equation}
Now we have to estimate $\be$. 
The critical term is obviously the term of $O(|\x|)$ because one cannot use estimate \eqref{ini4} since 
$\g_0\approx|\x|$.  
One can use an algebraic arguments.
First we recall that by \eqref{ini1} one has ${\oo}^{(0)}=\la^{(-1)}+\la^{(0)}(\x)$ with $\la^{(-1)}=j^{2}, \;\; j\in S^{+}$. On the other hand the term of order $\x$ of $\be$ \eqref{38bis} depends only
on the coefficients $a_{2}$ given in \eqref{NF2}. Hence in formula \eqref{38bis} we need to estimate ${\oo}^{(0)}\cdot k$ with $k\in(S^{+})^{d}$ but with only two components 
different from zero and not for $k\in \ZZZ^{d}$ as in \eqref{ini4}. This implies, using Lemma \ref{hyp3autaut}, in the term $a_{2}z_{xx}\cdot\del_{z}$ there are only trivial 
resonances, and hence  
\begin{equation}\label{monaco12}
\|\be\|_{\vec{v}_{0},p+\gotp_0}\leq \|\Pi_{K_{1}}a_{2}\|_{\vec{v}_{0},p+\gotp_0},\quad \|\be\|_{\vec{v}_0,\gotp_1}\leq C |\x|,
\end{equation}
for some constant $C$, that is a better estimates with respect to the one in \eqref{barad3}. In this way we get that the transformation 
is $\x-$close to the identity. Now the last step can be performed exactly as in Lemma \ref{regularization} because there are no other differences and one can estimate
the transformation $\Phi_4$ as done in \eqref{merda5} and \eqref{merda6}.
Thanks to the perturbative arguments in \eqref{monaco10} and \eqref{rondine}
one can fix 
$$
m_0:=\frac{1}{(2\pi)^{d+1}}\int_{\TTT^{d+1}}a_2 d xd \theta
$$
and \eqref{barad100} follows. 
Finally 
we defined
$$
h_{1}=\Phi_{4}^{-1}(\tilde{h}_{1}),
$$
where $\tilde{h}_1$ is defined as in \eqref{barad7} with $h=0$ and $\be$ defined in \eqref{38bis}.
The estimates \eqref{argento1000} follows by \eqref{monaco12}.
The fact that $\calL_1$ is compatible, according to Definition \ref{compa}, follows
by \eqref{cartabianca} and Corollary \ref{senonseicompa}.
This concludes the proof.

\end{proof}

\begin{rmk}\label{laghetto}
Note that the coefficient $m_0$ in \eqref{barad100} gives  the correction of order $O(|\x| j^{2})$ to the eigenvalues $j^{2}$ as we will see in Section
\ref{caspita2} (see equation \ref{integra2}). This term will remain the same at each step of our iteration since all the further correction 
will be of higher order in $|\x|$. 
\end{rmk}

The main result of this Section is the following:

\begin{proposition}\label{andrea}
There exists a sequence of maps $\calL_n$, $n\geq 1$ that satisfies \eqref{odio}, \eqref{satana} and \eqref{satana2} with $\ka_3,\ka_{1},\ka_2$,
$\gotp_1,\gotp_2,\mu,\e_{0}$ given in \ref{sceltapar} and \eqref{parametri},  such that
the $n-$th vector field $F_n$ is defined on $\calO_0$ and  on $\calO_n$ in \eqref{oscurosignore}
satisfies bounds \eqref{lamorte}. Moreover $F_n$ is in $\calE$ (see Definitions \ref{pseudopseudo}, \ref{compatibili}) which satisfies \eqref{lamorte}
and can be written in the form \eqref{regu22} as
\begin{equation}\label{regu2222}
F_n:=(1+h_n)\Big(N^{(n)}_0+N_{n}^{(1)}+N_{n}^{(2)}+H_n
\Big)
\end{equation}
where ${N}^{(n)}_0=\oo_n\cdot\del_{\theta}+{\Omega}_{n}^{-1}w\cdot\del_{w}$ with ${\Omega}_n^{-1}=m_n \Omega^{(-1)}$, $\oo_n\in \RRR^{d}$ is diophantine and
\begin{equation}\label{barad22222}
|\oo_n-
{\oo}^{(0)}|_{\g}\leq  C|\x| \e_{0}, \quad |m_n-m_0|_{\g}\leq C |\x|\e_0,
\end{equation}
where $m_0$ is defined in \eqref{barad100} and 
\begin{equation}\label{barad2122}
d_{w}H^{(w)}(u)[\cdot]=0 
\end{equation}
In particular $N_{n}^{(1)},N_{n}^{(2)}$ have the form \eqref{barad20} and \eqref{angband2} and,
using the same notation as in \eqref{indiani},
the following estimates hold:
\begin{equation}\label{barad2200}
\begin{aligned}
C_{\vec{v}_n,p}(H_{n})\leq C_{\vec{v}_n,p}(\Pi_{\NN^{\perp}}G_n), \quad \|h_n\|_{\vec{v}_n,p},\;C_{\vec{v}_n,p}(N_n^{(1)})\leq C_{\vec{v}_n,p}(G_n).
\end{aligned}
\end{equation}
\begin{equation}\label{strega}
\g_{n}^{-1}\|H_n^{(\theta,0)}\|_{\vec{v}_n,\gotp_1}\leq \e_0 K_{n-1}^{-\ka_2+\mu+4}, \quad C_{\vec{v}_{n},\gotp_1}(N_n^{(2)})\leq \g_0\e_0 K_{n-1}^{-\ka_2+\mu+4}.
\end{equation}
We also have that
the coefficients $a_{i}^{(n)}, b_{i}^{(n)}$ of the field $N_{n}^{(1)}$ for $i=0,1$, 
have the form 
\begin{equation*}
a^{(n)}_i=a^{(0)}_{i}+\tilde{a}^{(n)}_{i}, \quad b_i:=b^{(0)}_{i}+\tilde{b}^{(n)}_{i}, \;\; i=0,1,
\end{equation*}
where the coefficients 
$a^{(0)}_{i},b^{(0)}_{i}$ for $i=0,1$ are given by formul\ae  \eqref{NF2}. The
the term $\mathscr{K}_{n}$ has the form \eqref{maiale} with coefficients $c_{j}^{(n)},d_{j}^{(n)}$, for $j=1,\ldots, \mathtt{N}_{n}$ such that
\begin{equation}\label{birk44}
\begin{aligned}
&\|\tilde{a}^{(n)}_{i}\|_{\vec{v}_{n},\gotp_1},\|\tilde{b}^{(n)}_{i}\|_{\vec{v}_{n},\gotp_1}\leq C|\x| \e_0, \quad i=0,1,\\
\end{aligned}
\end{equation}
and 
\begin{equation}\label{birk444}
\|c_{i}^{(n)}\|_{\vec{v},\gotp_1},\|d_{i}^{(n)}\|_{\vec{v}_{n},\gotp_1}\leq (1+\e_0 K_{n-1}^{-1})C_{\vec{v}_{n},\gotp_1}(\Pi_{\NN^{\perp}}(G_{0}) ),
\end{equation}
for $\mathtt{N}_{n}\sim \mathtt{N}_0+\mathtt{C}|S|n$.
In particular one has that
\begin{equation}\label{lamerd2}
\begin{aligned}
\|a_{i}^{(n)}-a_{i}^{(n-1)}\|_{\vec{v}_{n},\gotp_1},\|b_{i}^{(n)}-b_{i}^{(n-1)}\|_{\vec{v}_{n},\gotp_1}&\leq \g_0 \e_0 K_{n-1}^{-\ka_2+\mu+4}, \quad i=0,1,2,\\
\|c_{i}^{(n)}-c_{i}^{(n-1)}\|_{\vec{v}_{n},\gotp_1},\|d_{i}^{(n)}-d_{i}^{(n-1)}\|_{\vec{v}_{n},\gotp_1}&\leq \g_0 \e_0 K_{n-1}^{-\ka_2+\mu+4}, \quad i=1,\ldots,\mathtt{N}_{n} .
\end{aligned}
\end{equation}
\end{proposition}

\begin{proof}
We prove the Lemma by induction on $n$. If one assume that we already constructed the map $\calL_{n}$ such that   all the properties above are satisfied
then we proceed as follows.
First of all by \eqref{strega} one note that hypotheses \eqref{piccolopassero} of Lemma \ref{regularization} are satisfied.
Then we apply the Lemma to the field $F_{n}$. We set $\calL_{n+1}$ the map given by the Lemma 
 \ref{regularization}. It satisfies \eqref{odio}, \eqref{satana} and \eqref{satana2} thanks to bound  \eqref{trans2}
 (recall the definition of $\ka_{3}$ in \eqref{numeretti1000} in Corollary \ref{senonseicompa}).
We set
$\hat{F}_{n}:=(\calL_{n+1})_{*}F_n$ (see \eqref{mappaccia})
that has the form
\begin{equation}\label{strega1}
\hat{F}_n:=N_0+\hat{G}_n=(1+h_{n+1})\Big(\hat{N}^{(n)}_0+\hat{N}_{n}^{(1)}+\hat{N}_{n}^{(2)}+\hat{H}_n
\Big)
\end{equation} 
that clearly has all the bounds \eqref{argento},\eqref{barad1000}, \eqref{regu6}, \eqref{regu66}, \eqref{passero}, \eqref{passero2} and \eqref{merlo} hold and these bounds
together the inductive hypotheses implies that  the field $\hat{F}_{n}$ satifies bounds like \eqref{barad22222}-\eqref{birk444}
except for \eqref{strega}. Actually we prove  better bounds on $\hat{N}_n^{(1)}, \hat{H}_n^{(\theta,0)}$.
Indeed we have 
$\hat{N}^{(n)}_0=\oo_{n+1}\cdot\del_{\theta}+\hat{\Omega}_{n}^{-1}w\cdot\del_{w}$ with $\hat{\Omega}_n^{-1}=m_{n+1} \Omega^{(-1)}$, $\oo_{n+1}\in \RRR^{d}$ is diophantine (see \eqref{barad1000}) and
\begin{equation}\label{barad2222}
|\oo_{n+1}-{\oo}^{(0)}|_{\g}\leq  C|\x| \e_0, \quad |m_{n+1}-1|_{\g}\leq C |\x|,
\end{equation}
and, by \eqref{regu6}, \eqref{regu66}, \eqref{passero}, \eqref{passero2} and bounds \eqref{lamorte} for $F_n$, 
 \begin{equation}\label{regulari6}
 \begin{aligned}
  C_{\vec{v},\gotp_{1}}(\hat{N}_n^{(1)})&\leq \g_0 (1+K_{n}^{\h_1}\e_0 K_{n-1}^{-\ka_2})\mathtt{G}_{n}
\\
C_{\vec{v},\gotp_{2}}(\hat{N}_n^{(1)})&\leq  \g_0
(1+K_{n}^{\h_1}\e_0 K_{n-1}^{-\ka_2})\mathtt{G}_0( K_{n-1}^{\ka_1}+
K_{n}^{\h_1}\e_0 K_{n-1}^{\ka_1}),
 \end{aligned}
 \end{equation}
 \begin{equation}\label{regulari66}
 \begin{aligned}
 C_{\vec{v},\gotp_1}(\hat{N}_{n}^{(2)})&\leq \g_0(1+K_{n}^{\h_1}\e_0 K_{n-1}^{-\ka_2}) \Big[K_{n}^{-(\gotp_2-\gotp_1-\h_1)} \e_0 K_{n-1}^{\ka_1}
+K_{n}^{\h_1}K_{n-1}^{-2\ka_2}\e_0^{2}
\Big],\\
C_{\vec{v},\gotp_2}(\hat{N}_{n}^{(2)})&\leq \g_0\e_0 \mathtt{G}_0\Big( K_{n-1}^{\ka_1}+ K_{n}^{\h_1} K_{n-1}^{\ka_1}
\Big),
 \end{aligned}
\end{equation}
 \begin{equation}\label{passero20}
 \begin{aligned}
 C_{\vec{v},\gotp_{1}}(\hat{G}_n)&\leq  \g_0(1+K_{n}^{\h_1}\e_0 K_{n-1}^{-\ka_2})\mathtt{G}_n
\\
C_{\vec{v},\gotp_{2}}(\hat{G}_n)&\leq \g_0
(1+K_{n}^{\h_1}\e_0 K_{n-1}^{-\ka_2})\mathtt{G}_0(K_{n}^{\ka_1}+
K_{n}^{\h_1}\e_0 
),
\end{aligned}
\end{equation}
 \begin{equation}\label{passero22}
\|\hat{H}_{n}^{(\theta,0)}\|_{\vec{v}_2,\gotp_1}\leq   \g_0(1+K_{n}^{\h_1}\e_0 K_{n-1}^{-\ka_2})\Big[
K_{n}^{-(\gotp_2-\gotp_1)}\e_0 K_{n-1}^{-\ka_{2}}
+K_{n}^{\h_1}\e_0^{2}K_{n-1}^{-2\ka_2}
\Big],
\end{equation}
Clearly one has that, by \eqref{numeretti}, \eqref{posacenere10}, bounds in \eqref{barad2200}, the  \eqref{birk44} and \eqref{birk444} holds also for $\hat{F}_n$.
 The more critical conditions are those in \eqref{strega}. Using \ref{sceltapar}, 
 \eqref{numeretti1000}, then 
 by \eqref{regulari66} and \eqref{passero22} one gets
 \begin{equation}\label{pollini10}
\begin{aligned}
C_{\vec{v},\gotp_1}(\hat{N}_n^{(2)})&\leq\g_0  \e_0 K_{n}^{-\ka_2}, \qquad 
\|\hat{H}_n^{(\theta,0)}\|_{\vec{v},\gotp_1}\leq \g_0\e_0K_{n}^{-\ka_2}
\end{aligned}
\end{equation}
that are bounds even better than \eqref{strega}.
Reasoning in the same way, bounds \eqref{merlo1} and \eqref{merlo} togheter with the inductive
hypotheses \eqref{barad22}-\eqref{birk44} imply bounds \eqref{lamerd2} with $K_{n}^{-\ka_{2}}$
instead of $K_{n}^{-\ka_{2}+\mu+4}$.
We recall the following. By Proposition \ref{regularization}, we have that the rank of $\hat{\mathscr{K}}_{n}$
is increased proportionally to the cardinality of the set $|S|$ (hence we set  $\mathtt{N}_{n+1}\sim \mathtt{N}_{n}+\mathtt{C}|S|$).

Now, by the definition in Theorem \ref{thm:kambis}, we have that 
the field $F_{n+1}$ is given by $F_{n+1}=(\Phi_{n+1})_{*}\hat{F}_{n}$, where the map $\Phi_{n+1}$ is generated by the field $g_{n+1}$ in \eqref{lamorte} for $n\rightsquigarrow n+1$.
We have to show that the map $\Phi_{n+1}$ does not change the size of  the coefficients of $\hat{N}^{(1)}_{n},N^{(2)}_{n}$ in such a way the estimates on $F_{n+1}$
remains essentially the  same of those on  $\hat{F}_n$.
First we note that, by the form of the map $\Phi_{n+1}$ one has
\begin{equation}\label{santana}
F_{n+1}:=(\Phi_{n+1})_{*}\hat{F}_{n}=(1+h_{n+1})\Big(
(\Phi_{n+1})_{*}(\hat{N}^{(n)}_0)
+(\Phi_{n+1})_{*}(\hat{N}_{n}^{(1)}+\hat{N}_{n}^{(2)}+\hat{H}_{n})
\Big).
\end{equation}
Let us study first the term that does not contains the constant coefficients term $\hat{N}_0^{(n)}$.
Again by the form of the map $\Phi_{n+1}$, that is generated by $g_{n+1}\in \hat{\BB}$, we have that $\Phi_{n+1}$ preserves the pseudo-differential structure of the vector fields.
By setting
$$
\calF=(\Phi_{n+1})_{*}(\hat{N}_{n}^{(1)}+\hat{N}_{n}^{(2)}+\hat{H}_{n})
$$
we have that the coefficients of $(\Pi_{\NN}\calF)^{(w)}$ comes form the term $\Pi_{\NN}(\Phi_{n+1})_{*}(\hat{N}_{n}^{(1)}+\hat{N}_{n}^{(2)}+\Pi_{\NN}\hat{H}_n)$ or 
from the term $\Pi_{\NN}(\Phi_{n+1})_* (\Pi_{\NN^{\perp}}\hat{H}_{n})$. Obviously the first coefficient satisfies \eqref{strega} using $\eqref{regulari66}$ and the fact that, in the low norm $\gotp_1$,  $\Phi_{n+1} \approx \uno+O(\de K_{n}^{-\ka_{2}+\mu})$
.  The second terms satisfies \eqref{strega}  
because 
one has
\begin{equation}\label{santana2}
\Pi_{\NN}(\Phi_{n+1})_* (\Pi_{\NN^{\perp}}\hat{H}_{n}):=\Pi_{\NN}\int_{0}^{1}(\Phi_{n+1})_{*}^{s}[g_{n+1},\Pi_{\NN^{\perp}}\hat{H}_n]d s,
\end{equation}
and hence one gets estimates \eqref{strega} by using the estimate \eqref{lamorte} on $g_{n+1}$. 
Formula \eqref{santana2} follows by Remark $B.5$ in Appendix $B$ of \cite{CFP} (see also Lemma \ref{temavec}). 
Let us analyze the first term.
We note that $\hat{N}_{0}^{(n)}$ is diagonal according to Definition 
 \ref{norm}  since it has only constant coefficients. Hence we have
\begin{equation}\label{pollini20}
\begin{aligned}
\Pi_{\NN}(\Phi_{n+1})_{*}(\hat{N}_0^{(n)})&=\Pi_{\NN}\int_{0}^{1}d s (\Phi_{n+1})_{*}^{s}[g_{n+1},\hat{N}_0^{(n)}]=
\Pi_{\NN}\int_{0}^{1}d s (\Phi_{n+1})_{*}^{s}\Pi_{K_{n+1}}\Pi_{\calA}[g_{n+1},\hat{N}_0^{(n)}]\\
&=\Pi_{\NN}\int_{0}^{1}d s (\Phi_{n+1})_{*}^{s}\Pi_{K_{n+1}}\Pi_{\calA}\left[g_{n+1},(1+h_{n+1})\Big(\hat{N}_0^{(n)}+\hat{N}_n^{(1)}+\hat{N}_n^{(2)}+
\Pi_{\NN}\hat{H}_n\Big)\right]\\
&-\Pi_{\NN}\int_{0}^{1}d s (\Phi_{n+1})_{*}^{s}\Pi_{K_{n+1}}\Pi_{\calA}[g_{n+1},h_{n+1}\hat{N}_0^{(n)}]\\
&-\Pi_{\NN}\int_{0}^{1}d s (\Phi_{n+1})_{*}^{s}\Pi_{K_{n+1}}\Pi_{\calA}\left[g_{n+1},(1+h_{n+1})\Big(\hat{N}_n^{(1)}+\hat{N}_{n}^{(2)}+\Pi_{\NN}\hat{H}_n\Big)\right]
\end{aligned}
\end{equation}
Now we use the definition of $g_{n+1}$ and by item $(iii)$  in Definition \ref{pippopuffo3} we obtain that 
\begin{equation}\label{pollini21}
\begin{aligned}
\Pi_{K_{n+1}}\Pi_{\calA}\left[g_{n+1},(1+h_{n+1})\Big(\hat{N}_0^{(n)}+\hat{N}_n^{(1)}+\hat{N}_n^{(2)}+
\Pi_{\NN}\hat{H}_n\Big)\right]=\Pi_{K_{n+1}}\Pi_{\calA}\hat{F}_{n}+r_n
\end{aligned}
\end{equation}
where $r_{n}$ satisfies bounds \eqref{cribbio4} and $\Pi_{K_{n+1}}\Pi_{\calA}\hat{F}_{n}=(1+h_{n+1})\Pi_{\calA}\hat{H}_{n}$.
Equation \eqref{strega} simply follows by applying Lemma \ref{temavec} and the inductive hypotheses. The bounds \eqref{lamerd2}
follows similarly recalling the first estimate in  \eqref{torus2} of Lemma \ref{torus}.

In order to prove the inductive basis we reason as follows.
First we note that 
if $n=0$ then we cannot apply Lemma \ref{regularization} in order to define the map $\calL_1$ and the field $\hat{F}_0$.  On the other hand we apply Lemma
\ref{step0} which provides the same result. Then on can reason as done in above using the map $\Phi_{1}$.

\end{proof}

Now our aim is to give an explicit form to the sets $\calO'$ which satisfy
Mel'nikov conditions (see \ref{pippopuffo3}) for the vector field $\hat{F}_{n}$.
This is formalized in the 
following Proposition.

\begin{proposition}[{\bf The set of ``good'' parameters}]\label{pipino} 
 Recall \eqref{corno4} and set for $n\geq0$
\begin{equation}\label{strega7}
 n^{*}:=n+\log_{\frac{3}{2}}\frac{\ka_2}{\ka_4},
\end{equation}
Then for all $n\ge 0$, there exist a 
 sequence of purely imaginary numbers
\begin{equation}\label{corvo66}
\mu_{\s,j}(\x):=\mu_{\s,j}^{(n)}(\x):=-\ii \s (m_{n+1}j^{2}+r^{(n)}_{j}), \quad \s=\pm1, \; j\in \ZZZ_{+}\cap S^{c},
\end{equation}
with
\begin{equation}\label{strega20}
|r^{(n)}_{j}|_{\g_n}\leq C |\x|, \quad |r^{(n)}_{j}-r^{(0)}_{j}|_{\g_n}\leq C|\x|\e_0, \qquad |m_{n+1}-m_{n}|_{\g_n}\leq |\x|\e_0 K_{n}^{-\ka_2} , \qquad |\oo_{n+1}-\oo_{n}|_{\g_n}\leq|\x|
\e_0 K_{n}^{-\ka_2}, 
\end{equation}
with $r_{j}^{(0)}$ defined in \eqref{strega10},
 such that, 
 by defining 
 \begin{equation}\label{corno77}
 \begin{aligned}
 \Lambda_{n+1}^{2\g_{n}}&:=\left\{
\begin{array}{ll}
\x\in \calO_n: &\;  |\oo_{n+1}\cdot l \!+\!\mu_{\s,j}^{(n)}(\x)-\!\mu_{\s',j'}^{(n)}(\x)| \geq \frac{2\g_{n}|\s j^{2}-\s' j'^{2}|}{\langle l \rangle^{\tau}}
  \\  
& \forall l \in\ZZZ^{d},\; |l|\leq K_{n^*} \;\; \forall (\s,j),(\s',j')\in\{\pm1\}\times(\ZZZ_+ \cap S^{c})
\end{array}
\right\},\\
P_{n+1}^{2\g_n}&:=
\left\{\begin{array}{ll}
\x\in\calO_n : &|\oo_{n+1}\cdot l+\mu_{\s,j}^{(n)}(\x)|\geq
\frac{2\g j^{2}}{\langle l \rangle^{\tau}}, \\
& \forall \in\ZZZ^{d}, \;\; |l|\leq K_{n^{*}}\;\;\forall\; (\s,j)\in\{\pm1\}\times\ZZZ_{+}\cap S^{c}
\end{array}
\right\}.\\
\SSSS_{n+1}^{2\g_{n}}&:=\left\{
\begin{array}{ll}
\x\in \calO_n: &\;  |\oo_{n+1}\cdot l | \geq \frac{2\g_{n}}{\langle l \rangle^{\tau}}
   \forall l \in\ZZZ^{d},\; |l|\leq K_{n_\star} 
\end{array}
\right\},
\end{aligned}
 \end{equation}
one has that  
\begin{equation}\label{finefine}
\Lambda^{2\g_n}_{n+1}\cap P^{2\g_{n}}_{n+1}\cap\SSSS^{2\g_n}_{n+1} 
\end{equation}
%
satisfies the Mel'nikov conditions (see \ref{pippopuffo3}) with $(\hat{F}_n,K_n,\g_n,\calO_0,s_n,a_n,r_n)$. 
Here $\hat F_n$ is the vector field defined in \eqref{strega1}. 

\noindent
Then in Proposition \ref{andrea} we may choose 
$$
\calO_{n+1}:=\Lambda^{2\g_n}_{n+1}\cap P^{2\g_{n}}_{n+1}\cap\SSSS^{2\g_n}_{n+1} \,,\qquad \forall n\ge 0.
$$
\end{proposition}

\begin{proof}
We proceed by induction, assuming that our claim holds true up to $n$, We shall systematically use the bounds on $F_n,\hat F_n$ given in Proposition \ref{andrea}.  We show that 
for any parameter $\x\in\Lambda^{2\g_n}_{n+1}\cap P^{2\g_{n}}_{n+1}\cap\SSSS^{2\g_n}_{n+1}$  we can construct an approximate inverse $\gotW$.  As we have seen explicitely in \eqref{corvo3}-\eqref{corvo6}, the operator $\gotN=\Pi_{K_{n+1}}\Pi_\calX {\rm ad }(\Pi_\NN \hat{F}_{n})$ is block diagonal and decomposes in four equations, so that also $\gotW$ is made of four blocks.  
We have the trivial multiplication by $1/(1+h_{n+1})$ and
the following operators:
\begin{itemize}
	\item[(i)]  $W^{(n)}_0$ which is an approximate inverse of $\big[\Pi_{K_{n+1}}(\omega_{n+1} + \widehat H^{(\theta,0)}_n )\partial_\theta\big] $. This is used for the inversion of 
	the first two equation in \eqref{corvo6};
	
	\item[(ii)] $W^{(n)}_{+}$ which is an approximate inverse of $\Pi_{K_{n+1}}\big[(\omega_{n+1} + \widehat H^{(\theta,0)}_n )\partial_\theta+
	(\hat{F}_{n})^{(w,w)}(\theta)\big]$;
	note that this is a linear operator acting on $H^{\gotp_{0}}(\TTT;\ell_{a,p})\cap H^{p}(\TTT;\ell_{a,\gotp_0})$;
	
\item[(iii)]  $W^{(n)}_{-}$ which is an approximate inverse of 
$\Pi_{K_{n+1}}\big[(\omega_{n+1} + \Pi_{\NN}\widehat H^{(\theta,0)}_n )\partial_\theta
	-[(\hat{F}_{n})^{(w,w)}(\theta)\big]^{*}\big]$
	which 
	is a linear operator acting on $H^{\gotp_{0}}(\TTT;\ell_{a,p})\cap H^{p}(\TTT;\ell_{a,\gotp_0})$.
Note that in fact the adjoint act on the much bigger dual space, however we need to 
find an inverse on the space of regular vector field. Hence we need this stronger notion of invertibility.
\end{itemize}
 We show that $\gotW$ defined above 
 satisfy all the properties of 
Definition \ref{pippopuffo3}.
%
%
%
%
First of all, in order to deal with $(i)$, we need that $\oo_{n+1}$ is a diophantine 
vector of $\RRR^{d}$. 
Then we set $W^{(n)}_0:=(\oo_{n+1}\cdot\del_{\theta})^{-1}\Pi_{K_n}$. This is possible 
since by Lemma \ref{regularization} the size of $\widehat H^{(\theta,0)}_n $ is so small that it is possible to put it
inside the remainder term $u$, see formula \eqref{cribbio4}.
This choice of $W^{(n)}_0$ is possible since $\x\in \SSSS_{n+1}^{2\g_{n}}$.
We will see that 
this approximation is sufficient to get a good approximate solutions that satisfies the requirements in Definition \eqref{pippopuffo3}.

\noindent
In $(ii)$ we ignore 
$\widehat H^{(\theta,0)}_n $ exactly for the same reason as above. Moreover, recalling
$(\hat{F}_{n})^{(w,w,)}(\theta)=(\hat{N}_0^{(n)})^{(w,w)}+\hat{N}_n^{(1)}+\hat{N}_n^{(2)}+(\Pi_{\NN}\hat{H}_n)^{(w)}$,
we ignore the term $\hat{N}_n^{(2)}$ again due to Lemma \ref{regularization} (see bound
\eqref{pollini10}).
We study equation
\begin{equation}\label{corvo88}
(\hat{L}_n) g=\Big(\oo_{n+1}\cdot\del_{\theta}+\hat{\Omega}_{n}(\theta)\Big)g=f, 
\end{equation}
with $f\in {\bf Z}$ and $g\in {\bf X}$ (see Definition \ref{spaces200}) and where
\begin{equation}\label{strega3}
\hat{\Omega}_{n}:=(\Pi_{\NN}\hat{N}_0^{(n)})^{(w)}+\hat{N}_n^{(2)}+(\Pi_{\NN}\hat{H}_n)^{(w)}.
\end{equation}

The method we use 
to invert $\hat{L}_{n}$ is to approximately diagonalize it. Hence we get
the approximate solution in item $(iii)$
following the same diagonalization procedure, since the operators have the same eigenvalues.
In particular this ensure that $W^{(n)}_{-}$ acts on $H^{\gotp_{0}}(\TTT;\ell_{a,p})\cap H^{p}(\TTT;\ell_{a,\gotp_0})$ and not only on its dual space.

We claim that, by the construction of the field $\hat{F}_n$, the operator $\hat{L}_n$
satisfies the hypotheses of Proposition \ref{birk}. Indeed 
it has the form
\begin{equation}\label{ilbuono}
\begin{aligned}
\hat{L}_{n}&=\Pi_{S}^{\perp}\oo_{n+1}\cdot\del_{\theta}\uno+\Pi_{S}^{\perp}
\left(-\ii E\left(\begin{matrix}m_{n+1} & 0\\ 0 & m_{n+1} \end{matrix}\right)\del_{xx}
\right)\Pi_{S}^{\perp}+\hat{\mathscr{K}}_{n}\\
&+\Pi_{S}^{\perp}\left(
-\ii E\left( \begin{matrix} \hat{a}^{(n)}_1 & \hat{b}^{(n)}_{1} \\ {\ol{\hat{b}}^{(n)}_1} & {\ol{\hat{a}}^{(n)}_{1}} \end{matrix}\right)\del_{x}
-\ii E \left(
 \begin{matrix} \hat{a}_{0}^{(n)} & \hat{b}^{(n)}_{0} \\ {\ol{\hat{b}}^{(n)}_0} & {\ol{\hat{a}}^{(n)}_{0}} \end{matrix}
\right)
\right)
\Pi_{S}^{\perp},
\end{aligned}
\end{equation}
with $\oo_{n+1},m_{n+1}$ which satisfy \eqref{barad22222} and \eqref{strega20},
$\hat{\mathscr{K}}_{n}$ of the form \eqref{maiale} with coefficients $\hat{c}_{i}^{(n)},\hat{d}_{i}^{(n)}$, $i=1,\ldots,\mathtt{N}_{n+1}$.
By \eqref{merlo} in Proposition \ref{regularization} and the inductive estimates on $F_{n}$ (see Prop. \ref{andrea} eq. \eqref{birk44}) one has that
 bounds \eqref{lamerd2} hold on the coefficients $\hat{a}_{i}^{(n)},\hat{b}_{i}^{(n)},\hat{c}_{j}^{(n)},\hat{d}_{j}^{(n)}$, $i=0,1$ and $j=1,\ldots,\mathtt{N}_{n+1}$.
Hence the  smallness conditions in \eqref{birk4} follows. 

By applying Proposition \ref{birk} to the operator $\hat{L}_n$ in \eqref{corvo88}
we get a change of coordinates $\SSSS^{(n)}:=\uno+\Psi^{(n)}$ (given in \eqref{birk5})
such that 
the operator
\begin{equation}\label{strega4}
{L}^{+}_n:=(\SSSS^{(n)})^{-1}\hat{L}_{n}\SSSS^{(n)}=
\Pi_{S}^{\perp}(\oo_{n+1}\cdot\del_{\theta}+ \DD_n+\RR_n)\Pi_{S}^{\perp}
\end{equation}
of the form \eqref{corno1}
where
\begin{equation}\label{strega5}
\DD_{n}:=-\ii E\,{\rm diag}_{j\in \ZZZ_{+}\cap S^{c}}(m_{n+1}j^{2}+r_{j}^{(0)}), \qquad 
\RR_n:= -\ii E_{1}(L_{n})\del_{x}-\ii E_{0}(L_{n}),
\end{equation}
given in \eqref{birk7}. 
More precisely 
denoting $E_{i}(L_n):=E_{i}^{(n)}$ for $i=0,1$,
we set
\begin{equation}\label{lcappuccio}
E_{1}^{(n)}:= E\Pi_{S}^{\perp}\left( \begin{matrix} 0 & b^{(n,+)}_{1} \\ \bar{b}^{(n,+)}_1 & 0 \end{matrix}\right)\del_{x}\Pi_{S}^{\perp},\quad 
E_{0}^{(n)}:=E \Pi_{S}^{\perp}  \left(
 \begin{matrix} a^{(n,+)}_{0} & b^{(n,+)}_{0} \\ \bar{b}^{(n,+)}_0 & \bar{a}^{(n,+)}_{0} \end{matrix}
\right)
\Pi_{S}^{\perp}+\hat{\mathscr{K}}_{n}^{+}.
\end{equation}
By estimates \eqref{birk8} in Proposition \ref{birk}
and by Remark \eqref{opmolt},
we have
$$
|E^{(n)}_1|^{{\rm dec}}_{\vec{v},\gotp_1}+|E^{(n)}_0|^{{\rm dec}}_{\vec{v},\gotp_1}\leq C|\x|\e_0.
$$
We have that ${L}^{+}_n$ in \eqref{strega4} satisfies the hypotheses of Lemma \ref{kamreduc}.
In order to prove a bound like \eqref{cribbio4}
we fix the number $N>0$ in Lemma \ref{kamreduc} in such a way
one has
\begin{equation}\label{strega6}
N^{-\ka_4}\leq K_{n}^{-\ka_2}.
\end{equation}
Using \eqref{strega7}  may set that $N=K_{0}^{\left(\frac{3}{2}\right)^{n^{*}}}$. 
 Lemma \ref{kamreduc} is a KAM-like scheme, the point is that, if we are at step $n$ of the abstract algorithm
in Theorem \ref{thm:kambis}, we have to perform $n^{*}$ Kam steps in Lemma \ref{kamreduc}.  With this reduction procedure we approximately diagonalize $L_n$ up to a remainder which is so small that it is negligible in the construction of the approximate inverse $W_n$. 

By applying Lemma \ref{kamreduc} with $N=K_{0}^{(\frac{3}{2})^{n^*}}:=K_{n^{*}}$ to the truncated operator
\begin{equation}\label{strega8}
L_{n}^{+}:=\Pi_{S}^{\perp}(\oo_{n+1}\cdot\del_{\theta}+ \DD_n+(\Pi_{K_{n+1}}E_{1}^{(n)})D+\Pi_{K_{n+1}}(E_{0}^{(n)})
)\Pi_{S}^{\perp}
\end{equation}
we have that $ \Lambda_{n+1}^{2\g_{n}}$ of definition \eqref{corno77} coincides with $\Lambda_{K_{n^*}}^{2\g_n}$ defined in \eqref{corno7}. Hence for  $\xi\in\Lambda_{n+1}^{2\g_{n}}$  there is a map
$\Phi:=\Phi_{K_{n^{*}}}$ that satisfies \eqref{corno11} and conjugates $L_n^{+}$ to the operator
\begin{equation}
\widehat{\calL}^{+}_{n}:=\Pi_{S}^{\perp}\Big(\oo_{n+1}\cdot\del_{\theta}+
\widehat\DD_{n}+\widehat\RR_{n}
\Big)\Pi_{S}^{\perp}
\end{equation}
 where ( $\widehat{L}^{+}_{n}$ is $(\calL_{n}^{+})_{n^*}$ in the notation of Lemma \ref{kamreduc})
 \begin{equation}\label{corno99}
 \begin{aligned}
 \widehat\DD_{n}&:={\rm diag}_{\s\in \{\pm1\}, j\in\ZZZ_+}(\mu_{\s,j}^{(n)}),\\
 \widehat\RR_{n}:&=\widehat E_{1}^{(n)}D+\widehat E_{0}^{(n)},
 \end{aligned}
 \end{equation}
 given by equation \eqref{corno9}. 
 The $\mu^{(n)}_{\s,j}$ satisfy \eqref{corvo66} and \eqref{strega20} as consequence of Lemma \ref{kamreduc}.
 Moreover using estimates \eqref{corno10}, \eqref{corno11}  and the Inductive Hypothesis
 we get the bounds:
 
  \begin{equation}\label{corno100}
 \begin{aligned}
 |\widehat E^{(n)}_1|^{{\rm dec}}_{\vec{v}_1,p}+|\widehat E^{(n)}_0|^{{\rm dec}}_{\vec{v}_1,p}&\leq K_{n+1}^{\ka_5}\Big(|E^{(n)}_1|^{{\rm dec}}_{\vec{v},p}+|E^{(n)}_0|^{{\rm dec}}_{\vec{v},p}
 \Big)
 K_{n^*}^{-\ka_4}, \quad \vec{v}_1:=(\g, \Lambda_{N}^{2\g},s,a),\\
  \end{aligned}
 \end{equation}
 Moreover one has that
 \begin{equation}\label{corno111}
 |\Phi^{\pm1}-\uno|^{\rm dec}_{\vec{v}_1,p}\leq \g^{-1}\Big(|E^{(n)}_1|^{{\rm dec}}_{\vec{v},p}+|E^{(n)}_0|^{{\rm dec}}_{\vec{v},p}\Big).
 \end{equation}
In particular this means that
\begin{equation}\label{strega21}
\begin{aligned}
 |\widehat E^{(n)}_1|^{{\rm dec}}_{\vec{v}_1,\gotp_1}+|\widehat E^{(n)}_0|^{{\rm dec}}_{\vec{v}_1,\gotp_1}&
 \leq K_{n+1}^{\ka_{5}}K_{n}^{-\ka_{1}}|\x|\e_0
\end{aligned}
\end{equation}

\noindent
For $\x\in P_{n+1}^{2\g_n}$ we set
\begin{equation}\label{pollini2}
g:=W_{+}^{(n)}f:=\Phi^{-1}(\oo_{+}\cdot\del_{\theta}+\widehat \DD_{n})^{-1}\Phi f,
\end{equation}
which is well defined since
 we have that $P_{n+1}^{2\g_n}$ coincides with  
$P_{K_{n*}}^{2\g_n}$ in \eqref{primedimerda}.  
Hence, by 
Lemma \ref{invertiamoo} and estimates \eqref{corno111}, one has that $g$ in \eqref{pollini2}
satisfy bound \eqref{buoni} with $f=X^{(w,0)}$, and hence 
by the inductive hypothesis on $F_{n}$, $g_{n+1}$ satisfies bounds \eqref{lamorte}.

Secondly we have that 
$$
(\oo_{n+1}\del_{\theta}+F_n^{(w,w)}(\theta))g= f+u
$$
with
\begin{equation}\label{pollini3}
\begin{aligned}
u&:=\Big(\hat{N}^{(2)}_{n}+\hat{H}_{n}^{(\theta,0)}\cdot\del_{\theta}+(\Pi_{K_{n+1}}^{\perp}E^{(n)}_{1})D+(\Pi_{K_{n+1}}^{\perp}E_{0}^{(n)})
\Big)g
+\\
&+\Phi^{-1}\widehat{R}_{n}\Phi g
\end{aligned}
\end{equation}
We claim that
\begin{equation}\label{pollini4}
\begin{aligned}
|u|_{\vec{v},\gotp_1}&\leq \g_{n}^{-1}\e_0 K_{n}^{-\ka_1}K_{n+1}^{\mu_1}|f|_{\vec{v},\gotp_1},\\
|u|_{\vec{v},\gotp_2}&\leq \g_{n}^{-1} K^{\mu_1}\Big(|f|_{\vec{v},\gotp_2}+|f|_{\vec{v},\gotp_1}\mathtt{G}_0 K_{n}^{\ka_1+\ka_{5}}
\Big)
\end{aligned}
\end{equation}
which are the bounds \eqref{cribbio4}  for the $w-$component. The \eqref{pollini4} follow
recalling \eqref{numeretti} 
and using \eqref{buoni} to estimate $g$,
\eqref{corno111} to estimate $\Phi^{\pm 1}$, \eqref{strega21} to estimate $\widehat\RR_{n}$ in \eqref{corno99},
 \eqref{regulari66} and \eqref{passero22} to estimates the terms $\hat{N}^{(1)}_{n}$ and $\hat{H}_{n}^{(\theta,0)}$.
\end{proof}

\begin{rmk}\label{autoautoval}
Cosider $F_n$ the sequence of vector fields  given by Theorem \ref{thm:kambis} and consider the approximate eigenvalues
 $\mu_{\s,j}^{(n)}(\x)$ given by Lemma \ref{pipino}. By Remark \ref{laghetto2} and \ref{laghetto} one has
 that
 \begin{equation}\label{aut}
 \mu_{\s,j}^{(n)}(\x)=c_0 j^{2}+r_0^{j}+o(|\x|)j^{2}+o(|\x|)=\Omega_{j}^{\rm int}+o(|\x|)j^{2}+o(|\x|),
 \end{equation}
 where $\Omega^{\rm int}$ is defined in Section \ref{caspita2}.
\end{rmk}

We have the following Lemma.
\begin{lemma}\label{distanza}
For $n\geq0$ consider the operators
$E_{i}^{(n)}:=E_{i}(L_n)$, $i=0,1$, given in \eqref{strega5} and \eqref{lcappuccio}. Then, for $\x\in \calO_{n}$ of Proposition \ref{pipino},  one has that

\begin{equation}\label{stimalip}
\max_{i=1,0} |E_{i}(L_{n-1})\!-\!E_{i}(L_{n})|_{\vec{v},\gotp_{0}}
\leq \g_0\e_0 K_{n-1}^{-\ka_2+\mu_1+4},
\end{equation}

\end{lemma}

\begin{proof}

We reason as follows. 
Recalling the form of the field $F_{n}$ in \eqref{regu2222} in Proposition \ref{andrea} we have that 
the linearized operator in the $w-$direction of the field $F_{n}$ has the form
\begin{equation}\label{corvo1000}
\begin{aligned}
L_{n}:=\Pi_{S}^{\perp}\oo_{n}\cdot\del_{\theta}+\Pi_{S}^{\perp}\left(-\ii E\left(\begin{matrix}m_{n} & 0\\ 0 & m_{n} \end{matrix}\right)\del_{xx}
-\ii E\left( \begin{matrix} a_{1}^{(n)} & b_{1}^{(n)}\\ -\bar{b}^{(n)}_1 & \bar{a}^{(n)}_{1} \end{matrix}\right)\del_{x}
-\ii E \left(\begin{matrix} a_{0}^{(n)} & b_{0}^{(n)}\\ -\bar{b}^{(n)}_0 & \bar{a}^{(n)}_{0} \end{matrix}
\right)
\right)\Pi_{S}^{\perp}+{\mathscr{K}}^{(n)},
\end{aligned}
\end{equation}
where $\hat{\mathscr{K}}^{(n)}$ has the form \eqref{maiale} with coefficients $c_{j}^{(n)},d_{j}^{(n)}$, $j=1,\ldots, N$. 
By 	the discussion above we know that 
the coefficients $\hat{a}_{i}^{(n)},\hat{b}_{i}^{(n)},\hat{c}_{j}^{(n)},\hat{d}_{j}^{(n)}$, $j=1,\ldots,N$,
satisfies
\begin{equation}\label{ilbuono2}
\|a_{0}^{(n)}-\hat{a}_{0}^{(n)}\|_{\vec{v},\gotp_0}\leq \g_0\e_0 K_{n-1}^{-\ka_1+\mu_1+4}.
\end{equation}
Bound \eqref{ilbuono2}
together with 
\eqref{lamerd2}
implies that 
\begin{equation}\label{ilbuono3}
\|\hat{a}_{0}^{(n-1)}-\hat{a}_{0}^{(n)}\|_{\vec{v},\gotp_0}\leq\g_0 \e_0 K_{n-1}^{-\ka_1+\mu_1+4}.
\end{equation}
We want to estimate the difference between ${a}_{0}^{(n-1,+)} $ and ${a}_{0}^{(n,+)}$
in \eqref{lcappuccio}.
One has that the transformation $\SSSS^{(n)}$ which conjugate $\hat{L}_{n}$ in \eqref{ilbuono}
to $L_{n}^{+}$ in \eqref{strega4} is generate by the function $s^{(n)}$ given in \eqref{birk16}.
The bounds \eqref{ilbuono3} implies the same bounds on the the difference $(s^{(n)}-s^{(n-1)})$.
Hence, by triangular inequality and the the form of the coefficients ${a}_{0}^{(n-1,+)} $ and ${a}_{0}^{(n,+)}$ given in \eqref{birk12} we can conclude that 
$({a}_{0}^{(n-1,+)} -{a}_{0}^{(n,+)})$ satisfies bound \eqref{ilbuono2}. The same holds for the other coefficients. This implies the thesis.

\end{proof}

\zerarcounters
\section{Measure estimates}\label{caspita2}
In this last Section we prove that the measure of the set of ``good'' parameters is large as $\x\to0$. In particular in Section \ref{sbroo}
we have seen that Theorem \ref{teoremap} holds in the set
\begin{equation}\label{lamerd5}
\mathcal{C}_{\e}:=\bigcap_{n\geq1}\calO_{n}, \qquad \calO_n:=\Lambda^{2\g_n}_{\nu}\cap P^{2\g_{n}}_{\nu}\cap\SSSS^{2\g_n}_{\nu},
\end{equation}
with $\nu$ defined in \eqref{strega7} (see Lemma \ref{pipino}).
Before performing such measure estimates we first prove that the map which link the parameters $\x$ to the frequency $\oo(\x)$ and $\x\to \mu_{\s,j}(\x)$
is a diffeomorphism.

\subsection{The ``twist'' condition}

Recall that, by definition (see \eqref{ini1}), we can write $F=N_{0}+G$ where $N_{0}:=(\la^{(-1)}+\la^{(0)}(\x))\cdot\del_{\theta}+\Omega^{(-1)}w\del_{\theta}$
where $ \la^{(-1)}_{j}:=j^{2}, \la^{(0)}_{j}(\x):=-(\MM\x)_{j}, \; j\in S^{+}$ (see \eqref{omeghino}) and $(\Omega^{(-1)})^{\s}_{\s}=\ii\s {\rm diag} j^{2}$, $(\Omega^{(-1)})_{\s}^{-\s}=0$. Hence
by equation \eqref{ini2}  we have
\begin{equation}\label{NF11}
\begin{aligned}
\Pi_{\NN}F&=\oo(\x,\theta)\del_{\theta}+\Omega(\theta,\x)w\del_{w},\\
\oo(\x,\theta)&=\oo^{(-1)}+\oo^{(0)}(\x)+G^{(\theta,0)}(\x,\theta), \\
 \Omega(\theta,\x)&=\Omega^{(-1)}+\Omega^{(0)}(\theta,\x)=d_{w}F^{(w)}(\theta,0,0)[\cdot]
=\Omega^{(-1)}+d_{w}G^{(w)}(\theta,0,0)[\cdot]
\end{aligned}
\end{equation}
Let us study in particular the linear operator $\Omega$ on
$ \Pi_S^\perp{\bf h}_{\rm odd}^{a_0,p}$.
We have that $\Omega^{(-1)}:=-iE\del_{xx} :  \Pi_S^\perp{\bf h}_{\rm odd}^{a_0,0}\to  \Pi_S^\perp{\bf h}_{\rm odd}^{a_0,p-2}$
where $E:={\rm diag}\{1,-1\}$. Moreover one has that $\Omega^{(0)}=((\Omega^{(0)})_{\s}^{\s'})_{\s,\s'=\pm1}$ can be seen as a 2 times 2
matrix  whose components are operator on $H^{a_{0},p}$. 

We have the vector field
\begin{equation}\label{approx1}
\left\{
\begin{aligned}
&\dot\theta=\omega(\x,\theta)+\Pi_{\NN^{\perp}}G^{(\theta)}(\theta,y,w)\\
&\dot y = G^{(y)}(\theta,y,w)\\
&\dot w = \Omega(\x,\theta)w+\Pi_{\NN^{\perp}}G^{(w)}(\theta,y,w)
\end{aligned}
\right.
\end{equation}
where $G$ is small. In order to run our algorithm we need to reduce the matrix $\Omega(\omega(\x) t; \x)$. In order to do this perturbatively we need to impose second Melnikov conditions. The minimal requirement (so that the reduction algorithm runs at least at a formal level) is that the difference of the eigenvalues is not identically zero as function of $\xi$, namely 
 \begin{itemize}
\item {\em Twist}. Denote by $\mu_j(\x)$  for $j\in S^{c}$ the  {\em eigenvalue functions} of $\Omega(\theta,\x)$. For all $l,j,k,\s_1,\s_2$  such that: if $\s_1=\s_2$ then $(l,j,k)\neq (0,j,j)$  and moreover $\sum_i l_i+ \s_1=\s_2$ consider the map 
\begin{equation}\label{imu}
\xi \to  {\omega}^{(0)}(\xi) \cdot l  + \s_1\mu_j(\x) - \s_2 \mu_k(\x)
\end{equation}
where we defined $\oo(\theta,\x):={\oo}^{(0)}(\x)+\Pi_{\NN}G^{(\theta)}=\la^{(-1)}+\la^{(0)}(\x)+\Pi_{\NN}G^{(\theta)}$.
We require that  these maps 
are never  identically zero. 
\end{itemize}
 This is the reason why we needed to introduce $\la^{(0)},\Omega^{(0)}$ since clearly
$$
\la^{(-1)} \cdot l + \Omega^{(-1)}_j \pm \Omega^{(-1)}_k\equiv 0
$$ 
for infinitely many choices of $l,j,k$. 

We split:
$$
\Omega(\theta,\x)= \Omega^{\rm int}(\x) +\widetilde{\Omega}^{(0)}(\theta,\x) 
$$ 
where
\begin{equation}\label{integra2}
\Omega^{(\rm int)}(\x)=\Omega^{(-1)}+ {\rm diag}(\tm_j\cdot\xi)_{j\in S^c\cap \ZZZ_+}=  {\rm diag}(j^2+\tm_j\cdot\xi)_{j\in S^c\cap\ZZZ_+},
\end{equation}

 $$
\tm_j^i=  \frac{1}{4}\big(C_{j}^{\tv_{i}}+C_{j}^{-\tv_i}\big)= \frac{1}{2}\Big(2\ta_{1}-\ta_{2}(j ^{2}+ \tv_i^2)+\ta_{3}\tv_{i}^2
-2\ta_{6}\tv_i^{2}-\ta_{7}\tv_i^{2}
-\ta_{4}\tv_{i}^2j^{2}
-2\ta_{5}\tv_{i}^{4}j ^{2}\Big)
$$


$$
\tm_j= \frac{1}{2}\Big(2\ta_{1}-\ta_{2}(j ^{2}+ \tV^2)+(\ta_{3}-2\ta_{6}-\ta_{7})\tV^2-\ta_{4}j^2\tV^2
-2\ta_{5}j^{2}\tV^{4}\Big) \vec 1\,,
$$

where $\tV:={\rm diag}_{i}( \mathtt{v}_i)$.
Note that  $\widetilde\Omega$ is of the same order as $\Omega^{(\rm int)}$, however it turns out that for {\em generic} choices of $\ta_1,\ldots,\ta_{5},\ta_{6},\ta_{7},\ta_{8},\tv_1,\ldots,\tv_d$:
\begin{itemize}
\item $\omega^{(0)}, \Omega_{\rm int}^{(0)}(\x)$ satisfy the {\em twist conditions}, namely $\forall l,j,k$ $\s_1,\s_2=0,\pm 1$  the affine maps
$$
\xi \to  \omega^{(0)}(\xi) \cdot l  + \s_1\Omega^{(\rm int)}_j(\x) - \s_2 \Omega^{(\rm int)}_k(\x)
$$
are not identically zero (with the usual restrictions on $(l,j,k)$).
\item The twist condition above implies the corresponding twist condition for the $\mu_j$ (see \eqref{imu}).
\end{itemize}

\noindent
Now we prove that our normal form satisfies the twist condition. First we introduce the following non-resonance condition.

  \begin{defi}
 We   say that $(\ta):=(\ta_1,\ta_2,\ta_3,\ta_4,\ta_{5},\ta_{6},\ta_{7},\ta_{8})$ is
 \emph{non-resonant} if one of the following occurs: 
 \begin{enumerate}
 \item $\ta_{5}\neq0$, 
 \item $\ta_{5}=0$ and $\ta_{1}\neq0$,
 \item $\ta_{5}=\ta_{1}=0$, $-\ta_4+\ta_{8}\neq0$ and  one of the following holds:
 \begin{itemize}
 \item  $\ta_{4}=0$, or
\item  $\ta_{4}\neq0$ and $(2d-1)\ta_4-\ta_{8}\neq0$ or
\end{itemize}
\item  $\ta_{5}=\ta_{1}=-\ta_4+\ta_{8}=0$, $\ta_{3}-\ta_2-\ta_{6}-\ta_{7}\neq 0$
and one of the following holds:
\begin{itemize}
\item  $\ta_{2}=0$ and $\ta_{3}-3\ta_{6}-\ta_{7}=0$, or
\item $\ta_{2}\neq0$, $\ta_{3}-\ta_2-3\ta_{6}-\ta_{7}\neq0$, or
\item $\ta_{2}=0$, $\ta_{3}-3\ta_{6}-\ta_{7}\neq0$ and $ \ta_3-\frac{6d+1}{2d+1}\ta_{6}-\ta_{7}\neq0$, or
\item $\ta_{2}\neq0$, $\ta_{3}-\ta_2-3\ta_{6}-\ta_{7}=0$ and 
$d \ta_{2}\neq \ta_{6}$.
\end{itemize} 
 \end{enumerate}


 \end{defi}

Note that a non-resonant vector $(\ta)$ is ``generic'' in the sense of Definition \ref{defgene}.

\begin{lemma}\label{twist1}
For  all
non-resonant choices of $(\ta)$
there exists a ``generic'' choice of 
the tangential sites $S^{+}=\{\tv_{1},\ldots,\tv_{d}\}\subset \NNN$ such that the map
\begin{equation}\label{NF3}
\e^{2}\Lambda\ni\x \to \oo^{(0)}(\x)= \la^{(-1)}(\x)-\MM \xi
\end{equation}
is a affine diffeomorphism. 
\end{lemma}
\begin{proof}
Since $\omega^{(0)}$ is affine we only need to show that $\MM$ is invertible. Recalling that $\MM_{ij}=(1/4)\big(C_{\tv_i}^{\tv_j}
+C_{\tv_i}^{-\tv_j}\big)$ for $i,j=1,\ldots,d$.
It is convenient to represent 
\begin{equation}\label{integra3}
\MM= \frac14 \sum_{k=0}^3 \MM^{(2k)} 
\end{equation}
where the matrix elements $\MM^{(2k)} $ are homogeneous of degree $2k$ in the variables $\tv_1,\ldots,\tv_d$.
More precisely setting $V=$ diag$(\tv_i)$,  $A_{ij}=1$, forall $i,j$, we have
$$\MM^{(0)}= \ta_1 (4A - \uno )  \,, \quad \MM^{(6)}= -\ta_{5} \tV^2 (4A - \uno ) \tV^4
$$
$$
\MM^{(2)}:=-(\ta_{3}-\ta_2-\ta_{6}-\ta_{7})\tV^{2}+
(\ta_{3}-\ta_2-3\ta_{6}-\ta_{7})2A\tV^{2}- 2\ta_{2}\tV^2 A 
$$
$$
\MM^{(2)}_{ik}=\left\{\begin{array}{ll}  (\ta_{3}-\ta_2)\tv_i^2 -2 \ta_2 \tv_i^2 -3\ta_{6} \tv_{i}^{2}-\ta_{7}\tv_i^{2}
\; & \text{ if}\quad i=k\\
2(\ta_{3}-\ta_2)\tv_{k}^{2}
-2\ta_{2}\tv_{i}^{2} -4\ta_{6}\tv_{k}^{2}-2\ta_{7} \tv_{k}^{2}
\; &\text{ if}\quad i\neq k\end{array}\right.
$$
$$
\MM^{(4)}_{ik}=\left\{\begin{array}{ll}  (-\ta_4-\ta_{8})\tv_i^4 \; & \text{ if}\quad i=k\\
-2\ta_{4}\tv_{i}^{2}\tv^2_{k} 
\; &\text{ if}\quad i\neq k\end{array}\right.
$$

$$
\MM^{(4)}:=-(-\ta_4+\ta_{8})\tV^{4}
-2\ta_4\tV^{2}A\tV^2
$$

We now compute $P(\ta,\tv):=$det $(\MM)$ which is a  non trivial polynomial in $(\ta_1,\ldots,\ta_{5},\ta_{6},\ta_{7},\ta_{8},\tv_1,\ldots,\tv_d)$.
Indeed  $P(\ta,0)=$ det$(\MM^{(0)})=\ta_1^d (2d -1)$, so  {\em for any } $\ta$ such that $\ta_1\neq 0$ we  impose 
$P(\ta,\tv)\neq 0$ as generiticity condition on the $\tv$.

In the same way, the term of highest degree in $\tv$ is  det$(\MM^{(6)})=\ta_{5} (2d-1) \prod_i \tv_i^6$, so again {\em for any } $\ta$ such that $\ta_{5}\neq 0$ we  impose 
$P(\ta,\tv)\neq 0$ as generiticity condition on the $\tv$.

We are left with the case $\ta_1=\ta_{5}=0$.
Now the term of minimal degree is det$(\MM^{(2)})$ while term of maximal degree is det$(\MM^{(4)})$.
Now we show that for generic choices of $\ta_{i}$ then det$(\MM^{(2)})$ 
is not identically zero as function of the $\tv_{i}$. First of all we set
\begin{equation}\label{dylan3}
\la:=-(\ta_{3}-\ta_2-\ta_{6}-\ta_{7}), \quad \al:=(\ta_{3}-\ta_2-3\ta_{6}-\ta_{7})2, \quad \be:=-2\mathtt{a}_{2}.
\end{equation}
Hence we can write
\begin{equation*}
\MM^{(2)}=\la \tV^{2}+\al A\tV^{2}+\be \tV^{2}A.
%
\end{equation*}
Assume that $\la\neq0$.
Now if $\be =0$ and $\al\neq0$ then one has
$$
\MM^{(2)}:=(\uno+\frac{\al}{\la}A)\tV^{2}.
$$
The first matrix in the product is invertible if has all the eigenvalues different from zero. Hence we impose that 
\begin{equation}\label{dylan}
1+\frac{\al}{\la}d\neq0, \quad {\rm i.e.} \quad \ta_3-\ta_2-\frac{6d+1}{2d+1}\ta_{6}-\ta_{7}\neq0.
\end{equation}
If on the contrary $\al=0$ and $\be\neq0$ then 
$$
\MM^{(2)}:=\tV^{2}(\uno+\frac{\be}{\la}A),
$$
which is invertible if
\begin{equation}\label{dylan2}
1+\frac{\be}{\la}d\neq0, \quad {\rm i.e.} \quad \ta_3-(2d+1)\ta_2-\ta_{6}-\ta_{7}\neq0.
\end{equation}
%
%
%
%
Consider the case $\al\neq0$ and $\be\neq0$. Then we have 
\begin{equation*}
\MM^{(2)}= (\uno+\frac{\al}{\la} A+\frac{\be}{\la} \tV^{2}A\tV^{-2})\tV^{2}.
\end{equation*}
 The invertibility of $\MM^{(2)}$ relies on the invertibility of the matrix $\uno+R:=\uno+\frac{\al}{\la} A+\frac{\be}{\la} \tV^{2}A\tV^{-2}$. 
  We have that $R$ has at most rank 2, hence has at most  two eigenvalues different
  from zero.
  Say that $\mu_{1,2}=\mu_{1,2}(\tv_{i})$ is such eigenvalues that in principle depends on the $\tv_i$. Now one has that $\uno+R$ has $d-2$ eigenvalues equals to $1$
  and two equals to $1+ \mu_{1,2}(\tv_i)$. One must have that $1+ \mu_{1,2}(\tv_i)\neq0$. Hence if $\mu(\tv_i)$ is not a trivial polynomial in the variables $\tv_{i}$
  then one get the invertibility of $\MM^{(2)}$ as generiticity condition on $\tv_i$. Otherwise one has to exclude some values of $\frac{\al}{\la}$ and $\frac{\be}{\la}$ by imposing a generiticity condition on $\ta_{2},\ta_{3},\ta_{6},\ta_{7}$ (as done in equations \eqref{dylan} and \eqref{dylan2}) and then taking a generic choice of $\tv_i$. 
  This second option does not occur. Indeed one note that
  the vector $\vec{w}_{1}:=(1,\ldots,1)\in \RRR^{d}$ is orthogonal to the kernel of the matrix $\frac{\al}{\la}A$.  
  Moreover the vector $\vec{w}_{2}:=\vec{v}$, where $\vec{v}:=(\tv^{2}_{1},\ldots,\tv^{2}_{d})$, is orthogonal to the kernel of the matrix
  $y\tV^{2}A\tV^{-2}$. Hence the range of the matrix $R$ is generated by $\{\vec{w}_{1},\vec{w}_{2}\}$. 
  One can note that
  \begin{equation}\label{zephyr}
  \begin{aligned}
  &(\frac{\al}{\la} A+\frac{\be}{\la} \tV^{2}A\tV^{-2}) \vec{w}_{1}=\frac{\al}{\la}d\vec{w}_1+\frac{\be}{\la}C_{1} \vec{w}_{2},\\
  &(\frac{\al}{\la} A+\frac{\be}{\la} \tV^{2}A\tV^{-2}) \vec{w}_{2}=\frac{\al}{\la}C_{2}\vec{w}_{1}+\frac{\be}{\la}d \vec{w}_{2},\\
  &{\rm where}  \quad 
  C_{1}:=\sum_{i=1}^{d}\frac{1}{\tv_{i}^{2}}, \quad C_{2}:=\sum_{i=1}^{d}\tv_{i}^{2}.
  \end{aligned}
  \end{equation}
  The $2\times2$ matrix which represent the matrix $R$ has eigenvalues given by
  \begin{equation}\label{zephyr2}
  \mu_{1,2}:=\frac{1}{2\la}\left(d(\al+\be)\pm\sqrt{d^{2}(\be-\al)^{2}+4\al\be C_{1}C_{2}}\right)
  \end{equation}
  
  The dimension of the range of $R$,
  for any $\al/\la \neq0$ and $\be/\la \neq0$,
depends  only on the $\tv_{i}$ for $i=1,\ldots, d$.
  
  
The same reasoning  holds verbatim if $\ta_1=\ta_{5}=0$, $\la=0$ (see \eqref{dylan3}) but  
\begin{equation}\label{dylan4}
\la_{1}:=-(\ta_{8}-\ta_4),
\end{equation}
indeed one can write
\begin{equation*}
\MM^{(4)}:=\tV^{4}
+\frac{-2\ta_4}{\la_1}\tV^{2}A\tV^{2}.
\end{equation*}
Here, as in the case of $\MM^{(2)}$, we get some additional conditions
on $\ta_i$ :
if $\ta_{4}=0$, then $\MM^{(4)}$ is invertible, otherwise we have the invertibility of the matrix if
\begin{equation}\label{dylan5}
\begin{aligned}
\ta_4\neq0\;\;\; {\rm and } \;\;
\ta_{4}(2d-1)-\ta_{8}\neq0.
\end{aligned}
\end{equation}
Suppose that $\la=\la_1=0$
then 
$$
\MM= (\ta_{3}-\ta_2-3\ta_{6}-\ta_{7})2A\tV^{2}
-2\ta_2\tV^{2}A-2\ta_4\tV^{2}A\tV^{2},
$$
which has at most rank $2$.
 \end{proof}

\begin{lemma}\label{twist666}
For all non-resonant choices of $(\ta_1,\ta_2,\ta_3,\ta_4\ta_{5},\ta_{6},\ta_{7},\ta_{8})$ there exists a no-trivial polynomial in the $\tv_i$ 
such that for all choices of $(\tv_1,\ldots,\tv_d)$ with $\tv_i$ ``generic'' with respect to the polynomial the following holds.
For all $\ell,j,k,\s_1,\s_2$  such that: if $\s_1=\s_2$ then $(l,j,k)\neq (0,j,j)$  and moreover $\sum_i l_i+ \s_1=\s_2$ the affine map
\begin{equation}\label{affinemap}
\xi \to  \omega^{(0)}(\xi) \cdot l  + \s_1\Omega^{(\rm int)}_j(\x) - \s_2 \Omega^{(\rm int)}_k(\x)
\end{equation}
is not identically zero.
\end{lemma}
\begin{proof}
$$
\omega^{(0)}(\xi) \cdot l  + \s_1\Omega^{(\rm int)}_j(\x) - \s_2 \Omega^{(\rm int)}_k(\x)=
$$
$$\la^{(-1)}\cdot \ell  + \s_1j^2-\s_2 k^2 + \big(\MM^T\ell+ \s_1 \tm_j -\s_2 \tm_k\big)\cdot \xi$$
then if $\la^{(-1)}\cdot \ell  + \s_1j^2-\s_2 k^2=0$, and using \eqref{integra2},   we look at the vector
$$
\MM^T\ell+ \frac{1}{2} (\s_1-\s_2)\big( 2\ta_{1}+(\ta_3-\ta_2 -2\ta_{6}-\ta_{7})\tV^2\big)\vec 1 +
$$
$$
+\frac{1}{2}
(\s_1j ^{2}-\s_2k^2)\big(-\ta_2-\ta_4)\tV^2 -2\ta_{5}\tV^4)\vec 1=
  $$
$$
\Big(\MM^T +\frac{1}{2}\Big(\ta_2+\ta_4\tV^2+
2\ta_{5}\tV^4\Big)A\tV^2-\frac{1}{2}\big( 2\ta_{1}+(\ta_3-\ta_2 -2\ta_{6}-\ta_{7})\tV^2\big)A \Big) \ell 
$$
since $(\s_1j^2-\s_2 k^2) \vec 1= -   A\tV^2 \ell $  and $(\s_1-\s_2)\vec 1= -A \ell$.
Hence, by using \eqref{integra3}, 
we say that a list $(\ta,\tv)$ is acceptable if for all $\ell,j,k$ such that $\sum_i \ell_i= -\s_1+\s_2$ one has 
\begin{equation}\label{puffo}
-\frac{1}{4}\Big(\ta_1\uno+(\ta_3-\ta_2-\ta_{6}-\ta_{7})\tV^{2}+(-\ta_4+\ta_{8})\tV^{4}-\ta_{5}\tV^{6}\Big)\ell\neq0
\end{equation}
then one only needs to require that none of the $\tv_i$ satisfy  
$$
p(x):=\ta_1 +(\ta_3-\ta_2-\ta_{6}-\ta_{7})x^2+(-\ta_4+\ta_{8})x^4-\ta_{5}x^{6}=0.
$$

The hypothesis of non resonance implies that $p$ is not identically zero.
%
%
\end{proof}

\begin{rmk}\label{NF6}
Just to fix the ideas we give some examples of cubic non linearity (see \eqref{6.5}) for which the extraction of parameters give the twist condition on 
the tangential sites. The classical cubic NLS with $a_{1}=1$,
$a_{i}\equiv0$ for $i=2,\ldots,8$. 
The derivative NLS $a_{3}=1$, $a_{i}=0$ for $i=1,\ldots,8$ 
(this case has been studied in
\cite{ZGY}).
\end{rmk}

\subsection{ The estimates of ``good'' parameters}

We prove the following Proposition.
\begin{proposition}\label{misu2}
Consider the set $\mathcal{C}_\e$ defined in \eqref{lamerd5}. One has that
\begin{equation}\label{misu}
|\calO_0\backslash \mathcal{C}_{\e}|\leq \mathtt{c}, \quad  \g_0=\mathtt{c}|\x|.
\end{equation}
\end{proposition}
For simplicity we set $G_{n}^{(1)}:=\Lambda_{\nu}^{2\g_n}(n), G_{n}^{(2)}:=\calP_{\nu}^{2\g_n}(n), G_{n}^{(3)}:=\SSSS_{\nu}^{2\g_n}(n)$ (see equation \eqref{corno77}).
In order to prove \eqref{misu} we prove 
prove by induction that, for any $n\geq0$, one has
\begin{equation}\label{eq1444}
|G_{0}^{(i)}\backslash G_{1}^{(i)}|\leq C_{\star}\g_0\e^{2(d-1)}, \quad |G^{(i)}_{n}\backslash G^{(i)}_{n+1}|\leq C_{\star}\g_0 \e^{2(d-1)} K_{n}^{-1},\;\;\; n\geq1, \quad i=1,2,3.
\end{equation}
%
We follows the same strategy used in Section $6$ of \cite{FP}
and we bounds the measure only of the sets $G_{n}^{(1)}$ which is the more difficult case. The other estimates can be 
obtained in the same way.
First of all write, dropping the index $1$ and recalling the definition of $\g_n$ in \eqref{numeretti},
\begin{eqnarray}\label{eq1455}
&G_{n}\backslash G_{n+1}:=\bigcup_{\substack{\s,\s'\in\{\pm1\}, j,j'\in\ZZZ_{+} \\ l \in\ZZZ^{n}}} 
R_{l j j'}^{\s,\s'}({n})\\
R_{l j j'}^{\s,\s'}({n})&:=\left\{\la \in G_{n}:
|\ii{\oo}_{n+1}\cdot\ell+\mu^{(n)}_{\s,j}\!-\!
\mu^{(n)}_{\s',j'}|<\frac{2\g_{n}|\s j^{2}-\s' j'^{2}|}
{\langle l \rangle^{\tau}}
\right\}.\nonumber
\end{eqnarray}
%

By \eqref{ini4} we have $R_{l j j}^{\s,\s}({{n}})=\emptyset$ and moreover recalling \eqref{strega7} for $|l|\leq K_{n+n^{*}}$ one has $R_{l j j'}^{\s,\s'}(n)=\emptyset$. In the following we assume that if $\s=\s'$, then $j\neq j'$.
Important properties of the sets $R_{l j j'}^{\s,\s'}(n)$ are the following. The proofs are quite standard
and follow very closely
 Lemmata 5.2 and 5.3 in \cite{BBM1}. 

\begin{lemma}\label{lemma6.133}
For any $n\geq0$, $|\ell|\leq K_{n+n^{*}}$, one has, for $|\x|$ small enough,
\begin{equation}\label{eq1466}
R_{l j j'}^{\s,\s'}({n})\subseteq R_{l j j'}^{\s,\s'}({n-1}).
\end{equation}
Moreover, 
\begin{equation}\label{eq1477}
{\rm if} \;\;\;\;\; R_{l j j'}^{\s,\s'}\neq \emptyset, \;\;\;\;\;\; {\rm then}\;\;\;\;\; 
|\s j^{2}-\s' j'^{2}|\leq 8|\tilde{\oo}\cdot l |.
\end{equation}
\end{lemma}
\begin{proof}
We first prove the (\ref{eq1477}); note that if $(\s,j)=(\s',j')$ then it is trivially true.
If $R_{l j j'}^{\s,\s'}({n})\neq\emptyset$, then, by definition (\ref{eq1455}), there exists a $\x\in\calO_0$
such that
\begin{equation}\label{eq1588}
|\mu_{\s,j}^{(n)}-\mu_{\s',j'}^{(n)}|<2\g_{n}|\s j^{2}-\s'j'^{2}|\langle l \rangle^{-\tau}+2|{\oo}_{n+1}\cdot l |.
\end{equation}
On the other hand, for $\x$ small and since $(\s,j)\neq(\s',j')$,
\begin{equation}\label{eq1599}
|\mu_{\s,j}^{(n)}-\mu_{\s',j'}^{(n)}|
\stackrel{(\ref{strega20})}{\geq}
\frac{1}{2}(1-C|\x|)|\s j^{2}-\s'j'^{2}|-C|\x| \geq\frac{1}{3}|\s j^{2}-\s' j'^{2}|.
\end{equation}
By the (\ref{eq1588}), (\ref{eq1599}) and $\g_{n}\leq 2\g$ follows
\begin{equation}\label{eq160}
2|{\oo}_{n+1}\cdot l |\geq\left(\frac{1}{3}-\frac{4\g}{\langle\ell\rangle^{\tau}}\right)|\s j^{2}-\s'j'^{2}|
\geq \frac{1}{4}|\s j^{2}-\s' j'^{2}|,
\end{equation}
since $\g\leq \g_{0}$, by choosing $\g_{0}$ small enough. It is sufficient $\g_{0}<1/48$. Then, the (\ref{eq1477}) hold.

\noindent
In order to prove the (\ref{eq1466}) we need to understand the variation of the eigenvalues $\mu_{\s,j}^{(n)}$ with respect to $n$. In other word 
the eigenvalues of the linearized operator of the field $F_n$.
If we assume that
\begin{equation}\label{eq1677}
|(\mu_{\s,j}^{(n)}-\mu_{\s',j'}^{(n)})-(\mu_{\s,j}^{(n-1)}-\mu^{(n-1)}_{\s',j'})|
\leq C \e_0 |\s j^{2}-\s'j'^{2}|K_{n+n^{*}}^{-\ka_4},
\end{equation}
then, for $j\neq j'$, $|l|\leq K_{n+n^*}$, and $\x\in G_{n}$, we have
\begin{eqnarray}\label{eq1688}
|\ii{\oo}_{n+1}\cdot l+\mu^{(n)}_{\s,j}&-&\mu_{\s',j'}^{(n)}|
\stackrel{(\ref{eq1677})}{\geq}2\g_{n-1}|\s j^{2}-\s' j'^{2}|\langle l \rangle^{-\tau}\\
&-&C\e_0|\s j^{2}-\s'j'^{2}|K_{n+n^*}^{-\ka_4}
\geq2\g_{n}|\s j^{2}-\s'j'^{2}|\langle\ell\rangle^{-\tau},\nonumber
\end{eqnarray}
because $C\de K_{n+n^*}^{\tau-\ka_4}2^{n+1}\leq1$. We complete the proof by verifying   (\ref{eq1677}).

\noindent
We use the $({\bf S4})_{\nu}$  (for $\nu\leq n+n^*$) 
of Lemma \ref{teo:KAMKAM}
with
$\g=\g_{n-1}$ and $\g-\rho=\g_{n}$,  and with $L_{1}=L_{n-1}$, $L_{2}=L_{n}$, 
(where $L_{n-1}$ and $L_n$ are the linearized operator of the vector fields $F_{n-1}$ and $F_{n}$ respectively).
By Lemma \ref{distanza} we have that
for $|\x|$ small enough, one has
\begin{equation*}
 C N_{n+n^*}^{\tau}\max_{i=1,0} |E_{i}(L_{n-1})\!-\!E_{i}(L_{n})|_{\vec{v},\gotp_{0}}
\leq \g_{n-1}-\g_{n}=:\rho=\g2^{-n},
\end{equation*}
which implies condition \eqref{eq:4.2510}.
Hence $({\bf S4})_{\nu}$ implies that 
\begin{equation}\label{eq1622}
\Lambda_{\nu}^{\g_{n-1}}({n-1})\subseteq\Lambda_{\nu}^{\g_{n}}({n}).
\end{equation}
%
 Furthermore we also  note that,
\begin{eqnarray}\label{eq164}
G_{n}&\stackrel{(\ref{corno77})}{\subseteq}&\Lambda_{\nu}^{2\g_{n-1}}({n})
\stackrel{(\ref{eq1622})}{\subseteq}
\Lambda_{\nu}^{\g_{n}}({n+1}).
\end{eqnarray}
This means that 
$\x\in G_{n}\subset \Lambda_{\nu}^{\g_{n-1}}({n})\cap\Lambda_{\nu}^{\g_{n}}({n+1})$,
and hence, we can apply the ${\bf (S3)}_{\nu}$, with $\nu=n+1$, in Lemma \ref{teo:KAMKAM} to get
\begin{equation}\label{eq166}
|r_{\s,j}^{\nu}({n})-r_{\s,j}^{\nu-1}({n-1})|{\leq}
\e K_{n+n^*}^{-\al}.
\end{equation}
Then, by (\ref{barad22222}) and (\ref{eq166}), one has that the (\ref{eq1677}) hold and the proof of Lemma (\ref{lemma6.133}) is complete.
\end{proof}

The next Lemma is fundamental. 
A similar result can be found in \cite{FP}.
Anyway in the autonomous cases it is slightly more difficult. This is due to the fact that
if one ``move'' the parameters $\x$, then $\oo(\x)$ and $\mu_{\s,j}$ moves together. This is why one need to prove that the entire map 
in \eqref{affinemap} must have the ``twist'', and it is not enough to ask that $\x\to\oo(\x)$ is a diffeomorphism.

\begin{lemma}\label{cappericapperi}
For all $n\geq0$, one has
\begin{equation}\label{eq1499}
|R_{l j j'}^{\s,\s'}({n})|\leq C \g \e^{2(d-1)} \langle\ell\rangle^{-\tau+1}.
\end{equation}
\end{lemma}

\begin{proof}
Let us define the map $\psi :\calO_0\to \CCC$
\begin{equation}\label{cap}
\begin{aligned}
\psi(\x)&:=\ii \oo_{n+1}(\x)\cdot l+\mu_{\s,j}^{(n)}(\x)-\mu_{\s',k}^{(n)}(\x)\\
&\stackrel{(\ref{corvo66}),(\ref{strega20})}{=}\ii\Big(
\la^{(-1)}\cdot l+ \s j^{2}-\s' k^{2}+
\la^{(0)}(\x)\cdot l+ m_0(\x)(\s j^{2}-\s' k^{2})+\s r_0^{j}-\s' r_0^{k}
\Big)\\
&+\ii\Big(
(\oo_{n+1}-\tilde{\oo})\cdot l+
(m_{n+1}-m_0)(\s j^{2}-\s k^{2})+(\s (r^{(n)}_{j}-r_0^{j})-\s' (r^{(n)}_{k}-r_0^{k})
\Big),
\end{aligned}
\end{equation}
where $m_0$ and $r_0^{j}$ are defined in \eqref{barad100} and \eqref{strega10}. In other words 
the terms that are linear in $\x$ are given by ${\oo}^{(0)}(\x)=\la^{(-1)}+\la^{(0)}(\x)$, defined in \eqref{omeghino},
and $j^{2}+m_0 j^{2}+r_0^{j}=\Omega^{\rm int}_{j}$ defined in \eqref{integra2}. In order to get \eqref{cap} we need a lower bound
on the lipschitz semi- norm $|\psi|^{lip}$ (as done in \cite{FP}).
First of all assume that $\la^{(-1)}\cdot l+ \s j^{2}-\s' k^{2}=0$. Then the \eqref{cap} becomes
\begin{equation}\label{cap2}
\psi:=\ii
A l\cdot\x+
\ii\Big((m_{n+1}-m_0)(\s j^{2}-\s k^{2})+(\s (r^{(n)}_{j}-r_0^{j})-\s' (r^{(n)}_{k}-r_0^{k})
\Big),
\end{equation}
where by formula \eqref{puffo} we have set $$A=-(1/4)(\ta_1\uno+(\ta_3-\ta_2-\ta_{6}-\ta_{7})\tV^{2}+(-\ta_4+\ta_{8})\tV^{4}-\ta_{5}\tV^{6}
).$$ Hence we have
for $\x_1\neq\x_2$
\begin{equation}\label{cap3}
\frac{|\psi(\x_1)-\psi(\x_2)|}{|\x_1-\x_2|}\geq \frac{c}{2d}|l|-C\e_0 |l|\geq \frac{c'}{d}|l|,
\end{equation}
for a suitable pure constant $c'>0$.
To obtain \eqref{cap3} we used  the invertibility of the matrix $A$, equation \eqref{eq1477} and  \eqref{strega20} to estimate the Lipschitz semi-norm of the constants $(r_{j}^{(n)}-r_0^{j})(\x)$ and $(m_{n+1}-m_0)(\x)$.
This implies that, in one direction one has a good estimates of the measure in terms of $\g$, while one uses the Fubini Theorem to get
\begin{equation}\label{cap4}
|R_{l j j'}^{\s,\s'}({n})|
\stackrel{(\ref{cap3})}{\leq} \tilde{C}\e^{2(d-1)}\g \langle l\rangle^{-\tau+1},
\end{equation}
which implies \eqref{eq1499}. 
Let us now consider the case $\la^{(-1)}\cdot l+ \s j^{2}-\s' k^{2}:=Z\neq0 $. We first prove the following Lemma.
\begin{lemma}\label{cavolfiore}
Assume that 
\begin{equation}\label{cavol1}
|l|,|\s j^{2}-\s' k^{2}|\leq\frac{Z}{\sqrt{|\x|}}, 
\end{equation}
then $|\psi|\geq1/4$.
\end{lemma}
\begin{proof}
Since we have  $\la^{(-1)}\cdot l+ \s j^{2}-\s' k^{2}\neq0$ we obtain
\begin{equation}\label{cavol2}
\begin{aligned}
|\psi|&\geq| \la^{(-1)}\cdot l+ \s j^{2}-\s' k^{2}| -\Big(|l| |\oo_{n+1}-\la^{(-1)}|+|\s j^{2}-\s' k^{2}| |m_{n+1}-1|
+|r_{j}^{(n)}|+|r_{k}^{(n)}|
\Big)\\
&\geq 
\frac{1}{2}-C_1\frac{|\x|}{\sqrt{|\x|}}+C_2 |\x|\geq \frac{1}{4}
\end{aligned}
\end{equation}
\end{proof}
\noindent
Lemma \ref{cavolfiore} we have that if \eqref{cavol1} hold, then there is no small divisor, and hence $R_{l j j'}^{\s,\s'}({n})=\emptyset$.
In the last case we rewrite \eqref{cap} as
\begin{equation*}
\psi:=Z+\MM^{T}l \cdot \x+m_0(\x)(Z-\oo^{(-1)}\cdot l)+O(|\x|\de |l|)=
Z+A\x\cdot l+m_0(\x)Z+O(|\x|\de)
\end{equation*}
hence one has
\begin{equation}\label{cavol4}
\begin{aligned}
\frac{|\psi(\x_1)-\psi(\x_2)|}{|\x_1-\x_2|}\stackrel{(\ref{barad100})}{\geq} c |l|  -Z C
\end{aligned}
\end{equation}
for some suitable constant $c,C>0$. Now we use that $|l|\geq Z/\sqrt{\x}$ to conclude, for $\x$ small, that one has
\begin{equation}\label{cavol5}
|\psi|^{lip}\geq c'|l|.
\end{equation}
Reasoning as in \eqref{cap4}, we have that \eqref{cavol5} implies the \eqref{eq1499}.
\end{proof}

The previous results implies that one has
\begin{equation}\label{trieste}
|\calO_0\backslash \mathcal{C}_{\e}|\leq C\g \e^{2(d-1)}\leq C\e^{2d}{\cc},
\end{equation}
 where we have used  the definition of $\g$ in \eqref{ini4} and \eqref{sottosottosotto}. In particular one gets that the relative measure of $\e^{-2d}|\calO_0\backslash\mathcal{C}_\e|\leq O({\cc})$. 
 This implies that the relative measure of the cantor set $\mathcal{C}_\e$ is positive if ${\cc}$ is small.

\subsection{Proof of Theorem \ref{teoremap}}
Consider the vector field $F$ in \eqref{system}. By Lemma 
\ref{INIZIAMO}, the choices of parameters in \eqref{parametri}
and Lemma \ref{smallcondi}, we have that $F$ satisfies all the hypotheses of Theorem \ref{thm:kambis}.
Hence in the set $\calO_{\infty}$ given by Theorem \ref{thm:kambis}
the result of Theorem \eqref{teoremap} holds. It remains to prove that $\calO_{\infty}$
satisfies the measure estimate in \eqref{asyaut}.

Proposition \ref{pipino} guarantees that the set $\mathcal{C}_{\e}$ in \eqref{lamerd5} is contained in $\calO_{\infty}$. We choose  $\mathcal{C}_{\e}$ as the set on which \ref{teoremap} holds.
In particular Proposition \ref{misu2} implies that $\mathcal{C}_{\e}$ satisfies \eqref{asyaut}.
This concludes the proof.

\subsection{Proof of Theorem \ref{teoremap2}}

Concerning Theorem \ref{teoremap2} in which the nonlinearity $\ff$ is merely differentiable, we just give
a sketch
of the proof since it is very similar to the one of Theorem \ref{teoremap}.

One  can repeat  word by word the arguments of Sections \ref{weakuffa} and \ref{action}. One gets that  
the vector field
in \eqref{system} is defined in the domain \eqref{domain}
with $s_0=a_0\equiv 0$. This implies that the norm $\|\cdot\|_{H^{p}(\TTT^d_s\times \TTT_a)}$ is the 
Sobolev norm $\|\cdot\|_{H^{p}(\TTT^d\times \TTT)}\sim \|\cdot\|_{0,0,p}$ see Remark \ref{orolo} and  \eqref{totalnorm}.
Again, by Lemma \ref{INIZIAMO}, we have that $F$ satisfies all the hypotheses of Theorem \ref{thm:kambis}
which implies that 
 in the set $\calO_{\infty}$ given by Theorem \ref{thm:kambis}
the result of Theorem \eqref{teoremap2} holds. 
We now  give a sketch of the proof that
 $\calO_{\infty}$ satisfies \eqref{asyaut1000}. The reasoning we follow is very similar to  
 the one used to prove Theorem \ref{teoremap}.
The main difference is that we set  $\calL_{n}=\uno$ for $n\geq0$, recall that  $\calL_{n}$
 are the \emph{compatible} changes
 of variables introduced in Definition \ref{compa}. This is due to the fact that the diffeomorphisms of the torus (form which the $\calL_n$ are chosen in the analytic case)  are not {\em close to identity} in Sobolev class, i.e. we do not have the second formula in \eqref{diffeo4}.

One can show by induction that linearized operator of the vector field $F_{n}$, for $n\geq0$, 
 has the form in \eqref{regu22}-\eqref{barad22}, with
  $\de_{\gotp_1}^{(1)},\de_{\gotp_1}^{(2)}\sim O(|\x|)$ (see equation \eqref{corvo}).
 Again  recall that in the analytic case we choose the map $\calL_{n}$
so that the size of $\de_{\gotp_1}^{(1)},\de_{\gotp_1}^{(2)}$ decreases as $n$ go to infinity.

We claim that   Proposition \ref{pipino}  holds also when  $\de_{\gotp_1}^{(1)},\de_{\gotp_1}^{(2)}\to 0$. 
Indeed such condition is only used in order to prove that the sequence of analiticity radiuses $a_n$ does not go to zero.
Indeed the proof of Proposition \ref{pipino} relies on the existece of changes of variables which approximately diagonalize the linearized vector field. 
Such changes of variables are defined in Sections \ref{settesette} and \ref{sec7aut} and work also in Sobolev class.

As shown in \cite{FP} (and \cite{BBM},\cite{BBM1} etc..)  in order to prove Proposition \ref{pipino} one does not need to apply the diagonalizing changes of variables to the  vector field but only to know that they exist and satisfy certain estimates. This is proved in Sections \ref{settesette} and \ref{sec7aut} so Proposition \ref{pipino} follows.

 Proposition \ref{pipino} ensures that the set  $\mathcal C_\e$ in \eqref{lamerd5} is contained in $\calO_{(\infty)}$. Then  one uses Proposition \ref{misu2} in order to ensure the measure estimates.

\bibliography{bibliografiaNLS2.bib}

\end{document}